%% file: qle_continuity_final.tex
\date{} 
\title{Liouville quantum gravity and the Brownian map II:\\
 geodesics and continuity of the embedding}
\author{Jason Miller and Scott Sheffield}
\documentclass[12pt,naturalnames]{article}
\usepackage{amsmath}
\usepackage{amssymb}
\usepackage{amsthm}
\usepackage{amsfonts}
\usepackage{graphicx}
\usepackage{tabularx}
\usepackage{subfigure}
\usepackage{enumerate}	
\usepackage{comment}
\usepackage{mathrsfs}

\pdfoutput=1

\def\@rst #1 #2other{#1}
\newcommand\MR[1]{\relax\ifhmode\unskip\spacefactor3000 \space\fi
  \MRhref{\expandafter\@rst #1 other}{#1}}
\newcommand{\MRhref}[2]{\href{http://www.ams.org/mathscinet-getitem?mr=#1}{MR#2}}

\usepackage[usenames,dvipsnames]{xcolor}
\usepackage[pdftitle={Liouville quantum gravity and the Brownian map II: geodesics and continuity of the embedding}, pdfauthor={Jason Miller and Scott Sheffield}, colorlinks=true,linkcolor=NavyBlue,urlcolor=RoyalBlue,citecolor=PineGreen,bookmarks=true,bookmarksopen=true,bookmarksopenlevel=2,unicode=true,linktocpage]{hyperref}

\usepackage[margin=1.19in]{geometry}

\newcommand{\CF}{{\mathcal F}}
\newcommand{\CS}{{\mathcal S}}

\newcommand{\supp}{\mathrm{supp}}

\renewcommand{\dh}{d_{\mathrm H}}

\allowdisplaybreaks

\numberwithin{equation}{section}
\numberwithin{figure}{section}

\newtheorem{theorem}{Theorem}
\numberwithin{theorem}{section}
\newtheorem{corollary}[theorem]{Corollary}
\newtheorem{lemma}[theorem]{Lemma}
\newtheorem{proposition}[theorem]{Proposition}

\newtheorem{problem}[theorem]{Problem}

\theoremstyle{remark}
\theoremstyle{remark}\newtheorem{remark}[theorem]{Remark}

\newcommand{\R}{\mathbf{R}}
\renewcommand{\C}{\mathbf{C}}
\newcommand{\D}{\mathbf{D}}
\newcommand{\Z}{\mathbf{Z}}
\newcommand{\T}{\mathbf{T}}
\newcommand{\Q}{{\mathbf Q}}
\newcommand{\N}{\mathbf{N}}

\newcommand{\HH}{\mathbf{H}}
\newcommand{\h}{\HH}

\definecolor{purple}{rgb}{0.7,0,0.7}
\definecolor{gray}{rgb}{0.6,0.6,0.6}
\definecolor{dgreen}{rgb}{0.0,0.4,0.0}
\definecolor{dblue}{rgb}{0.0,0.0,0.5}

\def\gmsspace{{\mathcal M}_{\mathrm{SPH}}}

\def\mmspace{{\mathcal M}}
\def\mmsigma{{\mathcal F}}

\newcommand{\mustwo}{\mu_{\mathrm{SPH}}^2}

\newcommand{\fb}[2]{B^\bullet(#1,#2)}

\def\gmsspace{{\mathcal M}_{\mathrm{SPH}}}

\def\mmspace{{\mathcal M}}
\def\mmsigma{{\mathcal F}}

\newcommand{\re}{{\mathrm {Re}}}

\newcommand{\Fh}{{\mathfrak h}}
\newcommand{\CC}{{\mathcal C}}
\newcommand{\ol}{\overline}
\newcommand{\ul}{\underline}
\newcommand{\wh}{\widehat}
\newcommand{\wt}{\widetilde}

\newcommand{\qlegrowth}{\Gamma}

\newcommand{\CA}{{\mathcal A}}

\newcommand{\CN}{{\mathcal N}}
\newcommand{\CX}{{\mathcal X}}
\newcommand{\CB}{{\mathcal B}}
\newcommand{\CD}{{\mathcal D}}

\newcommand{\CT}{{\mathcal T}}

\newcommand{\s}{{\mathbf S}}
\newcommand{\giv}{\,|\,}
\newcommand{\bes}{\mathrm{BES}}

\newcommand{\one}{{\mathbf 1}}

\def\diam{\mathop{\mathrm{diam}}}

\def\dist{\mathop{\mathrm{dist}}}

\newcommand{\SLE}{{\rm SLE}}
\newcommand{\QLE}{{\rm QLE}}
\newcommand{\CLE}{{\rm CLE}}
\newcommand{\CQ}{{\mathcal Q}}
\newcommand{\CH}{{\mathcal H}}
\newcommand{\CM}{{\mathcal M}}
\newcommand{\CU}{{\mathcal U}}

\newcommand{\CZ}{{\mathcal Z}}

\newcommand{\strip}{{\mathscr{S}}}
\newcommand{\cyl}{\mathscr{C}}
\newcommand{\ppp}{p.p.p.}

\newcommand{\free}{{\mathrm F}}
\newcommand{\dirichlet}{{\mathrm D}}

\newcommand{\p}{{\mathbf P}}
\newcommand{\pr}[1]{\p\!\left[#1\right]}
\newcommand{\ex}[1]{\E\!\left[#1\right]}

\newcommand{\prstart}[2]{\p^{#1}\!\left[#2\right]}

\def\Ito/{It\^o}

\def \P {{\bf P}}

\def \p {{\P}}
\def \E {{\bf E}}

\newcommand{\qdist}{d_\CQ}
\newcommand{\oqdist}{\ol{d}_\CQ}
\newcommand{\qlength}{\ell_\CQ}

\newcommand{\ball}[2]{B(#1,#2)}
\newcommand{\qball}[2]{B_\CQ(#1,#2)}

\newcommand{\qhull}[2]{B_\CQ^\bullet(#1,#2)}

\newcommand{\SM}{\mathsf M}

\newcommand{\MstwoW}{\SM_{\mathrm{SPH,W}}^2}
\newcommand{\MstwoD}{\SM_{\mathrm{SPH,D}}^2}

\newcommand{\Mstwo}{\SM_{\mathrm{SPH}}^2}

\newcommand{\Mdonel}{\SM_{\mathrm{DISK}}^{1, L}}
\newcommand{\Mdl}{\SM_{\mathrm{DISK}}^L}

\newcommand{\conelaw}{{\mathsf m}}
\newcommand{\qconeUnex}{{\mathsf m}}

\newcommand{\cadlag}{{c\`adl\`ag}}

\newcommand{\LN}{\mathrm{LN}}
\newcommand{\LQG}{\mathrm{LQG}}

\newcommand{\treeequivspace}{\CA}

\begin{document}

\maketitle

\begin{abstract}
We endow the $\sqrt{8/3}$-Liouville quantum gravity sphere with a metric space structure and show that the resulting metric measure space agrees in law with the Brownian map.  Recall that a Liouville quantum gravity sphere is \emph{a priori} naturally parameterized by the Euclidean sphere~$\s^2$.  Previous work in this series used quantum Loewner evolution ($\QLE$) to construct a metric~$\qdist$ on a countable dense subset of~$\s^2$.  Here we show that $\qdist$ a.s.\ extends uniquely and continuously to a metric $\oqdist$ on all of $\s^2$. Letting $d$ denote the Euclidean metric on $\s^2$, we show that the identity map between $(\s^2, d)$ and $(\s^2, \oqdist)$ is a.s.\ H\"older continuous in both directions. We establish several other properties of $(\s^2, \oqdist)$, culminating in the fact that (as a random metric measure space) it agrees in law with the Brownian map. We establish analogous results for the Brownian disk and plane. Our proofs involve new estimates on the size and shape of $\QLE$ balls and related quantum surfaces, as well as a careful analysis of $(\s^2, \oqdist)$ geodesics.
\end{abstract}

\newpage
\newgeometry{margin=1in}
{\small \tableofcontents}
\newpage

\restoregeometry

\parindent 0 pt
\setlength{\parskip}{0.25cm plus1mm minus1mm}

\medbreak{\noindent\bf Acknowledgements.} \input{acknowledgements.tex}

We thank two anonymous referees for many helpful comments which have led to substantial improvements in the exposition.

\input{support_acknowledgements.tex}

\section{Introduction}
\label{sec::intro}

\subsection{Overview}
\label{sec::overview}
This article is the second in a three part series that proves the equivalence of two fundamental and well studied objects: the $\sqrt{8/3}$-Liouville quantum gravity (LQG) sphere and the Brownian map (TBM). Both of these objects can be understood as random measure-endowed surfaces. However, an instance $\CS$ of the $\sqrt{8/3}$-LQG sphere comes with a conformal structure, which means that it can be parameterized by the Euclidean sphere $\s^2$ in a canonical way (up to M\"obius transformation), and an instance of TBM comes with a metric space structure. The problem is to endow each object with the \emph{other's} structure in a natural way, and to show that once this is accomplished the two objects agree in law. Although they are part of the same series, the three articles are extremely different from one another in terms of what they accomplish and the methods they use. To briefly summarize the current series of articles:

\begin{enumerate}
\item The first article \cite{qlebm} used a ``quantum natural time'' form of the so-called quantum Loewner evolution ($\QLE$), as introduced in \cite{ms2013qle}, to define a distance $\qdist$ on a countable, dense collection of points $(x_n)$ chosen as i.i.d.\ samples from the area measure that lives on the instance $\CS$ of a $\sqrt{8/3}$-LQG sphere.  Moreover, it was shown that for any $x$ and $y$ sampled from the area measure on $\CS$, the value $\qdist(x,y)$ is a.s.\ determined by $\CS$, $x$, and $y$. This implies in particular that the distance function $\qdist$, as defined on $(x_n)$, is a.s.\ determined by $\CS$ and the sequence $(x_n)$, so that there is no additional randomness required to define $\qdist$.
\item The current article shows that there is a.s.\ a unique continuous extension $\oqdist$ of $\qdist$ to all of $\CS$, and that the pair $(\CS,  \oqdist)$, interpreted as a random metric measure space, agrees in law with TBM.  Moreover, $\oqdist$ is a.s.\ determined by $\CS$. Thus a $\sqrt{8/3}$-LQG sphere has a \emph{canonical metric space structure} that effectively makes an instance of the $\sqrt{8/3}$-LQG sphere into an instance of TBM. This statement, which appears as Theorem~\ref{thm::tbm_lqg} below, is the first major equivalence theorem for TBM and the $\sqrt{8/3}$-LQG sphere.
\item  The third article \cite{qle_determined} will show that it is a.s.\ possible to recover $\CS$ when one is given just the metric measure space structure of the corresponding instance of TBM. In other words, the map (established in the current article) from $\sqrt{8/3}$-LQG sphere instances to instances of TBM is a.e.\ invertible --- which means that an instance of TBM can a.s.\ be embedded in the sphere in a canonical way (up to M\"obius transformation) --- i.e., an instance of TBM has a \emph{canonical conformal structure}. In particular, this allows us to define Brownian motion on Brownian map surfaces, as well as various forms of $\SLE$ and $\CLE$.
\end{enumerate}

Thanks to the results in these three papers, every theorem about TBM can be understood as a theorem about $\sqrt{8/3}$-LQG, and vice versa.

But let us focus on the matter at hand.  Assume that we are given an instance $\CS$ of the $\sqrt{8/3}$-LQG sphere, endowed with the metric $\qdist$ on a countable dense set $(x_n)$. How shall we go about extending $\qdist$ to $\oqdist$? 

By way of analogy, let us recall that in an introductory probability class one often constructs Brownian motion by \emph{first} defining its restriction to the dyadic rationals, and \emph{second} showing (via the so-called Kolmogorov-\u Centsov theorem \cite{KS98,RY04}) that this restriction is a.s.\ a H\"older continuous function \emph{on the dyadic rationals}, and hence a.s.\ extends uniquely to a H\"older continuous function on all of $\R_+$. The work in \cite{qlebm} is analogous to the first step in that construction (it constructs $\qdist$ on a countable dense set), and Sections~\ref{sec::boundary_area}, \ref{sec::euclidean_inner_outer_radius_bounds}, and~\ref{sec::continuity_to_bm} of the current article are analogous to the second step. These sections derive H\"older continuity estimates that in particular imply that $\qdist$ can a.s.\ be continuously extended to all of $\s^2$.

Precisely, these sections will show that if $(x_n)$ are interpreted as points in $\s^2$ (which parameterizes $\CS$), then for some fixed $\alpha,\beta >0$ it is a.s.\ the case that, for some (possibly random) $C_1, C_2 > 0$,
\begin{equation}
\label{eqn::holderbounds}
C_1 d(x_i,x_j)^\alpha \leq \qdist(x_i, x_j) \leq C_2 d(x_i, x_j)^\beta
\end{equation}
where $d$ is the Euclidean metric on $\s^2$.  This will immediately imply that $\qdist$ can be uniquely extended to a continuous function $\oqdist : \s^2 \times \s^2 \to \R$ that satisfies the same bounds, i.e.,
\begin{equation}
\label{eqn::holderbounds2}
C_1 d(x_i,x_j)^\alpha \leq \oqdist(x_i, x_j) \leq C_2 d(x_i, x_j)^\beta,
\end{equation}
and is also a metric on $\s^2$. Another way to express the existence of $C_1$ and $C_2$ for which~\eqref{eqn::holderbounds2} holds is to say that the identity map between $(\s^2, d)$ and $(\s^2, \oqdist)$ is a.s.\ H\"older continuous (with some deterministic exponent) in both directions.  We will also show that the metric $\oqdist$ is a.s.\ geodesic, i.e.\ that it is a.s.\ the case that every pair of points $x,y$ can be connected by a path whose length with respect to $\oqdist$ is equal to $\oqdist(x,y)$.

Once we have established this, Sections~\ref{sec::sle_6_moment_bounds}, \ref{sec::reverse_explorations}, and~\ref{sec::geodesics_and_levy_net} will show that this geodesic metric space agrees in law with TBM.  The proof makes use of several basic results about LQG spheres derived in \cite{quantum_spheres}, along with several properties that follow from the manner in which~$\qdist$ was constructed in \cite{qlebm}. A fundamental part of the argument is to show that certain paths that \emph{seem} like they should be geodesics on the LQG-sphere side actually \emph{are} geodesics w.r.t.\ $\oqdist$, which will be done by studying a few approximations to these geodesics. We will ultimately conclude that, as a random metric measure space, $(\CS, \oqdist)$ satisfies the properties that were shown in \cite{map_making} to uniquely characterize TBM.

We remark that the results of the current series of articles build on a large volume of prior work by the authors and others on imaginary geometry \cite{MS_IMAG,MS_IMAG2,MS_IMAG3,MS_IMAG4}, conformal welding  \cite{SHE_WELD}, conformal loop ensembles \cite{SHE_CLE,SHE_WER_CLE}, and the mating of trees in infinite and finite volume settings \cite{dms2014mating, quantum_spheres}, as well as the above mentioned works on quantum Loewner evolution \cite{ms2013qle} and TBM \cite{map_making}. We also cite foundational works by many other authors on Liouville quantum gravity, Schramm-Loewner evolution, L\'evy trees, TBM, continuous state branching processes, and other subjects. 
There has been a steady accumulation of theory in this field over the past few decades, and we hope that the proof of the equivalence of TBM and $\sqrt{8/3}$-LQG will be seen as a significant milestone on this continuing journey.

\subsection{Main results}
\label{subsec::main_results}

In this subsection, we state the results summarized in Section~\ref{sec::overview} more formally as a series of theorems.
In \cite{qlebm}, it was shown that if~$\CS$ is a unit area $\sqrt{8/3}$-LQG sphere \cite{dms2014mating,quantum_spheres} and $(x_n)$ is an i.i.d.\ sequence chosen from the quantum measure on~$\CS$ then a variant of the $\QLE(8/3,0)$ processes introduced in \cite{ms2013qle} induces a metric space structure $\qdist$ on $(x_n)$ which is a.s.\ determined by~$\CS$.  Our first main result is that the map $(x_i,x_j) \mapsto \qdist(x_i,x_j)$ a.s.\ extends to a function $\oqdist$ on all of $\s^2 \times \s^2$ such that $(x,y) \mapsto \oqdist(x,y)$ is H\"older continuous on $\s^2 \times \s^2$.

\begin{theorem}
\label{thm::continuity}
Suppose that $\CS = (\s^2,h)$ is a unit area $\sqrt{8/3}$-LQG sphere, $(x_n)$ is an i.i.d.\ sequence chosen from the quantum measure on $\CS$, and $\qdist$ is the associated $\QLE(8/3,0)$ metric on $(x_n)$.  Then $(x_i,x_j) \mapsto \qdist(x_i,x_j)$ is a.s.\ H\"older continuous with respect to the Euclidean metric $d$ on $\s^2$.  In particular, $\qdist$ uniquely extends to a H\"older continuous function $\oqdist \colon \s^2 \times \s^2 \to \R_+$ (with deterministic H\"older exponent).  Finally, $\oqdist$ is a.s.\ determined by $\CS$.
\end{theorem}

Our next main result states that $\oqdist$ induces a metric on $\s^2$ which is isometric to the metric space completion of $\qdist$, and provides some relevant H\"older continuity.

\begin{theorem}
\label{thm::metric_completion}
Suppose that $\CS = (\s^2,h)$ is a unit area $\sqrt{8/3}$-LQG sphere and that $\oqdist$ is as in Theorem~\ref{thm::continuity}.  Then $\oqdist$ defines a metric on $\s^2$ which is a.s.\ isometric to the metric space completion of $\qdist$.  Moreover, the identity map from $(\s^2,d)$ to $(\s^2,\oqdist)$ is a.s.\ H\"older continuous in both directions (with deterministic H\"older exponent) where $d$ denotes the Euclidean metric on $\s^2$.
\end{theorem}

As we mentioned in the statements, the H\"older exponents in Theorem~\ref{thm::continuity} and Theorem~\ref{thm::metric_completion} are deterministic but are not optimal.  The optimal H\"older exponents were computed recently in \cite{dfgps2019weak}.

Recall that a metric space $(M,d)$ is said to be \emph{geodesic} if for all $x,y \in M$ there exists a path $\gamma_{x,y}$ whose length is equal to $d(x,y)$.  Our next main result is that the metric space $\oqdist$ is a.s.\ geodesic.

\begin{theorem}
\label{thm::geodesic_metric_space}
Suppose that $\CS = (\s^2,h)$ is a unit area $\sqrt{8/3}$-LQG sphere and that $\oqdist$ is as in Theorem~\ref{thm::continuity}.  The metric space $\oqdist$ is a.s.\ geodesic.  Moreover, it is a.s.\ the case that for all $x, y \in \s^2$, each geodesic path $\gamma_{x,y}$, viewed as a map from a real time interval to $(\s^2,d)$, is H\"older continuous, where $d$ denotes the Euclidean metric on $\s^2$.
\end{theorem}

Combining Theorem~\ref{thm::metric_completion} and Theorem~\ref{thm::geodesic_metric_space} with the axiomatic characterization for TBM given in \cite{map_making} and the results in the first paper of this series \cite{qlebm}, as well as some additional work carried out in the present article, we will find that the law of the metric space with metric $\oqdist$ is indeed equivalent to the law of TBM.

\begin{theorem}
\label{thm::tbm_lqg}
Suppose that $\CS = (\s^2,h)$ is a unit area $\sqrt{8/3}$-LQG sphere and that $\oqdist$ is as in Theorem~\ref{thm::continuity}.  Then the law of the metric measure space $(\s^2, \oqdist,\mu_h)$ is the same as that of the unit area Brownian map.
\end{theorem}

Theorem~\ref{thm::tbm_lqg} implies that there exists a coupling of the law of a $\sqrt{8/3}$-LQG unit area sphere~$\CS$ and an instance $(M,d,\nu)$ of TBM such that the metric measure space $(\s^2,\oqdist,\mu_h)$ associated with $\CS$ is a.s.\ isometric to $(M,d,\nu)$.  Moreover, by the construction of $\qdist$ given in \cite{qlebm} we have that $\oqdist$ and hence $(M,d,\nu)$ is a.s.\ determined by~$\CS$.  That is, the metric measure space structure $(M,d,\nu)$ of TBM is a measurable function of~$\CS$.  The converse is the main result of the subsequent work in this series \cite{qle_determined}.  In other words, it will be shown in \cite{qle_determined} that TBM a.s.\ determines its embedding into $\sqrt{8/3}$-LQG via $\QLE(8/3,0)$.

We can extract from Theorem~\ref{thm::tbm_lqg} the equivalence of the $\QLE(8/3,0)$ metric on a unit boundary length $\sqrt{8/3}$-quantum disk \cite{dms2014mating} and the random metric disk with boundary called the \emph{Brownian disk}.  The Brownian disk is  defined in different ways in \cite{bettinelli_miermont} and \cite{map_making} and is further explored in \cite{alg}.  The equivalence of the Brownian disk definitions in \cite{bettinelli_miermont} and \cite{map_making} was proved by Le Gall in \cite{lg2019disksnake}; and will be approached from another angle in the forthcoming work \cite{emmanuel_miermont_disks}.  We can similarly extract from  Theorem~\ref{thm::tbm_lqg}  the equivalence of the $\QLE(8/3,0)$ metric on a $\sqrt{8/3}$-quantum cone \cite{SHE_WELD,dms2014mating} and the Brownian plane \cite{cl2012brownianplane}.  We state this result as the following corollary.

\begin{corollary}
\label{cor::disk_plane_equivalences}
\begin{enumerate}[(i)]
\item\label{it::disk_equivalence} Suppose that $\CD = (\D,h)$ is a unit boundary length $\sqrt{8/3}$-LQG disk.  Then the law of the metric measure space $(\D, \oqdist,\mu_h)$ is the same as that of the unit boundary length Brownian disk.  Moreover, the identity map from $(\D,d)$ to $(\D,\oqdist)$ is a.s.\ locally H\"older continuous (i.e., H\"older continuous on compact sets) in both directions where $d$ denotes the Euclidean metric on $\D$.  Moreover, the identity map extends to a homeomorphism of $\ol{\D}$.
\item\label{it::plane_equivalence} Suppose that $\CC = (\C,h,0,\infty)$ is a $\sqrt{8/3}$-quantum cone.  Then the law of the metric measure space $(\C, \oqdist,\mu_h)$ is the same as that of the Brownian plane.  Moreover, the identity map from  $(\C,d)$ to $(\C,\oqdist)$ is a.s.\ locally H\"older continuous in both directions where $d$ denotes the Euclidean metric on $\C$.
\end{enumerate}
In both cases, $\oqdist$ is a.s.\ determined by the underlying quantum surface.
\end{corollary}

We emphasize that in Part~\eqref{it::disk_equivalence} of Corollary~\ref{cor::disk_plane_equivalences}, we have not proved that the identity map from $(\D,d)$ to $(\D, \oqdist)$ extends to be a bi-H\"older continuous homeomorphism from $(\ol{\D},d)$ to $(\ol{\D},\oqdist)$.  Rather, the statement is that the identity map is H\"older continuous on compact subsets of $\D$ and is a homeomorphism from $(\ol{\D},d)$ to $(\ol{\D},\oqdist)$.

Part~\eqref{it::disk_equivalence} of Corollary~\ref{cor::disk_plane_equivalences} follows from Theorem~\ref{thm::tbm_lqg} because both a unit boundary length quantum disk and the Brownian disk can be realized as the complement of the filled metric ball.  That is, if $(\CS,x,y)$ denotes a doubly-marked instance of TBM (resp.\ $\sqrt{8/3}$-LQG surface) (with associated metric $\oqdist$) then for each $r > 0$, on the event $\oqdist(x,y) > r$, the law of the $y$-containing component of the complement of the ball centered at $x$ of radius $r$ conditioned on its boundary length is that of a Brownian disk (resp.\ quantum disk), weighted by its area.  Indeed, this follows in the case of the Brownian disk from its construction given in \cite{map_making} (see also \cite[Proposition~2.17]{map_making} which implies that the filled metric ball is a measurable function of the metric measure space structure of $(\CS,x,y)$ and therefore so is its complement) and this follows in the case of a $\sqrt{8/3}$-LQG sphere from the basic properties of $\QLE(8/3,0)$ established in \cite{qlebm}.  Moreover, Proposition~\ref{prop::internal_metric} implies that the internal metric $\oqdist^U$ associated with any fixed domain $U$ is a.s.\ determined by the restriction $h|_U$ of $h$ to $U$.  Note also that $\oqdist^U = \inf_{V \subseteq U} \oqdist^V$ where the infimum is over all domains $V \subseteq U$.  In fact, if $(U_n)$ is a sequence of domains so that $U_n \subseteq U_{n+1}$ and $\cup_n U_n = U$ then we have that $\oqdist^U = \inf_n \oqdist^{U_n}$.  By applying these facts to the countable collection of domains which consist of finite, connected unions of Euclidean balls with rational centers and rational radii, we thus see that the internal metric associated with the filled metric ball complement is a.s.\ determined by the field restricted to the filled metric ball complement.  Therefore the metric in this case is determined by the underlying quantum surface.  The local bi-H\"older continuity of the identity map immediately follows from the corresponding statement in the case of the $\sqrt{8/3}$-LQG sphere.

Let us now explain why the identity map is a homeomorphism from $(\ol{\D},d)$ to $(\ol{\D},\oqdist)$.  In the coupling of the area-weighted quantum disk $(\ol{\D},\oqdist)$ with an instance of the $\sqrt{8/3}$-LQG sphere as a filled metric ball complement explained above, let $\varphi \colon \ol{\D} \to \CS$ be the embedding map.  If we parameterize $\CS$ by $\s^2$ and (by an abuse of notation) let also $d$ denote the Euclidean metric on $\s^2$, then we know that $\varphi$ extends to be H\"older continuous up to $\partial \D$ and is a homeomorphism (see \cite[Proposition~5.12]{qlebm}), using the Euclidean metric on both sides.  By definition, we have for all $x,y \in \ol{\D}$ that $\oqdist(x,y)$ is equal to the distance between $\varphi(x)$ and $\varphi(y)$ computed using the interior-internal metric associated with the $\sqrt{8/3}$-LQG metric on $\CS$.  This, in turn, is at least the distance between $\varphi(x)$ and $\varphi(y)$ using the overall $\sqrt{8/3}$-LQG metric on $\CS$ (i.e., not using the interior-internal metric anymore).  By Theorem~\ref{thm::metric_completion}, this is in turn at least $c_0 d(\varphi(x),\varphi(y))^\alpha$ where $\alpha > 0$ is deterministic and $c_0 > 0$ is random.  Since $\varphi$ is a homeomorphism, we have that $d(\varphi(x),\varphi(y))$ is bounded from below by the inverse of the modulus of continuity of $\varphi^{-1}$ applied to $d(x,y)$.  This proves that the map from $(\ol{\D},\oqdist)$ to $(\ol{\D},d)$ is continuous.  Since both spaces are compact (as the Brownian disk is compact), it follows that the identity map is continuous in the opposite direction and is therefore a homeomorphism.

Part~\eqref{it::plane_equivalence} of Corollary~\ref{cor::disk_plane_equivalences} follows from Theorem~\ref{thm::tbm_lqg} because a $\sqrt{8/3}$-quantum cone is given by the local limit of a $\sqrt{8/3}$-LQG sphere near a quantum typical point \cite[Propositions~4.13, A.13]{dms2014mating} and likewise the Brownian plane is given by the local limit of TBM near a typical point sampled from TBM's intrinsic area measure \cite[Theorem~1]{cl2012brownianplane}.

It will also be shown in \cite{qle_determined} that the unit boundary length Brownian disk (resp.\ Brownian plane) a.s.\ determines its embedding into the corresponding $\sqrt{8/3}$-LQG surface via $\QLE(8/3,0)$.

The proofs of Theorems~\ref{thm::continuity}--\ref{thm::tbm_lqg} and Corollary~\ref{cor::disk_plane_equivalences} require us to develop a number of estimates for the Euclidean size and shape of the regions explored by $\QLE(8/3,0)$.  While we do not believe that our estimates are in general optimal, we are able to obtain the precise first order behavior for the Euclidean size of a metric ball in $\oqdist$ centered around a quantum typical point.  We record this result as our final main theorem.

Throughout this work, we will make use of the following notation.  We will write $\ball{z}{\epsilon}$ for the open Euclidean ball centered at $z$ of radius $\epsilon$ and write $\qball{z}{\epsilon}$ for the ball with respect to $\oqdist$.  We will also write $\diam(A)$ to denote the Euclidean diameter of a set $A$.

\begin{theorem}
\label{thm::typical_ball_size}
Suppose that $\CS= (\s^2,h)$ is a unit area $\sqrt{8/3}$-LQG sphere and that $z$ is picked uniformly from the quantum measure on $\CS$.  Then we have (in probability) that
\[ \frac{\log \diam \qball{z}{\epsilon}}{\log \epsilon} \to 6 \quad\text{as}\quad \epsilon \to 0.\]
That is, the typical Euclidean diameter of $\qball{z}{\epsilon}$ for quantum typical $z$ is $\epsilon^{6(1+o(1))}$ as $\epsilon \to 0$.  The same also holds if we replace $\CS$ with the unit boundary length $\sqrt{8/3}$-LQG disk or a finite mass open subset of a $\sqrt{8/3}$-quantum cone.
\end{theorem}

To put this result in context, recall that a typical radius $\epsilon$ ball in TBM has Brownian map volume $\epsilon^4$ and that we expect that TBM can be covered by $\epsilon^{-4}$ such balls.  If the overall Brownian map has unit area, then among these $\epsilon^{-4}$ balls the \emph{average} ball has to have Euclidean volume of order at least $\epsilon^4$. But ``average'' and ``typical'' can be quite different. Theorem~\ref{thm::typical_ball_size} states that in some sense a \emph{typical} Brownian map ball has Euclidean diameter of order $\epsilon^{6}$ and hence Euclidean volume of order at most $\epsilon^{12}$, much smaller than this average. Based on this fact it is natural to conjecture that when a random triangulation with $n^4 = N$ triangles is conformally mapped to $\s^2$ (with three randomly chosen vertices mapping to three fixed points on $\s^2$, say) \emph{most} of the triangles end up with Euclidean volume of order $n^{-12} = N^{-3}$, even though the \emph{average} triangle has Euclidean volume of order $n^{-4} = N^{-1}$.

We remark that there are approximate variants of Theorem~\ref{thm::typical_ball_size} that could have been formulated without the metric construction of this paper. This is because even before one constructs a metric on $\sqrt{8/3}$-LQG, it is possible to construct a set one would expect to ``approximate'' a radius $\epsilon$ ball in the random metric: one does this by considering a typical point $x$ and taking the \emph{Euclidean} ball centered at $x$ with radius chosen so that its LQG volume is exactly $\epsilon^4$. Scaling results involving these ``approximate metric balls'' are derived e.g.\ in \cite{DS08}. Once Theorems~\ref{thm::continuity} and~\ref{thm::metric_completion} are established, Theorem~\ref{thm::typical_ball_size} is deduced by bounding the extent to which the ``approximate metric balls'' differ from the actual radius $\epsilon$ balls in the random metric.

\subsection{Outline}
\label{subsec::outline}

As partially explained above, the remaining sections of the paper can be divided into three main parts (not counting the open problem list in Section~\ref{sec::open_problems}):
\begin{enumerate}
\item  Section~\ref{sec::preliminaries} provides background, definitions, and results.
\item Sections~\ref{sec::boundary_area}, \ref{sec::euclidean_inner_outer_radius_bounds}, and~\ref{sec::continuity_to_bm} establish the fact that $\qdist$ a.s.\ extends uniquely to $\oqdist$ (Theorem~\ref{thm::continuity}), along with the H\"older continuity of the identity map and its inverse between $(\s^2, d)$ and $(\s^2, \oqdist)$ (Theorem~\ref{thm::metric_completion}), that $\oqdist$ is geodesic (Theorem~\ref{thm::geodesic_metric_space}), and the scaling exponent describing the Euclidean size of typical small metric balls (Theorem~\ref{thm::typical_ball_size}). These results are proved in Section~\ref{sec::continuity_to_bm} using estimates derived in Sections~\ref{sec::boundary_area} and~\ref{sec::euclidean_inner_outer_radius_bounds}. 
\item Sections~\ref{sec::sle_6_moment_bounds}, \ref{sec::reverse_explorations}, and~\ref{sec::geodesics_and_levy_net} establish the fact that, when viewed as a random metric measure space,  $(\s^2, \oqdist ,\mu_h)$ has the law of TBM (Theorem~\ref{thm::tbm_lqg}).  This is proved in
Section~\ref{sec::geodesics_and_levy_net}, using estimates derived in Sections~\ref{sec::sle_6_moment_bounds} and~\ref{sec::reverse_explorations}.
\end{enumerate}

The reader who mainly wants to know how to interpret an instance of the $\sqrt{8/3}$-LQG sphere as a random metric measure space homeomorphic to the sphere can stop reading after the first two parts. Theorems~\ref{thm::continuity}--\ref{thm::geodesic_metric_space} and~\ref{thm::typical_ball_size} provide a way to endow an instance of the $\sqrt{8/3}$-LQG sphere with a metric $\oqdist$ and answer some of the most basic questions about the relationship between $(\s^2,d)$ and $(\s^2,\oqdist)$. These four theorems are already significant. On the other hand, the third part may be the most interesting for many readers, as this is where the long-conjectured relationship between TBM and LQG is finally proved.

We conclude this introduction below with Section~\ref{subsec::strategy}, which gives a brief synopsis of the proof strategies employed in the later parts of the paper, along with summaries of some of the lemmas and propositions obtained along the way. Section~\ref{subsec::strategy} is meant as a road map of the paper, to help the reader keep track of the overall picture without getting lost, and to provide motivation and context for the many estimates we require.

\subsection{Strategy}
\label{subsec::strategy}

\subsubsection{Remark on scaling exponents} \label{subsubsec::scalingremark}

Throughout this paper, for the sake of intuition, the reader should keep in mind the ``1-2-3-4 rule'' of scaling exponents for TBM and for corresponding discrete random surfaces. Without being too precise, we will try to briefly summarize this rule here, first in a discrete context.  Consider a uniform infinite planar triangulation centered at a triangle $y$ and let $\partial B(y,r)$ denote the outer boundary of the set of triangles in the dual-graph ball $B(y,r)$. The rule states that the length of a geodesic from $y$ to $\partial B(y,r)$ is $r$, the outer boundary length $|\partial B(y,r)|$ is of order $r^2$, the sum $\sum_{i=0}^r |\partial B(y,i)|$ is of order $r^3$, and the volume of $B(y,r)$ (as well as the volume of the whole region cut off from $\infty$ by $\partial B(y,r)$) is of order $r^4$.

The $r^3$ exponent corresponds to the number of triangles explored by the first $r$ layers of the \emph{peeling process}, as presented e.g.\ in \cite{angel2003growth}. Also, as explained e.g.\ in \cite[Section~2]{ms2013qle}, if the vertices of the planar triangulation are colored with i.i.d.\ coin tosses, one can define an ``outward-reflecting'' percolation interface starting at $y$ and (by comparison with the peeling procedure) show that the length of a percolation interface (run until $r^4$ triangles have been cut off from $\infty$) is also of order $r^3$, while the \emph{outer boundary} of the set of triangles in that interface has length of order $r^2$.  (These exponents in the setting of the UIPT were derived in \cite{angel2003growth}.)

The continuum analog of this story is that the Hausdorff dimension $d_H$ of a set $S$ on TBM (defined using the intrinsic metric on TBM) \emph{should} be
\begin{itemize}
\item $d_H=1$ if $S$ is a geodesic,
\item $d_H=2$ if $S$ is the outer boundary of a metric ball, or the outer boundary of an (appropriately defined) $\SLE_6$ curve, or an (appropriately defined) $\SLE_{8/3}$ curve,
\item $d_H=3$ if $S$ is an (appropriately defined) $\SLE_6$ curve itself, or if $S$ is the union of the outer boundaries of balls of radius $r$ (as $r$ ranges over an interval of values), and
\item $d_H=4$ if $S$ is an open subset of the entire Brownian map.
\end{itemize}
Similarly, on an instance of the $\sqrt{8/3}$-LQG sphere, the number of Euclidean balls of quantum area $\delta$ required to cover a geodesic, a metric ball boundary (or $\SLE_{8/3}$ curve), an $\SLE_6$ curve, and the entire sphere should be respectively of order $\delta^{-1/4}$, $\delta^{-1/2}$, $\delta^{-3/4}$ and $\delta^{-1}$.

We will not prove these precise statements in this paper (though in the case of $\SLE_6$, $\SLE_{8/3}$, or the entire sphere the scaling dimension follows from the KPZ theorem as stated e.g.\ in \cite{DS08}).  On the other hand, in the coming sections we will endow all of these sets with fractal measures that scale in the appropriate manner: i.e., if one adds a constant to $h$ so that overall volume is multiplied by $C^4$, then geodesic lengths are multiplied by $C$, metric ball boundary lengths are multiplied by $C^2$, and QLE trace measures and $\SLE_6$ quantum natural times are both multiplied by $C^3$.

The distance function for $\gamma \in (0,2)$ was constructed in \cite{dddf2019tightness,dfgps2019weak,gm2019local,gm2019confluence,gm2019uniqueness,gm2019conformalcov}.  The analog of the ``1-2-3-4'' rule for $\gamma$-LQG surfaces with $\gamma \neq \sqrt{8/3}$ has never been completely worked out.  The ``1'' should presumably remain unchanged (a geodesic always has dimension one) but the ``4'' should presumably be replaced by the fractal dimension of the surface, which is expected to increase from $2$ to $4$ continuously as $\gamma$ increases from $0$ to $\sqrt{8/3}$ (see \cite[Section~3]{ms2013qle} for further discussion of this point, including a controversial conjectural formula due to Watabiki that applies to all $\gamma \in [0,2]$). The ``3'' should be replaced by two possibly distinct values (the quantum dimensions of the $\QLE(\gamma^2,0)$ trace and of $\SLE_{\kappa'}$, both drawn on a $\gamma$-LQG surface, where $\kappa' = 16/\gamma^2$), while the ``2'' should also be replaced by two possibly distinct values (the quantum dimensions of the outer boundaries of the $\QLE(\gamma^2,0)$ trace and of $\SLE_{\kappa'}$, when each is generated up to a stopping time).

\subsubsection{Remark on variants of measures on unit area surfaces} \label{subsubsec::diskvariants}
The unit area Brownian map, or unit area $\sqrt{8/3}$-LQG sphere, is not always the easiest or most natural object to work with directly.  If one considers a doubly marked unit area surface, together with an $\SLE_6$ curve from one endpoint to the other, then the disks cut out by the $\SLE_6$ cannot be completely conditionally independent of one another (given their boundary length) because we know that the total sum of their areas has to be $1$. To produce a setting where this type of conditional independence \emph{does} hold exactly, we will often be led to consider either
\begin{enumerate}
\item \emph{probability} measures on the space of \emph{infinite} volume surfaces, such as the Brownian plane and the (to be shown to be equivalent) $\sqrt{8/3}$-LQG cone with a $\sqrt{8/3}$-$\log$ singularity, or
\item \emph{infinite} measures on the space of \emph{finite} volume surfaces, where the law of the total area $A \in (0,\infty)$ is (up to multiplicative constant) an infinite measure given by $A^{\alpha}dA$ for some $\alpha$, and where once one conditions on a fixed value of $A$, the conditional law of the surface is a rescaled unit area Brownian map or the (to be shown to be equivalent) unit area $\sqrt{8/3}$-LQG sphere.
\end{enumerate}

In order to simplify proofs, we will prove some of our results \emph{first} in the setting where they are easiest and cleanest, and only later transfer them to the other settings. We will do a fair amount of work in the quantum cone setting in Sections~\ref{sec::boundary_area}, \ref{sec::euclidean_inner_outer_radius_bounds}, and~\ref{sec::continuity_to_bm}, a fair amount of work in the (closely related) quantum wedge setting in Section~\ref{sec::sle_6_moment_bounds},
 and a fair amount of work in the ``infinite measure on space of finite volume surfaces'' setting in Sections~\ref{sec::reverse_explorations} and~\ref{sec::geodesics_and_levy_net}.

 In this article, we will often abuse notation and refer to an infinite measure as a law or say that we sample from an infinite measure defined on a measure space $(E,\CA,\mu)$.  This is a convenient abuse of notation because several of the natural measures that we will consider are in fact infinite measures but become probability measures when conditioning on some event or value.  By this, we mean that we have a measurable function $X$ into $E$ so that for any $A \in \CA$ we have that the measure of $\{\omega : X(\omega) \in A\}$ is given by $\mu(A)$.  If $A \in \CA$ is such that $\mu(A) \in (0,\infty)$ then the law of $X$ conditioned on $X \in A$ makes sense as a probability measure in the usual way that conditional probability is defined for positive measure events.  One can also understand conditioning on certain zero measure events in the same way.  In particular, suppose that $(E,\CA,\mu)$ is $\sigma$-finite and $(A_n)$ is a sequence in $\CA$ with $A_n \subseteq A_{n+1}$ for all $n$ such that $\cup_n A_n = E$ with $\mu(A_n) \in (0,\infty)$ for all $n$.  Suppose further that we know that a regular conditional probability exists for the probability measure $X$ given $X \in A_n$ for every $n$ and some given $\sigma$-algebra.  Then we can speak of the regular conditional probability given just the $\sigma$-algebra.

\subsubsection{Strategy for background}

This is a long and somewhat technical paper, but many of the estimates we require in later sections can be expressed as straightforward facts about classical objects like the Gaussian free field, Poisson point processes, stable L\'evy process, and continuous state branching processes (which can be understood as time-changed stable L\'evy processes). In Section~\ref{sec::preliminaries} we enumerate some of the background results and definitions necessary for the current paper and suggest references in which these topics are treated in more detail.

We begin Section~\ref{sec::preliminaries} by recalling the definitions of quantum disks, spheres, cones, and wedges, as well as the construction of quantum Loewner evolution given in \cite{ms2013qle}. We next make an elementary observation: that the proof of the standard Kolmogorov-\u{C}entsov theorem --- which states that a.s.\ $\gamma$-H\"older continuity of a random field $X_u$, indexed by $u \in [0,1]^d$, can be deduced from estimates on moments of $|X_u - X_v|$ --- can be adapted to bound the law of the corresponding $\gamma$-H\"older norm. We then proceed to give some bounds on the probability that maximal GFF circle averages are very large. We finally define continuous state branching processes and present a few facts about them to be used later, along with some basic observations about stable L\'evy processes and Poisson point processes.

\subsubsection{Strategy for constructing metric and proving H\"older continuity}

We will consider a $\QLE(8/3,0)$ process $\Gamma_r$ (a random increasing family of closed sets indexed by $r$) defined on a certain infinite volume quantum surface called a quantum cone. To establish the desired H\"older continuity, we will need to control the law of the amount of time it takes a QLE growth started at a generic point $x_i$ to reach a generic point $x_j$, and to show that, in some appropriate local sense, these random quantities can a.s.\ be uniformly bounded above and below by random constants times appropriate powers of $|x_i - x_j|$.

To this end, we begin by establishing some control on how the Euclidean diameter of $\Gamma_r$ (started at zero) changes as a function of $r$. We do not \emph{a priori} have a very simple way to describe the growth of the Euclidean diameter of $\Gamma_r$ as a function of $r$.  On the other hand, based on the results in \cite{qlebm, ms2013qle}, we \emph{do} have a simple way to describe the evolution of the boundary length of $\Gamma_r$, which we denote by $B_r$, and the evolution of the area cut off from $\infty$ by $\Gamma_r$, which we denote by $A_r$. These processes can be described using the continuous state branching processes discussed in Section~\ref{sec::preliminaries}. 

Sections~\ref{sec::boundary_area},~\ref{sec::euclidean_inner_outer_radius_bounds}, and~\ref{sec::continuity_to_bm} are a sort of a dance in which one first controls the most accessible relationships (between $r$, $A_r$ and $B_r$) and other reasonably accessible relationships (between Euclidean and quantum areas of Euclidean disks, or between Euclidean and quantum lengths of boundary intervals --- here uniform estimates are obtained from basic information about the GFF), then combines them to address the \emph{a priori} much less accessible relationship between $r$ and the Euclidean diameter of $\Gamma_r$, and then uses this to address the general relationship between $|x_i - x_j|$ and the amount of time it takes for a branching QLE exploration to get from $x_i$ to $x_j$.

As explained in Section~\ref{subsubsec::scalingremark}, one would expect $B_r$ to be of order $r^2$, and it is natural to expect \begin{equation} \label{eqn::brsup} \sup_{0 \leq s \leq r} B_s \end{equation} to also be of order $r^2$.    Similarly, as explained in Section~\ref{subsubsec::scalingremark}, we expect $A_r$ to be of order $r^4$.  In Section~\ref{sec::boundary_area} we obtain three important results:

\begin{enumerate}
\item  Lemma~\ref{lem::boundary_length_radius} uses standard facts about continuous state branching processes to bound the probability that~\eqref{eqn::brsup} is much larger or smaller than $r^2$.
\item Lemma~\ref{lem::metric_hull_volume} uses standard facts about continuous state branching processes (CSBPs) to bound the probability that $A_r$ is much smaller than $r^4$.
\item Proposition~\ref{prop::cone_holder} uses simple Gaussian free field estimates to put a lower bound on the probability of the \emph{event} that (within a certain region of an appropriately embedded quantum cone) the quantum mass of \emph{every} Euclidean ball is at most some universal constant times a power of that ball's radius.  In what follows, it will frequently be useful to truncate on this event --- i.e., to prove bounds conditioned on this event occurring.
\end{enumerate}

Section~\ref{sec::euclidean_inner_outer_radius_bounds} uses the estimates from Section~\ref{sec::boundary_area} to begin to relate $r$ and the Euclidean diameter of $\Gamma_r$. There are a number of incremental lemmas and propositions used internally in Section~\ref{sec::euclidean_inner_outer_radius_bounds}, but the results cited in later sections are these:

\begin{enumerate}
\item Propositions~\ref{prop::euclidean_diameter_lower_bound} and~\ref{prop::ball_size_ubd} begin the game of relating $r$ and the Euclidean diameter of $\Gamma_r$. Proposition~\ref{prop::euclidean_diameter_lower_bound} states that \emph{on the event} described in Proposition~\ref{prop::cone_holder}, the Euclidean diameter of $\Gamma_r$ is very unlikely to be less than some power of $r$, and Proposition~\ref{prop::ball_size_ubd} states that (without any truncation) the Euclidean diameter of $\Gamma_r$ is very unlikely to be more than some other power of $r$.  (In fact, under a certain truncation, a bound on the fourth moment of $\diam(\Gamma_r)$ is given.) To show that $\Gamma_r$ is unlikely to have small Euclidean diameter, one applies the bounds from Section~\ref{sec::boundary_area} in a straightforward way. (If $\Gamma_r$ had small Euclidean diameter, then either $A_r$ would be unusually small or a small Euclidean-diameter region would have an unusually large amount of quantum mass, both scenarios that were shown in Sections~\ref{sec::boundary_area} to be improbable.) To show that $\Gamma_r$ is unlikely to have large Euclidean diameter, the hard part is to rule out the possibility that $\Gamma_r$ has large diameter despite having only a moderate amount of quantum area --- perhaps because it has lots of long and skinny tentacles. On the other hand, we understand the law of the quantum surface that forms the complement of $\Gamma_r$ (it is independent of the surface cut off by $\Gamma_r$ itself, given the boundary length) and can use this to show (after some work) that these kinds of long and skinny tentacles do not occur.
\item Corollary~\ref{cor::qle_area_bound} (which follows from Propositions~\ref{prop::euclidean_diameter_lower_bound} and~\ref{prop::ball_size_ubd}) implies that the total quantum \emph{area} cut off by $\Gamma_r$ has a certain power law decay on the special event from Proposition~\ref{prop::cone_holder}.  (The power law exponent one obtains after truncating on this event is better than the one that can be derived using the direct relationship between $r$ and $A_r$ without this truncation.)
\item Proposition~\ref{prop::gff_boundary_length} shows that when $h$ is an appropriately normalized GFF with free boundary conditions, the boundary length measure is \emph{very} unlikely to be much smaller than one would expect it to be.
\item Lemma~\ref{lem::gff_infimum} (used in the proof of Proposition~\ref{prop::gff_boundary_length}, as well as later on) is an elementary but useful tail bound on the maximum (over a compact set $K$) of the projection of the Gaussian free field onto the space of functions harmonic on some $U \supseteq K$.
\end{enumerate}
 
 In Section~\ref{sec::continuity_to_bm} we use the estimates from Section~\ref{sec::euclidean_inner_outer_radius_bounds} to show that the $\QLE(8/3,0)$ metric extends to a function which is H\"older continuous with respect to the Euclidean metric.  This will allow us to prove Theorem~\ref{thm::continuity} and Theorem~\ref{thm::metric_completion}.  We will also give the proof of Theorem~\ref{thm::geodesic_metric_space} and Theorem~\ref{thm::typical_ball_size} in Section~\ref{sec::continuity_to_bm}.

\subsubsection{Strategy for proving metric measure space has law of TBM} \label{subsubsec::strategytbmlaw}

Sections~\ref{sec::sle_6_moment_bounds}, \ref{sec::reverse_explorations}, and~\ref{sec::geodesics_and_levy_net} will show that the law of $(\s^2, \oqdist)$ is the law of TBM. They will do this by making use of the axiomatic characterization of TBM given in \cite{map_making}. Let us recall some notation and results from  \cite{map_making}.

A triple $(S, d, \nu)$ is called a {\bf metric measure space} if $(S,d)$ is a complete separable metric space and $\nu$ is a measure on the Borel $\sigma$-algebra generated by the topology generated by $d$, with $\nu(S) \in (0, \infty)$. We remark that one can represent the same space by the quadruple $(S, d, \wt \nu, m)$, where $m = \nu(S)$ and $\wt \nu = m^{-1} \nu$ is a probability measure.  This remark is important mainly because some of the literature on metric measure spaces requires $\nu$ to be a probability measure. Relaxing this requirement amounts to adding an additional parameter $m \in (0, \infty)$.

Two metric measure spaces are considered equivalent if there is a measure-preserving isometry from a full measure subset of one to a full measure subset of the other.  Let $\mmspace$ be the space of equivalence classes of this form. Note that when we are given an element $(S,d,\nu)$ of $\mmspace$, we have no information about the behavior of $S$ away from the support of $\nu$.

Next, recall that a measure on the Borel $\sigma$-algebra of a topological space is called {\bf good} if it has no atoms and it assigns positive measure to every open set. Let $\gmsspace$ be the space of geodesic metric measure spaces that can be represented by a triple $(S,d,\nu)$ where $(S,d)$ is a geodesic metric space homeomorphic to the sphere and $\nu$ is a good measure on $S$.

Note that if $(S_1, d_1, \nu_1)$ and $(S_2, d_2, \nu_2)$ are two such representatives, then the a.e.\ defined measure-preserving isometry $\phi \colon S_1 \to S_2$ is necessarily defined on a dense set, and hence can be extended to the completion of its support in a unique way so as to yield a continuous function defined on all of $S_1$ (similarly for $\phi^{-1}$). Thus $\phi$ can be uniquely extended to an \emph{everywhere} defined measure-preserving isometry.  In other words, the metric space corresponding to an element of $\gmsspace$ is uniquely defined, up to measure-preserving isometry.

As we are ultimately interested in probability measures on $\mmspace$, we will need to describe a $\sigma$-algebra $\CF$ on $\mmspace$, and more generally a $\sigma$-algebra $\CF^k$ on elements of $\mmspace$ with $k$ marked points. We will also need that $\gmsspace$ belongs to that $\sigma$-algebra, so that in particular it makes sense to talk about measures on $\mmspace$ that are supported on $\gmsspace$. We would like to have a $\sigma$-algebra that can be generated by a complete separable metric, since this would allow us to define regular conditional probabilities for all subsets. Such a $\sigma$-algebra is introduced in~\cite{map_making}.

Let $\gmsspace^2$ denote the space of sphere-homeomorphic metric measure spaces equipped with a good measure and with two marked points $x$ and $y$.  Given an element of this space, one can consider the union of the boundaries $\partial \fb{x}{r}$ taken over all $r \in [0,d(x,y)]$, where $\fb{x}{r}$ is the set of all points cut off from $y$ by the closed metric ball $\ol{B(x,r)}$.  (That is, $\fb{x}{r}$ is the complement of the component of $\ol{B(x,r)}$ containing $y$.)  This union is called the \emph{metric net} from $x$ to $y$.

It will be important for us to refer to \emph{leftmost} and \emph{rightmost} geodesics in a geodesic sphere.  For this, we need an orientation.  One way of specifying an orientation on a geodesic sphere $(S,d,\nu,x,y)$ marked by two distinct points $x,y$ is to specify three additional distinct marked points on $\partial \fb{x}{r}$ for some $r \in (0,d(x,y))$.  We say that two such spheres are equivalent if they are equivalent as doubly marked geodesic spheres and the extra marked points induce the same orientation.  We let $\gmsspace^{2,O}$ denote the set of equivalence classes and $\mmsigma^{2,O}$ the corresponding $\sigma$-algebra.  For the purposes of this work, we will be interested in the tree of leftmost geodesics from points in the metric net back to the root (i.e., $x$) as well as how they are identified in the metric net.  We call this structure the \emph{unembedded metric net}.

Let us explain further in what space the unembedded metric net lives.  Let $\T_1$ be the one-dimensional torus (i.e., $[0,1]$ with $0$ and $1$ identified) and $\T_2 = \T_1 \times \T_1$ be the two-dimensional torus.  We let $\treeequivspace$ be the set of pairs $(X,K)$ where $X \colon \T_1 \to \R_+$ is a continuous function with $\inf_{t \in \T_1} X_t = 0$ which is not constant in any interval of $\T_1$ and $K \subseteq \T_2$ is a compact set.  We say that pairs $(X,K)$ and $(Y,A)$ in $\treeequivspace$ are equivalent if there exists an increasing homeomorphism $\phi \colon \T_1 \to \T_1$ so that $X = Y \circ \phi$ and $K = \phi^{-1}(A)$ where $\phi^{-1}(A) = \{ (\phi^{-1}(x), \phi^{-1}(y)) : (x,y) \in A\}$.  We can define a metric on $\treeequivspace$ as follows.  Let $\dh$ denote the Hausdorff distance between compact subsets of $\T_2$.  For $(X,K), (Y,A) \in \treeequivspace$, we set
\[ d( (X,K), (Y,A)) = \inf_\phi \bigg( \| X  - Y \circ \phi \|_\infty + \dh(K,\phi^{-1}(A)) \bigg) \]
where the infimum is over all homeomorphisms $\phi$ as above.  We equip $\treeequivspace$ with its Borel $\sigma$-algebra.

It was proved in \cite[Proposition~2.22]{map_making} that there exists a Borel measurable map $\gmsspace^{2,O} \to \treeequivspace$ which associates with a doubly-marked and oriented geodesic sphere $(S,d,\nu,x,y)$ its unembedded metric net.  The unembedded metric net is only a non-trivial object when the leftmost geodesics are strongly coalescent, which means that for each value of $0 < s < r$, the number of points on $\partial \fb{x}{s}$ which are visited by leftmost geodesics from points on $\partial \fb{x}{r}$ to $x$ is finite.  This turns out to be equivalent to the tree of leftmost geodesics being precompact, hence its completion is a compact planar real tree and has a contour function.  In this case, the function $X \colon \T_1 \to \R_+$ is the contour function for the tree of leftmost geodesics and $K$ encodes a topologically closed equivalence relation on $\T_2$ which describes which points on the tree encoded by $X$ are identified in the metric net.  More precisely, if $\rho \colon \T_1 \to S$ is the map which visits the (completion of) the leftmost geodesic tree in contour order with the same time parameterization as $X$ then $(s,t) \in K$ if and only if $\rho(s) = \rho(t)$.

When $(S,d,\nu,x,y)$ is an instance of the doubly marked Brownian map, the unembedded metric net is the so-called \emph{$\alpha$-stable L\'evy net}, as defined in \cite[Section~3.3]{map_making}, with $\alpha = 3/2$.  More specifically, the $3/2$-stable L\'evy net is an infinite measure on pairs consisting of a planar real tree (encoded by a continuous function $\T_1 \to \R_+$ and defined modulo monotone parameterization) and a topologically closed equivalence relation on $\T_1$ (encoded by a compact subset of $\T_2$).  Multiple equivalent constructions of the \emph{$\alpha$-stable L\'evy net} appear in \cite[Section~3]{map_making}. (See Figure~\ref{fig::levy_net_sketch} for an informal description of the L\'evy net.)  We will also give a brief overview in Section~\ref{subsubsec::metric_net_is_levy_net}.  We now cite the following from \cite[Theorem~4.11]{map_making}.

\begin{figure}[ht!]
\begin{center}
\includegraphics[scale=0.85]{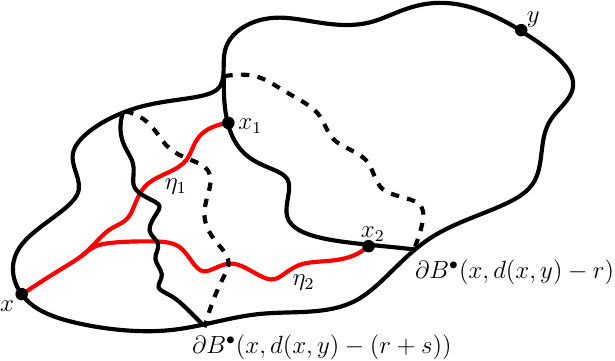}
\end{center}
\caption{\label{fig::levy_net_sketch} Shown is a doubly-marked sphere $(S,x,y)$ equipped with a metric $d$.  We assume that, for each $r \in (0,d(x,y))$, $\partial \fb{x}{d(x,y)-r}$ comes equipped with a boundary length measure.  For a fixed value of $r \in (0,d(x,y))$, the points $x_1,x_2$ shown in the illustration are assumed to be sampled from the boundary measure on $\partial \fb{x}{d(x,y)-r}$ and the red paths are leftmost geodesics from $x_1,x_2$ back to $x$.  Roughly, the unembedded metric net of $(S,x,y)$ from $x$ to $y$ is the $3/2$-L\'evy net if it is the case that boundary lengths of the clockwise and counterclockwise segments of $\partial \fb{x}{d(x,y) - (r+s)}$ between the leftmost geodesics from $x_1,x_2$ back to $x$ evolve as independent $3/2$-stable CSBPs as $s$ varies in $[0,d(x,y)-r]$ (and the same holds more generally for any finite collection of points chosen on $\partial \fb{x}{d(x,y)-r}$.  In addition, one needs that for each $r > 0$ the metric measure space $U_r$ is conditionally independent of $S \setminus \fb{x}{d(x,y)-r}$ given its boundary length.  The main focus of Section~\ref{sec::geodesics_and_levy_net} is to show that the unembedded metric net associated with a $\sqrt{8/3}$-LQG sphere is the $3/2$-L\'evy net.}
\end{figure}

\begin{theorem}
\label{thm::levynetbasedcharacterization}
Up to a positive multiplicative constant, the doubly marked Brownian map measure $\mustwo$ is the unique (infinite) measure on $(\gmsspace^{2,O}, \mmsigma^{2,O})$ which satisfies the following properties, where an instance is denoted by $(S,d,\nu,x,y)$.  
\begin{enumerate}
\item Given $(S,d,\nu)$, the conditional law of $x$ and $y$ is that of two i.i.d.\ samples from~$\nu$ (normalized to be a probability measure). In other words, the law of the doubly marked surface is invariant under the Markov step in which one ``forgets'' $x$ (or~$y$) and then resamples it from the given measure.
\item The law on $\treeequivspace$ (real trees with an equivalence relation) induced by the unembedded metric net from $x$ to $y$ (whose law is an infinite measure) by the measurable map defined in \cite[Proposition~2.22]{map_making} has the law of an $\alpha$-L\'evy net for some $\alpha \in (1,2)$.  In other words, the metric net is a.s.\ strongly coalescent (as defined in \cite[Section~2.5]{map_making}) and the law of the contour function of the leftmost geodesic tree and set of identified points agrees with that of the L\'evy height process used in the $\alpha$-L\'evy net construction.
\item Fix $r>0$ and consider the circle that forms the boundary $\partial \fb{x}{r}$ (an object that is well-defined a.s.\ on the finite-measure event that the distance from $x$ to $y$ is at least $r$).  Then the inside and outside of $\fb{x}{r}$ (with the orientation induced by $S$) are conditionally independent, given the boundary length of $\partial \fb{x}{r}$ (as defined from the L\'evy net structure) and the orientation of $S$.  Moreover, the conditional law of the outside of $\fb{x}{r}$ does not depend on $r$.
\end{enumerate}
\end{theorem}

The ultimate goal of Sections~\ref{sec::sle_6_moment_bounds}, \ref{sec::reverse_explorations}, and 
\ref{sec::geodesics_and_levy_net} is to show that the metric measure space we construct using QLE satisfies the conditions of Theorem~\ref{thm::levynetbasedcharacterization}.
 
\begin{itemize}
	\item The fact that our metric space is topologically a sphere and that the identity map from $\s^2$ equipped with $\oqdist$ to $\s^2$ equipped with the Euclidean metric is H\"older continuous with H\"older continuous inverse is proved in Sections~\ref{sec::boundary_area}--\ref{sec::continuity_to_bm}.
	\item It follows from the limiting construction developed in \cite[Appendix~A]{dms2014mating} of the doubly marked $\sqrt{8/3}$-LQG sphere that its law is preserved by the operation of forgetting the points $x$ and $y$ and resampling them independently from the underlying measure.  See Section~\ref{subsubsec::spheres}.
	\item The independence of the inside and outside of the filled metric ball follows from the construction of $\QLE(8/3,0)$ given in \cite{qlebm}, but care is needed to deal with a distinction between forward and reverse explorations, see Section~\ref{sec::reverse_explorations}.
	\item The fact that the unembedded metric net has the law of a $3/2$-L\'evy net is proved in Section~\ref{sec::geodesics_and_levy_net}.
	\end{itemize}

In order to do this, we will recall that some hints of the relationship with TBM, and more specifically with the $3/2$-L\'evy net, were already present in \cite{qlebm}.  One can define the ``outer boundary length'' process for growing QLE clusters and for growing Brownian map metric balls, and it was already shown in \cite{qlebm} that both of these processes can be understood as continuous state branching process excursions, and that their laws agree. In both cases, the ``jumps'' correspond to times at which disks of positive area are ``swallowed'' by the growing process; these disks are removed from the ``unexplored region'' at these jump times (i.e., the complement of the growing process or equivalently the complementary component which contains the target point). In both cases, it is possible to time-reverse the ``unexplored region'' process so that disks of positive area are ``glued on'' (at single ``pinch points'') at these jump times, and in both cases one can show that the location of the pinch point is uniformly random, conditioned on all that has happened before. One can use this to generate a coupling between the L\'evy net and QLE. However, it is not obvious that the geodesic paths of the L\'evy net actually correspond to geodesics of $\oqdist$.  This is the part that takes a fair amount of work and requires the analysis of a sequence of geodesic approximations.

In Section~\ref{sec::sle_6_moment_bounds}, we will prove moment bounds for the quantum distance between the initial point and tip of an $\SLE_6$ on a $\sqrt{8/3}$-quantum wedge as well as between two boundary points on a $\sqrt{8/3}$-quantum wedge separated by a given amount of quantum length.  These bounds will be used later to control the law of the length of certain geodesic approximations.

In Section~\ref{sec::reverse_explorations} we will describe the time-reversal of the $\SLE_6$ and $\QLE(8/3,0)$ unexplored-domain processes and deal with some technicalities regarding time reversal definitions. The QLE definition on an LQG sphere involves ``reshuffling'' every $\delta$ units of time during a certain time interval $[0,T]$ parameterizing a L\'evy process excursion; but technically speaking if $T$ is random and not necessarily a multiple of $\delta$, it makes a difference whether one marks the increments starting from $0$ (so their endpoints are $\delta, 2\delta, \ldots$).

Finally Section~\ref{sec::geodesics_and_levy_net} will use the results of Sections~\ref{sec::sle_6_moment_bounds} and~\ref{sec::reverse_explorations} to control various geodesic approximations and ultimately show that the geodesics of $\oqdist$ correspond to the L\'evy net in the expected way. This will enable us to complete the proofs of Theorem~\ref{thm::tbm_lqg} and Corollary~\ref{cor::disk_plane_equivalences}.

\section{Preliminaries}
\label{sec::preliminaries}

The purpose of this section is to review some background and to establish a number of preliminary estimates that will be used to prove our main theorems.  We begin in Section~\ref{subsec::spheres_and_disks} by reminding the reader of the construction of quantum disks, spheres, cones, and wedges.  We will then construct $\QLE(8/3,0)$ on a $\sqrt{8/3}$-quantum cone in Section~\ref{subsec::qle_on_cone}.  This process is analogous to the $\QLE(8/3,0)$ process constructed in \cite{qlebm} on a $\sqrt{8/3}$-LQG sphere.  Next, we will establish a quantitative version of the Kolmogorov-\u Centsov theorem in Section~\ref{subsec::quantitative_kc}.  Then, in Section~\ref{subsec::gff_extremes}, we will use the results of Section~\ref{subsec::quantitative_kc} to bound the extremes of the GFF.  Finally, we record a few basic facts about continuous state branching processes in Section~\ref{subsec::csbp_estimates}, an estimate of the tail of the supremum of an $\alpha$-stable process in Section~\ref{subsubsec::alpha_stable_supremum}, and an estimate of the tail of the Poisson distribution in Section~\ref{subsubsec::poisson_deviations}.

\subsection{Quantum disks, spheres, cones, and wedges}
\label{subsec::spheres_and_disks}

The purpose of this section is to give a brief overview of the construction of quantum disks, spheres, cones, and wedges.  We refer the reader to \cite[Section~4]{dms2014mating} for a much more in depth discussion of these objects.  See also the discussion in \cite{SHE_WELD,quantum_spheres}.

Suppose that $h$ is an instance of the Gaussian free field (GFF) on a planar domain $D$ and $\gamma \in (0,2)$.  The $\gamma$-LQG measure associated with $h$ is formally given by $e^{\gamma h(z)} dz$ where $dz$ denotes Lebesgue measure on $D$.  Since $h$ does not take values at points, it is necessary to use a regularization procedure in order to make sense of this expression rigorously.  This has been accomplished in \cite{DS08}, for example, by considering the approximation $\epsilon^{\gamma^2/2} e^{\gamma h_\epsilon(z)} dz$ where $h_\epsilon(z)$ denotes the average of $h$ on $\partial B(z,\epsilon)$ and $\epsilon^{\gamma^2/2}$ is the normalization factor which is necessary for the limit to be non-trivial.  A {\bf marked quantum surface} is an equivalence class of triples consisting of a domain $D$, a vector of points $\ul{z} \in \ol{D}$, and a distribution $h$ on $D$ where two triples $(D,h,\ul{z})$ and $(\wt{D},\wt{h},\ul{\wt{z}})$ are said to be equivalent if there exists a conformal transformation $\varphi \colon D \to \wt{D}$ which takes each element of $\ul{z}$ to the corresponding element of $\ul{\wt{z}}$ and such that $h = \wt{h} \circ \varphi + Q \log |\varphi'|$ where $Q = \tfrac{2}{\gamma} + \tfrac{\gamma}{2}$.  We will refer to a particular choice of representative of a marked quantum surface as its \emph{embedding}.  In order to specify the law of a marked quantum surface, we only have to specify the law of $h$ with one particular choice of embedding.

We will often refer to a quantum surface by specifying an embedding $(D,h)$, though when we say quantum surface we always mean modulo the equivalence relation mentioned above.  If $U \subseteq D$, we will often also abuse notation and write $(U,h)$ for the quantum surface (or embedding of a quantum surface) which corresponds to $(U,h|_U)$.

Throughout, we consider the infinite strip $\strip = \R \times (0,\pi)$ and the infinite cylinder $\cyl = \R \times [0,2\pi]$ (with the top and the bottom identified).  We denote by $\cyl_\pm = \R_\pm \times [0,2\pi]$ (with the top and bottom identified) the positive and negative half-infinite cylinders.  For $\CX \in \{\strip,\cyl,\cyl_\pm,\C,\h\}$, we let $H(\CX)$ be the closure of $C_0^\infty(\CX)$ with respect to the Dirichlet inner product
\begin{equation}
\label{eqn::dirichlet_inner_product}
(f,g)_\nabla = \frac{1}{2\pi} \int \nabla f(x) \cdot \nabla g(x) dx.
\end{equation}
For $\CX \in \{\strip,\cyl,\cyl_\pm\}$, we note that $H(\CX)$ admits the orthogonal decomposition $\CH_1(\CX) \oplus \CH_2(\CX)$ where $\CH_1(\CX)$ (resp.\ $\CH_2(\CX)$) consists of those functions on $\CX$ which are constant (resp.\ have mean zero) on vertical lines; see, e.g.\ \cite[Lemma~4.2]{dms2014mating}.  For $\CX = \C$, we have that $H(\C)$ admits the orthogonal decomposition $\CH_1(\C) \oplus \CH_2(\C)$ where $\CH_1(\C)$ (resp.\ $\CH_2(\C)$) consists of those functions on $\C$ which are radially symmetric about $0$ (resp.\ have mean zero on circles centered at $0$).  The same is likewise true for $H(\h)$ except with circles centered at $0$ replaced by semicircles centered at $0$.

The starting point for the construction of the unit boundary length quantum disk as well as the unit area quantum sphere is the infinite excursion measure $\nu_\delta^\bes$ associated with the excursions that a Bessel process of dimension $\delta$ ($\bes^\delta$) makes from $0$ for $\delta \in (0,2)$.  This measure can be explicitly constructed as follows.
\begin{itemize}
\item Sample a lifetime $t$ from the infinite measure $c_\delta t^{\delta/2-2} dt$ where $dt$ denotes Lebesgue measure on $\R_+$ and $c_\delta > 0$ is a constant.
\item Given $t$, sample a $\bes^\delta$ excursion from $0$ to $0$ of length $t$.
\end{itemize}

In the above description of $\nu_\delta^\bes$, we have used an abuse of notation since the first step involved ``sampling'' from an infinite measure (i.e., cannot be normalized to be a probability measure).  We will be working with infinite measures frequently in this article (since it is natural to consider infinite measures for a number of types of quantum surfaces) and we will frequently use this same abuse of notation.

The law of a $\bes^\delta$ process with $\delta \in (0,2)$ can then be sampled from by first picking a Poisson point process (\ppp)~$\Lambda$ with intensity measure $du d\nu_\delta$ where $du$ denotes Lebesgue measure on $\R_+$ and then concatenating together the elements $(u,e) \in \Lambda$ ordered by $u$.  It is still possible to sample a \ppp\ $\Lambda$ as above when $\delta \leq 0$, however it is not possible to concatenate together the elements of $\Lambda$ in chronological order to form a continuous process because there are too many short excursions.  (See \cite{py82bessel} as well as the text just after \cite[Theorem~1]{py96maximum}.)

\subsubsection{Quantum disks}
\label{subsubsec::disks}

As explained in \cite[Definition~4.21]{dms2014mating}, one can use $\nu_\delta^\bes$ to define an infinite measure $\CM$ on quantum surfaces $(\strip,h)$ as follows.
\begin{itemize}
\item Take the projection of $h$ onto $\CH_1(\strip)$ to be given by $2 \gamma^{-1} \log Z$ where $Z$ is sampled from $\nu_\delta^\bes$ with $\delta=3-\tfrac{4}{\gamma^2}$, reparameterized (by all of $\R$) to have quadratic variation $2du$.
\item Take the projection of $h$ onto $\CH_2(\strip)$ to be given by the corresponding projection of a free boundary GFF on $\strip$ sampled independently of $Z$.
\end{itemize}

The above construction defines a doubly marked quantum surface parameterized by the infinite cylinder; however it only determines $h$ up to a free parameter corresponding to the ``horizontal translation.''  We will choose this horizontal translation depending on the context.

If we condition $\CM$ on the quantum boundary length being equal to~$1$, then we obtain the law of the unit boundary length quantum disk.  More generally, we can sample from the law of $\CM$ conditioned on having quantum boundary length equal to $L$ by first sampling from the law of the unit boundary length quantum disk and then adding $\tfrac{2}{\gamma} \log L$ to the field.  We will denote this law by $\Mdl$.  The points which correspond to $\pm \infty$ are independently and uniformly distributed according to the quantum boundary length measure conditional on $\CS$ \cite[Proposition~A.8]{dms2014mating}.  The law $\Mdonel$ is obtained by weighting $\Mdl$ by its quantum area.  This corresponds to adding an extra marked point which is uniformly distributed from the quantum measure.

\subsubsection{Quantum spheres}
\label{subsubsec::spheres}

As is also explained in \cite[Definition~4.21]{dms2014mating}, one can use $\nu_\delta^\bes$ to define an infinite measure $\CM_\bes$ on doubly-marked quantum surfaces $(\cyl,h,-\infty,+\infty)$ as follows.
\begin{itemize}
\item Take the projection of $h$ onto $\CH_1(\cyl)$ to be given by $2 \gamma^{-1} \log Z$ where $Z$ is sampled from $\nu_\delta^\bes$ with $\delta=4-\tfrac{8}{\gamma^2}$, reparameterized to have quadratic variation $du$.
\item Take the projection of $h$ onto $\CH_2(\cyl)$ to be given by the corresponding projection of a whole-plane GFF on $\cyl$ sampled independently of $Z$.
\end{itemize}

As in the case of quantum disks, we have not yet fully specified $h$ as a distribution on the infinite cylinder because there is still one free parameter which corresponds to the ``horizontal translation.''  We will choose this horizontal translation depending on the context.

If we condition on the quantum area associated with $\CM_\bes$ to be equal to $1$, then we obtain the law of the unit area quantum sphere.  Given $\CS$, the points which correspond to $\pm \infty$ are uniformly and independently distributed according to the quantum measure \cite[Proposition~A.13]{dms2014mating}.

As explained in \cite[Theorem~1.2]{quantum_spheres}, in the special case that $\gamma = \sqrt{8/3}$ the measure $\CM_\bes$ admits another description in terms of the infinite excursion measure for a $3/2$-stable L\'evy process with only upward jumps from its running infimum; see \cite{bertoin96levy} for more details on this measure.  In this construction, one uses that if we start off with a quantum sphere sampled from $\CM_\bes$ and then draw an independent whole-plane $\SLE_6$ process $\eta'$ from $-\infty$ to $+\infty$, then the law of ordered, oriented (by whether $\eta'$ traverses the boundary points in clockwise or counterclockwise order --- i.e., whether the loop is on the left or right side of $\eta'$), and marked (last point on the disk boundary visited by $\eta'$) disks cut out by $\eta'$ can be sampled from as follows:
\begin{itemize}
\item Sample an excursion $e$ from the infinite excursion measure for $3/2$-stable L\'evy processes with only upward jumps from its running infimum.  (The time-reversal $e(T-\cdot)$ of $e \colon [0,T] \to \R_+$ at time $t$ is equal to the quantum boundary length of the component of $\CS \setminus \eta'([0,t])$ which contains $y$.)
\item For each jump of $e$, sample a conditionally independent quantum disk whose boundary length is equal to the size of the jump.
\item Orient the boundary of each quantum disk either to be clockwise or counterclockwise with the toss of a fair coin flip and mark the boundary of each with a uniformly chosen point from the quantum measure.
\end{itemize}
Moreover, it is shown in \cite[Theorem~1.2]{quantum_spheres} that the information contained in the doubly-marked sphere and $\eta'$ can be uniquely recovered from the ordered collection of marked and oriented disks.

A quantum sphere produced from $\CM_\bes$ is doubly marked.  If we parameterize the surface by $\cyl$ as described above, the marked points are located at $\pm \infty$.  In general, we will indicate such a doubly marked quantum sphere with the notation $(\CS,x,y)$ where $\CS$ denotes the quantum surface and $x,y$ are the marked points and we will indicate the corresponding measure by $\Mstwo$.

\subsubsection{Quantum cones}
\label{subsubsec::cones}

Fix $\alpha < Q$.  An $\alpha$-quantum cone \cite[Section~4.3]{dms2014mating} is a doubly marked quantum surface which is homeomorphic to $\C$.  The two marked points are referred to as the ``origin'' and ``infinity.''  Bounded neighborhoods of the former all a.s.\ contain a finite amount of mass and neighborhoods of the latter a.s.\ contain an infinite amount of mass.  It is convenient to parameterize a quantum cone by either $\cyl$ or $\C$, depending on the context.  In the former case, we will indicate the quantum cone with the notation $(\cyl,h,-\infty,+\infty)$ (meaning that $-\infty$ is the origin and $+\infty$ is infinity) and the law of $h$ can be sampled from by:
\begin{itemize}
\item Taking the projection of $h$ onto $\CH_1(\cyl)$ to be given by $2\gamma^{-1} \log Z$ where $Z$ is a $\bes^\delta$ with $\delta = 2 + \tfrac{4}{\gamma}(Q-\alpha)$, reparameterized to have quadratic variation $du$.
\item Taking the projection of $h$ onto $\CH_2(\cyl)$ to be given by the corresponding projection of a whole-plane GFF on $\cyl$.
\end{itemize}

It is often convenient in the case of quantum cones to take the horizontal translation so that the projection of $h$ onto $\CH_1(\cyl)$, which can be understood as a function of one real variable (since it is constant on vertical line segments), last hits $0$ on the line $\re(z) = 0$.

When $h$ is an instance of the GFF, the projection of $h$ onto $\CH_1(\cyl)$ is (as a function of the horizontal coordinate) a Brownian motion with drift. In order to construct an $h$ that corresponds to an instance of the quantum cone, we can take the projection onto $\CH_1(\cyl)$ to be as follows:
\begin{itemize}
\item For $u < 0$, it is equal to $B_{-u} + (Q-\alpha) u$ where $B$ is a standard Brownian motion with $B_0 = 0$.
\item For $u \geq 0$, it is equal to $\wt{B}_u + (Q-\alpha)u$ where $\wt{B}$ is a standard Brownian motion independent of $B$ conditioned so that $\wt{B}_u + (Q-\alpha)u \geq 0$ for all $u \geq 0$.
\end{itemize}
The definition of $\wt{B}$ involves conditioning on an event with probability zero, but it is explained in \cite[Remark~4.3]{dms2014mating}, for example, how to make sense of this conditioning rigorously.

If we parameterize by $\C$ instead of $\cyl$, we first sample the process $A_u$ by:
\begin{itemize}
\item For $u > 0$ taking it to be $B_u + \alpha u$ where $B$ is a standard Brownian motion with $B_0 = 0$.
\item For $u \leq 0$ taking it to be $\wt{B}_{-u}  + \alpha u$ where $\wt{B}$ is a standard Brownian motion with $\wt{B}_0 = 0$ conditioned so that $\wt{B}_u + (Q-\alpha) u > 0$ for all $u \geq 0$.
\end{itemize}
Then we take the projection of $h$ onto $\CH_1(\C)$ to be equal to $A_{e^{-u}}$ and the projection of $h$ onto $\CH_2(\C)$ to be the corresponding projection for a whole-plane GFF.  We will use the notation $(\C,h,0,\infty)$ for a quantum cone parameterized by $\C$ where $0$ (resp.\ $\infty$) is the origin (resp.\ infinity).

We will refer to the particular embedding of a quantum cone into $\C$ described just above as the \emph{circle average embedding}.

As explained in \cite[Theorem~1.18]{dms2014mating}, it is natural to explore a $\sqrt{8/3}$-quantum cone (parameterized by $\C$) with an independent whole-plane $\SLE_6$ process $\eta'$ from~$0$ to~$\infty$.  If one parameterizes $\eta'$ by quantum natural time \cite{dms2014mating}, then the quantum boundary length of the unbounded component of $\C \setminus \eta'([0,t])$ evolves in~$t$ as a $3/2$-stable L\'evy process with only downward jumps conditioned to be non-negative \cite[Corollary~12.2]{dms2014mating}.  (See \cite[Chapter~VII, Section~3]{bertoin96levy} for more details on the construction of a L\'evy process with only downward jumps conditioned to be non-negative.  In particular, \cite[Chapter~VII, Proposition~14]{bertoin96levy} gives the existence of the process started from~$0$.)  Moreover, the surface parameterized by the unbounded component of $\C \setminus \eta'([0,t])$ given its quantum boundary length is conditionally independent of the surfaces cut off by $\eta'|_{[0,t]}$ from $\infty$.  If the quantum boundary length is equal to $u$, then we will write this law as $\qconeUnex^u$.  By scaling, we can sample from the law of $\qconeUnex^u$ by first sampling from the law $\qconeUnex^1$ and then adding the constant $2 \gamma^{-1} \log u$, $\gamma=\sqrt{8/3}$, to the field.  (One can think of a sample produced from $\qconeUnex^u$ as corresponding to a quantum disk with boundary length equal to $u$ and conditioned on having infinite quantum area.)

Let $\kappa' = 4/\gamma^2 > 4$. It is also shown in \cite{dms2014mating} that it is natural to explore a $\gamma$-quantum cone $(\C,h,0,\infty)$ with a space-filling $\SLE_{\kappa'}$ process $\eta'$ \cite{MS_IMAG4} from $\infty$ to $\infty$ which is sampled independently of the quantum cone and then reparameterized by quantum area, i.e., so that $\mu_h(\eta'([s,t])) = t-s$ for all $s < t$ and normalized so that $\eta'(0) = 0$.  It is in particular shown in \cite[Theorem~1.13]{dms2014mating} that the joint law of $h$ and $\eta'$ is invariant under the operation of translating so that $\eta'(t)$ is taken to $0$.  That is, as doubly-marked path-decorated quantum surfaces we have that
\[ (h,\eta') \stackrel{d}{=} (h(\cdot + \eta'(t)),\eta'(\cdot+t) - \eta'(t))\]
This fact will be important for us in several places in this article.

\subsubsection{Quantum wedges}

Fix $\alpha < Q$.  An $\alpha$-quantum wedge \cite[Section~4.2]{dms2014mating} (see also \cite{SHE_WELD}) is a doubly-marked surface which is homeomorphic to~$\h$.  As in the case of a quantum cone, the two marked points are the origin and infinity.  It is natural to parameterize a quantum wedge either by $\strip$ or by~$\h$.  In the former case, we can sample from the law of the field $h$ by:
\begin{itemize}
\item Taking its projection onto $\CH_1(\strip)$ to be given by $2 \gamma^{-1} \log Z$ where $Z$ is a $\bes^\delta$ with $\delta = 2 + \tfrac{2}{\gamma}(Q-\alpha)$ reparameterized to have quadratic variation $2 du$.
\item Taking its projection onto $\CH_2(\strip)$ to be given by the corresponding projection of a GFF on $\strip$ with free boundary conditions.
\end{itemize}
As in the case of an $\alpha$-quantum cone, we can also describe the projection of $h$ onto $\CH_1(\strip)$ in terms of Brownian motion \cite[Remark~4.5]{dms2014mating}.  In fact, the definition is the same as for an $\alpha$-quantum cone except with $B_u$, $\wt{B}_u$ replaced by $B_{2u}$, $\wt{B}_{2u}$.  (The variance is twice as large because the strip is half as wide as the cylinder.)

If we parameterize the surface with~$\h$, then we can sample from the law of the field $h$ by (see \cite[Definition~4.4]{dms2014mating}):
\begin{itemize}
\item Taking its projection onto $\CH_1(\h)$ to be given by $A_{e^{-u}}$ where $A$ is as in the definition of an $\alpha$-quantum cone parameterized by $\C$ except with $B_u$, $\wt{B}_u$ replaced by $B_{2u}$, $\wt{B}_{2u}$.
\item Taking its projection onto $\CH_2(\h)$ to be given by the corresponding projection of a GFF on~$\h$ with free boundary conditions.
\end{itemize}

\subsection{$\QLE(8/3,0)$ on a $\sqrt{8/3}$-quantum cone}
\label{subsec::qle_on_cone}

The idea of $\QLE(8/3,0)$ is to define a growth process on a $\sqrt{8/3}$-LQG surface which should be interpreted as a form of the Eden growth model \cite{eden1961growth}.  Recall that the Eden growth model on a graph $G = (V,E)$ is an increasing sequence of clusters $C_n \subseteq V$ where, for a given initial vertex $x \in V$, we take $C_0 = \{x\}$.  Given that $C_0,\ldots,C_n$ have been defined, we define $C_{n+1} = C_n \cup \{v\}$ where $\{u,v\}$ is an edge chosen uniformly at random among those with $u \in C_n$ and $v \notin C_n$.  In $\QLE(8/3,0)$, the uniform measure on edges is replaced by the $\sqrt{8/3}$-LQG boundary length measure and rather than adding a vertex to a discrete cluster at each stage, one adds a $\delta$-length segment of radial $\SLE_6$.  Here, the time parameterization for the $\SLE_6$ is the quantum natural time developed in \cite{dms2014mating}, which is the continuum analog of the time parameterization which one obtains when performing a percolation exploration on a random planar map and each unit of time corresponds to an edge traversed.  Further intuition and motivation for this construction is developed in \cite[Section~2.2]{ms2013qle} and \cite[Section~3]{qlebm}.

In \cite[Section~6]{qlebm}, we constructed a ``quantum natural time'' \cite{dms2014mating} variant of the $\QLE(8/3,0)$ process from \cite{ms2013qle} on a $\sqrt{8/3}$-LQG sphere and showed that this process defines a metric on a countable, dense set of points chosen i.i.d.\ from the quantum area measure on the sphere.  In many places in this article, it will be convenient to work on a $\sqrt{8/3}$-quantum cone instead of a $\sqrt{8/3}$-LQG sphere.  We will therefore review the construction and the basic properties of the process in this context.  We will not give detailed proofs here since they are the same as in the case of the $\sqrt{8/3}$-LQG sphere. We refer the reader to \cite[Section~6]{qlebm} for additional detail.

\begin{figure}[ht!]
\begin{center}
\includegraphics[scale=0.85]{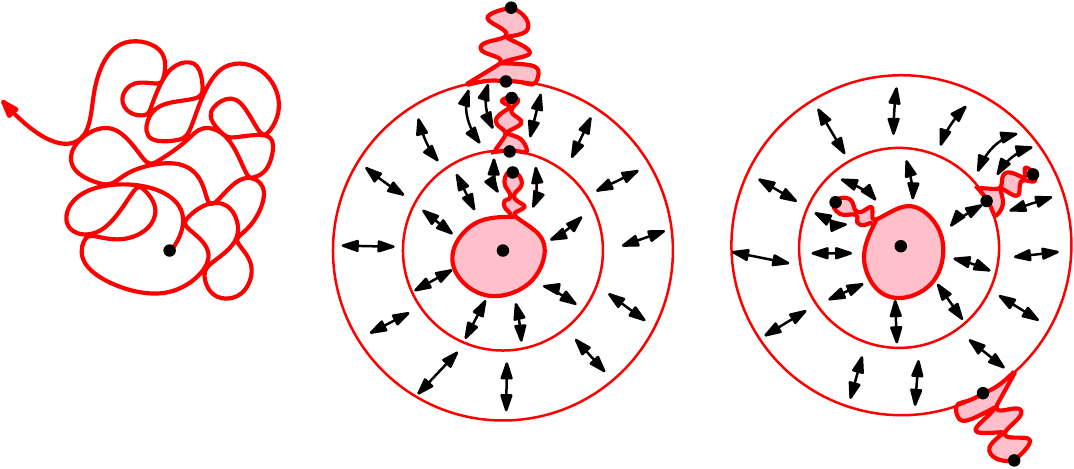}	
\end{center}
\caption{\label{fig::sle_to_qle} {\bf Left:} Independent whole-plane $\SLE_6$ from $0$ to $\infty$ drawn on top of a $\sqrt{8/3}$-quantum cone.  {\bf Middle:} We can represent the path-decorated surface as a collection of $\delta$-quantum natural time length necklaces which serve to encode the bubbles cut off by the $\SLE_6$ in each of the $\delta$-length intervals of time.  Each necklace has an inner and an outer boundary, is doubly marked by the initial and terminal points of the $\SLE_6$, the necklaces are conditionally independent given their inner and outer boundary lengths, and each necklace is a.s.\ determined by the collection of marked and oriented bubbles cut off by the $\SLE_6$ in the corresponding time interval.  The length of the outer boundary of each necklace is equal to the length of the inner boundary of the next necklace.  If we glue together the necklaces as shown, then we recover the $\sqrt{8/3}$-quantum cone decorated by the independent $\SLE_6$.  {\bf Right:} If we ``rotate'' each of the necklaces by a uniformly random amount and then glue together as shown, the underlying surface is a $\sqrt{8/3}$-quantum cone which is decorated with the $\delta$-approximation to $\QLE(8/3,0)$.  The left and right pictures are naturally coupled together so that the bubbles cut out by the $\SLE_6$ and $\QLE(8/3,0)$ are the same as quantum surfaces and the evolution of the boundary length of both is the same, up to a time-change.}
\end{figure}

We suppose that $(\C,h,0,\infty)$ is a $\sqrt{8/3}$-quantum cone and that $\eta'$ is a whole-plane $\SLE_6$ from $0$ to $\infty$ sampled independently of $h$ and then reparameterized by quantum natural time.  Fix $\delta > 0$.  We define the $\delta$-approximation of $\QLE(8/3,0)$ starting from~$0$ as follows.  First, we take $\qlegrowth_t^\delta$ to be the complement of the unbounded component of $\C \setminus \eta'([0,t])$ for each $t \in [0,\delta]$.  We also let $g_t^\delta \colon \C\setminus \qlegrowth_t^\delta \to \C\setminus \ol{\D}$ be the unique conformal map which fixes and has positive derivative at $\infty$.  Fix $j \in \N$ and suppose that we have defined paths $\eta_1',\ldots,\eta_j'$, where each $\eta_i$ for $1 \leq i \leq j$ is defined in $[(i-1)\delta,i\delta]$, and a growing family of hulls $\qlegrowth^\delta$ with associated uniformizing conformal maps $(g_t^\delta)$ for $t \in [0,j\delta]$ such that the following hold:
\begin{itemize}
\item The conditional law of the surface parameterized by the complement of $\qlegrowth_{j\delta}^\delta$ given its quantum boundary length $\ell$ is the same as in the setting of exploring a $\sqrt{8/3}$-quantum cone with an independent whole-plane $\SLE_6$.  That is, it is given by~$\conelaw^\ell$.
\item $\eta_j'(j\delta)$ is distributed uniformly according to the quantum boundary measure on $\partial \qlegrowth_{j\delta}^\delta$ conditional on $\qlegrowth_{j \delta}^\delta$ (as a path decorated quantum surface).
\item The joint law of the components (viewed as quantum surfaces) separated from $\infty$ by time $j \delta$, given their quantum boundary lengths, is the same as in the case of whole-plane $\SLE_6$.  That is, they are given by conditionally independent quantum disks given their boundary lengths and their boundary lengths correspond to the downward jumps of a $3/2$-stable L\'evy process starting from $0$ and conditioned to be non-negative.
\end{itemize}
We then let $\eta_{j+1}'$ be an independent radial $\SLE_6$ defined in the time-interval $[j\delta, (j+1)\delta]$ starting from a point on $\partial \qlegrowth_{j\delta}^\delta$ which is chosen uniformly from the quantum boundary measure conditionally independently of everything else (i.e., we resample the location of the tip $\eta_j'(j \delta)$ of $\eta_j'$).  For each $t \in [j\delta,(j+1)\delta]$, we also let $\qlegrowth_t^\delta$ be the complement of the unbounded component of $\C\setminus (\qlegrowth_{j\delta}^\delta \cup \eta_{j+1}'([j \delta,t]))$.  Then by the construction, all three properties described above are satisfied by the process up to time $(j+1)\delta$.

A convenient way to visualize the construction of the $\delta$-approximation to $\QLE(8/3,0)$ is illustrated in Figure~\ref{fig::sle_to_qle}.  We refer to the path-decorated quantum surface which is parameterized by the region that $\eta_j'$ separates from $\infty$ and decorated by $\eta'$ as part of a \emph{necklace}.  An $\SLE_6$ necklace is simply this path-decorated surface together with the boundary of the cluster grown up to just before we draw $\eta_j'$.  Thus a necklace consists of an inner boundary (boundary of the cluster before $\eta_j'$ is drawn) and an outer boundary (boundary of the cluster after $\eta_j'$ is drawn).  One can similarly decompose an $\SLE_6$ into necklaces by considering the successive path decorated surfaces which correspond to $\delta$ units of quantum natural time.  One can then apply a transformation from $\SLE_6$ to the $\delta$-approximation of $\QLE(8/3,0)$ by taking the necklace decomposition of the former and then changing how the necklaces are glued together by randomizing the inner boundary point of each necklace from which the $\SLE_6$ is grown using quantum boundary length.

By repeating the compactness argument given in \cite[Section~6]{qlebm}, we see that there exists a deterministic sequence $(\delta_k)$ which tends to $0$ as $k \to \infty$ along which the $\delta$-approximations converge weakly and the limiting process satisfies properties which are analogous to the three properties described above.

We note that it is shown in \cite{qlebm} that if $(x_n)$ is a sequence of points chosen i.i.d.\ from the quantum measure on a $\sqrt{8/3}$-LQG sphere, then the joint law of the hitting times of the $(x_n)$ by the subsequentially limiting $\QLE(8/3,0)$ does not depend on the choice of sequence $(\delta_k)$ and is a.s.\ determined by the underlying quantum surface.  Fix $R > 0$.  Suppose that we apply the map $\log(z)$ from $\C$ to $\cyl$ so that $0$ is taken to $-\infty$ and then we take the horizontal translation so that the projection of the field onto $\CH_1(\cyl)$ first hits $R$ at $u=0$.  Then the law of the restriction of the field to $\cyl_-$ is the same as the corresponding law for a quantum sphere parameterized by $\cyl$ sampled from $\CM_\bes$ conditioned on the projection onto $\CH_1(\cyl)$ exceeding $R$ and with the horizontal translation taken in the same way.  Since $R > 0$ was arbitrary, it therefore follows that the same is also true for $\QLE(8/3,0)$ on a $\sqrt{8/3}$-quantum cone.  This alone does not imply that the $\delta$-approximations to $\QLE(8/3,0)$ converge as $\delta \to 0$ (in other words, it is not necessary to pass along a sequence of positive numbers $(\delta_k)$ which tend to $0$ as $k \to \infty$) because these hitting times may not determine the law of the process itself.  This, however, will be a consequence of the continuity results established in the present article.  It will also be a consequence of the present article that one has convergence in probability because we will show that the $\QLE(8/3,0)$ is a.s.\ determined by the underlying field.

In the case of a whole-plane $\SLE_6$ exploration of a $\sqrt{8/3}$-quantum cone, we know from \cite[Corollary~12.2]{dms2014mating} that the boundary length of the outer boundary evolves as a $3/2$-stable L\'evy process with only downward jumps conditioned to be non-negative.  The compactness argument of \cite[Section~6]{qlebm} also implies that the subsequentially limiting $\QLE(8/3,0)$ with the quantum natural time parameterization has the same property.

Recall from \cite{qlebm} that we change time from the quantum natural time to the \emph{quantum distance time parameterization} using the time-change
\begin{equation}
	\label{eqn::qle_time_change}
	\int_0^t \frac{1}{X_s} ds
\end{equation}
where $X_s$ is the quantum boundary length of the outer boundary of the process at quantum natural time $s$.  (The intuition for using this particular time change is that in the Eden growth model, the rate at which new edges are added to the outer boundary of the cluster is proportional to the boundary length of the cluster.)  If we perform this time-change, then the outer boundary length of the $\QLE(8/3,0)$ evolves as the time-reversal of a $3/2$-stable continuous state branching process (CSBP; we will give a review of CSBPs in Section~\ref{subsec::csbp_estimates} below, including the relationship between CSBPs and L\'evy processes via time-change).

\begin{lemma}
\label{lem::quantum_disk_exit}
Suppose that $(\D,h)$ has the law of a quantum disk with boundary length $L > 0$ and that $z \in \D$ is distributed uniformly according to the quantum area measure.  Then the $\QLE(8/3,0)$ stopped upon first hitting $\partial \D$ intersects $\partial \D$ at a unique point a.s.  Finally, if $D_L$ has the law of the amount of quantum distance time required by the $\QLE(8/3,0)$ to hit $\partial \D$ then $D_L \stackrel{d}{=} L^{1/2} D_1$.
\end{lemma}
\begin{proof}
The first assertion of the lemma is established in \cite[Lemma~7.6]{qlebm}.  We will deduce the second assertion of the lemma using the following scaling calculation.  Recall that if we add the constant $C$ to the field then quantum boundary length is scaled by the factor $e^{\gamma C/2}$ and that quantum natural time is scaled by the factor $e^{3 \gamma C/4}$ (see \cite[Section~6.2]{quantum_spheres}).  Equivalently, if we start off with a unit boundary length quantum disk, $L > 0$, and we scale the field so that the boundary length is equal to $L$ then quantum natural time is scaled by the factor $L^{3/2}$.  Recall also that if $X_t$ denotes the quantum boundary length of the outer boundary of the $\QLE(8/3,0)$ growth at quantum natural time $t$, then the quantum distance time elapsed by quantum natural time $T$ is equal to
\begin{equation}
\label{eqn::qle_distance_time}
\int_0^T \frac{1}{X_s} ds.
\end{equation}
Combining~\eqref{eqn::qle_distance_time} with the scaling given for boundary length and quantum natural time given above, we see that if we start out with a unit boundary length quantum disk and then scale the field so that the boundary length is $L$, then the amount of quantum distance time elapsed by the resulting $\QLE(8/3,0)$ is given by
\begin{equation}
\label{eqn::qle_scale_time}
\int_0^{L^{3/2} T} \frac{1}{L X_{L^{-3/2} s}} ds.
\end{equation}
Making the substitution $t = L^{-3/2} s$ in~\eqref{eqn::qle_scale_time}, we see that~\eqref{eqn::qle_scale_time} is equal to
\begin{equation}
\label{eqn::qle_distance_time_simplified}
L^{1/2} \int_0^T \frac{1}{X_t} dt.
\end{equation}
The final claim follows from~\eqref{eqn::qle_distance_time_simplified}.
\end{proof}

Using the same scaling argument used to establish Lemma~\ref{lem::quantum_disk_exit}, we can also determine how quantum distances scale when we add a constant $C$ to the field.
\begin{lemma}
\label{lem::quantum_distance_scale}
Suppose that $(D,h)$ is a $\sqrt{8/3}$-LQG surface and let $\qdist$ be the distance function associated with the $\QLE(8/3,0)$ metric.  Fix $C \in \R$.  Then the distance function associated with the field $h+C$ is given by $e^{\gamma C/4} \qdist$ with $\gamma = \sqrt{8/3}$.
\end{lemma}
We note that $\qdist$ is \emph{a priori} only defined on a countable dense subset of $D$ chosen i.i.d.\ from the quantum area measure.  However, upon completing the proof of Theorem~\ref{thm::continuity} and Theorem~\ref{thm::metric_completion}, the same scaling result immediately extends to $\oqdist$ by continuity.
\begin{proof}[Proof of Lemma~\ref{lem::quantum_distance_scale}]
This follows from the same argument used to establish~\eqref{eqn::qle_distance_time},~\eqref{eqn::qle_scale_time}, and~\eqref{eqn::qle_distance_time_simplified}. 
\end{proof}

\subsection{Quantitative Kolmogorov-\u Centsov}
\label{subsec::quantitative_kc}

The purpose of this section is to establish a quantitative version of the Kolmorogov-\u Centsov continuity criterion \cite{KS98,RY04}.  We will momentarily apply this result to the case of the circle average process for the GFF, which will be used later to establish the continuity results for $\QLE(8/3,0)$.

\begin{proposition}[Kolmogorov-\u Centsov continuity criterion]
\label{prop::kc_continuity}
Suppose that $(X_u)$ is a random field indexed by $u \in [0,1]^d$.  Assume that there exist constants $\alpha,\beta, c_0 > 0$ such that for all $u,v \in [0,1]^d$ we have that
\begin{equation}
\label{eqn::moment_condition}
\ex{ |X_u - X_v|^\alpha } \leq c_0 |u-v|^{d+\beta}.
\end{equation}
Then there exists a modification of $X$ (which we shall write as $X$) such that for each $\gamma \in (0,\beta/\alpha)$ there exists $M > 0$  such that
\begin{equation}
\label{eqn::holder_continuous}
|X_u - X_v| \leq M|u-v|^\gamma \quad\text{for all}\quad u,v \in [0,1]^d.
\end{equation}
Moreover, if we define $M$ to be $\sup_{u \neq v} |X_u-X_v|/|u-v|^\gamma$, then there exists $c_1 > 0$ depending on $\alpha,\beta,\gamma,c_0$ such that
\begin{equation}
\label{eqn::holder_norm_growth}
\pr{ M \geq t } \leq c_1 t^{-\alpha} \quad\text{for all}\quad t \geq 1.
\end{equation}
\end{proposition}
The first statement of the proposition is just the usual Kolmogorov-\u Centsov continuity criterion.  One sees that~\eqref{eqn::holder_norm_growth} holds by carefully following the proof.  For completeness, we will work out the details here. 
\begin{proof}[Proof of Proposition~\ref{prop::kc_continuity}]
Applying Chebyshev's inequality, we have from~\eqref{eqn::moment_condition} that
\begin{equation}
\label{eqn::mkc_cheby}
\pr{ |X_u - X_v| \geq \delta } \leq c_0 \delta^{-\alpha} | u-v|^{d+\beta} \quad\text{for all}\quad u,v \in [0,1]^d.
\end{equation}
For each $k$, let $\CD_k$ consist of those $x \in [0,1]^d$ with dyadic rational coordinates that are integer multiples of $2^{-k}$.  Let $\wt{\CD}_k$ consist of those pairs $\{u,v\}$ in $\CD_k$ which are adjacent, i.e., differ in only one coordinate and have $|u-v|=2^{-k}$.  By~\eqref{eqn::mkc_cheby}, we have that
\begin{equation}
\label{eqn::mkc_one_coord}
 \pr{ |X_u - X_v| \geq t 2^{-\gamma k} } \leq c_0 t^{-\alpha} 2^{-k(d+\beta - \alpha \gamma)} \quad\text{for all}\quad u,v \in \wt{\CD}_k.
\end{equation}
Noting that $|\wt{\CD}_k| = O(2^{dk})$, by applying a union bound and using~\eqref{eqn::mkc_one_coord} we have for some constant $c_1 > 0$ that
\begin{equation}
\label{eqn::mkc_union_bound}
\pr{ \max_{\{u,v\} \in \wt{\CD}_k} |X_u - X_v| \geq t 2^{-\gamma k} } \leq c_1 t^{-\alpha} 2^{-k(\beta - \alpha \gamma)}.
\end{equation}
Thus, by a further union bound and using~\eqref{eqn::mkc_union_bound}, we have for some constant $c_2 > 0$ that
\begin{equation}
\label{eqn::mkc_union_bound2}
\pr{ \sup_{k \in \N} \max_{\{u,v\} \in \wt{\CD}_k} 2^{\gamma k} |X_u - X_v| \geq t  } \leq c_2 t^{-\alpha}.
\end{equation}
It is not difficult to see that there exists some constant $c_3 > 0$ such that on the event $\sup_{k \in \N} \max_{\{u,v\} \in \wt{\CD}_k} 2^{\gamma k} |X_u - X_v| \leq t$ considered in~\eqref{eqn::mkc_union_bound2} we have that $|X_u - X_v| \leq c_3 t | u-v|^\gamma$ for all $u,v \in \cup_k \CD_k$.  This, in turn, implies the result.
\end{proof}

\subsection{GFF extremes}
\label{subsec::gff_extremes}

In this section, we will establish a result regarding the tails of the maximum of the circle average process associated with a whole-plane GFF.  We refer the reader to \cite[Section~3]{DS08} for more on the construction of the circle average process.  We also refer the reader to \cite[Section~3.2]{SHE_WELD} for more on the whole-plane GFF.

\begin{proposition}
\label{prop::gff_maximum}
Suppose that $h$ is a whole-plane GFF.  For each $r > 0$ and $z \in \C$ we let $h_r(z)$ be the average of $h$ on $\partial B(z,r)$.  We assume that the additive constant for $h$ has been fixed so that $h_1(0) = 0$.  For each $\xi \in (0,1)$ there exists a constant $c_0 > 0$ such that for each fixed $r \in (0,1/2)$ and all $\delta > 0$ we have that
\begin{equation}
\label{eqn::gff_max_decay}
\pr{ \sup_{z \in B(0,1/2)} |h_r(z)| \geq (2+\delta) \log r^{-1} } \leq c_0 r^{2\delta(1-\xi)}.
\end{equation}
\end{proposition}

Before giving the proof of Proposition~\ref{prop::gff_maximum}, we are first going to deduce from it a result which bounds the growth of $|h_r(z)|$ for $z \in \C$ with $|z|$ large and $r$ proportional to $|z|$.

\begin{corollary}
\label{cor::gff_whole_plane_maximum}
Suppose that we have the same setup as described in Proposition~\ref{prop::gff_maximum}.  For $a,C > 0$ we let 
\begin{equation}
\label{eqn:gff_good_event}
 E_{a,C} = \bigcap_{k \in \N} \left\{ \sup_{z \in B(0,e^{k+1}) \setminus B(0,e^k)} |h_{e^{k-1}}(z)| \leq C + a k \right\}.
\end{equation}
Then we have that
\begin{equation}
\label{eqn::gff_max_decay}
\pr{ E_{a,C} } \to 1 \quad\text{as}\quad C \to \infty \quad\text{(with $a > 0$ fixed)}.
\end{equation}
The same likewise holds if $\alpha < Q$ and $h= h_1 + \alpha\log|\cdot|$ where $(\C,h_1,0,\infty)$ is an $\alpha$-quantum cone with the circle average embedding.
\end{corollary}

Before establishing Corollary~\ref{cor::gff_whole_plane_maximum}, we first record the following Gaussian tail bound, which is easy to derive directly from the standard Gaussian density function.

\begin{lemma}
\label{lem::gaussian_tail_estimate}
Suppose that $Z \sim N(0,1)$.  Then we have that
\[ \pr{ Z \geq \lambda} \asymp \sqrt{\frac{2}{\pi}} \lambda^{-1} \exp\left( - \frac{\lambda^2}{2} \right) \quad\text{as}\quad \lambda \to \infty.\]
\end{lemma}

\begin{proof}[Proof of Corollary~\ref{cor::gff_whole_plane_maximum}]
We are first going to deduce the result in the case of a whole-plane GFF from Proposition~\ref{prop::gff_maximum} and a union bound.

Note that $h - h_{e^{k+2}}(0)$ has the law of the whole-plane GFF with the additive constant fixed so that $h_{e^{k+2}}(0) = 0$.  Applying the scale-invariance of the whole-plane GFF in the equality and Proposition~\ref{prop::gff_maximum} with $\xi=1/2$, $r = e^{-3}$, $\delta = C/3+ak/6-2$ in the inequality, we have for each $k \in \N$ that
\begin{align}
  & \pr{ \sup_{z \in B(0,e^{k+1}) \setminus B(0,e^k)} \left| h_{e^{k-1}}(z) - h_{e^{k+2}}(0)\right| \geq C + ak/2 } \notag\\
=& \pr{ \sup_{z \in B(0,e^{-1}) \setminus B(0,e^{-2})} |h_{e^{-3}}(z)| \geq C + ak/2 }
\leq  c_0 e^{6-C-ak/2}. \label{eqn::whole_plane_max_bound1}
\end{align}
Since $h_{e^{k+2}}(0)$ is a Gaussian random variable with mean $0$ and variance $k+2$, it follows from Lemma~\ref{lem::gaussian_tail_estimate} that
\begin{align}
 \pr{ |h_{e^{k+2}}(0)| \geq a k /2} \lesssim e^{-a^2 k/8}. \label{eqn::whole_plane_max_bound2}
\end{align}
Combining~\eqref{eqn::whole_plane_max_bound1} with~\eqref{eqn::whole_plane_max_bound2} and the Borel-Cantelli lemma implies that there a.s.\ exists $k_0 \in \N$ such that $k \geq k_0$ implies that
\[ \left\{ \sup_{z \in B(0,e^{k+1}) \setminus B(0,e^k)} |h_{e^{k-1}}(z)| \leq C + a k \right\} \quad\text{holds}.\]
This implies~\eqref{eqn::gff_max_decay} as $\sup_{z \in B(0,e^{k+1}) \setminus B(0,e^k)} |h_{e^{k-1}}(0)|$ is a.s.\ finite for all $1 \leq k \leq k_0$.

We will now extract the corresponding result for an $\alpha$-quantum cone.  Suppose that $h=h_1+\alpha\log|\cdot|$ where $h_1$ is an $\alpha$-quantum cone with $\alpha < Q$ and the embedding as in the statement of the corollary.  In this setting, $h|_{\D}$ has the same law as a whole-plane GFF with the additive constant fixed so that its average on $\partial \D$ is equal to $0$.  For each $z \in \C$ and $r > 0$ we let $h_{1,r}(z)$ be the average of $h_1$ on $\partial B(z,r)$.  Then we have that $h_{1,e^r}(0)$ for $r \geq 0$ evolves as $B_r -\alpha r$ where $B$ is a standard Brownian motion with $B_0 = 0$ conditioned so that $B_r + (Q-\alpha) r \geq 0$ for all $r \geq 0$.  Therefore $h_r(0)$ evolves as a standard Brownian motion $B$ with $B_0 = 0$ conditioned so that $B_r + (Q-\alpha)r \geq 0$ for all $r \geq 0$.  Note that such a process is stochastically dominated from above by a standard Brownian motion $B$ with $B_0 = 1$ conditioned so that $B_r + (Q-\alpha)r \geq 0$ for all $r \geq 0$ and that in this case we are conditioning on a positive probability event.  Such a process is also stochastically dominated from below by a standard Brownian motion $B$ with $B_0 = 0$ (with no conditioning).  Combining, it follows that~\eqref{eqn::whole_plane_max_bound2} holds in this setting.  Moreover, \eqref{eqn::whole_plane_max_bound1} also holds by using that the projection of $h$ onto the functions with mean-zero on all of the circles $\partial B(0,r)$ for $r > 0$ is given by the corresponding projection of a whole-plane GFF and the projection of $h$ onto the functions which are constant on such circles is stochastically dominated from above and below as we have just described.
\end{proof}

\begin{lemma}
\label{lem::gff_circ_covariance}
Suppose that we have the same setup as in Proposition~\ref{prop::gff_maximum}.  For each $\alpha > 0$ there exists a constant $c_0 > 0$ such that the following is true.  For all $z,w \in B(0,1/2)$ and $r,s \in (0,1/2)$ we have that
\[ \ex{ | h_r(z) - h_s(w)|^\alpha } \leq c_0 \left( \frac{ |(z,r) - (w,s)|}{r \wedge s} \right)^{\alpha/2}.\]
\end{lemma}
\begin{proof}
This is the content of \cite[Proposition~2.1]{HMP10} in the case of a GFF on a bounded domain $D \subseteq \C$ with Dirichlet boundary conditions.  The proof in the case of a whole-plane GFF is the same.
\end{proof}

\begin{proof}[Proof of Proposition~\ref{prop::gff_maximum}]
By combining Lemma~\ref{lem::gff_circ_covariance} (with a sufficiently large value of~$\alpha$) with Proposition~\ref{prop::kc_continuity} we have that the following is true.  For each $\varsigma > 0$, there exists $M > 0$ (random) such that for all $z,w \in B(0,1/2)$ and $r \in (0,1/2)$ we have that
\begin{equation}
\label{eqn::h_eps_continuity1}
 |h_r(z) - h_r(w)| \leq M r^{-1/2+\varsigma} |z-w|^{1/2-\varsigma}.
\end{equation}
Moreover, Lemma~\ref{lem::gff_circ_covariance} and Proposition~\ref{prop::kc_continuity} imply that, for each $\alpha > 0$, there exists a constant $c_0 > 0$ depending only on $\alpha$ such that:
\begin{equation}
\label{eqn::m_decay}
\pr{ M \geq t } \leq c_0 t^{-\alpha} \quad\text{for all}\quad t \geq 1.
\end{equation}

Fix $a_0 \in (0,1)$, $j \in \N$, and let $E_{j,a_0} = \{ M \geq e^{a_0 j/ 4} \}$.  On $E_{j,a_0}^c$,~\eqref{eqn::h_eps_continuity1} implies that
\begin{align}
|h_{e^{-j}}(z) - h_{e^{-j}}(w)|
&\leq M e^{j(1/2-\varsigma)} |z-w|^{1/2-\varsigma} \notag\\
&\leq e^{-a_0(1/4-\varsigma)j} \quad\text{for all}\quad |z-w| \leq e^{-(1+a_0)j}. \label{eqn::h_eps_continuity2}
\end{align}

Combining Lemma~\ref{lem::gaussian_tail_estimate} with the explicit form of the variance of $h_\epsilon$ \cite[Proposition~3.2]{DS08}, we have that there exists a constant $c_1 > 0$ such that for each $\alpha,\delta > 0$ that
\begin{equation}
\label{eqn::h_eps_decay}
\pr{ h_\epsilon(z) \geq (\alpha+\delta) \log \epsilon^{-1} } \leq c_1 \exp\left( - \frac{(\alpha+\delta)^2 (\log \epsilon^{-1})^2}{2 \log \epsilon^{-1}} \right) \leq c_1 \epsilon^{\alpha^2/2+\alpha \delta}
\end{equation}

We are now going to use~\eqref{eqn::h_eps_decay} to perform a union bound over a grid of points with spacing $e^{-(1+a_0)j}$.  The result will then follow by combining this with~\eqref{eqn::m_decay} and~\eqref{eqn::h_eps_continuity2}.

Let $\CC_{j,a_0} = \{ z \in e^{-j(1+a_0)} \Z^2 : z \in B(0,1/2)\}$.  Note that $|\CC_{j,a_0}| \asymp e^{2j(1+a_0)}$.  By~\eqref{eqn::h_eps_decay}, we have that
\begin{equation}
\label{eqn::hej_bound}
\pr{ h_{e^{-j}}(z) \geq (2+\delta)j } \leq c_1 e^{-2(1+\delta)j}.
\end{equation}
Consequently, by a union bound and~\eqref{eqn::hej_bound}, there exists a constant $c_2 > 0$ such that with
\begin{equation}
\label{eqn::hej_max_bound}
F_{j,a_0} = \left\{ \max_{z \in \CC_{j,a_0}} h_{e^{-j}}(z) \leq (2+\delta)j \right\} \quad\text{we have}\quad \pr{ F_{j,a_0}^c } \leq c_2 e^{2j(a_0-\delta)}.
\end{equation}
Suppose that $u \in B(0,1/2)$ is arbitrary.  Then there exists $z \in \CC_{j,a_0}$ such that $|u-z| \leq \sqrt{2} \cdot e^{-j(1+a_0)}$.  On $E_{j,a_0}^c$, by~\eqref{eqn::h_eps_continuity2} we have for a constant $c_3> 0$ that 
\[ |h_{e^{-j}}(z) - h_{e^{-j}}(u)| \leq c_3 e^{-a_0(1/4-\varsigma)j}.\]
Thus, on $E_{a_0}^c \cap F_{j,a_0}$, we have that
\[ h_{e^{-j}}(u) \leq c_3 e^{-a_0(1/4-\varsigma) j} + h_{e^{-j}}(z) \leq c_3 e^{-a_0(1/4-\varsigma)j} + (2+\delta) j.\]
That is,
\[ \sup_{u \in B(0,1/2)} h_{e^{-j}}(u) \leq c_3 e^{-a_0(1/4-\varsigma)j} + (2+\delta)j.\]
Choose $\alpha > 0$ sufficiently large so that, applying~\eqref{eqn::m_decay} with this value of $\alpha$, we have that
\begin{equation}
\label{eqn::ej_alpha}
\pr{ E_{j,a_0} } \leq c_0 e^{2j(a_0-\delta)}.
\end{equation}
By~\eqref{eqn::h_eps_decay} and~\eqref{eqn::ej_alpha}, we have that
\[ \pr{ E_{j,a_0}^c \cap F_{j,a_0} } \geq 1 - (c_0+c_2) e^{2j(a_0-\delta)} = 1 - c_4 e^{2j(a_0-\delta)}\]
where $c_4 = c_0 + c_2$.  This proves the result for $r = e^{-j}$.  The result for general $r \in (0,1/2)$ is proved similarly.
\end{proof}

\subsection{Continuous state branching processes}
\label{subsec::csbp_estimates}

The purpose of this section is to record a few elementary properties of continuous state branching processes (CSBPs); see \cite{legall1999spatial,kyp2006levy_fluctuations} for an introduction.

Suppose that $Y$ is a CSBP with branching mechanism $\psi$.  Recall that this means that $Y$ is the Markov process on $\R_+$ with $Y_0 = a \geq 0$ deterministic whose transition kernels are characterized by the property that
\begin{equation}
\label{eqn::csbp_def}
\ex{ \exp(-\lambda Y_t) \giv Y_s } = \exp(-Y_s u_{t-s}(\lambda)) \quad\text{for all}\quad t > s \geq 0
\end{equation}
where $u_t(\lambda)$, $t \geq 0$, is the non-negative solution to the differential equation
\begin{equation}
\label{eqn::csbp_diffeq}
\frac{\partial u_t}{\partial t}(\lambda) = -\psi(u_t(\lambda))\quad\text{for}\quad u_0(\lambda) = \lambda.
\end{equation}
Let
\begin{equation}
\label{eqn::phi_def}
\Phi(q) = \sup\{\theta \geq 0 : \psi(\theta) = q\}
\end{equation}
and let
\begin{equation}
\label{eqn::extinction}
\zeta = \inf\{t \geq 0: Y_t = 0\}
\end{equation}
be the extinction time for $Y$.  Then we have that \cite[Corollary~10.9]{kyp2006levy_fluctuations}
\begin{equation}
\label{eqn::csbp_exponential_integral}
\ex{ e^{-q \int_0^\zeta Y_s ds} } = e^{-\Phi(q) Y_0}.
\end{equation}

A $\psi$-CSBP can be constructed from a L\'evy process with only positive jumps and vice-versa \cite{lamp1967csbp} (see also \cite[Theorem~10.2]{kyp2006levy_fluctuations}).  Namely, suppose that $X$ is a L\'evy process with Laplace exponent $\psi$.  That is,
\[ \E[ e^{-\lambda X_t}] = e^{\psi(\lambda) t}.\]
Let
\begin{equation}
\label{eqn::levy_to_csbp}
s(t) = \int_0^t \frac{1}{X_u} du \quad\text{and}\quad s^*(t) = \inf\{ r > 0 : s(r) > t\}.
\end{equation}
Then the time-changed process $Y_t = X_{s^*(t)}$ is a $\psi$-CSBP.  That is, $Y_{s(t)} = X_t$.  Conversely, if $Y$ is a $\psi$-CSBP and we let
\begin{equation}
\label{eqn::csbp_to_levy}
t(s) = \int_0^s Y_u du \quad\text{and}\quad t^*(s) = \inf\{ r > 0 : t(r) > s\}
\end{equation}
then $X_s = Y_{t^*(s)}$ is a L\'evy process with Laplace exponent $\psi$.  That is, $X_{t(s)} = Y_s$.

We will be interested in the particular case that $\psi(u) = u^\alpha$ for $\alpha \in (1,2)$.  For this choice, we note that
\begin{equation}
\label{eqn::csbp_u_form}
u_t(\lambda) = \left( \lambda^{1-\alpha} + (\alpha-1)t\right)^{1/(1-\alpha)}.
\end{equation}
Combining~\eqref{eqn::csbp_def} and~\eqref{eqn::csbp_u_form} implies that $u^\alpha$-CSBPs (which we will also later refer to as \emph{$\alpha$-stable CSBPs}) satisfy a certain scaling property.  Namely, if $Y$ is a $u^\alpha$-CSBP starting from $Y_0$ then $\wt{Y}_t = \beta^{1/(1-\alpha)}Y_{\beta t}$ is a $u^\alpha$-CSBP starting from $\wt{Y}_0 = \beta^{1/(1-\alpha)} Y_0$.  In particular, if $Y$ is a $u^{3/2}$-CSBP starting from $Y_0$ then $\wt{Y}_t = \beta^{-2} Y_{\beta t}$ is a $3/2$-stable CSBP starting from $\wt{Y}_0 = \beta^{-2} Y_0$.

\subsection{Tail bounds for stable processes and the Poisson law}
\label{subsec::auxiliary}

\subsubsection{Supremum of an $\alpha$-stable process}
\label{subsubsec::alpha_stable_supremum}

\begin{lemma}
\label{lem::stable_maximum}
Suppose that $X$ is an $\alpha$-stable process with $X_0 = 0$ and without positive jumps.  For each $t \geq 0$, let $S_t = \sup_{s \in [0,t]} X_s$.  There exist constants $c_0,c_1 > 0$ such that
\begin{equation}
\label{eqn::stable_maximum_tail} \pr{ S_t \geq u } \leq c_0 \exp(- c_1 t^{-1/\alpha} u).
\end{equation}
\end{lemma}
\begin{proof}
For each $t \geq 0$, we let $S_t = \sup_{s \in [0,t]} X_s$.  Fix $q > 0$ and let $\tau(q)$ be an exponential random variable with parameter $q$ which is sampled independently of $X$.  Let $\Phi(\lambda) = a_0^{-1/\alpha} \lambda^{1/\alpha}$ be the inverse of the Laplace exponent $\psi(\lambda) = a_0 \lambda^\alpha$ of $X$.  By \cite[Chapter~VII, Corollary~2]{bertoin96levy}, we have that $S_{\tau(q)}$ has the exponential distribution with parameter $\Phi(q)$.  In particular, we have that
\[ \pr{ S_{\tau(q)} \geq u } = \exp(-\Phi(q) u).\]
Therefore we have that
\begin{align*}
     \pr{ S_{q^{-1}} \geq u }
&\leq \pr{ S_{\tau(q)} \geq u \giv \tau(q) \geq q^{-1} }\\
& \leq c_0 \pr{ S_{\tau(q)} \geq u }\\
&\leq c_0 \exp(-\Phi(q) u)
\end{align*}
where $c_0 = 1/\pr{ \tau(q) \geq q^{-1}} = e$.
\end{proof}

\subsubsection{Poisson deviations}
\label{subsubsec::poisson_deviations}

\begin{lemma}
\label{lem::poisson_deviation}
If $Z$ is a Poisson random variable with mean $\lambda$ then for each $\alpha \in (0,1)$ we have that
\begin{equation}
\label{eqn::poisson_deviation_lbd}
\pr{ Z  \leq \alpha \lambda } \leq \exp\big( \lambda( \alpha- \alpha \log \alpha -1 ) \big).
\end{equation}
Similarly, for each $\alpha > 1$ we have that
\begin{equation}
\label{eqn::poisson_deviation_ubd}
\pr{ Z  \geq \alpha \lambda } \leq \exp\big( \lambda( \alpha- \alpha \log \alpha -1 ) \big).
\end{equation}
\end{lemma}
\begin{proof}
Recall that the moment generating function for a Poisson random variable with mean $\lambda$ is given by $\exp(\lambda(e^t-1))$.  Therefore the probability that a Poisson random variable $Z$ of mean $\lambda$ is smaller than a constant $c$ satisfies for each $\beta > 0$ the inequality
\[ \pr{ Z \leq c } = \pr{ e^{-\beta Z} \geq e^{-\beta c} } \leq e^{\beta c} \ex{ e^{-\beta Z} } = \exp(\beta c + \lambda (e^{-\beta} -1)).\]
If we take $c = \alpha \lambda$, the above becomes
\[ \pr{ Z \leq \alpha \lambda } \leq \exp(\lambda (\alpha \beta + e^{-\beta} -1 ) ).\]
Note that $\beta \mapsto \alpha \beta  + e^{-\beta} - 1$ is minimized with $\beta = - \log \alpha$ and taking $\beta$ to be this value implies the lower bound.  The upper bound is proved similarly.
\end{proof}

\section{Quantum boundary length and area bounds}
\label{sec::boundary_area}

The purpose of this section is to derive tail bounds for the quantum boundary length of the outer boundary of a $\QLE(8/3,0)$ metric ball (Section~\ref{subsec::boundary_length}), for the quantum area surrounded by a $\QLE(8/3,0)$ (Section~\ref{subsec::area}), and also to establish the regularity of the quantum area measure on a $\gamma$-quantum cone (Section~\ref{subsec::liouville_holder}).  The estimates established in this section will then feed into the Euclidean size bounds for $\QLE(8/3,0)$ derived in Section~\ref{sec::euclidean_inner_outer_radius_bounds}.

\subsection{Quantum boundary length of $\QLE(8/3,0)$ hull}
\label{subsec::boundary_length}

\begin{lemma}
\label{lem::boundary_length_radius}
Suppose that $(\C,h,0,\infty)$ is a $\sqrt{8/3}$-quantum cone, let $(\Gamma_r)$ be the $\QLE(8/3,0)$ starting from $0$ with the quantum distance parameterization, and for each $r > 0$ let $B_r$ be the quantum boundary length of the outer boundary of $\Gamma_r$.  There exist constants $c_0,\ldots,c_3 > 0$ such that for each $r > 0$ and $t > 1$ we have both
\begin{equation}
\label{eqn::boundary_length_deviations}
\pr{ \sup_{0 \leq s \leq r} B_s \leq r^2/t} \leq c_0 e^{-c_1 t^{1/2}} \quad\text{and}\quad \pr{\sup_{0 \leq s \leq r} B_s \geq r^2 t} \leq c_2 e^{- c_3 t}.
\end{equation}
\end{lemma}

Recall from the construction of $\QLE(8/3,0)$ on a $\sqrt{8/3}$-quantum cone given in Section~\ref{subsec::qle_on_cone} that $B$ evolves as the time-reversal of a $3/2$-stable CSBP.  Consequently, Lemma~\ref{lem::boundary_length_radius} is in fact a statement about $3/2$-stable CSBPs.  In order to prove Lemma~\ref{lem::boundary_length_radius}, we will make use of the scaling property for $3/2$-stable CSBPs explained at the end of Section~\ref{subsec::csbp_estimates}.  Namely, if $Y$ is a $3/2$-stable CSBP starting from $Y_0 = x$ and $\alpha > 0$ then $\alpha^{-2} Y_{\alpha t}$ is a $3/2$-stable CSBP starting from $\alpha^{-2} x$.  We will also make use of the relationship between a $3/2$-stable L\'evy process with downward jumps conditioned to be non-negative and the law of a $3/2$-stable L\'evy process run until the first time that it hits $0$.  Results of this type are explained in \cite[Chapter~VII, Section~4]{bertoin96levy}.

\begin{proof}[Proof of Lemma~\ref{lem::boundary_length_radius}]
Let $Y$ be a $3/2$-stable CSBP and let $\zeta = \inf\{t > 0 : Y_t = 0\}$ starting from $Y_0$.  For each $x \geq 0$, we let $\prstart{x}{\cdot}$ be the law under which $Y_0 = x$.

In order to prove the first inequality of~\eqref{eqn::boundary_length_deviations} it suffices to show that the following is true.  There exist constants $c_0,c_1 > 0$ such that the probability that there is an interval of length at least $r$ during which $Y$ is contained in $[0,r^2/t]$ is at most $c_0 e^{-c_1 t^{1/2}}$ under the law $\p^x$ with $x \geq r^2/t$.  By applying scaling as described at the end of Section~\ref{subsec::csbp_estimates}, it in turn suffices to show that the probability of the event $E$ that there is an interval of length at least $t^{1/2}$ during which $Y$ is contained in $[0,1]$ is at most $c_0 e^{-c_1 t^{1/2}}$ under the law $\p^x$ with $x \geq 1$.  

To see that this is the case, we define stopping times inductively as follows.  Let $\tau_0 = \inf\{ t \geq 0: Y_t \leq 1\}$ and $\sigma_0 = \zeta \wedge \inf\{ t \geq \tau_0 : Y_t \geq 2\}$.  Assuming that we have defined stopping times $\tau_0,\ldots,\tau_k$ and $\sigma_0,\ldots,\sigma_k$ for some $k \in \N$, we let $\tau_{k+1} = \inf\{ t \geq \sigma_k : Y_t \leq 1\}$ and $\sigma_{k+1} = \zeta \wedge \inf\{ t \geq \tau_{k+1} : Y_t \geq 2\}$.  Let $N = \min\{ k : Y_{\sigma_k} = 0\}$.  Then $N$ is distributed as a geometric random variable.   Note that there exist constants $c_0,c_1$ such that for each $k$, we have that $\p[ \sigma_k - \tau_k \geq t^{1/2} \giv N \geq k] \leq c_0  e^{- c_1 t^{1/2}}$ because in each round of length $1$, $Y$ has a uniformly positive chance of exiting $(0,2)$.  Observe that
\begin{align}
       \p[ E ]
&\leq \sum_k \p[ \sigma_k - \tau_k \geq t^{1/2},\ N \geq k] \notag\\
&= \sum_k \p[ \sigma_k -\tau_k \geq t^{1/2} \giv N \geq k]\p[ N \geq k] \notag\\
&\leq c_0 e^{- c_1 t^{1/2}} \sum_k \p[ N \geq k] = \E[ N] c_0 e^{- c_1 t^{1/2}}. \label{eqn::boundary_length_deviations_final_bound}
\end{align}
The first inequality of~\eqref{eqn::boundary_length_deviations} thus follows by possibly increasing the value of $c_0$.

We will now prove the second inequality of~\eqref{eqn::boundary_length_deviations}.  It suffices to show that there exist constants $c_2,c_3 > 0$ such that the probability that there is an interval of length at most $r$ in which $Y$ starts at $r^2 t$ and then exits at $0$ is at most $c_2 e^{-c_3 t}$.  By scaling, it suffices to show that there exist constants $c_2,c_3 > 0$ such that the probability of the event $E$ that there is an interval of length at most $t^{-1/2}$ in which $Y$ starts at $1$ and then exits at $0$ is at most $c_2 e^{- c_3 t}$.  To show that this is the case we assume that we have defined stopping times $\sigma_k,\tau_k$ and $N$ as in our proof of the first inequality of~\eqref{eqn::boundary_length_deviations}.  Note that (recall~\eqref{eqn::csbp_def} and~\eqref{eqn::csbp_u_form})
\begin{equation}
\label{eqn::csbp_hitting}
\prstart{x}{ \zeta \leq v } = \lim_{\lambda \to \infty} \E^x[\exp(-\lambda Y_v) ] = \lim_{\lambda \to \infty} \exp(-x u_v(\lambda)) = \exp(-4 x/v^2).
\end{equation}
Evaluating~\eqref{eqn::csbp_hitting} at $x = 1$ and $v = t^{-1/2}$ implies that there exist constants $c_2,c_3 > 0$ such that $\p[ \sigma_k - \tau_k \leq t^{-1/2} \giv N \geq k] \leq c_2 e^{- c_3 t}$.  Thus the second inequality in~\eqref{eqn::boundary_length_deviations} follows the calculation in~\eqref{eqn::boundary_length_deviations_final_bound} used to complete the proof of the first inequality of~\eqref{eqn::boundary_length_deviations}.
\end{proof}

\subsection{Quantum area of $\QLE(8/3,0)$ hull}
\label{subsec::area}

\begin{lemma}
\label{lem::metric_hull_volume}
Let $(\C,h,0,\infty)$ be a $\sqrt{8/3}$-quantum cone, let $(\Gamma_r)$ be the $\QLE(8/3,0)$ growing from $0$ with the quantum distance parameterization, and for each $r > 0$ let $A_r$ be the quantum area of $\Gamma_r$.  There exist constants $a_0,c_0,c_1 > 0$ such that
\begin{equation}
\label{eqn::metric_hull_area_deviations}
\pr{ A_r  \leq r^4 / t } \leq c_0 \exp(-c_1 t^{a_0}) \quad\text{for all}\quad r > 0,\ t \geq 1.
\end{equation}
\end{lemma}

Before we give the proof of Lemma~\ref{lem::metric_hull_volume}, we first need to record the following fact.

\begin{lemma}
\label{lem::quantum_disk_area}
There exists a constant $c_0 > 0$ such that the following is true.  Suppose that $(\strip,h)$ has the law of a quantum disk with quantum boundary length $\ell$.  Then
\begin{equation}
\label{eqn::quantum_disk_area}
\E[ \mu_h(\CS)] = c_0 \ell^2.
\end{equation}
\end{lemma}
\begin{proof}
Recall that the law of a quantum disk with boundary length $\ell$ can be sampled from by first picking $(\strip,h)$ from the law of the unit boundary length quantum disk and then taking the field $h + 2\gamma^{-1} \log \ell$.  Note that adding $2 \gamma^{-1} \log \ell$ to the field has the effect of multiplying quantum boundary lengths (resp.\ areas) by $\ell$ (resp.\ $\ell^2$).  \cite[Proposition~6.5]{quantum_spheres} implies that the law of a quantum disk with given boundary length weighted by its quantum area makes sense as a probability measure which is equivalent to the quantum area having finite expectation.  Combining this with the aforementioned scaling implies the result.
\end{proof}

\begin{proof}[Proof of Lemma~\ref{lem::metric_hull_volume}]
For each $r > 0$, we let $B_r$ be the quantum length of the outer boundary of $\Gamma_r$.  Fix $r > 0$.  Then we know from Lemma~\ref{lem::boundary_length_radius} that there exist constants $c_0,c_1 > 0$ such that
\begin{equation}
\label{eqn::mhv_bl_small}
\pr{ \sup_{0 \leq s \leq r} B_s \leq r^{2}/t } \leq c_0 \exp(-c_1 t^{1/2}) \quad\text{for each}\quad  t \geq 1.
\end{equation}
Suppose that $X$ is a $3/2$-stable L\'evy process with only downward jumps and let $\prstart{x}{\cdot}$ be the law under which $X_0 = x$.  Let $W$ have law $\pr{\cdot \giv X \geq 0}$ (see \cite{bertoin96levy} for a careful definition of this law) and write $\prstart{w}{\cdot}$ for the law of $W$ under which $W_0 = w$.  Then we know that the law of $B$ is equal to the law of $W$ under $\p^0$ after performing the time change as in~\eqref{eqn::csbp_to_levy} (recall the importance of this time-change in the context of $\QLE(8/3,0)$, as discussed around~\eqref{eqn::qle_time_change}).  Fix $t \geq 1$.  It then follows from~\eqref{eqn::mhv_bl_small} that the probability that $W$ hits $r^2 / t$ before the time which corresponds to when $\Gamma$ has quantum radius $r$ is at least $1- c_0 \exp(-c_1 t^{1/2})$.

We are now going to argue that, by possibly adjusting the values of $c_0,c_1 > 0$, we have that the probability that $W$ takes less than $r^3 / t^3$ units of time to hit $r^2/t$ is at most $c_0 \exp(-c_1 t)$.  To see this, we let $\tau$ be the first time that $W$ hits $r^2/(2t)$.  Then it suffices to show that the probability that $W$ starting from $r^2/(2t)$ takes less than $r^3 / t^3$ time to hit $r^2 / t$ is at most $c_0 \exp(-c_1 t)$.  Since the probability that a $3/2$-stable L\'evy process with only downward jumps starting from $r^2 / (2t)$ to hit $r^2/t$ before hitting $0$ is uniformly positive in $r > 0$ and $t \geq 1$ (by scaling), it suffices to show that the probability that $X$ starting from $r^2 / (2t)$ hits $r^2/t$ in less than $r^3 / t^3$ time is at most $c_0 \exp(-c_1 t)$.  This, in turn, follows from Lemma~\ref{lem::stable_maximum}.

Suppose that $0 < a < b < \infty$.  The number of downward jumps made by $X$ in time $r^3 / t^3$ of size between $a$ and $b$ is distributed as a Poisson random variable with mean given by a constant times
\begin{equation}
\label{eqn::ball_counts}
\frac{r^3}{t^3} \int_a^b s^{-5/2} ds = \frac{2}{3} \cdot \frac{r^3}{t^3} (a^{-3/2} - b^{-3/2}).
\end{equation}
In particular, the number of jumps made by $X$ in time $r^3/t^3$ of size between $\tfrac{1}{2} r^{2} t^{-8/3}$ and $r^2 t^{-8/3}$ is Poisson with mean proportional to $t$.  Therefore it follows from Lemma~\ref{lem::poisson_deviation} that there exist constants $c_2,c_3 > 0$ such that the probability of the event that the number of such jumps is fewer than $1/2$ its mean is at most $c_2 \exp(-c_3 t)$.  It follows from the argument of the previous paragraph that the same holds for $W$.  We note that each of the jumps of $W$ corresponds to a quantum disk cut out by $\Gamma|_{[0,r]}$ and the size of the jump corresponds to the quantum boundary length of the disk.  Since the law of a quantum disk with boundary length $\ell$ can be obtained from the law of a quantum disk with boundary length $1$ and then adding $2 \gamma^{-1} \log \ell$ to the field, we have that the following is true.  There exists $a > 0$ so that the probability that fewer than $1/2$ of these disks have quantum area which is larger than $a$ times the conditional expectation of the quantum area given its quantum boundary length is at most $c_4 \exp(-c_5 t)$ where $c_4,c_5 > 0$ are constants.  By Lemma~\ref{lem::quantum_disk_area}, the conditional mean of the quantum area of such a quantum disk given its quantum boundary length is proportional to $r^4 t^{-16/3}$ (when the boundary length is proportional to $r^2 t^{-8/3}$), combining all of our estimates implies~\eqref{eqn::metric_hull_area_deviations}.
\end{proof}

\subsection{Regularity of the quantum area measure on a $\gamma$-quantum cone}
\label{subsec::liouville_holder}

The purpose of this section is to record an upper bound for the quantum area measure associated with a $\gamma$-quantum cone.

\begin{proposition}
\label{prop::cone_holder}
Fix $\gamma \in (0,2)$ and let
\begin{equation}
\label{eqn::alpha_def}
\alpha = \frac{(\gamma^2-4)^2}{4(4+\gamma^2)}.
\end{equation}
Suppose that $(\C,h,0,\infty)$ is a $\gamma$-quantum cone with the circle average embedding.  Fix $\zeta \in (0,\alpha)$ and let $H_{R,\zeta}$ be the event that for every $z \in \C$ and $s \in (0,R)$ such that $\ball{z}{s} \subseteq \D$ we have that $\mu_h(\ball{z}{s}) \leq s^{\alpha-\zeta}$.  Then $\pr{ H_{R,\zeta} } \to 1$ as $R \to 0$ with $\zeta > 0$ fixed.
\end{proposition}
\begin{proof}
We first suppose that $h$ is a whole-plane GFF on $\C$ with the additive constant fixed so that $h_1(0) = 0$ and let $\mu_h$ be the associated quantum area measure.  Fix $q \in (0,4/\gamma^2)$.  Then \cite[Proposition~3.7]{rv2010revisited} implies that there exists a constant $c_q > 0$ such that with
\[ \xi(q) = \left(2+\frac{\gamma^2}{2}\right) q - \frac{\gamma^2}{2} q^2\]
we have that
\begin{equation}
\label{eqn::area_moment_bound}
\E[ \mu_h(B(z,s))^q] \leq c_q s^{\xi(q)}.
\end{equation}
Let $\alpha$ be as in~\eqref{eqn::alpha_def} and fix $\zeta \in (0,\alpha)$.  It therefore follows from~\eqref{eqn::area_moment_bound} and Markov's inequality that
\begin{equation}
\label{eqn::quantum_area_bound_moment}
\pr{ \mu_h(B(z,s)) \geq s^{\alpha-\zeta} } \leq c_q s^{\xi(q) - (\alpha-\zeta)q}.
\end{equation}
Let
\[ q^* = \frac{4+\gamma^2}{2\gamma^2} \in \left(0,\frac{4}{\gamma^2} \right)\]
be the value of $q$ that maximizes $\xi(q)$.  Note that 
\[ \alpha = \frac{\xi(q^*) -2}{q^*}\]
so that the exponent on the right side of~\eqref{eqn::quantum_area_bound_moment} with $q=q^*$ is strictly larger than $2$.  Therefore applying the Borel-Cantelli lemma along with~\eqref{eqn::quantum_area_bound_moment} on a dyadic partition of $\D$ implies the result in the case of the whole-plane GFF.

We are now going to deduce the result in the case of a $\gamma$-quantum cone from the result in the case of the whole-plane GFF using absolute continuity.  We suppose now that $(\C,\wt{h},0,\infty)$ is a $\gamma$-quantum cone with the circle average embedding.  If $B \subseteq \D$ is any box with positive distance to $0$, we have that the law of $h|_B$ is mutually absolutely continuous with respect to the law of $\wt{h}|_B$.  In particular, if we define $\wt{H}_{R,\zeta}^B$ in the same manner as $H_{R,\zeta}$ except with $\mu_{\wt{h}}$ restricted to $B$ in place of $\mu$ then we have that $\pr{ \wt{H}_{R,\zeta}^B} \to 1$ as $R \to 0$ with $\zeta \in (0,\alpha)$ fixed.

Let $\wt{\eta}'$ be a space-filling $\SLE_{\kappa'}$ from $\infty$ to $\infty$ sampled independently of $\wt{h}$ and then reparameterized by quantum area as assigned by $\wt{h}$.  That is, we have that $\mu_{\wt{h}}(\wt{\eta}'([s,t])) = t-s$ for all $s < t$.  We normalize time so that $\wt{\eta}'(0) = 0$.  Then we know from \cite[Theorem~1.13]{dms2014mating} that the joint law of $(\wt{h},\wt{\eta}')$ is the same as the joint law of $(\wt{h}(\cdot+\wt{\eta}'(t)),\wt{\eta}'(\cdot+t)-\wt{\eta}'(t))$ (i.e., the field and path after recentering so that $\wt{\eta}'(t)$ becomes the origin) and then rescaling so that the new field has the circle average embedding. 

Note that for $t > 0$ small we have that $\wt{\eta}'(t)$ has probability arbitrarily close to $1$ of being in a box $B$ as above with rational coordinates.  The result therefore follows by scaling.
\end{proof}

\section{Euclidean size bounds for $\QLE(8/3,0)$}
\label{sec::euclidean_inner_outer_radius_bounds}

The purpose of this section is to establish bounds for the Euclidean size of a $\QLE(8/3,0)$ process growing on a $\sqrt{8/3}$-quantum cone.  The lower bound is obtained in Section~\ref{subsec::diameter_lower_bound} by combining Proposition~\ref{prop::cone_holder} established just above with the lower bound on the quantum area cut off from $\infty$ by a $\QLE(8/3,0)$ established in Lemma~\ref{lem::metric_hull_volume}.  In Section~\ref{subsec::qle_diameter_upper_bound} we will first give an upper bound on the Euclidean diameter of a $\QLE(8/3,0)$ and then combine this with the results of Section~\ref{subsec::liouville_holder} to obtain an upper bound on the quantum area of the hull of a $\QLE(8/3,0)$.

\subsection{Diameter lower bound}
\label{subsec::diameter_lower_bound}

\begin{proposition}
\label{prop::euclidean_diameter_lower_bound}
Suppose that $(\C,h,0,\infty)$ is a $\sqrt{8/3}$-quantum cone with the circle average embedding.  Let $H_{R,\zeta}$ be the event from Proposition~\ref{prop::cone_holder}.  There exist constants $c_0,\ldots,c_3 > 0$ depending only on $R, \zeta$ such that the following is true.  Let $(\Gamma_r)$ be the hull of a $\QLE(8/3,0)$ process starting from $0$ parameterized by quantum distance.  For each $r \in (0,R)$ we have that
\begin{equation}
\label{eqn::euclidean_diameter_lower_bound}
\pr{ \diam(\Gamma_r) \leq r^{c_0},\ H_{R,\zeta} } \leq c_1 \exp(-c_2 r^{-c_3}).
\end{equation}
\end{proposition}
\begin{proof}
This follows by combining~\eqref{eqn::metric_hull_area_deviations} of Lemma~\ref{lem::metric_hull_volume} with the definition of $H_{R,\zeta}$.
\end{proof}

\subsection{Diameter upper bound}
\label{subsec::qle_diameter_upper_bound}

\begin{proposition}
\label{prop::ball_size_ubd}
Suppose that $(\C,h,0,\infty)$ is a $\sqrt{8/3}$-quantum cone with the circle average embedding.  Let $(\Gamma_r)$ be a $\QLE(8/3,0)$ process starting from $0$ with the quantum distance parameterization.  For each $p > 0$ there exists a constant $a_0 = a_0(p) > 0$ so that
\begin{equation}
\label{eqn::prob_diam_bound}
\pr{ \diam(\Gamma_r) \geq r^{a_0}} = O(r^p) \quad\text{as}\quad r \to 0.
\end{equation}
Moreover, there exist constants $c_1 > 0$ and $a_1 > 4$ such that
\begin{equation}
\label{eqn::fourth_moment_diam_bound}
\ex{ \diam(\Gamma_r)^4 \one_{\{ \diam(\Gamma_r) \leq 1\}}} \leq c_1 r^{a_1} \quad\text{for all}\quad r > 0.
\end{equation}
\end{proposition}

The part of Proposition~\ref{prop::ball_size_ubd} asserted in~\eqref{eqn::prob_diam_bound} will be used in the proof of Theorem~\ref{thm::continuity} and Theorem~\ref{thm::metric_completion}.  The part which is asserted in~\eqref{eqn::fourth_moment_diam_bound} will be used in the proof of the main result of \cite{qle_determined}.

\begin{figure}[ht!]
\begin{center}
\includegraphics[scale=0.85]{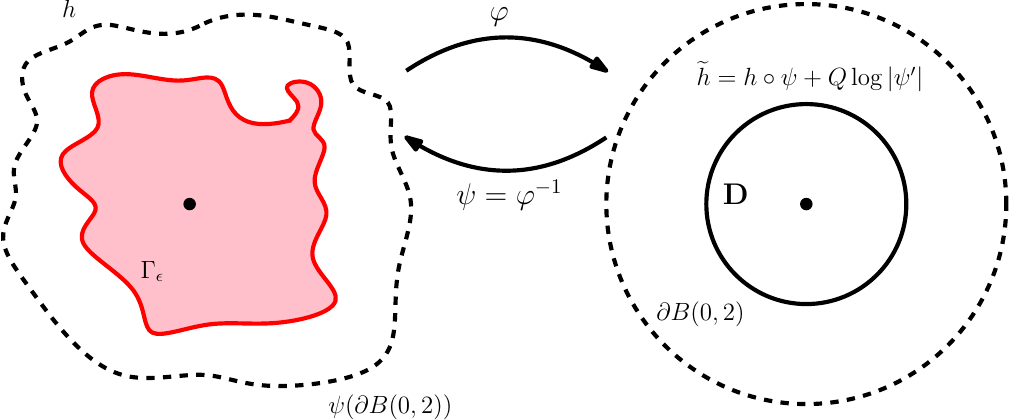}	
\end{center}
\caption{\label{fig::qle_size_ubd_argument} Illustration of the argument used to prove Proposition~\ref{prop::ball_size_ubd}.  Shown on the left is a $\QLE(8/3,0)$ process $\Gamma$ on a $\sqrt{8/3}$-quantum cone $(\C,h,0,\infty)$ starting from $0$ run up to quantum distance time $\epsilon > 0$.  If $\ell$ denotes the quantum boundary length of $\Gamma_\epsilon$, then the conditional law of the surface parameterized by $\C \setminus \Gamma_\epsilon$ is given by $\qconeUnex^\ell$.  The map $\varphi$ takes $\C \setminus \Gamma_\epsilon$ to $\C \setminus \ol{\D}$ which fixes and has positive derivative at $\infty$.  To bound $\diam(\Gamma_\epsilon)$, it suffices to bound the Euclidean length of $\psi(\partial B(0,2))$ where $\psi = \varphi^{-1}$.  By solving for $\log |\psi'|$ in the change of coordinates formula $\wt{h} = h \circ \psi + Q \log|\psi'|$ for quantum surfaces and using that $\log|\psi'|$ is harmonic, it in turn suffices to bound the extremes of the harmonic extensions of $h$ and $\wt{h}$ from $\partial \Gamma_\epsilon$ to $\C \setminus \Gamma_\epsilon$ and from $\partial \D$ to $\C \setminus \ol{\D}$, respectively.}
\end{figure}

We will divide the proof of Proposition~\ref{prop::ball_size_ubd} into three steps.  The first step, carried out in Section~\ref{subsubsec::quantum_boundary_length_free_boundary_tail_bounds}, is to give a tail bound for the quantum boundary length of $\partial \cyl_+$ assigned by a free boundary GFF on $\cyl_+$ with the additive constant fixed so that the average on $\partial \cyl_+$ is equal to $0$.  Using the resampling characterization of the unexplored region of a $\sqrt{8/3}$-quantum cone established in \cite{quantum_spheres}, we will then deduce from this in Section~\ref{subsubsec::quantum_boundary_length_qcone_tail_bounds} that it is very unlikely for the harmonic extension of the values of the field from $\partial \cyl_+$ to $\cyl_+$ restricted to $\cyl_+ + r$ to be large where $r > 0 $ is fixed.  We will then use this result to complete the proof of Proposition~\ref{prop::ball_size_ubd} in Section~\ref{subsubsec::proof_of_ball_size_ubd}.

Before we proceed to the proof, we will first deduce an upper bound on the quantum area in the hull of a $\QLE(8/3,0)$.

\begin{corollary}
\label{cor::qle_area_bound}
Let $H_{R,\zeta}$ be as in Proposition~\ref{prop::cone_holder}.  For every $\beta > 0$ there exists $r_0,\alpha \in (0,1)$ such that the following is true.  Let $(\C,h,0,\infty)$ be a $\sqrt{8/3}$-quantum cone with the circle average embedding and let $(\Gamma_r)$ be a $\QLE(8/3,0)$ starting from $0$ with the quantum distance parameterization.  For each $r > 0$, let $A_r$ be the quantum area cut off by $\Gamma_r$ from $\infty$.  Then there exists a constant $c_0 > 0$ such that
\[ \pr{ A_r \geq r^\alpha,\ H_{R,\zeta}} \leq c_0 r^\beta \quad\text{for all}\quad r \in (0,r_0).\]
\end{corollary}
\begin{proof}
Fix $\beta > 0$ and let $\delta = \beta$ so that the assertion of~\eqref{eqn::prob_diam_bound} from Proposition~\ref{prop::ball_size_ubd} holds with probability $c_3 r^\beta$.  Then it is easy to see from the definition of $H_{R,\zeta}$ that the result holds for $r_0 = R^{1/a_0}$ and a value of $\alpha \in (0,1)$ sufficiently small.
\end{proof}

\subsubsection{Quantum boundary length tail bounds for the free boundary GFF}
\label{subsubsec::quantum_boundary_length_free_boundary_tail_bounds}

We turn to establish a tail bound for the quantum boundary length assigned by a free boundary GFF on $\cyl_+$ to $\partial \cyl_+$ where the additive constant is set so that its average on $\partial \cyl_+$ is equal to $0$.  This result is analogous to \cite[Lemma~4.5]{DS08} and we will make use of a similar strategy for the proof.

\begin{proposition}
\label{prop::gff_boundary_length}
Suppose that $h$ is a free boundary GFF on $\cyl_+$ with the additive constant fixed so that its average on $\partial \cyl_+$ is equal to $0$.  Fix $\gamma \in (0,2)$.  There exist constants $c_0,c_1 > 0$ such that the following is true.  Let $B$ be the quantum boundary length of $\partial \cyl_+$ and $\wt{B} = 2 \gamma^{-1} \log B$.  Then
\begin{equation}
\label{eqn::gff_boundary_length_small}
\pr{ \wt{B} \leq \eta } \leq c_0 e^{-c_1 \eta^2} \quad\text{for all}\quad \eta \in \R_-.
\end{equation}
Let $\Fh$ be the function which is harmonic in $\cyl_+$ with boundary values given by those of $h$ on $\partial \cyl_+$.  Then the same is also true if we let $r > 0$ and then fix the additive constant for $h$ so that $\sup_{z \in (\cyl_+ + r)} \Fh(z) = 0$.
\end{proposition}

We need three preparatory lemmas in order to establish Proposition~\ref{prop::gff_boundary_length}.

\begin{lemma}
\label{lem::quadratic_exponential_decay}
Suppose that $f \colon \R_- \to [0,1]$ is an increasing function such that there exist constants $c_0,c_1 > 0$, $\alpha \in (1/\sqrt{2},1)$, and $\eta_0 \in \R_-$ such that
\begin{equation}
\label{eqn::f_assump}
f(\eta) \leq e^{-c_0 \eta^2} + (f(\alpha \eta - c_1))^2 \quad\text{for all}\quad \eta \leq \eta_0.
\end{equation}
Then there exists a constant $c_2 > 0$ and $\eta_1 \in \R_-$ such that
\begin{equation}
\label{eqn::f_conclusion}
f(\eta) \leq e^{-c_2 \eta^2} \quad\text{for all}\quad \eta \leq \eta_1.
\end{equation}
\end{lemma}
\begin{proof}
We set $a_K = \eta$ and then inductively set $a_{k-1} = \alpha a_k - c_1$ for $k \geq 1$ where we have chosen $K$ so that $a_0 \geq -2c_1$.  Let
\[ q_k = \frac{f(a_k)}{e^{- c_0 a_k^2}} \quad\text{for each}\quad k \in \N.\]
We have that
\begin{align*}
q_k
&\leq 1 + q_{k-1}^2 e^{-c_0 (2a_{k-1}^2 - a_k^2)} \quad\text{(by~\eqref{eqn::f_assump})}\\
&\leq 1 + q_{k-1}^2 e^{-c_0 (2\alpha^2-1) a_k^2}.
\end{align*}
It is not difficult to see from this that $q_k$ is bounded by a constant which does not depend on $\eta$, from which the result follows.
\end{proof}

\begin{lemma}
\label{lem::gff_infimum}
Suppose that $h$ is a GFF with zero boundary conditions on a bounded domain $D$, $U \subseteq D$ is open with $\dist(\partial U,\partial D) > 0$, and $K \subseteq U$ is compact.  Let $\wt{h}$ be the projection of $h$ onto the subspace of functions in $H(D)$ which are harmonic on $U$.  There exist constants $c_0,c_1 > 0$ depending only on $U$, $K$, and $D$ such that
\begin{equation}
\label{eqn::gff_harm_tail_bound}
\pr{ \sup_{z \in K} |\wt{h}(z)| \geq \eta } \leq c_0 e^{-c_1 \eta^2} \quad\text{for all}\quad \eta \geq 0.
\end{equation}
The same is also true if $h$ is a whole-plane GFF with the additive constant fixed so that its average on $\partial \D$ is equal to $0$, $U \subseteq \D$ is open with $\dist(U,\partial \D) > 0$, and $K \subseteq U$ is compact.
\end{lemma}
\begin{proof}
We will give the proof in the case that $h$ is a GFF on a bounded domain $D$ with zero-boundary conditions.  The proof in the case of the whole-plane GFF is analogous.

Fix $r_0 > 0$ such that $z \in K$ implies that $B(z,r_0)$ has distance at least $r_0$ to $\partial U$ and let $r_1 = \tfrac{1}{2} r_0$ and $r_2 = \tfrac{1}{2} r_1$.  Fix $z \in K$ and, for $w \in B(z,r_1)$, let $\mu_{z,w}$ denote harmonic measure in $B(z,r_1)$ as seen from $w$.  Then we can write
\[ \wt{h}(w) = \int \wt{h}(u) d\mu_{z,w}(u)\]
and therefore
\[ |\wt{h}(w)| \leq \int |\wt{h}(u)| d\mu_{z,w}(u).\]
Note that there exists a constant $c_0 > 0$ such that
\[ \sup_{w \in B(z,r_2)} |\wt{h}(w)| \leq c_0 \int |\wt{h}(u)| d\mu_{z,z}(u).\]
By the compactness of $K$, it suffices to show that there exist constants $c_1,c_2 > 0$ such that
\[ \p\!\left[ \int |\wt{h}(u)| d\mu_{z,z}(u) \geq \eta \right] \leq c_1 e^{-c_2 \eta^2} \quad\text{for all}\quad \eta \geq 0.\]
Fix $\alpha > 0$.  By two applications of Jensen's inequality, we have that
\begin{align}
\label{eqn::harmonic_exp_square_bound}
   \ex{ \exp\left(\left(\alpha \int |\wt{h}(u)| d\mu_{z,z}(u)\right)^2\right) }
&\leq \int \ex{  e^{ \alpha^2 |\wt{h}(u)|^2} }  d\mu_{z,z}(u).
\end{align} 
The right hand side of~\eqref{eqn::harmonic_exp_square_bound} is finite for $\alpha > 0$ small enough uniformly in $z \in K$ since $\wt{h}(u)$ is a Gaussian with variance which is uniformly bounded over $u \in \partial B(z,r_1)$ for $z \in K$.  This, in turn, implies the result.
\end{proof}

\begin{lemma}
\label{lem::vertical_line_average_bl}
Suppose that $h$ is a GFF on $\D$ with zero boundary conditions.  Fix $\gamma \in (0,2)$.  Let $B$ be the quantum boundary length of $[-1/2,1/2]$ measured using the field $\sqrt{2} h$ and let $\wt{B} = 2 \gamma^{-1} \log B$.  There exist constants $c_0,c_1 > 0$ such that 
\[ \pr{ \wt{B} < \eta } \leq c_0 e^{-c_1 \eta^2} \quad\text{for all}\quad \eta \in \R_-.\]
\end{lemma}

By the odd/even decomposition \cite[Section~3.2]{SHE_WELD}, it follows that the law of the restriction of $\sqrt{2} h$ as in the statement of Lemma~\ref{lem::vertical_line_average_bl} is mutually absolutely continuous with respect to the law of the corresponding restriction of a free boundary GFF on~$\h$.  Consequently, the quantum boundary length of $[-1/2,1/2]$ assigned by $\sqrt{2} h$ is well-defined.

\begin{proof}[Proof of Lemma~\ref{lem::vertical_line_average_bl}]
Let $\wt{h}$ be the projection of $h$ onto the subspace of functions which are harmonic in $C_- = B(-1/4,1/4)$ and $C_+ = B(1/4,1/4)$.  Then we have that $\wh{h} = h-\wt{h}$ is given by a pair of independent zero-boundary GFFs in $C_-,C_+$.  Let $B_-$ (resp.\ $B_+$) be the quantum boundary length of $[-3/8,-1/8]$ (resp.\ $[1/8,3/8]$) computed using the GFF $\sqrt{2} \wh{h}$.  Let $\ul{h}$ be the infimum of $\wt{h}$ on $[-3/8,-1/8] \cup [1/8,3/8]$ and let $\wt{B}_\pm = 2\gamma^{-1} \log B_\pm$.  Then $\wt{B}_-$, $\wt{B}_+$ are independent, $\wt{B}_-,\wt{B}_+ \stackrel{d}{=} \wt{B} - Q \log 4$, and
\begin{equation}
\label{eqn::b_lbd}
\wt{B} \geq \max(\wt{B}_-,\wt{B}_+) + \sqrt{2} \ul{h}.
\end{equation}

For each $\eta \leq 0$, we let $f(\eta) = \pr{ \wt{B} < \eta }$.  Fix $\alpha > 0$.  Then we have that
\begin{align*}
   f(\eta)
&\leq \pr{ \ul{h} \leq \alpha \eta } + \pr{ \wt{B} < \eta,\ \ul{h} > \alpha \eta}\\
&\leq c_0 e^{-c_1 \alpha^2 \eta^2} + \pr{ \wt{B}_- + \sqrt{2} \alpha \eta < \eta,\ \wt{B}_+ + \sqrt{2} \alpha \eta < \eta} \quad\text{(Lemma~\ref{lem::gff_infimum}) and~\eqref{eqn::b_lbd})}\\
&\leq c_0 e^{-c_1 \alpha^2 \eta^2} + (f( \wt{\alpha} \eta -  Q \log 4))^2 \quad\text{(with $\wt{\alpha} = 1-\sqrt{2} \alpha$)}.
\end{align*}
Assume that $\alpha > 0$ is chosen sufficiently small so that $\wt{\alpha} \in (1/\sqrt{2},1)$.  Then Lemma~\ref{lem::quadratic_exponential_decay} implies that there exist constants $c_2,c_3 > 0$ such that $f(\eta) < c_2 e^{-c_3 \eta^2}$, which gives the result.
\end{proof}

\begin{proof}[Proof of Proposition~\ref{prop::gff_boundary_length}]
Lemma~\ref{lem::vertical_line_average_bl} implies the result when we work in the modified setting that $h$ is a GFF on $\D$ with zero boundary conditions and $B$ is the quantum boundary length of $[-1/2,1/2]$ measured using $\sqrt{2} h$.  We will deduce the result from this and conformal mapping.  We begin by letting $\varphi$ be a M\"obius transformation which sends $[-1/2,1/2]$ to $X = \{ \tfrac{1}{2} e^{i \theta} : \theta \in [0,\pi]\}$, i.e.\ the semi-circle of radius $1/2$ in~$\ol{\h}$ centered at the origin, and let $\wh{h} = h \circ \varphi^{-1} + Q\log|(\varphi^{-1})'|$.  Let $\wh{B}$ be the quantum boundary length assigned to $X$ by $\sqrt{2} \wh{h}$.  Since $(\varphi^{-1})'$ is bounded from above and below on $X$, it follows that there exist constants $c_0,c_1 > 0$ such that
\[ \pr{ 2 \gamma^{-1} \log \wh{B} < \eta } \leq c_0 e^{-c_1 \eta^2} \quad\text{for all}\quad \eta \in \R_-.\]
Two applications of Lemma~\ref{lem::gff_infimum} and the Markov property imply that the same is true for the quantum length $\wh{B}$ assigned to $X$ by $\sqrt{2} \wh{h}$ where $\wh{h}$ is a zero-boundary GFF on~$\D$ and therefore by a union bound the same is true for the quantum length assigned to~$\tfrac{1}{2}\partial \D$ by~$\sqrt{2}\wh{h}$.  The result for the whole-plane GFF then follows by applying the Markov property and Lemma~\ref{lem::gff_infimum} again.  Finally, the result for the GFF on $\cyl$ with free boundary conditions follows by using the odd/even decomposition \cite[Section~3.2]{SHE_WELD} of the free boundary GFF on $\cyl_+$ in terms of the whole-plane GFF on $\cyl$.  The proof in the setting that we fix the additive constant for $h$ so that $\sup_{z \in \cyl_+ + r} \Fh(z) = 0$ is analogous.
\end{proof}

\subsubsection{Harmonic tail bound for the unexplored region of a quantum cone}
\label{subsubsec::quantum_boundary_length_qcone_tail_bounds}

We are now going to use Proposition~\ref{prop::gff_boundary_length} to show that the harmonic extension of the boundary values of $h$ sampled from $\qconeUnex^1$ (recall the definition from Section~\ref{subsubsec::cones}) is unlikely to be large when restricted to $\cyl_+ + r$ for any fixed $r > 0$.

\begin{proposition}
\label{prop::unexplored_tail_bound}
For each $r > 0$ there exist constants $c_0,c_1 > 0$ such that the following is true.  Suppose that $(\cyl_+,h)$ has the law $\qconeUnex^1$.  Let $\Fh$ be the harmonic extension of the values of $h$ from $\partial \cyl_+$ to $\cyl_+$.  Then we have that
\[ \pr{ \sup_{z \in \cyl_+ + r} \Fh(z) \geq \eta } \leq c_0 e^{-c_1 \eta^2} \quad\text{for all}\quad \eta \in \R_+.\]
\end{proposition}

We will need to collect two preliminary lemmas before we give the proof of Proposition~\ref{prop::unexplored_tail_bound}.  The first result gives that Proposition~\ref{prop::gff_boundary_length} holds when we choose the additive constant for $h$ in a slightly different way.

\begin{lemma}
\label{lem::conelaw_compare_he}
Fix $r > 0$.  Suppose that we have the same setup as in Proposition~\ref{prop::gff_boundary_length}, let $\Fh$ be the function which is harmonic in $\cyl_+$ with boundary values given by those of $h$ on $\partial \cyl_+$, and that we have taken the additive constant for $h$ so that $\sup_{z \in \partial \cyl_+ + r} \Fh(z)$ is equal to $0$.  Then~\eqref{eqn::gff_boundary_length_small} still holds.
\end{lemma}
\begin{proof}
This follows by a union bound using Proposition~\ref{prop::gff_boundary_length} with Lemma~\ref{lem::gff_infimum}.
\end{proof}

\begin{lemma}
\label{lem::condition_far_away}
For each $r > 0$, consider the law $\p_r$ on random fields $h_r$ defined as follows.
\begin{enumerate}
\item Sample $h$ from $\qconeUnex^1$.
\item Take $h_r$ to be equal to $h$ in $\cyl_+ + r$ and then sample $h_r$ in the annulus $[0,r] \times [0,2\pi]$ in $\cyl_+$ as a GFF with Dirichlet boundary conditions on $\partial \cyl_+ + r$ given by those of $h$ and free boundary conditions on $\partial \cyl_+$.
\end{enumerate}
Let $\Fh$ denote the harmonic extension of the values of $h_r$ from $\partial \cyl_+$ to $\cyl_+$ and let $A = \sup_{z \in \cyl_+ + 1} \Fh(z)$.  Let $B$ denote the quantum boundary length of $\partial \cyl_+$ and let $\wt{B} = 2 \gamma^{-1} \log B$.  Fix $x,y \in \R$ and let $I_{u,\epsilon} = [u,u+\epsilon]$ for $u \in \{x,y\}$.  Let $W_r$ be the average of $h$ on $\partial \cyl_+ + r$.  There exists a constant $c_0 > 0$ such that a.s.,
\[ \limsup_{\epsilon \to 0} \limsup_{r \to \infty} \frac{\p_r[A \in I_{x,\epsilon} \giv W_r]}{\p_r[ \wt{B} \in I_{y,\epsilon} \giv W_r]} \leq c_0 e^{(Q-\gamma)(x-y)}.\]
\end{lemma}

We recall from \cite[Proposition~6.5]{quantum_spheres} that the law of $h_r$ conditioned on $B=1$ is equal to $\conelaw^1$.  Thus the bound established in Lemma~\ref{lem::condition_far_away} will be useful in the proof of Proposition~\ref{prop::unexplored_tail_bound} given just below to rule out the possibility that $A$ takes on a large value given $B=1$ (via a Bayes' rule calculation).

\begin{proof}[Proof of Lemma~\ref{lem::condition_far_away}]
For each $r > 0$ we let $W_r$ be the average of $h$ on $\partial \cyl_+ + r$.  The resampling properties for $\qconeUnex^1$ (see, e.g., \cite[Proposition~6.5]{quantum_spheres}) imply that
\begin{equation}
\label{eqn::w_r_expression}
W_r = (Q-\gamma) r + U_r + X
\end{equation}
where $U_r$ is a standard Brownian motion with $U_0 = 0$ and $X$ is a.s.\ finite ($X$ is given by the average of $h$ on $\partial \cyl_+$).  Under $\p_r$, the conditional law of the average of the field on $\partial \cyl_+$ given $W_r$ is that of a Gaussian random variable with mean $W_r$ and variance $r$.  It therefore follows that there exists a constant $c_0 > 0$ such that
\begin{equation}
\label{eqn::avg_x_epsilon_give_w_r}
\p_r[ A \in I_{x,\epsilon} \giv W_r] \leq \frac{c_0 \epsilon}{\sqrt{r}} e^{-(W_r - x)^2 / 2 r}.
\end{equation}
We similarly have that there exists a constant $c_1 > 0$ such that
\begin{equation}
\label{eqn::avg_y_epsilon_give_w_r}
\p_r[ \wt{B} \in I_{y,\epsilon} \giv W_r] \geq \frac{c_1 \epsilon}{\sqrt{r}} e^{-(W_r - y)^2 / 2 r}.
\end{equation}
The result follows by combining~\eqref{eqn::avg_x_epsilon_give_w_r} and~\eqref{eqn::avg_y_epsilon_give_w_r} and using~\eqref{eqn::w_r_expression}.
\end{proof}

\begin{proof}[Proof of Proposition~\ref{prop::unexplored_tail_bound}]
Let $W_r$, $\p_r$, $A$, $B$, $\wt{B}$, $I_{x,\epsilon}$, and $I_{y,\epsilon}$ be as in Lemma~\ref{lem::condition_far_away}.  By Bayes' rule we have that
\begin{align}
     \p_r[ A \in I_{x,\epsilon} \giv \wt{B} \in I_{y,\epsilon}, W_r]
&= \frac{\p_r[ A \in I_{x,\epsilon} \giv W_r]}{\p_r[ \wt{B} \in I_{y,\epsilon} \giv W_r]} \p_r[ \wt{B} \in I_{y,\epsilon} \giv A \in I_{x,\epsilon}, W_r]. \label{eqn::bayes_a_b}
\end{align}
Lemma~\ref{lem::condition_far_away} implies that the $\limsup$ as $\epsilon \to 0$ and $r \to \infty$ of the first term on the right hand side is a.s.\ at most $c_0 e^{(Q-\gamma)(x-y)}$.  We also have that the $\limsup$ as $\epsilon \to 0$ and $r \to \infty$ of $\epsilon^{-1} \p_r[ \wt{B} \in I_{y,\epsilon} \giv A \in I_{x,\epsilon}, W_r]$ is equal to the conditional density of $\wt{B}$ at $y$ of the law of a GFF on $\cyl_+$ with free boundary conditions plus the function $r \mapsto (Q-\gamma)r$ with the additive constant fixed so that $A=x$.  Call this function $g_x(y)$.  Similarly, the $\limsup$ as $\epsilon \to 0$ and $r \to \infty$ of $\epsilon^{-1} \p_r[ A \in I_{x,\epsilon} \giv \wt{B} \in I_{0,\epsilon}, W_r]$ is equal to the density of $A$ at $x$ under $\conelaw^1$.  Call this function $f(x)$.  Combining, we have that
\[ f(x) \leq c_0 e^{(Q-\gamma)x} g_x(0).\]
Note that $g_x(0) = g(-x)$ where $g$ is the density of $\wt{B}$ under the law of a GFF on $\cyl_+$ with free boundary conditions plus the function $r \mapsto (Q-\gamma)r$ with the additive constant fixed so that $A=0$.  Proposition~\ref{prop::gff_boundary_length} implies that there exist constants $c_1, c_2 > 0$ so that for each $k \geq 0$ we have that
\[ \int_k^{k+1} g(-s) ds \leq c_1 e^{-c_2 k^2}.\]
Combining, we have for $\eta \geq 0$ and $k_0 = \lfloor \eta \rfloor$ that
\begin{align*}
   \p[ A \geq \eta ] 
&= \int_\eta^\infty f(s) ds
 \leq \sum_{k=k_0}^{\infty} \int_k^{k+1} c_0 e^{(Q-\gamma)(k+1)} g(-s) ds\\
&\leq \sum_{k=k_0}^{\infty} c_0 e^{(Q-\gamma) (k+1)} \times c_1 e^{-c_2 k^2}
 \leq c_3 e^{-c_4 \eta^2}
\end{align*}
for constants $c_3,c_4 > 0$.  That is, under $\qconeUnex^1$, we have that the probability that the supremum of the harmonic extension of the field from $\partial \cyl_+$ to $\cyl_+$ restricted to $\partial \cyl_+ + 1$ is at least $\eta$ is at most $c_3 e^{-c_4 \eta^2}$.  The same argument applies to bound the tail of the supremum of the harmonic extension of the field from $\partial \cyl_+$ to $\cyl_+$ restricted to $\partial \cyl_+ + r$ for any fixed value of $r > 0$.
\end{proof}

\subsubsection{Proof of Proposition~\ref{prop::ball_size_ubd}}
\label{subsubsec::proof_of_ball_size_ubd}

Suppose that $(\C,h,0,\infty)$ is a $\sqrt{8/3}$-quantum cone with the circle average embedding as in the statement of the proposition and let $(\Gamma_r)$ be the $\QLE(8/3,0)$ growing from $0$ to $\infty$.

Throughout, we let $\gamma=\sqrt{8/3}$.  Fix $\epsilon > 0$ and let $\ell_\epsilon$ be the quantum boundary length of the outer boundary of $\Gamma_\epsilon$.  Let $\varphi \colon \C \setminus \Gamma_\epsilon \to \C \setminus \ol{\D}$ be the unique conformal transformation with $\varphi(\infty) = \infty$ and $\varphi'(\infty) > 0$ and let $\psi = \varphi^{-1}$.  We then let $h_1 = h \circ \psi + Q \log|\psi'| - 2\gamma^{-1} \log \ell_\epsilon$ so that $(\C \setminus \ol{\D},h_1)$ has the law $\qconeUnex^1$.

Let $R_\epsilon^* = 4\pi \sup_{z \in \partial B(0,2)} |\psi'(z)|$ and note that
\begin{equation}
\label{eqn::diam_rstar_bound}
\diam(\Gamma_\epsilon) \leq \int_{\partial B(0,2)} |\psi'(z)| dz \leq R_\epsilon^*
\end{equation}
where $dz$ denotes Lebesgue measure on $\partial B(0,2)$.  It therefore suffices to show that for each $p > 0$ there exist $a_0 = a_0(p)$ such that
\begin{equation}
\label{eqn::ball_size_suffices}
\pr{  R_\epsilon^* \geq \epsilon^{a_0}} = O(\epsilon^p) \quad\text{as}\quad \epsilon \to 0.
\end{equation}

Fix $\zeta > 0$ and let $E_1 = \{ \ell_\epsilon \leq \epsilon^{2-\zeta}\}$.  By Lemma~\ref{lem::boundary_length_radius}, we have for constants $c_1,c_2 > 0$ that $\pr{ E_1^c } \leq c_1 \exp(-c_2 \epsilon^{-\zeta/2})$.  It therefore suffices to work on $E_1$.

Write $h_2 = h+ \gamma\log|\cdot|$.  By the change of coordinates formula for quantum surfaces, we have on the event $E_1$ that
\begin{align}
      Q \log|\psi'|
&= \frac{2}{\gamma} \log \ell_\epsilon + \gamma \log|\psi(\cdot)| +  h_1 - h_2 \circ \psi \notag\\
&\leq \frac{4-2\zeta}{\gamma} \log \epsilon + \gamma \log|\psi(\cdot)| + h_1 - h_2 \circ \psi. \label{eqn::coord_change_bound}
\end{align}
Let $\Fh_1$ (resp.\ $\Fh_2$) be the function which is harmonic in $\C \setminus \ol{\D}$ (resp.\ $\C \setminus \Gamma_\epsilon$) with boundary values given by those of $h_1$ (resp.\ $h_2$) on $\partial \D$ (resp.\ $\partial \Gamma_\epsilon$).  Proposition~\ref{prop::unexplored_tail_bound} implies that there exist constants $c_3,c_4 > 0$ such that with
\[ E_2 = \left\{ \sup_{z \in \partial B(0,2)} \Fh_1(z) \leq \frac{\zeta}{\gamma} \log \epsilon^{-1} \right\} \quad\text{we have}\quad \pr{E_2^c} \leq c_3 \exp(-c_4 \zeta^2 (\log \epsilon^{-1})^2).\] 
Therefore it suffices to work on $E_2$.

Since the left side of~\eqref{eqn::coord_change_bound} is harmonic in $\C \setminus \ol{\D}$ it follows that~\eqref{eqn::coord_change_bound} holds with $\Fh_1,\Fh_2$ in place of $h_1,h_2$ so that on $E_1 \cap E_2$ we have for $z \in \partial B(0,2)$ that
\begin{align}
 Q\log|\psi'(z)|
&\leq \frac{4-2\zeta}{\gamma} \log \epsilon + \gamma \log|\psi(z)| + \Fh_1(z) - \Fh_2(\psi(z)) \notag\\
&\leq \frac{4-3\zeta}{\gamma} \log \epsilon + \gamma \log|\psi(z)| - \Fh_2(\psi(z)). \label{eqn::q_log_psi_p_ubd}
 \end{align}

For $z \in B(0,2)$ we note that
\begin{equation}
\label{eqn::psi_rstar_bound}
|\psi(z)| \leq \diam(\psi(B(0,2))) \leq R_\epsilon^*.
\end{equation}
Thus by taking the supremum of both sides of~\eqref{eqn::q_log_psi_p_ubd} over $z \in \partial B(0,2)$ we arrive at the inequality
\begin{equation}
Q \log R_\epsilon^* \leq \frac{4-3\zeta}{\gamma} \log \epsilon + \gamma \log R_\epsilon^* - \inf_{z \in \partial B(0,2)} \Fh_2(\psi(z)). \label{eqn::q_log_psi_p_ubd2}	
\end{equation}

Let $z^*$ be a point in $\partial B(0,2)$ where $\inf_{z \in \partial B(0,2)} \Fh_2(\psi(z))$ is attained.  We can write $-\Fh_2(\psi(z^*)) = -h_{2,r}(\psi(z^*)) + Z$ where we take $r = \sup\{ e^k : k \in \Z, e^k \leq \dist(\psi(z^*), \Gamma_\epsilon)\}$ and $Z$ has a Gaussian tail (with bounded variance).  In particular, the probability of the event $E_3 = \{|Z| \leq \zeta/\gamma \log \epsilon^{-1}\}$ is $1 - O(\exp(-c_5 (\zeta/\gamma)^2 (\log \epsilon^{-1})^2))$ for a constant $c_5 > 0$.  We therefore may assume that we are working on $E_3$.  That is, $-\Fh_2(\psi(z^*)) \leq -h_{2,r}(\psi(z^*)) + \tfrac{\zeta}{\gamma} \log \epsilon^{-1}$.  Fix $a > 0$ so that $Q-\gamma-a > 0$ and $C > 0$.  We assume that $C$ is chosen so that if $E_{a,C}$ is as in the statement of Corollary~\ref{cor::gff_whole_plane_maximum} in terms of the field $h_2$ we have that $\pr{E_{a,C}} \geq 1/2$.  Fix $\delta > 0$ and let $A_\delta = \cap_{k \in \N} \{\sup_{z \in B(0,1/2)} |h_{2,e^{-k}}(z)| \geq (2+\delta) k\}$ be the event from the statement of Proposition~\ref{prop::gff_maximum}. On $A_\delta \cap E_{a,C}$, we thus have that
\begin{equation}
\label{eqn::fh2_inf}
-\inf_{z \in \partial B(0,2)} \Fh_2(\psi(z)) \leq
\begin{cases}
a \log R_\epsilon^* + \frac{\zeta}{\gamma} \log \epsilon^{-1} + C \quad&\text{if}\quad R_\epsilon^* \geq 1/2\\
(2+\delta) \log (R_\epsilon^*)^{-1} + \frac{\zeta}{\gamma}  \log \epsilon^{-1} \quad&\text{if}\quad R_\epsilon^* < 1/2.
\end{cases}
\end{equation}
Suppose that $R_\epsilon^* \geq 1/2$.  Using~\eqref{eqn::psi_rstar_bound} and~\eqref{eqn::fh2_inf} we have from~\eqref{eqn::q_log_psi_p_ubd} the upper bound
\begin{align}
Q \log R_\epsilon^*
&\leq \frac{4-4\zeta}{\gamma} \log \epsilon + (\gamma+a) \log R_\epsilon^*   + c_6 \label{eqn::deriv_a_priori}
\end{align}
where $c_6 > 0$ is a constant.  Rearranging~\eqref{eqn::deriv_a_priori} gives for a constant $c_7 > 0$ that
\begin{equation}
\label{eqn::r_star_bound}
\log R_\epsilon^* \leq \frac{4-4\zeta}{\gamma(Q-\gamma-a)} \log \epsilon + c_7.
\end{equation}
Suppose that $R_\epsilon^* \leq 1/2$.  Arguing as before, in this case, we have for a constant $c_8 > 0$ that
\begin{equation}
\label{eqn::r_star_bound2}
\log R_\epsilon^* \leq \frac{4-4\zeta}{\gamma(Q-\gamma+2+\delta)} \log \epsilon + c_8.
\end{equation}
Combining~\eqref{eqn::r_star_bound} and \eqref{eqn::r_star_bound2} implies that there exists $a_0 > 0$ such that $\p[ R_\epsilon^* \geq \epsilon^{a_0},\ A_\delta,\ E_{a,C}]$ decays to $0$ as $\epsilon \to 0$ faster than any polynomial.  Note that $Q_{\epsilon,\delta} = \{ R_\epsilon^* \geq \epsilon^{a_0} \} \cap A_\delta$ (resp.\ $E_{a,C}$) depends only on $h$ restricted to $\D$ (resp.\ the complement of $\D$) provided $\epsilon > 0$ is small enough.  Let $\wt{h}_2$ be a sample from the law of $h_2$ conditioned on $E_{a,C}$ occurring taken to be independent of $h_2$.  Let $g$ be the function which is harmonic in $\D$ with boundary values given by $\wt{h}_2-h_2$, let $\phi \in C_0^\infty(\D)$ be such that $\phi|_{B(0,1/2)} \equiv 1$, and let $\wt{g} = \phi g$.  Then $h_2 + \wt{g} = \wt{h}_2$ in $\D$.  Moreover, the Radon-Nikodym derivative of the law of $h_2+\wt{g}$ with respect to the law of $h_2$ is given by $\CZ = \exp((h_2,\wt{g})_\nabla - \| \wt{g} \|_\nabla^2/2)$ (see, e.g., Lemma~\ref{lem::gff_change_bc} below).  That is, weighting the law of $h_2$ by $\CZ$ and then restricting to $B(0,1/2)$ is the same as the law of $h_2$ given $E_{a,C}$ restricted to $B(0,1/2)$.  Applying the Cauchy-Schwarz inequality in the first inequality and recalling that $\p[E_{a,C}] \geq 1/2$ so that $1/\p[E_{a,C}] \leq 2$,  we thus have that
\begin{align*}
   \p[Q_{\epsilon,\delta}]
&= \E[\one_{Q_{\epsilon,\delta}} \CZ \CZ^{-1}]
 = \E[\one_{Q_{\epsilon,\delta}} \CZ^{-1} \giv E_{a,C}]
 \leq \p[Q_{\epsilon,\delta} \giv E_{a,C}]^{1/2} \E[ \CZ^{-2} \giv E_{a,C}]^{1/2}\\
&\leq 2 \p[Q_{\epsilon,\delta}, E_{a,C}]^{1/2} \E[ \CZ^{-2}]^{1/2}.
\end{align*}
As we explained above, $\p[Q_{\epsilon,\delta}, E_{a,C}]$ decays to $0$ as $\epsilon \to 0$ faster than any polynomial of $\epsilon$ and, by Lemma~\ref{lem::gff_infimum}, we have that
\[ \E[ \CZ^{-2}] = \E[ \exp( 3\| \wt{g} \|_\nabla^2) ] < \infty.\]
Consequently, $\p[Q_{\epsilon,\delta}]$ decays to $0$ as $\epsilon \to 0$ faster than any polynomial of $\epsilon$.  This completes the proof of~\eqref{eqn::prob_diam_bound} as we have that
\[ \p[ R_\epsilon^* \geq \epsilon^{a_0}] \leq \p[Q_{\epsilon,\delta}] + \p[A_\delta^c] \leq \p[Q_{\epsilon,\delta}] + c_0 \epsilon^{\delta} \quad\text{(by Proposition~\ref{prop::gff_maximum} with $\xi=1/2$)}.\]
In particular, we can make the right hand side be $O(\epsilon^p)$ by taking $\delta=p$.

On the event that $\diam(\Gamma_\epsilon) \leq 1$, we have that the term $\gamma\log|\psi(z)|$ on the right side of~\eqref{eqn::q_log_psi_p_ubd} is bounded.  We also have, using Proposition~\ref{prop::gff_maximum}, that $-\inf_{z \in \partial B(0,2)} \Fh_2(z)$ is at most $(2+\delta)\log (R_\epsilon^*)^{-1}$ off an event which occurs with probability at most a constant times $\epsilon^{2\delta(1-\zeta)}$.  That is, by rearranging~\eqref{eqn::q_log_psi_p_ubd} we get for constants $c_{9},c_{10} > 0$ that
\begin{equation}
\label{eqn::r_star_prop_bound}
\log R_\epsilon^* \leq \frac{4-3\zeta}{\gamma(Q-\gamma+2+\delta)} \log \epsilon + c_{9}
\end{equation}
off an event which occurs with probability at most $c_{10} \epsilon^{2\delta(1-\zeta)}$.  This implies~\eqref{eqn::fourth_moment_diam_bound} because we have that
\[ 4\left( \frac{4-3\zeta}{\gamma(Q-\gamma+2+\delta)} \right) + 2 \delta(1-\zeta) > 4 \quad\text{for all}\quad \delta > 0\]
provided we fix $\zeta > 0$ small enough.
\qed

\section{H\"older continuity of the $\QLE(8/3,0)$ metric}
\label{sec::continuity_to_bm}

We will prove Theorems~\ref{thm::continuity}--\ref{thm::geodesic_metric_space} and Theorem~\ref{thm::typical_ball_size} in this section.  We will prove the first two results in the setting of a $\sqrt{8/3}$-quantum cone.  As we will see, this setting simplifies some aspects of the proofs because a quantum cone is invariant under the operation of multiplying its area by a constant.

To prove Theorem~\ref{thm::continuity}, we suppose that $\CC =(\C,h,0,\infty)$ is a $\sqrt{8/3}$-quantum cone.  We want to get an upper bound on the amount of quantum distance time that it takes for the $\QLE(8/3,0)$ process $(\Gamma_r)$ starting from $0$ to hit a point $w \in \C$ with $|w|$ small.  There are two possibilities if $(\Gamma_r)$ does not hit $w$ in a given amount of quantum distance time $r$.  First, it could be that $w$ is contained in the hull of $\Gamma_r$ in which case we can use the bound established in Section~\ref{subsec::disk_diameter_upper_bound} just below for the quantum diameter of the hull of $\Gamma_r$ to get that the quantum distance of $0$ and $w$ is not too large.  The second possibility is that $w$ is not contained in the hull of $\Gamma_r$ in which case due to our lower bound on the Euclidean hull diameter established in Section~\ref{subsec::diameter_lower_bound}, we would get that the distance of $w$ to the hull of $\Gamma_r$ is much smaller than the Euclidean diameter of $\Gamma_r$.  This implies that if we apply the unique conformal map which takes the unbounded component of the complement of~$\Gamma_r$ to $\C \setminus \ol{\D}$ which fixes and has positive derivative at~$\infty$ then the image of $w$ will have modulus which is very close to $1$.  Therefore we need to get an upper bound on the quantum distance of those points in a surface sampled from $\qconeUnex^1$ parameterized by $\C \setminus \ol{\D}$ which are close to $\partial \D$.  We accomplish this in Section~\ref{subsec::close_to_disk_distance}.

We put all of our estimates together to complete the proof of Theorem~\ref{thm::continuity} in Section~\ref{subsubsec::proof_of_continuity} using a Kolmogorov-\u Centsov type argument, except we subdivide our space using a sequence of i.i.d.\ points chosen from the quantum measure rather than the usual dyadic subdivision.

In Section~\ref{subsubsec::metric_completion}, we will prove Theorem~\ref{thm::metric_completion} using an argument which is similar to that given in Section~\ref{subsubsec::proof_of_continuity} using the upper bound on the Euclidean diameter of a $\QLE(8/3,0)$ hull established in Section~\ref{subsec::qle_diameter_upper_bound}.

The estimates used to prove Theorem~\ref{thm::continuity} and Theorem~\ref{thm::metric_completion} will easily lead to the proofs of Theorem~\ref{thm::geodesic_metric_space} in Section~\ref{subsec::geodesic_existence} and Theorem~\ref{thm::typical_ball_size} in Section~\ref{subsec::typical_ball_proof}.

\subsection{Quantum diameter of $\QLE(8/3,0)$ hull}
\label{subsec::disk_diameter_upper_bound}

We are now going to give an upper bound on the tail of the quantum diameter of a $\QLE(8/3,0)$ on a $\sqrt{8/3}$-quantum cone.  More precisely, we will bound the tail of the amount of additional time it requires a $\QLE(8/3,0)$ on a $\sqrt{8/3}$-quantum cone run for a given amount of time to fill all of the components that it has separated from $\infty$.  In what follows, it will be necessary to truncate on the event $H_{R,\zeta}$ from Proposition~\ref{prop::cone_holder} in order to ensure that the tail decays to $0$ sufficiently quickly.

\begin{lemma}
\label{lem::qle_quantum_disk_size}
Suppose that $(\Gamma_r)$ is a $\QLE(8/3,0)$ process on a $\sqrt{8/3}$-quantum cone $(\C,h,0,\infty)$ starting from $0$ with the quantum distance parameterization.  Let $H_{R,\zeta}$ be the event as in Proposition~\ref{prop::cone_holder}.  Fix $\epsilon > 0$ and let $d_\epsilon^*$ be the supremum of the amount of time that it takes $(\Gamma_r)$ to fill all of the quantum disks which have been separated from $\infty$ by quantum distance time $\epsilon$.  For every $\beta > 0$ there exists $\alpha \in (0,1)$ and $c_0 > 0$ such that
\begin{equation}
\label{eqn::qle_hull_upper_bound}
\pr{ d_\epsilon^* \geq \epsilon^\alpha,\ H_{R,\zeta} } \leq c_0 \epsilon^\beta.
\end{equation}
\end{lemma}
We note that on $d_\epsilon^* \leq \epsilon^\alpha$ the quantum diameter of the hull of $\Gamma_\epsilon$ is at most $2(\epsilon+\epsilon^\alpha)$.

The main input into the proof of Lemma~\ref{lem::qle_quantum_disk_size} is the following lemma which gives the tail for the amount of time that it takes a $\QLE(8/3,0)$ growth starting from the boundary of a quantum disk to hit every point in the disk.  We will deduce Lemma~\ref{lem::quantum_disk_given_area_diameter_bound} using that the branching structure of a $\QLE(8/3,0)$ exploration starting from the boundary of a quantum disk is the same as in the Brownian disk, which allows us to make use of tail bounds for the diameter of the Brownian map (e.g., \cite{serlet_ldp}).  We expect, however, that it is possible to derive a sufficiently good upper bound directly from the branching structure of the $\QLE(8/3,0)$ exploration.

\begin{lemma}
\label{lem::quantum_disk_given_area_diameter_bound}
Fix $0 < a < \infty$ and suppose that $(D,h)$ is a quantum disk with boundary length $\ell \in (0,1]$.  Let $d^*$ be the amount of time that it takes the $\QLE(8/3,0)$ exploration starting from~$\partial D$ to hit every point in~$D$.  There exists a constant $c_0 \geq 1$ depending only on $a$ such that
\[ \pr{ \mu_h(D) \leq a \giv d^* \geq r } \leq c_0 \exp( - c_0^{-1} r^{4/3})\]
for all $r > 0$.
\end{lemma}

The $\QLE(8/3,0)$ exploration from $\partial D$ is defined because the conditional law of the components cut off from $\infty$ by the $\QLE$ exploration are given by conditionally independent quantum disks given their boundary length.

\begin{proof}[Proof of Lemma~\ref{lem::quantum_disk_given_area_diameter_bound}]
The lemma is a consequence of \cite[Proposition~4.23]{map_making} and the branching structure of $\QLE(8/3,0)$.  In particular, the joint law of the evolution of the boundary length of a $\QLE(8/3,0)$ on a quantum disk together with the total quantum area is the same as that of the evolution of the boundary length of the metric growth from the boundary of a Brownian disk determined in \cite{map_making} and the amount of area.  Therefore the joint law of the amount of time required for a $\QLE(8/3,0)$ starting from the boundary of a quantum disk to fill the entire disk together with the total area (i.e., the pair $(d^*, \mu_h(D))$) has the same joint law as the amount of time that a metric exploration from the boundary of a Brownian disk takes to fill the entire disk together with the amount of area.
\end{proof}

\begin{proof}[Proof of Lemma~\ref{lem::qle_quantum_disk_size}]
Suppose that $\alpha,\alpha' \in (0,1)$.  We will adjust their values in the proof.  For each $\epsilon > 0$, we let $\tau_\epsilon$ be the first $r > 0$ such that $(\Gamma_r)$ cuts off a bubble such that the amount of time it takes to subsequently fill is at least $\epsilon^\alpha$.  We also let $\sigma_\epsilon$ be the first $r > 0$ that $(\Gamma_r)$ cuts off a bubble with either quantum boundary length at least~$\epsilon^{\alpha'/4}$ or quantum area at least~$\epsilon^{\alpha'}$.  Fix $\beta > 0$.  We have that
\begin{align*}
   \p[ d_\epsilon^* \geq \epsilon^\alpha,\ H_{R,\zeta}]
&=     \p[ \tau_\epsilon \leq \epsilon,\ H_{R,\zeta}]\\
&\leq \p[ \sigma_\epsilon \leq \tau_\epsilon \leq \epsilon,\ H_{R,\zeta}] + \p[ \tau_\epsilon < \sigma_\epsilon,\ H_{R,\zeta}]\\
&\leq \p[ \sigma_\epsilon \leq \epsilon,\ H_{R,\zeta}] + \p[ \tau_\epsilon < \sigma_\epsilon].
\end{align*}
At the time $\sigma_\epsilon$, there are two possibilities.  Either $\Gamma_{\sigma_\epsilon}$ has just cut off a bubble with quantum area at least $\epsilon^{\alpha'}$ or a bubble with quantum boundary length at least $\epsilon^{\alpha'/4}$.  Given that we are in the latter situation, the conditional probability that this bubble has area smaller than $\epsilon^{\alpha'}$ decays to $0$ as $\epsilon \to 0$ faster than any power of $\epsilon$.  It therefore follows that if we let $A_\epsilon$ be the quantum area separated by $\Gamma_\epsilon$ from $\infty$, then the first term above is bounded from above by $\p[ A_\epsilon \geq \epsilon^{\alpha'},\ H_{R,\zeta}]$ plus an error term which tends to $0$ as $\epsilon \to 0$ faster than any power of $\epsilon$.  Corollary~\ref{cor::qle_area_bound} implies that we can make $\alpha' \in (0,1)$ small enough so that $\p[ A_\epsilon \geq \epsilon^{\alpha'},\ H_{R,\zeta}] \leq c_0 \epsilon^\beta$.

We now consider $\pr{ \tau_\epsilon < \sigma_\epsilon}$.  Let $U_\epsilon$ be the bubble which is cut off by $(\Gamma_r)$ at the time $\tau_\epsilon$.  Given its quantum boundary length $\ell$, $U_\epsilon$ is a quantum disk with quantum boundary length $\ell$ conditioned so that the amount of time it takes a $\QLE(8/3,0)$ exploration starting from the boundary to fill is at least $\epsilon^{\alpha}$.  Therefore $\pr{ \tau_\epsilon < \sigma_\epsilon}$ is at most the probability that a quantum disk with quantum boundary length $\ell \leq \epsilon^{\alpha'/4}$ conditioned to have quantum diameter at least $\epsilon^\alpha$ has quantum area at most $\epsilon^{\alpha'}$.  Lemma~\ref{lem::quantum_disk_given_area_diameter_bound} implies that we can choose $\alpha \in (0,1)$ sufficiently small so that this probability decays to $0$ as $\epsilon \to 0$ faster than any power of $\epsilon$, which completes the proof.
\end{proof}

\subsection{Euclidean disks are filled by $\QLE(8/3,0)$ growth}
\label{subsec::close_to_disk_distance}

\begin{figure}[ht!]
\begin{center}
\includegraphics[scale=0.85,page=1]{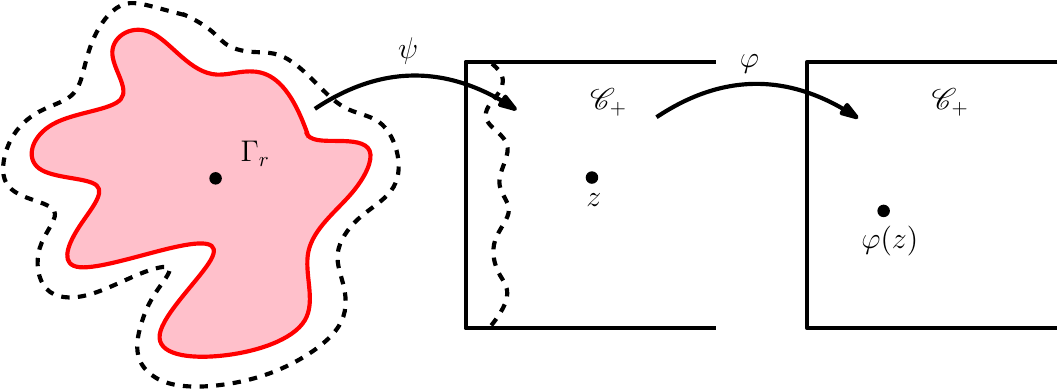}	
\end{center}
\caption{\label{fig::euclidean_disks_swallowed} Illustration of the setup and the argument of Proposition~\ref{prop::disk_diameter_bounds}, which shows that Euclidean disks are filled by the $\QLE(8/3,0)$ growth.  {\bf Left:} a $\QLE(8/3,0)$ process $\Gamma_r$ starting from the origin of a $\sqrt{8/3}$-quantum cone run up to a given radius $r > 0$.  The dashed curve indicates the range of $\Gamma$ at time $r+\epsilon^\zeta$ for $\zeta > 0$ very small.  {\bf Middle:} The map $\psi$ is the unique conformal map from $\C \setminus \Gamma_r$ to $\cyl_+$ with $\infty$ sent to $+\infty$ and with positive derivative at $\infty$.  The region bounded by the dashed curve is the image under $\psi$ of the corresponding region on the left.  {\bf Right:}  The map $\varphi$ is the unique conformal map from the unbounded complementary component of the dashed region to $\cyl_+$ with $\varphi(z)-z \to 0$ as $z \to +\infty$.  To prove the result (see Figure~\ref{fig::euclidean_disks_swallowed2} for an illustration), we show in the proof of Proposition~\ref{prop::disk_diameter_bounds} that by making $\zeta > 0$ sufficiently small the event that for every $z$ with $\re(z) \in [\epsilon/2,\epsilon]$ we have that $\re(\varphi(z)) < \epsilon/4$ occurs with overwhelming probability.  (We take $\re(\varphi(z)) = 0$ for points $z$ which are to the left of the dashed line.)  Iterating this implies there exists $\beta > 0$ such that, with overwhelming probability, the $\QLE(8/3,0)$ growing from $\partial \cyl_+$ absorbs all such $z$ in time $\epsilon^\beta$.}	
\end{figure}

\begin{figure}[ht!]
\begin{center}
\includegraphics[scale=0.85,page=2]{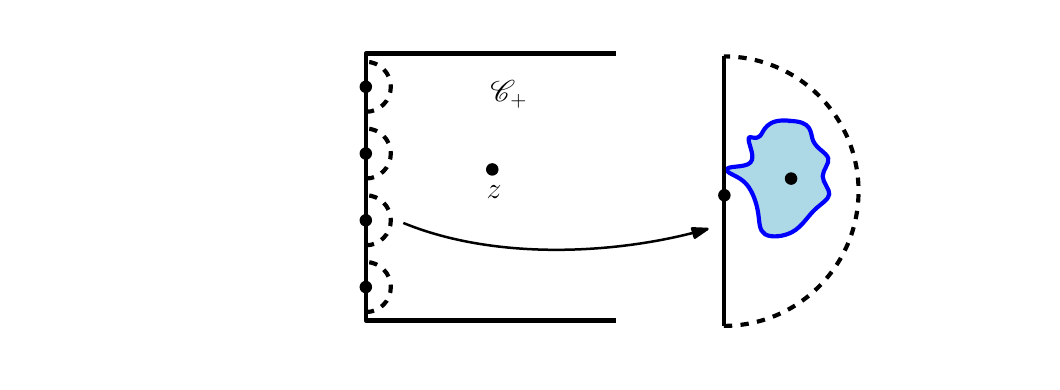}	
\end{center}
\vspace{-0.05\textheight}
\caption{\label{fig::euclidean_disks_swallowed2} (Continuation of Figure~\ref{fig::euclidean_disks_swallowed}.) Shown on the left is a copy of the middle part of Figure~\ref{fig::euclidean_disks_swallowed}, scaled so that the law of the surface is given by $\qconeUnex^1$.  Suppose that $\re(z) \in [\epsilon/2,\epsilon]$ and that $\varphi$ is as in Figure~\ref{fig::euclidean_disks_swallowed}.  In order to show that $\re(\varphi(z)) \leq \epsilon/4$ with overwhelming probability, we place semi-disks of radius $\epsilon / (\log \epsilon^{-1})^2$ with equal spacing $\epsilon / (\log \epsilon^{-1})$ along $\partial \cyl_+$.  Shown on the right is an enlargement of one of the semi-disks.  The restriction of the field to each semi-disk is mutually absolutely continuous with respect to the law of a quantum disk.  By making such a comparison, we see that if we pick a uniformly random point inside of the semi-disk and then grow the $\QLE(8/3,0)$ starting from that point until it hits the boundary, then there is a positive chance that the $\QLE(8/3,0)$ first exits in $\partial \cyl_+$ and does so in time at most $\epsilon^\sigma$.  By the metric property, the range of this $\QLE(8/3,0)$ is then contained in the $\QLE(8/3,0)$ growing from $\partial \cyl_+$ for time $\epsilon^\sigma$.  Since the behavior of the field in each of the semi-disks is approximately independent, with overwhelming probability, there cannot be a collection of consecutive semi-disks so that the event does not occur for any of them.  In particular, there must exist a semi-disk which close enough to $z$ to show that $\re(\varphi(z)-z)$ is bounded from above by a given negative number.  Iterating this yields the desired bound.}	
\end{figure}

We will now give an upper bound on the amount of quantum distance time that it takes for the $\QLE(8/3,0)$ hull growing in $\cyl_+$ from $\partial \cyl_+$ to fill a neighborhood of $\partial \cyl_+$ where the quantum surface has law $\qconeUnex^1$.  Similar to the setting of Lemma~\ref{lem::quantum_disk_given_area_diameter_bound} considered above, it makes sense to talk about the $\QLE(8/3,0)$ hull growing from $\partial \cyl_+$ because $\qconeUnex^1$ gives the conditional law of the quantum surface parameterized by the unbounded component when performing a $\QLE(8/3,0)$ exploration of a $\sqrt{8/3}$-quantum cone, after rescaling so that the boundary length is equal to $1$.  The main result is the following proposition.

\begin{proposition}
\label{prop::disk_diameter_bounds}
Suppose that $(\cyl_+,h)$ has law $\qconeUnex^1$.  For each $\beta > 0$ there exist constants $c_0,\alpha,\zeta > 0$ such that the following is true.  Let $E_{\alpha,\zeta}$ be the event that every $z \in \cyl_+$ with $\re(z) < \epsilon^\alpha$ is contained in the $\QLE(8/3,0)$ hull of radius $\epsilon^\zeta$ growing from $\partial \cyl_+$.  Then $\pr{ E_{\alpha,\zeta}^c } \leq  c_0 \epsilon^\beta$.  Moreover, if we fix $\sigma > 0$ and let $A_{\alpha,\sigma,\epsilon}$ be the event that the quantum area of $\{z \in \cyl_+ : \re(z) < \epsilon^\alpha\}$ is at most $\epsilon^\sigma$, then (with $\alpha$ fixed) for each $\beta > 0$ there exists $\zeta \in (0,1)$ such that $\pr{ E_{\alpha,\zeta}^c, A_{\alpha,\sigma,\epsilon} } \leq  c_0 \epsilon^\beta$.
\end{proposition}

We begin by recording an elementary lemma which gives the Radon-Nikodym derivative of the GFF with mixed boundary conditions when we change the boundary conditions on the Dirichlet part.

\begin{lemma}
\label{lem::gff_change_bc}
Suppose that $D \subseteq \C$ is a bounded Jordan domain and $\partial D = \partial^\free \cup \partial^\dirichlet$ where $\partial^\free,\partial^\dirichlet$ are non-empty, disjoint intervals.  Let $h_1,h_2$ be GFFs on $D$ with free (resp.\ Dirichlet) boundary conditions on $\partial^\free$ (resp.\ $\partial^\dirichlet$).  Let $U \subseteq D$ be open with positive distance from $\partial^\dirichlet$ and let $g$ be the function which is harmonic in $D$ with Neumann (resp.\ Dirichlet) boundary conditions $\partial^\free$ (resp.\ $\partial^\dirichlet$) where the Dirichlet boundary conditions are given by those of $h_1-h_2$.  Let $\wt{g} = g \phi$ where $\phi \in C^\infty(D)$ with $\phi|_U \equiv 1$ and which vanishes in a neighborhood of $\partial^\dirichlet$.  The Radon-Nikodym derivative $\CZ$ of the law of $h_1|_U$ with respect to the law of $h_2|_U$ is given by
\begin{equation}
\label{eqn::gff_change_bc_rn}
\CZ = \E\!\left[ \exp\!\left( (h_2, \wt{g} )_\nabla - \| \wt{g} \|_\nabla^2/2 \right) \giv h_2|_U \right].
\end{equation}
\end{lemma}
\begin{proof}
We first recall that if $h$ is a GFF on a domain $D \subseteq \C$ and $f \in H(D)$ then the Radon-Nikodym derivative of the law of $h+f$ with respect to the law of $h$ is given by $\exp((h,f)_\nabla - \| f\|_\nabla^2/2)$.  (This is proved by using that the Radon-Nikodym derivative of the law of a $N(\mu,1)$ random variable with respect to the law of a $N(0,1)$ random variable is given by $e^{x \mu-\mu^2/2}$.)  We can extract from this the result as follows.  By the definition of $\wt{g}$, we have that $(h_2+\wt{g})|_U$ has the law of $h_1|_U$.  Moreover, we have that the Radon-Nikodym derivative of the law of $h_2 + \wt{g}$ with respect to the law of $h_2$ is given by
\[ \exp\big( (h_2 , \wt{g})_\nabla  - \| \wt{g} \|_\nabla^2/2 \big).\]
From this, the result immediately follows.
\end{proof}

Suppose that we are in the setting of Lemma~\ref{lem::gff_change_bc} and that there exists a constant $M > 0$ such that
\[ \sup_{z,w \in W} |g(z) -g(w)| \leq M\]
where $W$ is a neighborhood of the support of $\phi$.  Then elementary regularity estimates for harmonic functions yield that $\sup_{z} \|\phi(z)\|_\nabla \leq c_0 M$ where the supremum is over the support of $\phi$ and $c_0$ is a constant depending on the support of $\phi$.  Thus for a constant $c_1 > 0$ (depending on the particular choice of $\phi$) we have that
\begin{equation}
\label{eqn::harmonic_gradient_bound}
\| \wt{g} \|_\nabla^2 \leq c_1 M^2.
\end{equation}
Combining the bound~\eqref{eqn::harmonic_gradient_bound} with~\eqref{eqn::gff_change_bc_rn} and using, for example, the Cauchy-Schwarz inequality gives us a uniform lower bound on the probability of an event which depends on $h_2|_U$ in terms of the probability of the corresponding event computed using the law of $h_1|_U$.  We will make use of this fact shortly.

\begin{lemma}
\label{lem::free_bd_quantum_disk_exit_pos_prob}
Suppose that $D \subseteq \C$ is a bounded Jordan domain and $\partial D = \partial^\free \cup \partial^\dirichlet$ where $\partial^\free,\partial^\dirichlet$ are non-empty, disjoint intervals.  There exists $U \subseteq D$ open with positive distance to $\partial^\dirichlet$, $p > 0$, and $K < \infty$ such that the following is true.  Suppose that $h$ is a GFF on $D$ with free (resp.\ Dirichlet) boundary conditions on $\partial^\free$ (resp.\ $\partial^\dirichlet$) where the Dirichlet part differs from a given constant $A$ by at most $K$.  Pick $z \in D$ uniformly from the quantum measure.  Then the $\QLE(8/3,0)$ starting from $z$ has chance at least $p$ of hitting $\partial U$ first in $\partial^\free$ before reaching quantum distance time $e^{\gamma/4 (A+K)}$, $\gamma=\sqrt{8/3}$.
\end{lemma}
\begin{proof}
Suppose first that $(\D,\wt{h})$ is a unit boundary length quantum disk and $z \in \D$ is picked from the quantum measure associated with $\wt{h}$.  Lemma~\ref{lem::quantum_disk_exit} implies that a $\QLE(8/3,0)$ starting from $z$ and stopped upon first hitting $\partial \D$ a.s.\ hits $\partial \D$ at a unique point $w$.  Therefore there exists $0 < \theta_1 < \theta_2 < 2\pi$ and $p > 0$ so that with $I$ given by the counterclockwise arc of $\partial \D$ from $e^{i \theta_1}$ to $e^{i \theta_2}$ we have that the probability of $E_1 = \{w \in I\}$ is at least $p$.  Suppose that $x,y \in \partial \D$ are chosen independently according to the quantum boundary length measure.  Fix $\epsilon > 0$ so that $\theta_1 > 4 \epsilon$ and $\theta_2 < 2\pi - 4\epsilon$.  Let $E_2$ be the event that $x$ (resp.\ $y$) is in the counterclockwise arc of $\partial \D$ from $e^{i (\theta_1 - \epsilon)}$ to $e^{i \theta_1}$ (resp.\ $e^{i \theta_2}$ to $e^{i (\theta_2+\epsilon)}$).  Since the quantum boundary length measure is a.s.\ good, it follows that by possibly decreasing the value of $p > 0$, we have that the probability of $E_1 \cap E_2$ is at least $p$.  Let $\varphi \colon \D \to \strip$ be the unique conformal transformation so that $\varphi(x) = -\infty$, $\varphi(y) = +\infty$, and $\varphi(e^{i \theta_1}) = 0$.  Then $\wh{h} = \wt{h} \circ \varphi^{-1} + Q \log| (\varphi^{-1})'|$ is the field which describes a unit area quantum disk with the embedding as described in the Bessel process construction from Section~\ref{subsubsec::disks} (up to a horizontal translation).

What we have shown implies that there is a compact interval $\wh{I} \subseteq \R$ so that the probability that a $\QLE(8/3,0)$ starting from a point chosen from the quantum measure in $\strip$ associated with $\wh{h}$ first exits in $\wh{I}$ with probability at least $p > 0$.  Note that fixing a constant $C \in \R$ and then replacing $\wh{h}$ with $\wh{h} + C$ does not affect the probability of this event.  This implies that the same statement holds if we replace $\wh{h}$ with a sample from the law $\CM_\bes$ conditioned on the projection of the field onto $\CH_1(\strip)$ exceeding $0$.  This further implies that the same holds if we replace $\wh{h}$ with a sample from the law $\CM_\bes$ conditioned on the projection of the field onto $\CH_1(\strip)$ exceeding any fixed $r \leq 0$ since under this law the conditional probability that the projection exceeds $0$ is positive.  Therefore if we take the horizontal translation so that the projection first hits $0$ at $u=0$, then we may assume further that $\wh{I} \subseteq \R_+$.  Since $\QLE(8/3,0)$ a.s.\ hits the boundary for the first time at a unique point, we can also find $\wh{U} \subseteq \strip$ open whose boundary has positive distance to the top of $\partial \strip$ so that, possibly reducing $p > 0$, the probability that the $\QLE(8/3,0)$ up until first hitting $\partial \strip$ is in addition contained in $\wh{U}$ and first exits in $\wh{I}$ is at least $p$.  We may assume that $\partial \wh{U} \cap \R = \wh{I}$.  Lemma~\ref{lem::gff_change_bc} implies that the restriction of $\wh{h}$ to $\wh{U}$ under this law is absolutely continuous with respect to the corresponding restriction of a GFF on $\strip$ with free (resp.\ Dirichlet) boundary conditions on the bottom (resp.\ top) of $\partial \strip$.  Therefore the result follows by applying a final conformal map which takes $\strip$ to $D$ with $\R$ taken to $\partial^\free D$.
\end{proof}

We will now argue that if we place small neighborhoods at evenly spaced points on $\partial \cyl_+$ then the law of the field sampled from $\qconeUnex^1$ restricted to each such neighborhood is approximately independent of the field restricted to the other neighborhoods, up to an additive constant.

\begin{lemma}
\label{lem::maximum_dominated}
Suppose that $h$ has the law of a GFF on the annulus $D = [0,2\pi]^2 \subseteq \cyl_+$ (so that the top and bottom of $[0,2\pi]^2$ are identified) with free (resp.\ Dirichlet) boundary conditions on the left (resp.\ right) side of $\partial D = \partial^\free \cup \partial^\dirichlet$.  Fix $\epsilon > 0$ very small and let $x_1,\ldots,x_n$ be equally spaced points on $\partial^\free$ with spacing $\epsilon (\log \epsilon^{-1})^{-1}$.  Let $r = r_\epsilon = \epsilon (\log \epsilon^{-1})^{-2}$.  For each $k$, let $U_k = B(x_k, r) \cap \cyl_+$ and let $\Fh_k$ be the function which is harmonic in $\cyl_+ \setminus \cup_{j \neq k} U_j$ with boundary conditions given by those of $h$ on $\cup_{j \neq k} \partial U_j$ and Neumann boundary conditions on $\partial \cyl_+ \setminus \cup_{j \neq k} U_j$.  Let
\[ \Delta_k = \sup_{z,w \in U_k} |\Fh_k(z) - \Fh_k(w)|.\]
For each $M > 0$ there exist constants $K,c_0 > 0$ such that if $E = \{ \max_k \Delta_k \leq K\}$ then
\[ \p[ E^c ] \leq c_0 \epsilon^M.\]
\end{lemma}
\begin{proof}
By the odd/even decomposition of the GFF with mixed boundary conditions (see \cite[Section~6.2]{DS08} or \cite[Section~3.2]{SHE_WELD}), we can represent $h$ as the even part of a GFF $h^\dagger$ on the annulus $D^\dagger = [-2\pi,2\pi] \times [0,2\pi] \subseteq \cyl$ (so that the top and the bottom are identified) with Dirichlet boundary conditions.  Fix a value of $1 \leq k \leq n$.  The conditional law of $h^\dagger$ in $B(x_k,r)$ given its values on $B(x_j,r)$ for $j \neq k$ is given by that of the sum of a GFF on $D^\dagger \setminus\cup_{j \neq k} B(x_j,r)$ with zero boundary conditions and a harmonic function $\Fh_k^\dagger$.  By the odd/even decomposition, we note that 
\[ \Fh_k(z) = \frac{1}{\sqrt{2}} (\Fh_k^\dagger(z) + \Fh_k^\dagger(z^*))\]
where $z^*$ is the reflection of $z$ about the vertical axis through $0$.  Proposition~\ref{prop::gff_maximum} implies that there exists a constant $c_0 > 0$ such that the probability of the event that $|\Fh_k(z)| \leq M \log \epsilon^{-1}$ for all $z \in B(x_k, r \log \epsilon^{-1})$ is at least $1-c_0 \epsilon^{2M}$.  Elementary regularity results for harmonic functions then tell us that there exists a constant $c_1 > 0$ such that, on this event, we have
\[ \sup_{z,w \in B(x_k,r)} |\Fh_k(z) - \Fh_k(w)| \leq \frac{c_1 M \log \epsilon^{-1}}{ r \log \epsilon^{-1}} \times r = c_1 M.\]
Applying a union bound over $1 \leq k \leq n$ implies the result.
\end{proof}

Lemma~\ref{lem::maximum_dominated} implies that the restrictions of $h$ to the sets $U_k$ are approximately independent off an event with small probability.  We will now use this result (together with Lemma~\ref{lem::free_bd_quantum_disk_exit_pos_prob}) to argue that in each of the sets $U_k$, there is a positive chance that a point chosen from the quantum measure on $U_k$ has quantum distance to $\partial \cyl_+$ which is not too large.  Since the amount of quantum measure which is close to $\partial \cyl_+$ is small, it is unlikely that there will be a consecutive string of these points which are arbitrarily close to $\partial \cyl_+$.  Therefore, by binomial concentration, a positive fraction of these points will be swallowed by the $\QLE(8/3,0)$ growth from $\partial \cyl_+$.

\begin{lemma}
\label{lem::many_succeed}
Suppose that $\gamma=\sqrt{8/3}$, $\alpha \in (0,Q-2)$, and let $\beta = \gamma (Q-2-\alpha)/4$.  Suppose that we have the same setup as in Lemma~\ref{lem::maximum_dominated} and fix $1 \leq k \leq n$.  There exist $p > 0$ and $M < \infty$ such that the following is true.  Assume that $w$ is picked from the quantum area measure in $U_k$.  Given $\Delta_k \leq K$, $\Fh_k(x_k) \leq (2+\alpha) \log \epsilon^{-1}$,  and $\Fh_k$, the conditional probability that the $\QLE(8/3,0)$ starting from $w$ exits $U_k$ in $\partial \cyl_+$ in at most $\epsilon^\beta$ quantum distance time is at least $p$.
\end{lemma}
\begin{proof}
We will deduce the result from Lemma~\ref{lem::free_bd_quantum_disk_exit_pos_prob}.  We note that if we perform a change of coordinates from $U_k$ to $\{z \in \D : \re(z) \geq 0\}$ via the map $z \mapsto \epsilon^{-1} (\log \epsilon^{-1})^2 (z - x_k)$ then the correction to the field which comes  from the change of coordinates is $Q (\log \epsilon - 2 \log \log \epsilon^{-1})$.  By the definition of the event that we assume to be working on, we have that 
\[ \sup_{z \in U_k} \Fh_k(z) \leq (2+\alpha) \log \epsilon^{-1} + K.\]
The result thus follows as
\[ \frac{\gamma}{4} \left( (2+\alpha) \log \epsilon^{-1} + Q \log \epsilon \right) = \beta \log \epsilon.\]
\end{proof}

We now prove a result which, when combined with Lemma~\ref{lem::many_succeed}, will give a lower bound on the rate at which the distance of the metric ball growth from a given point decreases.

\begin{lemma}
\label{lem::capacity_map_lower_bound}
There exists a constant $c_0 > 0$ such that the following is true.  Fix $\epsilon > 0$ and suppose that $K \subseteq \ol{\cyl_+}$ is compact such that:
\begin{itemize}
\item $\cyl_+ \setminus K$ is simply connected and
\item For every $z \in \cyl_+$ with $\re(z) \in [\tfrac{\epsilon}{2},\epsilon]$ there exists $w \in K$ with $\re(w) \geq \epsilon (\log \epsilon^{-1})^{-2}/2$ and $|z-w| \leq \epsilon$.
\end{itemize}
Let $\phi_K \colon \cyl_+ \setminus K \to \cyl_+$ be the unique conformal map which fixes $+\infty$ and has positive derivative at $+\infty$.  For all $z \in \cyl_+ \setminus K$ with $\re(z) \in [\tfrac{\epsilon}{2},\epsilon]$ we have that
\begin{equation}
\label{eqn::map_back_change}
\re(z) - \re(\phi_K(z)) \geq  \frac{c_0 \epsilon}{(\log \epsilon^{-1})^4}.
\end{equation}
\end{lemma}
\begin{proof}
Let $\E^w$ denote the expectation under the law where $B$ is a standard Brownian motion starting from $w \in \cyl_+$ and let $\sigma$ be the first time that $B$ leaves $\cyl_+ \setminus K$.  As $\re(w) - \re(\phi_K(w))$ is harmonic in $\cyl_+ \setminus K$ and $\re(\phi_K(z)) \to 0$ as $z \in \cyl_+ \setminus K$ tends to $K \cup \partial \cyl_+$, we therefore have that $\re(z) - \re(\phi_K(z)) = \E^z[ \re(B_\sigma) ]$.  From the assumptions, we thus see that the probability of the event that $\re(B_\sigma) \geq \epsilon (\log \epsilon^{-1})^{-2}/4$ is at least a constant times $(\log \epsilon^{-1})^{-2}$ when $B_0 = z$.  Combining implies the result.
\end{proof}

\begin{proof}[Proof of Proposition~\ref{prop::disk_diameter_bounds}]
Fix $\epsilon > 0$.  Let $U_1,\ldots,U_n$ be as in Lemma~\ref{lem::maximum_dominated} and Lemma~\ref{lem::many_succeed}.  Assume that $\gamma = \sqrt{8/3}$ and let $\alpha \in (0,Q-2)$, $\beta = \gamma(Q-2-\alpha)/4$, $p > 0$, and $M < \infty$ be as in Lemma~\ref{lem::many_succeed}.  Let $K$ be the hull of the $\QLE(8/3,0)$ grown from $\partial \cyl_+$ for quantum distance time $\epsilon^\beta$.  Lemma~\ref{lem::maximum_dominated} and Lemma~\ref{lem::many_succeed} together imply that with probability at least $1-\epsilon^{2\alpha}$ for every $z \in \cyl_+$ with $\re(z) \in [\epsilon/2,\epsilon]$ there exists $w \in K$ such that $\re(w) \geq \epsilon / (\log \epsilon^{-1})^2$ and $|z-w| \leq \epsilon$.  Let $\phi_K$ be as in Lemma~\ref{lem::capacity_map_lower_bound}.  Then we have that $\re(z) - \re(\phi_K(z)) \geq c_0 \epsilon / (\log \epsilon^{-1})^4$.

If we iterate this procedure a constant times $(\log \epsilon^{-1})^4$ times then we see that the following is true.  Suppose that $K$ denotes the hull of the $\QLE(8/3,0)$ grown from $\partial \cyl_+$ for quantum distance time given by a constant times $(\log \epsilon^{-1})^4 \epsilon^\beta$ and let $\phi_K$ be as above.  Then on an event which occurs with probability at least $1-c_1 (\log \epsilon^{-1})^4 \epsilon^{2\alpha}$ for a constant $c_1 > 0$ we have for all $z \in \cyl_+$ with $\re(z) \in [\tfrac{\epsilon}{2},\epsilon]$ that $\re(z) - \re(\phi_K(z)) \geq \tfrac{\epsilon}{4}$.  The first assertion of the proposition follows by iterating this over dyadic values of $\epsilon$.

We now turn to prove the second assertion of the proposition (namely when we truncate on the amount of quantum area which is close to $\partial \cyl_+$).  The reason that we had the exponent of $\alpha$ in the above is that we needed the field to have average at most $(2+\alpha) \log \epsilon^{-1}$ in each of the $B(x_j,r)$.  Thus, we just need to argue that if we truncate on the amount of quantum area close to $\partial \cyl_+$ being at most $\epsilon^\sigma$, then with very high probability the field averages are not larger than $(2+\alpha) \log \epsilon^{-1}$ for some fixed value $\alpha \in (0,Q-2)$.  This, in turn, follows from \cite[Lemma~4.6]{DS08}.  Indeed, \cite[Lemma~4.6]{DS08} tells us that inside such a ball it is very unlikely for the field to assign mass smaller than
\[ \epsilon^{\gamma Q} \times \epsilon^{-\gamma(2+\alpha)} = \epsilon^{\gamma (Q-2-\alpha)}\]
and it is easy to see that we can make this exponent larger than $\sigma > 0$ provided we make $\alpha$ sufficiently close to $Q-2$.
\end{proof}

\subsection{Proof of H\"older continuity}
\label{subsec::combining}

\subsubsection{Proof of Theorem~\ref{thm::continuity}}
\label{subsubsec::proof_of_continuity}

The first step (Proposition~\ref{prop::typical_close}) in the proof of Theorem~\ref{thm::continuity} is to combine the estimates of the previous sections to bound the moments of the Euclidean diameter of a $\QLE(8/3,0)$ starting from $0$ on a $\sqrt{8/3}$-quantum cone.  The purpose of the subsequent lemmas is to transfer this estimate to the setting in which the $\QLE(8/3,0)$ is starting from another point.

\begin{proposition}
\label{prop::typical_close}
Suppose that $(\C,h,0,\infty)$ is a $\sqrt{8/3}$-quantum cone with the circle average embedding.  For every $\beta, \zeta > 0$ there exist constants $c_0,\alpha > 0$ such that the following is true.  With $Y_\epsilon =  \sup_{z \in B(0,\epsilon)} \qdist(0,z)$ and $H_{R,\zeta}$ as in Proposition~\ref{prop::cone_holder} we have that
\[ \pr{ Y_\epsilon \geq \epsilon^\alpha,\ H_{R,\zeta} } \leq c_0 \epsilon^\beta \quad\text{for all}\quad \epsilon \in (0,1).\]
\end{proposition}
In the setting of Proposition~\ref{prop::typical_close}, $\qdist(0,z)$ denotes the amount of time that the $\QLE(8/3,0)$ starting from $0$ and targeted at $z$ takes to reach $z$.  This function is defined for every $z \in \C$ simultaneously (with the $\QLE(8/3,0)$ always starting from $0$) but at this point we have not shown that it corresponds to a metric in this generality.  We will later (Lemma~\ref{lem::qle_metric_on_quantum_cone}) show that $\qdist$ defines a metric on a certain countable dense set of $\C$ and, upon completing the proof of Theorem~\ref{thm::metric_completion}, show that it extends to a metric $\oqdist$ on all of $\C$.
\begin{proof}[Proof of Proposition~\ref{prop::typical_close}]
Fix $\beta > 0$.  Let $\alpha,\delta,\xi > 0$ be parameters.  We will adjust their values in the proof.  Let $\Gamma$ be the hull of the $\QLE(8/3,0)$ exploration starting from $0$ and stopped at the first time that it reaches quantum radius $\epsilon^\delta$.  Let $E$ be the event that the quantum diameter of the hull of $\Gamma$ is smaller than $\epsilon^\alpha$ and let $F$ be the event that $B(0,\epsilon) \subseteq \Gamma$.  Note that $E \cap F$ implies $Y_\epsilon < \epsilon^\alpha$.  Thus we have that
\begin{align*}
         \pr{ Y_\epsilon \geq \epsilon^\alpha,\ H_{R,\zeta} }
       \leq \pr{ E^c \cap H_{R,\zeta}} + \pr{F^c \cap H_{R,\zeta}}.
\end{align*}
By Lemma~\ref{lem::qle_quantum_disk_size} (and the comment just after the statement), we know that by making $\alpha/\delta > 0$ small enough we have that $\pr{ E^c \cap H_{R,\zeta} } \leq c_0 \epsilon^\beta$ for a constant $c_0 > 0$.

Thus, we are left to bound $\pr{ F^c \cap H_{R,\zeta} }$.  Let $\wt{\Gamma}$ be the hull of the $\QLE(8/3,0)$ process grown for quantum distance time $\epsilon^\delta/2$.  We let $G$ be the event that the Euclidean diameter of $\wt{\Gamma}$ is at least $\epsilon^\xi$.  Then we have that
\begin{equation}
\label{eqn::f_comp_ubd}
\pr{ F^c \cap H_{R,\zeta} } \leq \pr{ F^c \cap G \cap H_{R,\zeta} } + \pr{ G^c \cap H_{R,\zeta} }.
\end{equation}
By adjusting the value of $c_0 > 0$ if necessary and making $\delta > 0$ small enough, Proposition~\ref{prop::euclidean_diameter_lower_bound} implies that the second term in~\eqref{eqn::f_comp_ubd} is bounded by $c_0 \epsilon^\beta$.  To handle the first term, we let $\varphi \colon \C \setminus \wt{\Gamma} \to \cyl_+$ be the unique conformal map with $\varphi(\infty) = +\infty$ and $\varphi'(\infty) > 0$.  Since the diameter of $\wt{\Gamma}$ on $G$ is at least $\epsilon^\xi$, it follows from the Beurling estimate that there exists a constant $c_1 > 0$ such that on $G$ we have $\sup_{y \in B(0,\epsilon)} \re(\varphi(y)) \leq c_1\epsilon^{(1-\xi)/2}$.  Thus by possibly decreasing the value of $\xi > 0$ and increasing the value of $c_0 > 0$, Proposition~\ref{prop::disk_diameter_bounds} implies that $\pr{F^c \cap G \cap H_{R,\zeta}} \leq c_0 \epsilon^\beta$.
\end{proof}

We next show that $\QLE(8/3,0)$ defines a metric on a countable, dense subset of a $\sqrt{8/3}$-quantum cone.

\begin{lemma}
\label{lem::qle_metric_on_quantum_cone}
Let $\eta'$ be a space-filling $\SLE_6$ process from $\infty$ to $\infty$ on a $\sqrt{8/3}$-quantum cone $(\C,h,0,\infty)$ sampled independently of $h$ and then parameterized by quantum area.  Fix $s_1 < s_2$ and let $(t_j)$ be an i.i.d.\ sequence in $[s_1,s_2]$ chosen from Lebesgue measure independently of everything else.  Then $\QLE(8/3,0)$ defines a metric $\qdist$ on $\{ \eta'(t_j) : j \in \N \}$.
\end{lemma}
\begin{proof}
Throughout the proof, we will write $\qdist(s,t)$ for $\qdist(\eta'(s),\eta'(t))$.  First, we note that Proposition~\ref{prop::typical_close} implies that $\qdist(s,t) < \infty$ a.s.\ for any fixed $s,t \in \R$.  Fix $s \in \R$ and suppose that we have recentered the quantum cone so that $\eta'(t) = 0$ and then we rescale so that we have the circle average embedding.  By \cite[Theorem~1.13]{dms2014mating}, the resulting field has the same law as $h$.  Then Proposition~\ref{prop::ball_size_ubd} implies that the diameter of the $\QLE(8/3,0)$ running from $\eta'(t) = 0$ stopped at the first time that it hits $\eta'(s)$ is finite a.s.\ and that the same is true when we swap the roles of $\eta'(s)$ and $\eta'(t)$.  Fix $R > 0$.  As the restriction of $h$ to $B(0,R)$ is mutually absolutely continuous with respect to the corresponding restriction of a quantum sphere with large area, by fixing $R > 0$ sufficiently large it follows from the main result of \cite{qlebm} that $\qdist(t_i,t_j) = \qdist(t_i,t_j)$ for all $i,j \in \N$.  Applying the same argument but with three points implies that the triangle inequality is satisfied.
\end{proof}

\begin{lemma}
\label{lem::space_filling_sle_event}
For each $p \in (0,1)$ there exists $s_1, s_2 \in \R$ with $s_1 < s_2$, $z_0 \in \D \setminus \{0\}$, and $r_1 > 0$ with $\ol{B(z_0,r_1)} \subseteq \D \setminus \{0\}$ such that the following is true.  Suppose that $(\C,h,0,\infty)$ is a $\sqrt{8/3}$-quantum cone with the circle average embedding and let $\eta'$ be a space-filling $\SLE_6$ process from $\infty$ to $\infty$ sampled independently of $h$ and then reparameterized according to $\sqrt{8/3}$-LQG area.  Then with
\begin{equation}
\label{eqn::e_def}
E(z_0,r_1,s_1,s_2) = \{ B(z_0,r_1) \subseteq \eta'([s_1,s_2]) \subseteq \D\}
\end{equation}
we have that $\pr{ E(z_0,r_1,s_1,s_2) } \geq p$.
\end{lemma}
\begin{proof}
First, we consider the ball $B(\tfrac{1}{2},\tfrac{1}{4})$.  Fix $p > 0$.  Then we know that there exists $R_0 > 0$ such that for all $R \geq R_0$ we have that $\eta'([-R^2,R^2])$ (with the Lebesgue measure parameterization) contains $B(\tfrac{1}{2},\tfrac{1}{4})$ with probability at least $p$.  Fix $\epsilon > 0$.  By rescaling space by the factor $\epsilon/R$, we have that the probability that $\eta'([-\epsilon^2,\epsilon^2])$ (with the Lebesgue measure parameterization) contains $B(\tfrac{\epsilon}{2R}, \tfrac{\epsilon}{4R})$ is at least $p$.  The result follows because by making $\epsilon > 0$ sufficiently small, we can find $\delta > 0$ such that $\eta'([-\delta,\delta])$ (with the quantum area parameterization) is contained in $\D$ and contains $\eta'([-\epsilon^2,\epsilon^2])$ (with the Lebesgue measure parameterization) with probability at least~$p$.
\end{proof}

\begin{lemma}
\label{lem::quantum_cone_rescaling}
Suppose that $(\C,h,0,\infty)$ is a $\sqrt{8/3}$-quantum cone with the circle average embedding.  Let $\eta'$ be a space-filling $\SLE_6$ from $\infty$ to $\infty$ sampled independently of $h$ and then reparameterized by $\sqrt{8/3}$-LQG area.  For each $t \in \R$ and $r > 0$, let $\wt{h}^{t,r}$ be the field which is obtained by translating so that $\eta'(t)$ is sent to the origin and then rescaling by the factor $r$. Fix $0 < s_1 < s_2$.  For each $t \in [s_1,s_2]$, let $R(t)$ be such that $\wt{h}^{t,R(t)}$ has the circle average embedding.  Fix $r_1 > 0$ and $z_0 \in \D$ so that $\ol{B(z_0,r_1)} \subseteq \D \setminus \{0\}$ and let $E(z_0,r_1,s_1,s_2)$ be as in~\eqref{eqn::e_def}.  There exist constants $c_0,c_1 > 0$ such that for each $d \in (0,1)$ with
\begin{equation}
\label{eqn::f_def}
F(z_0,r_1,s_1,s_2,d) = \left\{ \forall t \in [s_1,s_2] : \eta'(t) \in B(z_0,r_1), R(t) \in [d,d^{-1}] \right\}
\end{equation}
we have
\begin{equation}
\label{eqn::f_prob}
\pr{F(z_0,r_1,s_1,s_2,d)^c \cap E(z_0,r_1,s_1,s_2)} \leq c_0 d^{c_1}.
\end{equation}
\end{lemma}
\begin{proof}
We note that $R(t) < r$ is equivalent to $\inf_{s \geq r} (h_s(\eta'(t)) + Q \log s) > 0$.  Therefore the event that $R(t) < r$ for some $t \in [s_1,s_2]$ so that $\eta'(t) \in B(z_0,r_1)$ is equivalent to
\[ \sup_{\substack{t \in [s_1,s_2]\\ \eta'(t) \in B(z_0,r_1)}} \left( \inf_{s \geq r} \big(h_s(\eta'(t)) + Q\log s \big) \right) > 0.\]
This event is in turn contained in $\sup_{z \in B(z_0,r_1)} (h_r(z) + Q \log r) > 0$.  Using that $Q > 2$, Proposition~\ref{prop::gff_maximum} implies that there exist constants $c_0,c_1 > 0$ such that 
\begin{equation}
\label{eqn::f_prob_1}
\pr{ \sup_{z \in B(z_0,r_1)}  \big( h_r(z) + Q \log r \big) > 0 } \leq c_0 r^{c_1}
\end{equation}
for all $r \in (0, \dist(B(z_0,r_1),\partial \D))$.  (Note that we can apply Proposition~\ref{prop::gff_maximum} here because, by our normalization, the law of $h$ restricted to $\D$ is equal to that of a whole-plane GFF plus $-\gamma \log |z|$ normalized to have average equal to $0$ on $\partial \D$.)  This implies the desired upper bound for the probability that $R(t) < d$ for some $t \in [s_1,s_2]$ with $\eta'(t) \in B(z_0,r_1)$.

The desired upper bound for the probability that $R(t) > d^{-1}$ for some $t \in [s_1,s_2]$ with $\eta'(t) \in B(z_0,r_1)$ follows because $h_s(0) + Q \log s > 0$ for all $s \geq 1$ because $h$ has the circle average embedding.  In fact, since $h_s(0) + Q \log s$ for $s \geq 1$ evolves as a time-change of a Brownian motion with positive drift $(Q-\gamma) > 0$ conditioned to be non-negative, the probability that $\inf_{s \geq r} (h_s(0) + Q \log s) \leq 1$ decays to $0$ faster than a negative power of $r$ as $r \to \infty$.  It therefore suffices to show that $\sup_{s \geq r} \sup_{z \in B(z_0,r_1)} |h_s(0) - h_s(z)| \geq 1$ decays to $0$ faster than a negative power of $r$ as $r \to \infty$.  This, in turn, follows from Proposition~\ref{prop::kc_continuity} together with Lemma~\ref{lem::gff_circ_covariance}.
\end{proof}

\begin{lemma}
\label{lem::pick_points_interior}
Suppose that $(\C,h,0,\infty)$ is a $\sqrt{8/3}$-quantum cone with the circle average embedding.  Let $E(z_0,r_1,s_1,s_2)$ be as in~\eqref{eqn::e_def}.  Let $(w_j)$ be an i.i.d.\ sequence of points picked from $\mu = \mu_h$ restricted to $\eta'([s_1,s_2])$ and let $N = \delta (\log \epsilon^{-1}) \epsilon^{-\gamma Q - (2+\delta)\gamma}$ where $\gamma = \sqrt{8/3}$.  Let $G$ be the event that $\{w_1,\ldots,w_N\} \cap B(z_0,r_1)$ forms an $\epsilon$-net of $B(z_0,r_1)$ (i.e., $B(z_0,r_1) \subseteq \cup_{j=1}^N B(w_j,\epsilon)$).  There exists a constant $c_0 > 0$ which depends on $z_0,r_1$ such that
\[ \pr{ G^c \cap E(z_0,r_1,s_1,s_2) } \leq c_0 \epsilon^{\delta}.\]
\end{lemma}
\begin{proof}
Let $z_1,\ldots,z_k$ be the elements of $\tfrac{\epsilon}{4} \Z^2$ which are contained in $B(z_0,r_1)$.  Proposition~\ref{prop::gff_maximum} implies that for each $\xi \in (0,1)$ there exists a constant $c_0 > 0$ such that
\begin{equation}
\label{eqn::circ_ave_ball_cone_min_bound}
\pr{ \min_{1 \leq j \leq k} h_\epsilon(z_j) \leq (2+\delta) \log \epsilon} \leq c_0 \epsilon^{2\delta(1-\xi)}.
\end{equation}
Combining~\cite[Lemma~4.6]{DS08} with~\eqref{eqn::circ_ave_ball_cone_min_bound}, we have by possibly adjusting the values of $c_0 > 0$ and $\xi$ that
\begin{equation}
\label{eqn::min_mass_bound_cone}
\pr{ \min_{1 \leq j \leq k} \mu_h(B(z_j,\epsilon)) \leq \epsilon^{\gamma Q + (2+\delta)\gamma}} \leq c_0 \epsilon^{2\delta(1-\xi)}.
\end{equation}
On the complement of the event in~\eqref{eqn::min_mass_bound_cone}, the probability that none of $w_1,\ldots,w_N$ are contained in $B(z_j,\epsilon)$ is at most
\[ (1-\epsilon^{\gamma Q + (2+\delta)\gamma})^N \leq \exp(-N \epsilon^{\gamma Q + (2+\delta)\gamma}) \leq \epsilon^{\delta}.\]
Combining this with~\eqref{eqn::min_mass_bound_cone} implies the result.
\end{proof}

We will now use Proposition~\ref{prop::typical_close} and Lemmas~\ref{lem::space_filling_sle_event}--\ref{lem::pick_points_interior} to prove that $\qdist$ is H\"older continuous with positive probability on $B(z_0,r_1)$.  We will afterwards explain how to deduce from this the almost sure local H\"older continuity of $\qdist$ on all of $\C$, thus finishing the proof of Theorem~\ref{thm::continuity}.

\begin{lemma}
\label{lem::continuity_on_event}
Suppose that $(\C,h,0,\infty)$ is a $\sqrt{8/3}$-quantum cone with the circle average embedding.  On the events $E(z_0,r_1,s_1,s_2)$, $F(z_0,r_1,s_1,s_2,d)$ from~\eqref{eqn::e_def}, \eqref{eqn::f_def}, we have that the quantum distance $\qdist$ restricted to pairs of points in $B(z_0,r_1)$ is a.s.\ H\"older continuous (with deterministic H\"older exponent).
\end{lemma}
\begin{proof}
Throughout, we shall assume that we are working on the event $H_{R,\zeta}$ of Proposition~\ref{prop::cone_holder} and we will prove the almost sure H\"older continuity on this event.  We note that it suffices to do so since Proposition~\ref{prop::cone_holder} implies that $\pr{H_{R,\zeta}} \to 1$ as $R \to 0$ with $\zeta$ fixed.

For each $j$, we let $N_j = e^{9j}$ (note that $9 > (Q+3) \gamma$ for $\gamma=\sqrt{8/3}$) and we pick $\CU_j = \{w_1^j,\ldots,w_{N_j}^j\}$ i.i.d.\ from the $\sqrt{8/3}$-LQG area measure restricted to $\eta'([s_1,s_2])$.  Equivalently, we can first pick $t_1^j,\ldots,t_{N_j}^j$ i.i.d.\ from $[s_1,s_2]$ uniformly using Lebesgue measure and then take $w_i^j = \eta'(t_i^j)$.  We assume that the $\CU_j$ are also independent as $j$ varies.  By Lemma~\ref{lem::pick_points_interior}, the probability of the event $E_j$ that $\CU_j$ forms an $e^{-j}$-net of $B(z_0,r_1)$ is at least $1- c_0 e^{-j}$ where $c_0 > 0$ is a constant.

By \cite[Theorem~1.13]{dms2014mating}, we have that the joint law of $(\C,h,0,\infty)$ and $\eta'$ is invariant under the operation of translating so that $\eta'(t_i^j)$ is sent to the origin and then rescaling so that the resulting surface has the circle average embedding.  Let $h^{i,j}$ be the resulting field.  Proposition~\ref{prop::typical_close} implies that for each $\beta > 0$ we can find $\alpha > 0$ such that for each $i,j$, the probability that the quantum diameter of $B(0,e^{-j})$ as measured using the field $h^{i,j}$ is larger than $e^{-\alpha j}$ is at most $c_1 e^{- \beta j}$ where $c_1 > 0$ is a constant.  Thus on the event $E_j$, this implies that the probability that the quantum diameter of $B(w_i^j,e^{-j})$ is larger than $e^{-\alpha j}$ is at most $c_2 e^{-\beta j}$ for a constant $c_2 > 0$.  Therefore by a union bound, the probability that the quantum diameter of any of the $B(w_i^j,e^{-j})$ is larger than $e^{-\alpha j}$ is at most $c_1 e^{(9-\beta) j}$.  Assume that $\beta > 9$.  Thus by the Borel-Cantelli lemma, it follows that there a.s.\ exists $J_0 < \infty$ (random) such that $j \geq J_0$ implies that $\CU_j$ is an $e^{-j}$-net of $B(z_0,r_1)$ and the maximal quantum diameter of $B(w_i^j,e^{-j})$ is $e^{-\alpha j}$ for all $1 \leq i \leq N_j$.

Let $\CU = \cup_j \CU_j$.  We will now extract the H\"older continuity of $(z,w) \mapsto \qdist(z,w)$ for $z,w \in \wt{\CU} = \CU \cap B(z_0,r_1)$.  Assume that $z,w,z',w' \in \wt{\CU}$ and assume that $\qdist(z,w) \geq \qdist(z',w')$.  By repeated applications of the triangle inequality (Lemma~\ref{lem::qle_metric_on_quantum_cone}), we have that
\begin{align}
         \qdist(z,w) - \qdist(z',w')
 &\leq \qdist(z,z') + \qdist(z',w) - \qdist(z',w') \notag\\
 &\leq \qdist(z,z') + \qdist(z',w') + \qdist(w',w) - \qdist(z',w') \notag\\
 &= \qdist(z,z') + \qdist(w,w'). \label{eqn::holder_triangle_bound}
 \end{align}
Consequently, it suffices to show that there exist constants $M, a > 0$ such that $\qdist(z,w) \leq M|z-w|^a$ for all $z,w \in \wt{\CU}$.  It in fact suffices to show that this is the case for all $z,w \in \wt{\CU}$ with $|z-w| \leq e^{-J_0}$ and $J_0$ as above.  Indeed, if this is the case and $z,w \in \wt{\CU}$ are such that $|z-w| > e^{-J_0}$ then we can find $z_0 = z,z_1,\ldots,z_{n-1},z_n = w \in \wt{\CU}$ with $n \leq e^{J_0}$ and then we can use the triangle inequality to get that
\[ \qdist(z,w) \leq \sum_{j=1}^n \qdist(z_{j-1},z_j) \leq M n e^{- a J_0} \leq M e^{J_0} |z-w|^\alpha.\]

Fix $z,w \in \wt{\CU}$ with $|z-w| \leq e^{-J_0}$ and take $j_0 \in \N$ so that $e^{-j_0-1} \leq |z-w| \leq e^{-j_0}$.  Then we can find $u_0=v_0$ in $\CU_{j_0}$ such that $|u_0 - w| \leq e^{-j_0}$ and $|v_0 - z| \leq e^{-j_0}$.  This implies that $\qdist(u_0,w) \leq e^{-\alpha j_0}$ and $\qdist(v_0,z) \leq e^{-\alpha j_0}$.  For each $j \geq j_0+1$, we can inductively find $u_{j-j_0}, v_{j-j_0} \in \CU_j$ with $|u_{j-j_0} - u_{j-j_0-1}| \leq e^{-j}$, $|v_{j-j_0} - v_{j-j_0-1}| \leq e^{-j}$, hence $\qdist(u_{j-j_0},u_{j-j_0-1}) \leq e^{-\alpha j}$ and $\qdist(v_{j-j_0},v_{j-j_0-1}) \leq e^{-\alpha j}$.  Therefore we have that
\begin{align*}
     \qdist(z,w)
&\leq  \qdist(u_0,w) + \qdist(v_0,z) + \sum_{i=1}^\infty \qdist(u_i,u_{i-1}) + \sum_{i=1}^\infty \qdist(v_i,v_{i-1})\\
&\leq M_0 \sum_{i=0}^\infty e^{-\alpha(i+j_0)}
  \leq M_1 e^{-\alpha j_0}
 \leq M_2 |z - w|^{\alpha}
\end{align*}
where $M_0,M_1,M_2 > 0$ are constants.  Therefore $(z,w) \mapsto \qdist(z,w)$ is a.s.\ H\"older continuous on $\wt{\CU}$.  Since the set $\wt{\CU}$ is a.s.\ dense in $B(z_0,r_1)$, it follows that $\qdist$ a.s.\ extends to be H\"older continuous in $(z,w)$ for $z,w \in B(z_0,r_1)$.
\end{proof}

\begin{proof}[Proof of Theorem~\ref{thm::continuity}]
By Lemma~\ref{lem::continuity_on_event}, we know that on the events $E(z_0,r_1,s_1,s_2)$ and $F(z_0,r_1,s_1,s_2,d)$ we have the almost sure H\"older continuity of $\qdist$ restricted to pairs $z,w \in B(z_0,r_1)$.  Fix $t \in (s_1,s_2)$.  Since translating by $-\eta'(t)$ and then rescaling so that the surface has the circle average embedding is a measure preserving transformation, it follows that the probability that $\qdist$ is H\"older continuous in a neighborhood of the origin is at least the probability $p$ of $E(z_0,r_1,s_1,s_2)$ and $F(z_0,r_1,s_1,s_2,d)$.  Since the law of $(\C,h,0,\infty)$ is invariant under the operation of multiplying its area by a constant, we have that the probability that $\qdist$ is H\"older continuous in any compact subset of $\C$ is also at least~$p$.  The almost sure continuity of the $\QLE(8/3,0)$ metric on a $\sqrt{8/3}$-quantum cone restricted to compact subsets of $\C$  follows because by Lemma~\ref{lem::space_filling_sle_event} and Lemma~\ref{lem::quantum_cone_rescaling} we know that by adjusting the values of $z_0,r_1,s_1,s_2,d$, we can make $p$ as close to $1$ as we want.  The result in the case of a $\sqrt{8/3}$-quantum sphere follows by absolute continuity.
\end{proof}

\subsubsection{Proof of Theorem~\ref{thm::metric_completion}}
\label{subsubsec::metric_completion}

We let $E(z_0,r_1,s_1,s_2)$ be as in~\eqref{eqn::e_def} and let $F(z_0,r_1,s_1,s_2,d)$ be as~\eqref{eqn::f_def}.  As in the proof of Lemma~\ref{lem::continuity_on_event}, we shall assume throughout that we are working on the event $H_{R,\zeta}$ of Proposition~\ref{prop::cone_holder}.  For each $j$, we let $N_j = e^{9 j}$ and we pick $\CU_j = \{w_1^j,\ldots,w_{N_j}^j\}$ i.i.d.\ from the $\sqrt{8/3}$-LQG measure restricted to $\eta'([s_1,s_2])$.  Equivalently, we can first pick $t_1^j,\ldots,t_{N_j}^j$ i.i.d.\ from $[s_1,s_2]$ uniformly and then take $w_i^j = \eta'(t_i^j)$.  We assume that the $\CU_j$ are also independent as $j$ varies.  By the proof of Lemma~\ref{lem::continuity_on_event}, we know that the probability that every point in $B(z_0,r_1)$ is within quantum distance at most $e^{-\alpha j}$ of some point in $\CU_j$ is at least $1-c_0 e^{-\beta j}$ for constants $\alpha,\beta,c_0 > 0$.  Moreover, by we can make $\beta > 0$ as large as we want by possibly decreasing the value of $\alpha > 0$.  It follows from Proposition~\ref{prop::ball_size_ubd} that the probability that the Euclidean diameter of $\qball{w_i^j}{e^{-\alpha j}}$ is larger than $e^{-\wt{\alpha}j}$ is at most $c_1 e^{-\wt{\beta} j}$ for constants $\wt{\alpha},\wt{\beta},c_1 > 0$.  Moreover, we are free to choose $\wt{\beta}$ as large as we want.  In particular, by taking $\beta, \wt{\beta} > 9$ so that
\[ \sum_j N_j \cdot e^{-\beta j} < \infty \quad\text{and}\quad \sum_j N_j \cdot e^{-\wt{\beta} j} < \infty ,\]
it follows from the Borel-Cantelli lemma that there a.s.\ exists $J_0 < \infty$ (random) such that $j \geq J_0$ implies that every point in $B(z_0,r_1)$ is contained in $\qball{w_i^j}{e^{-\alpha j}}$ for some $j$ and the Euclidean diameter of $\qball{w_i^j}{e^{-\alpha j}}$ is smaller than~$e^{-\wt{\alpha} j}$.

As explained in~\eqref{eqn::holder_triangle_bound} in the proof of Lemma~\ref{lem::continuity_on_event}, it suffices to show that there exist constants $M, a> 0$ such that $|z-w| \leq M (\oqdist(z,w))^a$ for all $z,w \in B(z_0,r_1)$.  In fact, it suffices to show that this is the case for all $z,w \in B(z_0,r_1)$ such that $\oqdist(z,w) \leq e^{-\alpha J_0}$.  So, suppose that $z,w \in B(z_0,r_1)$ are such that $\oqdist(z,w) \leq e^{-\alpha J_0}$ and let $j_0 \in \N$ be such that $e^{-\alpha(j_0+1)} \leq \oqdist(z,w) \leq e^{-\alpha j_0}$.  Then we can find $w_0=v_0 \in \CU_{j_0}$ such that $\oqdist(u_0,w) \leq e^{-\alpha j_0}$ and $\oqdist(v_0,z) \leq e^{-\alpha j_0}$.  This implies that $|u_0-w| \leq e^{-\wt{\alpha} j_0}$ and $|v_0-z| \leq e^{-\wt{\alpha} j_0}$.  For each $j \geq j_0+1$, we can inductively find $u_{j-j_0},v_{j-j_0} \in \CU_j$ with $\oqdist(u_{j-j_0}, u_{j-j_0-1}) \leq e^{-\alpha j}$, $\oqdist(v_{j-j_0}, v_{j-j_0-1}) \leq e^{-\alpha j}$ hence $|u_{j-j_0} - u_{j-j_0-1}| \leq e^{-\wt{\alpha} j}$ and $|v_{j-j_0} - v_{j-j_0-1}| \leq e^{-\wt{\alpha} j}$.  We therefore have that
\begin{align*}
         |z-w|
&\leq |z-v_0| + |w-w_0| + \sum_{i=1}^\infty |v_i-v_{i-1}| + \sum_{i=1}^\infty |w_i-w_{i-1}|\\
&\leq M_0 \sum_{i=1}^\infty e^{ -(i+j_0) \wt{\alpha}} \leq M_1 e^{-\wt{\alpha} j_0} \leq M_2 \oqdist(z,w)^{\alpha/\wt{\alpha}}
\end{align*}
where $M_0,M_1,M_2$ are constants.  This implies that $\oqdist$ is positive definite on $B(z_0,r_1)$ on the event $F(z_0,r_1,s_1,s_2,d)$.  Symmetry and the triangle inequality of $\oqdist$ are immediate from the corresponding properties of $\qdist$ and the continuity result from Theorem~\ref{thm::continuity}.  Altogether, we have shown that $\oqdist$ defines a metric on $B(z_0,r_1)$ on the event $F(z_0,r_1,s_1,s_2,d)$ and $(z,w) \mapsto |z-w|$ is H\"older continuous (with deterministic exponent) with respect to the metric defined by $\oqdist$ on $B(z_0,r_1)$ on the event $F(z_0,r_1,s_1,s_2,d)$.  The argument explained in the proof of Theorem~\ref{thm::continuity} then implies that $\oqdist$ is positive definite on $\C$ hence defines a metric on $\C$ and $(z,w) \mapsto |z-w|$ is a.s.\ H\"older continuous (with deterministic exponent) with respect to $\oqdist$ when restricted to a compact subset of $\C$.  Theorem~\ref{thm::continuity} implies the a.s.\ H\"older continuity (with deterministic exponent) of the map in the other direction, which completes the proof in the case of a $\sqrt{8/3}$-quantum cone.  The continuity in the case of a $\sqrt{8/3}$-quantum sphere follows by absolute continuity.
\qed

\subsubsection{Existence and continuity of geodesics: proof of Theorem~\ref{thm::geodesic_metric_space}}
\label{subsec::geodesic_existence}

We now show that the metric space that we have constructed is a.s.\ geodesic.  We will then show (Proposition~\ref{prop::geodesics_typical_unique}) that geodesics between quantum typical points are a.s.\ unique and (Proposition~\ref{prop::boundary_to_inside_unique}) that there is a.s.\ a unique geodesic from a typical point on the boundary of a metric ball to its center (assuming this center is also a ``typical'' point).

\begin{proposition}
\label{prop::geodesics_exist}
Suppose that $\CS$ is a $\sqrt{8/3}$-LQG sphere and let $\oqdist$ be the corresponding $\QLE(8/3,0)$ metric.  Then $(\CS,\oqdist)$ is a.s.\ geodesic.
\end{proposition}

Proposition~\ref{prop::geodesics_exist} is a consequence of the following two general observations about compact metric spaces and the proof of the metric property given in \cite{qlebm}.  In this section, we will use the notation $(X,d)$ for a metric space, $(X,d,\mu)$ for a metric measure space, and let $B(x,r) = \{y \in X : d(x,y) < r\}$ be the open ball of radius $r$ centered at $x$ in $X$.

\begin{lemma}
\label{lem::geodesic_characterization}
Suppose that $(X,d)$ is a compact metric space.  Then $(X,d)$ is geodesic if and only if for all $x,y \in X$ there exists $z \in X$ such that $d(x,z) = d(z,y) = d(x,y)/2$.
\end{lemma}
\begin{proof}
Suppose that $(X,d)$ is geodesic and $x,y \in X$.  Then there exists a geodesic $\eta \colon [0,d(x,y)] \to X$ connecting $x$ and $y$ and $z = \eta(d(x,y)/2)$ satisfies $d(x,z) = d(y,z) = d(x,y)/2$.

Conversely, we suppose that for all $x,y \in X$ there exists $z \in X$ with $d(x,z) = d(y,z) = d(x,y)/2$.  Fix $x,y \in X$.  We iteratively define a function $\eta$ on the dyadic rationals in $[0,1]$ as follows.  We first pick $z$ so that $d(x,z) = d(z,y) = d(x,y)/2$ and set $\eta(1/2) = z$.  By iterating this construction in the obvious way, we have that $\eta$ satisfies
\[ | \eta( r d(x,y)) - \eta( s d(x,y)) | = d(x,y)|r-s|.\]
Hence it is easy to see that $\eta$ extends to a continuous map $[0,1] \to X$ which (after reparameterizing time by the constant factor $d(x,y)$) is a geodesic connecting~$x$ and~$y$.
\end{proof}

\begin{lemma}
\label{lem::metric_measure_geodesic_characterization}
Suppose that $(X,d,\mu)$ is a good-measure endowed compact metric space (recall the definition from Section~\ref{subsubsec::strategytbmlaw}).  Then $(X,d)$ is geodesic if and only if the following property is true.  Suppose that $x,y$ are chosen from $\mu$ and $U \in [0,1]$ uniformly with $x,y,U$ independent.  With $r = U d(x,y)$ and $\ol{r} = d(x,y)-r$ there a.s.\ exists $z \in \partial B(x,r) \cap \partial B(y,\ol{r})$ such that $d(x,z) + d(z,y) = d(x,y)$.
\end{lemma}
\begin{proof}
It is clear that if $X$ is geodesic then the property in the lemma statement holds because we can take $z$ to be a point along a geodesic from $x$ to $y$ in $\partial B(x,r) \cap \partial B(y,\ol{r})$.

Suppose that the property in the lemma statement holds.  We will show that $(X,d)$ is geodesic by verifying the condition from Lemma~\ref{lem::geodesic_characterization}.  Suppose that $(x_n)$, $(y_n)$ are independent i.i.d.\ sequences chosen from $\mu$, that $(U_n)$ is an i.i.d.\ sequence of uniform random variables in $[0,1]$ which are independent of $(x_n)$, $(y_n$), and $r_n = U_n d(x_n,y_n)$ and $\ol{r}_n = d(x_n,y_n) - r_n$.  The following is then a.s.\ true for all $x,y \in X$ distinct and $k \in \N$.  There exists $n_k$ such that $d(x_{n_k},x) < 1/k$, $d(y_{n_k},y) < 1/k$, and $|U_{n_k} - 1/2| < 1/k$.  Let $z_{n_k} \in \partial B(x_{n_k},r) \cap \partial B(y_{n_k},r)$ be such that $d(x_{n_k},z_{n_k}) + d(z_{n_k},y_{n_k}) = d(x_{n_k},y_{n_k})$.  Let $(\wt{z}_m)$ be a convergent subsequence of $z_{n_k}$ and let $z = \lim_m \wt{z}_m$.  By the continuity of $d$, we have that $d(x,z) + d(z,y) = d(x,y)$ and $d(x,z) = d(y,z) = d(x,z)/2$.
\end{proof}

\begin{proof}[Proof of Proposition~\ref{prop::geodesics_exist}]
This follows by combining Lemma~\ref{lem::geodesic_characterization} and Lemma~\ref{lem::metric_measure_geodesic_characterization} and the construction of $\qdist$ given in \cite{qlebm}.
\end{proof}

\begin{proof}[Proof of Theorem~\ref{thm::geodesic_metric_space}]
The first part of the theorem follows from Proposition~\ref{prop::geodesics_exist}.  The second part of the theorem follows because a geodesic on a $\sqrt{8/3}$-LQG sphere is $1$-Lipschitz with respect to the $\QLE(8/3,0)$ metric, so the result follows by combining with Theorem~\ref{thm::continuity} and Theorem~\ref{thm::metric_completion}.
\end{proof}

\begin{proposition}
\label{prop::geodesics_typical_unique}
Suppose that $\CS$ is a $\sqrt{8/3}$-LQG sphere and that $\oqdist$ is the associated $\QLE(8/3,0)$ metric.  Assume that $x,y \in \CS$ are picked uniformly from the quantum measure.  Then there a.s.\ exists a unique geodesic connecting $x$ and $y$.
\end{proposition}
\begin{proof}
Proposition~\ref{prop::geodesics_exist} implies that there exists at least one geodesic $\eta$ connecting $x$ and $y$.  Suppose that $\ol{\eta}$ is another geodesic.  By \cite[Lemma~7.6]{qlebm}, if we let $r = Ud(x,y)$ where $U$ is uniform in $[0,1]$ independently of everything else and $\ol{r} = d(x,y) - r$ then $\partial B(x,r) \cap \partial B(y,\ol{r})$ a.s.\ intersect at a unique point.  This implies that $\eta(U d(x,y)) = \ol{\eta}(U d(x,y))$ which, in turn, implies that on a set of full Lebesgue measure in $[0,d(x,y)]$ we have that $\eta(t) = \ol{\eta}(t)$.  Therefore $\eta = \ol{\eta}$ by the continuity of the paths.
\end{proof}

\begin{proposition}
\label{prop::boundary_to_inside_unique}
Suppose that $\CS$ is a $\sqrt{8/3}$-LQG sphere, let $\oqdist$ be the associated $\QLE(8/3,0)$ metric, and assume that $x,y \in \CS$ are chosen independently from the quantum measure.  Fix $r > 0$ and assume that we are working on the event that $\qdist(x,y) > r$.  Suppose that $z$ is chosen uniformly from the quantum boundary measure on the boundary of the filled metric ball centered at $x$ of radius $r$.  Then there is a.s.\ a unique geodesic from $z$ to $x$.  The same holds if we replace $r$ with $\qdist(x,y) - r$.
\end{proposition}
\begin{proof}
This result is a consequence of Proposition~\ref{prop::geodesics_typical_unique} because we can sample from the law of $z$ by growing a metric ball from $y$ and taking $z$ to be the unique intersection point of this ball with the filled metric ball starting from $x$.
\end{proof}

\subsubsection{The internal metric}
\label{subsubsec::internal_metric}

We next turn to construct the internal metric $\oqdist^U$ associated with the restriction of $\oqdist$ to a domain $U$.  The almost sure finiteness of $\oqdist^U$ will rely on Theorems~\ref{thm::continuity}--\ref{thm::geodesic_metric_space}.
\begin{proposition}
\label{prop::internal_metric}
Suppose that $(\C,h,0,\infty)$ is a $\sqrt{8/3}$-quantum cone and that $U \subseteq \C$ is an open domain.  For each $x,y \in U$ we let $U_{x,y}$ be the set of paths in $U$ which connect $x,y$ and, for $\eta \in U_{x,y}$, we let $\qlength(\eta)$ be the $\oqdist$-length of $\eta$.  We let
\[ \oqdist^U(x,y) = \inf_{\eta \in U_{x,y}} \qlength(\eta) \quad\text{for}\quad x,y \in U.\]
Then $\oqdist^U$ defines a metric on $U$.  Moreover, $\oqdist^U$ is a.s.\ determined by $h|_U$.  The same holds with a $\sqrt{8/3}$-LQG sphere in place of the $\sqrt{8/3}$-quantum cone.
\end{proposition}
\begin{proof}
To show that $\oqdist^U$ a.s.\ defines a metric, we need to show that it is a.s.\ the case that for all $x,y \in U$ there exists $\eta \in U_{x,y}$ with $\qlength(\eta) < \infty$.  We may assume without loss of generality that $U$ is bounded.  We fix $c, \beta,r_0 > 0$ so that $\qball{x}{r}$ contains $B(x,c r^\beta)$ for all $x \in U$ and $r \in (0,r_0)$.  Such $c,\beta,r_0$ exist by Theorem~\ref{thm::metric_completion}.  Suppose that $x,y \in U$.  We then pick points $x_0 = x,x_1,\ldots,x_{k-1},x_k = y$ and radii $r_0,\ldots,r_k$ such that $\qball{x}{r_j} \subseteq U$ for all $0 \leq j \leq k$ and $B(x_j, c r_j^\beta) \cap B(x_{j+1}, c r_{j+1}^\beta) \neq \emptyset$ for all $0 \leq j \leq k-1$.  For each $0 \leq j \leq k-1$, we let $y_j \in B(x_j,c r_j^\beta) \cap B(x_{j+1},c r_{j+1}^\beta)$.  Then there exists $\oqdist$-geodesics connecting $x_j$ to $y_j$ and $y_j$ to $x_{j+1}$, which are respectively contained in $\qball{x_j}{r_j}$ and $\qball{x_{j+1}}{r_{j+1}}$ hence also in $U$.  Concatenating these paths yields an element of $U_{x,y}$ with $\oqdist$-length at most
\begin{equation*}
r_0 + 2(r_1 + \cdots + r_{k-1}) + r_k < \infty,
\end{equation*}
as desired.

To see that $\oqdist^U$ is a.s.\ determined by $h|_U$, we note that  the construction of $\QLE(8/3,0)$ given in \cite[Section~6]{qlebm} (and recalled in Section~\ref{subsec::qle_on_cone}) is local in the sense that for each $r > 0$ and $z \in U$, on the event that $\qball{z}{r} \subseteq U$ we have that $\qball{z}{r}$ is a.s.\ determined by $h|_U$.  Indeed, this follows because of the locality property for $\SLE_6$ and the quantum boundary length measure used to define $\QLE(8/3,0)$ is locally determined by $h$.  The claim thus follows because the collection of metric balls which are contained in $U$ determine $\oqdist^U$.
\end{proof}

\subsubsection{Proof of Theorem~\ref{thm::typical_ball_size}}
\label{subsec::typical_ball_proof}

We assume for simplicity that $(\C,h,0,\infty)$ is a $\sqrt{8/3}$-quantum cone with the circle average embedding.  Theorem~\ref{thm::continuity} and Theorem~\ref{thm::metric_completion} together imply that $\qball{0}{1}$ has Euclidean diameter which is finite and positive and, moreover, that $\mu_h(\qball{0}{1})$ is finite and positive.  Therefore if we let $B^I$ (resp.\ $B^O$) be the largest (resp.\ smallest) Euclidean ball centered at $0$ which is contained in (resp.\ contains $\qball{0}{1}$) then we also have that both $\mu_h(B^I)$ and $\mu_h(B^O)$ are finite and positive.  Fix $u > 0$.  Note that for $q \in \{I,O\}$ we have that
\[ \p[ \mu_h(\qball{0}{1}) \in [\epsilon^u,\epsilon^{-u}]] \to 1 \quad\text{and}\quad \p[ \mu_h(B^q) \in [\epsilon^u,\epsilon^{-u}]] \to 1 \quad\text{as}\quad \epsilon \to 0.\]
If we add $C = 4 \gamma^{-1} \log \epsilon$ to $h$, $\gamma=\sqrt{8/3}$, then Lemma~\ref{lem::quantum_distance_scale} implies that the radius of our quantum ball becomes $\epsilon$.  Therefore the scaling property of $\mu_h$ and the above implies that if we let $B_\epsilon^I$ (resp.\ $B_\epsilon^O$) be the largest (resp.\ smallest) Euclidean ball centered at $0$ which is contained in (resp.\ contains) $\qball{0}{\epsilon}$ then we have for $q \in \{I,O\}$ that
\[ \p[ \mu_h(\qball{0}{\epsilon}) \in [\epsilon^{4+u},\epsilon^{4-u}]] \to 1 \quad\text{and}\quad \p[ \mu_h(B_\epsilon^q) \in [\epsilon^{4+u},\epsilon^{4-u}]] \to 1 \quad\text{as}\quad \epsilon \to 0.\]
Let $\wt{B}_\epsilon^I$ (resp.\ $\wt{B}_\epsilon^O$) be the largest (resp.\ smallest) Euclidean ball centered at $0$ with quantum area at most $\epsilon^{4+u}$ (resp.\ at least $\epsilon^{4-u}$).  Then the above implies that with probability tending to $1$, we have that $\wt{B}_\epsilon^I \subseteq B_\epsilon^I$ and $B_\epsilon^O \subseteq \wt{B}_\epsilon^O$.  Let $\wt{r}_\epsilon^q$ be the radius of $\wt{B}_\epsilon^q$ for $q \in \{I,O\}$.  To finish the proof, we need to show that for every $v > 0$ there exists $u > 0$ so that
\[ \p[ \wt{r}_\epsilon^O \geq \epsilon^{6-v}] \to 1 \quad \text{and} \quad \p[ \wt{r}_\epsilon^I \leq \epsilon^{6+v}] \to 1 \quad\text{as}\quad \epsilon \to 0.\]
Equivalently, we need to prove that for every $v > 0$ there exists $u > 0$ so that
\[ \p[ \mu_h(B(0,\epsilon^{6-v})) \geq \epsilon^{4-u}] \to 1 \quad \text{and} \quad \p[ \mu_h(B(0,\epsilon^{6+v})) \leq \epsilon^{4+u}] \to 1 \quad\text{as}\quad \epsilon \to 0.\]

By \cite[Lemma~4.6]{DS08}, we have for each $r \in (0,1)$ that
\begin{equation}
\label{eqn:quantum_mass_cond_exp_formula}
\E[ \mu_h(B(0,r)) \giv h_r(0) ] = e^{\gamma h_r(0) + (2+\gamma^2/2) \log r}.
\end{equation}
Since $h_{e^{-t}}(0)$ evolves as a standard Brownian motion plus the linear drift $\gamma t$ in $t$, the typical value of $h_r(0)$ for $r \in (0,1)$ is $\gamma \log r^{-1} + O(\sqrt{\log r^{-1}})$.  Therefore the dominant term in the exponent on the right hand side of~\eqref{eqn:quantum_mass_cond_exp_formula} is equal to $(2-\gamma^2/2) \log r = (2/3) \log r$.  Applying this for $r = \epsilon^{6-v}$, we see that the probability that $\E[ \mu_h(B(0,r)) \giv h_r(0) ]$ exceeds $\epsilon^{4-u}$ with $u \in (0, 2v/3)$ fixed tends to $1$ as $\epsilon \to 0$.  The first claim thus follows from \cite[Lemma~4.5]{DS08}, which serves to bound the probability that $\mu_h(B(0,r))$ is much smaller than its conditional expectation given $h_r(0)$.  For $r = \epsilon^{6+v}$, the formula~\eqref{eqn:quantum_mass_cond_exp_formula} implies that the probability that $\E[ \mu_h(B(0,r)) \giv h_r(0) ]$ is at most $\epsilon^{4+u}$ with $u \in (0,2 v/3)$ tends to $1$ as $\epsilon \to 0$.  The second claim follows from this and Markov's inequality.  This completes the proof in the case of a quantum cone.

The result follows in general as the local behavior of a quantum sphere, quantum disk, or quantum cone in a bounded open set near a quantum typical point has the same behavior as~$h$ near~$0$. \qed

\section{Distance to tip of $\SLE_6$ on a $\sqrt{8/3}$-quantum wedge}
\label{sec::sle_6_moment_bounds}

\begin{figure}[ht!]
\begin{center}
\subfigure{\includegraphics[scale=0.85,page=1]{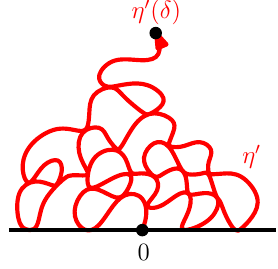}}
\hspace{0.05\textwidth}
\subfigure{\includegraphics[scale=0.85,page=2]{figures/sle_6_distance_to_tip_on_wedge}}
\hspace{0.05\textwidth}
\subfigure{\includegraphics[scale=0.85,page=3]{figures/sle_6_distance_to_tip_on_wedge}}
\end{center}
\vspace{-0.025\textheight}
\caption{\label{fig::moment_statements}{\bf Left:} A $\sqrt{8/3}$-quantum wedge parameterized by $\h$ decorated with an independent chordal $\SLE_6$ process $\eta'$ from $0$ to $\infty$ stopped at $\delta$ units of quantum natural time.  {\bf Middle:} Same as the left together with a $\oqdist$-shortest path from $0$ to $\eta'(\delta)$, indicated in blue.  In Proposition~\ref{prop::base_to_tip_expectation_bound}, we show that the length of this path has a finite $p$th moment for some $p > 1$.  {\bf Right:} A $\sqrt{8/3}$-quantum wedge together with the shortest path connecting $0$ to the point $x < 0$ with the property that the quantum length of $[x,0]$ is equal to $\delta$.  In Proposition~\ref{prop::left_right_shift_moment_bound}, we show that the length of the blue path also has a finite $p$th moment for some $p > 1$.  These moment bounds will be important in the arguments of Section~\ref{sec::geodesics_and_levy_net}.}
\end{figure}

The main purpose of this section is to prove the following two propositions, which will be used in Section~\ref{sec::geodesics_and_levy_net} to show that the unembedded metric net between two quantum typical points in a $\sqrt{8/3}$-LQG sphere has the law of the $3/2$-L\'evy net.  The first (Proposition~\ref{prop::base_to_tip_expectation_bound}) bounds the moments of the distance in a $\sqrt{8/3}$-quantum wedge between the origin and the tip of an independent $\SLE_6$ run for $\delta$ units of quantum natural time and the second (Proposition~\ref{prop::left_right_shift_moment_bound}) bounds the moments of the distance between the origin and a point which is $\delta$ units of quantum length along the boundary from the origin in a $\sqrt{8/3}$-quantum wedge.  Throughout, we will use the spaces $\CH_1(\CX)$ and $\CH_2(\CX)$ introduced in Section~\ref{subsec::spheres_and_disks}.

\begin{proposition}
\label{prop::base_to_tip_expectation_bound}
Suppose that $(\strip,h,-\infty,+\infty)$ is a $\sqrt{8/3}$-quantum wedge and let~$\eta'$ be an independent~$\SLE_6$ process from~$-\infty$ to~$+\infty$ with the quantum natural time parameterization.  There exists $p_0 > 1$ such that for all $p \in (0,p_0)$ there exists a constant~$c_p >0$ such that the following is true.  Let $D_\delta$ be the quantum distance between $-\infty$ and $\eta'(\delta)$ (with respect to the internal $\QLE(8/3,0)$ metric associated with $h$).  Then we have that
\begin{equation}
\label{eqn::base_to_tip_moment_bound}
\E[ D_\delta^p ] = c_p \delta^{p/3}.
\end{equation}
For each $\alpha > 0$, let $u_{\alpha,\delta} \in \R$ be where the projection of $h$ onto $\CH_1(\strip)$ first hits $\alpha \log \delta$, let $F_{\alpha,\delta} = \{ \sup_{t \in [0,\delta]} \re(\eta'(t)) \leq u_{\alpha,\delta}-1\}$ and let $D_{\alpha,\delta}$ be the quantum distance between $-\infty$ and $\eta'(\delta)$ with respect to the internal $\QLE(8/3,0)$ metric associated with $\strip_- + u_{\alpha,\delta}$.  For each $p \in (0,p_0)$ there exists $\alpha > 0$ and a constant $c_p > 0$ such that
\begin{equation}
\label{eqn::stronger_moment_bound}
\E[ D_{\alpha,\delta}^p \one_{F_{\alpha,\delta}}] \leq c_p \delta^{p/3}
\end{equation}
Finally, there exists $\alpha_0 > 0$ such that for all $\alpha \in (0,\alpha_0)$ and each $k > 0$ there exists a constant $c_{k} > 0$ such that
\begin{equation}
\label{eqn::path_displacement}
\p[ F_{\alpha,\delta}^c ] \leq c_k \delta^k.
\end{equation}
\end{proposition}

As we will see in the proof of Proposition~\ref{prop::base_to_tip_expectation_bound}, the exponent $p/3$ in~\eqref{eqn::base_to_tip_moment_bound} arises because adding a constant $C$ to $h$ has the effect of scaling the amount of quantum natural time elapsed by $\eta'$ by the factor $e^{3\gamma C/4}$, $\gamma=\sqrt{8/3}$, \cite[Section~6.2]{quantum_spheres} and the quantum distance by the factor $e^{\gamma C/4}$ (Lemma~\ref{lem::quantum_distance_scale}).  In particular, the quantum distance behaves like the quantum natural time to the power $1/3$.

\begin{proposition}
\label{prop::left_right_shift_moment_bound}
Suppose that $(\strip,h,-\infty,+\infty)$ is a $\sqrt{8/3}$-quantum wedge.  There exists $p_0 > 1$ such that for all $p \in (0,p_0)$ there exists a constant $c_p > 0$ such that the following is true.  For each $\delta > 0$ we let $x_\delta = \inf\{ x  \in \R : \nu_h([0,x]) \geq \delta\}$ and let $D_\delta$ be the quantum distance between $-\infty$ and $x_\delta$ (with respect to the internal $\QLE(8/3,0)$ metric associated with $h$).  Then we have that
\begin{equation}
\label{eqn::shift_distance}
\E[ D_\delta^p ] = c_p \delta^{p/2}.
\end{equation}
For each $\alpha > 0$, let $u_{\alpha,\delta}$ be where the projection of $h$ onto $\CH_1(\strip)$ first hits $\alpha \log \delta$, let $F_{\alpha,\delta} = \{x_\delta \leq u_{\alpha,\delta}-1 \}$, and let $D_{\alpha,\delta}$ be the quantum distance between $-\infty$ and $x_\delta$ with respect to the internal metric associated with $\strip_- + u_{\alpha,\delta}$.  For each $p \in (0,p_0)$ there exists $\alpha > 0$ and a constant $c_{\alpha,p} > 0$ such that
\begin{equation}
\label{eqn::shift_distance_internal}
\E[ D_{\alpha,\delta}^p \one_{F_{\alpha,\delta}}] \leq c_{\alpha,p} \delta^{p/2}.
\end{equation}
Finally, there exists $\alpha_0 > 0$ such that for all $\alpha \in (0,\alpha_0)$ and each $k > 0$ there exists a constant $c_{k} > 0$ such that
\begin{equation}
\label{eqn::shift_probability}
\p[ F_{\alpha,\delta} ] \leq c_k \delta^k.
\end{equation}
\end{proposition}

As we will see in the proof of Proposition~\ref{prop::left_right_shift_moment_bound}, the exponent $p/2$ in~\eqref{eqn::shift_distance} arises because adding a constant $C$ to $h$ has the effect of scaling quantum length by the factor $e^{\gamma C/2}$, $\gamma=\sqrt{8/3}$, and the quantum distance by the factor $e^{\gamma C/4}$ (Lemma~\ref{lem::quantum_distance_scale}).  In particular, the quantum distance behaves like the quantum length to the power $1/2$.

We note that~\eqref{eqn::stronger_moment_bound},~\eqref{eqn::path_displacement} of Proposition~\ref{prop::base_to_tip_expectation_bound} and~\eqref{eqn::shift_distance_internal},~\eqref{eqn::shift_probability} of Proposition~\ref{prop::left_right_shift_moment_bound} also hold in the setting of a quantum disk $(\strip,h)$ sampled from $\CM^\bes$ provided we condition on the event that the projection of $h$ onto $\CH_1(\strip)$ exceeds $\alpha \log \delta$.  (Note that this yields a probability measure since conditioning the maximum of a Bessel excursion to exceed a positive value is a positive and finite measure event.)  Indeed, this follows because in this case the law of $h$ restricted to the part of $\strip$ up until where the projection of $h$ onto $\CH_1(\strip)$ first hits $\alpha \log \delta$ is the same as the corresponding restriction of a $\sqrt{8/3}$-quantum wedge.  In fact, we will be applying these results in the setting of a quantum surface whose law is closely related to that of a quantum disk below.  As we will see, however, it will be more convenient to establish the above estimates in the setting of a quantum wedge because of the exact scaling properties that a quantum wedge possesses.

We will break the proof of Proposition~\ref{prop::base_to_tip_expectation_bound} into two steps.  The first step (carried out in Section~\ref{subsec::path_displacement}) is to establish~\eqref{eqn::path_displacement}.  The second step (carried out in Section~\ref{subsec::distance_bounds}) is to establish a moment bound between deterministic points in $\strip_-$.  As we will see upon completing the proof of Proposition~\ref{prop::base_to_tip_expectation_bound}, the proof of Proposition~\ref{prop::left_right_shift_moment_bound} will follow from the same set of estimates used to prove Proposition~\ref{prop::base_to_tip_expectation_bound} (though in this case the argument turns out to be simpler).

\subsection{Size of path with quantum natural time parameterization}
\label{subsec::path_displacement}

The purpose of this section is to bound the size of an $\SLE_6$ path drawn on top of an independent $\sqrt{8/3}$-quantum wedge equipped with the quantum natural time parameterization.

\begin{proposition}
\label{prop::process_does_not_leave_disk}
Suppose that $(\strip,h,-\infty,+\infty)$ has the law of a $\sqrt{8/3}$-quantum wedge.  Let $\eta'$ be an $\SLE_6$ from $-\infty$ to $+\infty$ which is sampled independently of $h$ and then parameterized according to quantum natural time.  Let $u_{\alpha,\delta}$ be where the projection of $h$ onto $\CH_1(\strip)$ first hits $\alpha \log \delta$.  There exists $\alpha_0 > 0$ such that for all $\alpha \in (0,\alpha_0)$ and each $k > 0$ there exists a constant $c_{k} > 0$ such that
\[ \p\!\left[ \sup_{t \in [0,\delta]} \re(\eta'(t)) \geq u_{\alpha,\delta} \right] \leq c_k \delta^k.\]
That is,~\eqref{eqn::path_displacement} from Proposition~\ref{prop::base_to_tip_expectation_bound} holds.
\end{proposition}

We will prove Proposition~\ref{prop::process_does_not_leave_disk} by first bounding in Lemma~\ref{lem::disk_bound} the number of quantum disks cut out by $\eta'|_{[0,\delta]}$ with large quantum area (a small, positive power of delta) and then argue in Lemma~\ref{lem::euclidean_area_lbd} and Lemma~\ref{lem::quantum_area_lbd} that if we run $\eta'$ until it first hits the line $\re(z) =  u_{\alpha,\delta}$ for $\alpha > 0$ small then it is very likely to cut out a large number of quantum disks with large quantum area.

\begin{lemma}
\label{lem::disk_bound}
Suppose that we have the same setup as in Proposition~\ref{prop::process_does_not_leave_disk}.  There exist $\alpha,\beta > 0$ such that for each $n \in \N$ there exists a constant $c_n > 0$ such that the following is true.  The probability that $\eta'|_{[0,\delta]}$ separates from $+\infty$ at least $n$ quantum disks with quantum area at least~$\delta^\alpha$ is at most~$c_n \delta^{\beta n}$.
\end{lemma}
\begin{proof}
Fix $k \in \N$.  We will prove the result by giving an upper bound on the probability that $\eta'|_{[0,\delta]}$ cuts out at least $k$ quantum disks with boundary length in three different regimes and then we will sum over all possibilities so that at least $n$ quantum disks with quantum area at least $\delta^\alpha$ are cut out by $\eta'|_{[0,\delta]}$.

Recall that for each $j \in \Z$ the number $N_j$ of quantum disks cut out by $\eta'|_{[0,\delta]}$ with quantum boundary length in $(e^{-j-1},e^{-j}]$ is distributed as a Poisson random variable with mean $\lambda_j$ which is given by a constant times $\delta e^{3/2 j}$ and that the $N_j$ are independent.  Moreover, recall from Lemma~\ref{lem::quantum_disk_area} that the expected quantum area in such a disk is given by a constant times $e^{-2j}$ (i.e., a constant times the square of its boundary length).  Therefore the probability that a given such disk has quantum area at least~$\delta^\alpha$ is, by Markov's inequality, at most a constant times $e^{-2j} \delta^{-\alpha}$.  By Lemma~\ref{lem::poisson_deviation}, there exists a constant $c_0 > 0$ such that
\begin{equation}
\label{eqn::poisson_negligible}
\p[ N_j \geq 2\lambda_j] \leq \exp(-c_0 \delta e^{3/2 j}).
\end{equation}
The upper bound in~\eqref{eqn::poisson_negligible} is negligible compared to any power of $\delta$ as $\delta \to 0$ provided we have for some $\epsilon > 0$ fixed that $j \geq \ell_0 = \tfrac{2}{3}(1 + \epsilon) \log \delta^{-1}$.  Let $E_{j,k}$ be the event that $\eta'|_{[0,\delta]}$ cuts out at least $k$ quantum disks with quantum area at least $\delta^\alpha$ and quantum boundary length in $(e^{-j-1},e^{-j}]$.  It thus follows that there exists a constant $c_1 > 0$ such that
\begin{equation}
\label{eqn::k_disks_small}
\p[ E_{j,k}] \leq c_1 ( \delta e^{3/2 j} )^k \times ( e^{-2j} \delta^{-\alpha})^k = c_1 \delta^{(1-\alpha) k} e^{-j k / 2} \quad\text{for all}\quad j \geq \tfrac{2}{3}(1 + \epsilon) \log \delta^{-1}.
\end{equation}
The number $N$ of quantum disks separated by $\eta'|_{[0,\delta]}$ from $+\infty$ with quantum boundary length between $\delta^{2/3(1+\epsilon)}$ and $\delta^{1/2}$ is Poisson with mean $\lambda$ proportional to $\delta^{-\epsilon}$.  By Lemma~\ref{lem::poisson_deviation}, we have for a constant $c_2 > 0$ that
\begin{equation}
\label{eqn::poisson_negligible2}
\p[ N \geq 2\lambda ] \leq \exp(-c_2 \delta^{-\epsilon}),
\end{equation}
hence decays to $0$ faster than any power of $\delta$ as $\delta \to 0$.  By Lemma~\ref{lem::quantum_disk_area}, the expected quantum area in such a quantum disk is at most a constant times $\delta$ so that, as before, the probability that any such disk has quantum area at least $\delta^\alpha$ is, by Markov's inequality, at most a constant times $\delta^{1-\alpha}$.  Let $F_k$ be the event that there are at least $k$ such disks.  Combining, it follows that there exists a constant $c_3 > 0$ such that
\begin{equation}
\label{eqn::k_disks_medium}
\p[ F_k ] \leq c_3 \delta^{-\epsilon k} \times \delta^{(1-\alpha)k} = c_3 \delta^{(1-\alpha-\epsilon)k}.
\end{equation}
Finally, the number of quantum disks separated by $\eta'|_{[0,\delta]}$ from $+\infty$ with boundary length larger than $\delta^{1/2}$ is Poisson with mean proportional to $\delta \times \delta^{-3/4} = \delta^{1/4}$.  Thus if we let $G_k$ be the event that there are at least $k$ such quantum disks cut out by $\eta'|_{[0,\delta]}$ then it follows that there exists a constant $c_4 > 0$ such that
\begin{equation}
\label{eqn::k_disks_large}
\p[ G_k ] \leq c_4\delta^{k/4}.
\end{equation}

We can deduce the result from~\eqref{eqn::k_disks_small},~\eqref{eqn::k_disks_medium}, and~\eqref{eqn::k_disks_large} as follows.  For each sequence $\ul{a} = (i,j,k_{\ell_0},k_{\ell_0+1},\ldots)$ of  non-negative integers with
\[ i+j+ \sum_{\ell \geq \ell_0} k_\ell = n\]
we let $F_{\ul{a}}$ be the event that $\eta'|_{[0,\delta]}$ separates from $+\infty$ at least $i$ (resp.\ $j$) quantum disks of quantum area at least $\delta^\alpha$ and quantum boundary length in $[\delta^{2/3(1+\epsilon)},\delta^{1/2}]$ (resp.\ larger than $\delta^{1/2}$) and at least $k_\ell$ quantum disks of quantum area at least $\delta^\alpha$ and quantum boundary length in $(e^{-\ell-1},e^{-\ell}]$ for each $\ell \geq \ell_0$.  Assume that we have chosen $\alpha,\epsilon$ such that $1-\alpha-\epsilon \in  (0,1/4)$.  Then~\eqref{eqn::k_disks_small},~\eqref{eqn::k_disks_medium}, and~\eqref{eqn::k_disks_large} together imply that there exists a constant $c_5 > 0$ such that
\[ \p[ F_{\ul{a}} ] \leq c_5 \delta^{(1-\alpha-\epsilon)n} \prod_{\ell \geq \ell_0} e^{-\ell k_\ell/2}.\]
The result follows by summing over all such multi-indices $\ul{a}$.
\end{proof}

\begin{lemma}
\label{lem::euclidean_area_lbd}
Suppose that we have the same setup as in Proposition~\ref{prop::process_does_not_leave_disk}.  For each $x \in \R$, we let $\tau_x = \inf\{t \geq 0 : \re(\eta'(t)) \geq x\}$.  For each $\rho \in (0,1)$ there exists $r > 0$ such that the conditional probability given $\eta'|_{[0,\tau_x]}$ that there exists $z \in [x+r,x+1-r] \times [0,\pi]$ such that $\eta'|_{[\tau_x,\tau_{x+1}]}$ separates $B(z,r)$ from $+\infty$ is at least $\rho$.
\end{lemma}
\begin{proof}
This follows from the conformal Markov property for $\SLE_6$.
\end{proof}

\begin{lemma}
\label{lem::quantum_area_lbd}
Suppose that we have the same setup as in Proposition~\ref{prop::process_does_not_leave_disk} and let $\tau_x$ be as in Lemma~\ref{lem::euclidean_area_lbd}.  For each $\beta > 0$ there exists $\alpha > 0$ such that the following is true.  Fix $\delta > 0$, let $x =  u_{\alpha,\delta}$, and let $m \in \N$.  For each $p > 0$ there exists a constant $c_{p,m} > 0$ such that the probability that $\eta'|_{[0,\tau_x]}$ separates from $+\infty$ fewer than $m$ quantum disks of quantum area at least $\delta^\beta$ is at most $c_{p,m} \delta^p$.
\end{lemma}
\begin{proof}
Fix $\rho \in (0,1)$ and let $r > 0$ be as in Lemma~\ref{lem::euclidean_area_lbd} for this value of $\rho$.  Fix $\sigma > 0$ small and let $n = \sigma \log \delta^{-1}$.  Lemma~\ref{lem::euclidean_area_lbd} implies that the number of components separated by $\eta'|_{[\tau_{x-n},\tau_x]}$ from $+\infty$ which contain a Euclidean disk of radius at least $r$ is stochastically dominated from below by a Binomial random variable with parameters $(n,\rho)$.  In particular, by choosing $\rho$ sufficiently close to $1$ we have that the probability that the number of components separated by $\eta'|_{[\tau_{x-n},\tau_{x}]}$ from $+\infty$ which contain a Euclidean disk of radius at least $r$, all contained in $[x-n,x] \times [0,\pi]$, is smaller than $n/2$ is at most a constant times $\delta^p$.

Assume that we are working on the complementary event that $\eta'|_{[\tau_{x-n},\tau_{x}]}$ separates at least $n/2$ such components $U_1,\ldots,U_{n/2}$ and, for each $j$, we let $z_j$ be such that $B(z_j,r) \subseteq U_j$.  For each $1 \leq j \leq n/2$, we let $\Fh_j$ be the harmonic extension of the values of $h$ from $\partial B(z_j,r)$ to $B(z_j,r)$.  By the Markov property for the GFF with free boundary conditions, we know that the conditional law of the restriction of $h$ to $B(z_j,r)$ given its values outside of $B(z_j,r)$ is that of a GFF in $B(z_j,r)$ with zero boundary conditions plus $\Fh_j$ conditioned so that $u_{\alpha,\delta}$ is the first time that $\alpha \log \delta$ is hit by the projection of $h$ onto $\CH_1(\strip)$.  For $\gamma=\sqrt{8/3}$, let
\[ X_j^* = \sup_{w \in B(z_j,r/2)} |\Fh(w) - (Q-\gamma) (\re(z_j) - u_{\alpha,\delta}) - \alpha \log \delta|.\]
The argument used to prove Lemma~\ref{lem::gff_infimum} implies that there exist constants $c_0,c_1 > 0$ such that
\begin{equation}
\label{eqn::harmonic_max_bound}
\p[ X_j^* \geq \eta] \leq c_0 \exp\left( - \frac{c_1 \eta^2}{\sigma \log \delta^{-1}} \right) \quad\text{for all}\quad \eta \geq 0.
\end{equation}
Fix $\epsilon > 0$.  It follows from~\eqref{eqn::harmonic_max_bound} that there exist constants $c_2,c_3 > 0$ such that
\begin{equation}
\label{eqn::harmonic_max_bound_all}
G = \left\{ \max_{1 \leq j \leq n/2} X_j^* \geq \epsilon \log \delta^{-1} \right\} \quad\text{we have}\quad \p[G] \leq c_2 \exp\left(-\frac{c_2 \epsilon^2}{\sigma} \log \delta^{-1} \right).
\end{equation}
In particular, by making $\sigma > 0$ small enough we can make it so that $\p[G]$ decays to $0$ faster than any fixed positive power of $\delta$.  Conditional on $G^c$, for each $j$ we have that the probability that the quantum area associated with $B(z_j,r)$ is at least a constant times $a \delta^{\gamma (\epsilon+\alpha)} \times \delta^{2 \sigma/3}$ is at least $\wt{\rho}$ where $\wt{\rho} \to 1$ as $a \to 0$.  (Here, we have used that $|\re(z_j) - u_{\alpha,\delta}| \leq \sigma \log \delta^{-1}$ so that on $G^c$, $\Fh(z_j)$ differs from $\alpha \log \delta$ by at most $(Q-\gamma) \sigma \log \delta^{-1} + \epsilon \log \delta^{-1}$.)  We note that these events are conditionally independent given the values of $h$ on the complement of $\cup_j B(z_j,r)$ and on the event $G^c$.  Therefore by choosing $a > 0$ small enough, the probability that we have fewer than $n/4$ disks with quantum area at least $a \delta^{\gamma (\epsilon +\alpha)}$ tends to $0$ faster than any power of $\delta$.  Combining implies the result.
\end{proof}

\begin{proof}[Proof of Proposition~\ref{prop::process_does_not_leave_disk}]
The result follows by combining Lemma~\ref{lem::disk_bound} with Lemma~\ref{lem::quantum_area_lbd}.  Indeed, Lemma~\ref{lem::disk_bound} implies that $\eta'|_{[0,\delta]}$ is very unlikely to separate at least a fixed number of components with quantum area a power of $\delta$ while Lemma~\ref{lem::quantum_area_lbd} implies that $\eta'|_{[0,\tau_x]}$ with $x = u_{\alpha,\delta}$ is very likely to separate at least a fixed number of components with quantum area a power of~$\delta$.
\end{proof}

\subsection{Quantum distance bounds}
\label{subsec::distance_bounds}

\begin{proposition}
\label{prop::quantum_distance_deviations}
Suppose that $(\strip,h,-\infty,+\infty)$ has the law of a $\sqrt{8/3}$-quantum wedge with the embedding into~$\strip$ so that the projection of $h$ onto $\CH_1(\strip)$ first hits $0$ at $u = 0$.  There exist $p > 1$ and constants $c_0,c_1 > 0$ such that the following is true.  For each $k \in \Z$ with $k < 0$ we let $D_k$ be the quantum distance from $k + i\pi/2$ to $k+1+i \pi/2$ (i.e., the midpoint of the line segment in $\strip$ with $\re(z) = k$ to the midpoint of the line segment with $\re(z) = k+1$) with respect to the $\QLE(8/3,0)$ internal metric in $[k-1,k+2] \times [\pi/4,3\pi/4]$.  Then we have that
\begin{equation}
\label{eqn::line_to_line_bound}
\E[ D_k^p] \leq c_0 e^{c_1 k}.
\end{equation}
Let $\wt{D}_k$ denote the quantum distance between the points $k + 3\pi/4 i$ and $k + i \pi /4$ with respect to the $\QLE(8/3,0)$ internal metric in $[k-1,k+1] \times [0,\pi]$.  We also have (for the same values of $p,c_0,c_1$) that
\begin{equation}
\label{eqn::vertical_line_bound}
\E[ \wt{D}_k^p] \leq c_0 e^{c_1 k}.
\end{equation}
\end{proposition}

We need to collect several lemmas before giving the proof of Proposition~\ref{prop::quantum_distance_deviations}.

\begin{lemma}
\label{lem::multifractal_spectrum}
Suppose that $(\strip,h,-\infty,+\infty)$ is a $\sqrt{8/3}$-quantum wedge with the embedding into~$\strip$ as in the statement of Proposition~\ref{prop::quantum_distance_deviations}.  Fix $k \in \Z$ with $k < 0$ and suppose that $B(z,r) \subseteq [k,k+1] \times [\pi/4,3\pi/4]$.  
With $\gamma=\sqrt{8/3}$, let
\[ \xi(q) = \left( 2 + \frac{\gamma^2}{2} \right) q - \frac{\gamma^2}{2} q^2 = \frac{10}{3} q - \frac{4}{3} q^2.\]  For each $M \in \R$, we let $A_M$ be the event that the value of the projection of $h$ onto $\CH_1(\strip)$ at $k+1$ is in $[M,M+1]$.  For each $q \in (0,4/\gamma^2) = (0,3/2)$ there exists a constant $c_q > 0$ such that
\[ \E[ \mu_h(B(z,r))^q  \giv A_M ] \leq c_q e^{ \gamma q M} r^{\xi(q)}.\]
\end{lemma}
\begin{proof}
This follows from the standard multifractal spectrum bound for the moments of the quantum measure; recall Proposition~\ref{prop::cone_holder}.
\end{proof}

Fix $k \in \Z$ with $k < 0$.  For $\alpha,\beta > 0$ and $j \in \N$, we say that a point $z \in \strip$ with $\re(z) \in [k,k+1]$ is $(\alpha,\beta,j)$-good if:
\begin{enumerate}
\item $B(z,e^{-j}) \subseteq \qball{z}{e^{-\alpha j}}$ and
\item $\qball{z}{e^{-\alpha j}} \subseteq B(z,e^{-\beta j})$.
\end{enumerate}

\begin{lemma}
\label{lem::good_prob}
Suppose that $(\strip,h,-\infty,+\infty)$ is a $\sqrt{8/3}$-quantum wedge with the embedding into~$\strip$ as in the statement of Proposition~\ref{prop::quantum_distance_deviations}.  Suppose that $z \in [-2,-1] \times [\pi/4,3\pi/4]$.  For each $\epsilon > 0$ there exists $\alpha,\beta, c_0 > 0$ such that the probability that $z$ is $(\alpha,\beta,j)$-good is at least $1- c_0 e^{- (25/12-\epsilon)j}$ for each $j \in \N$.
\end{lemma}
\begin{proof}
Fix $\epsilon > 0$.  By Lemma~\ref{lem::multifractal_spectrum} and Markov's inequality with $q = 5/4$ so that $\xi(q) = 25/12$, we can find $\alpha > 0$ and a constant $c_0 > 0$ such that with
\[ E = \{ \mu_h(B(z,e^{-j})) \leq e^{-\alpha j} \} \quad\text{we have}\quad \p[E^c]  \leq c_0 e^{- (25 / 12-\epsilon)j }.\]

We are now going to make a comparison between the law of $h$ near $z$ and the law of a quantum cone near its marked point since this is the setting in which our moment bounds for quantum distance in Section~\ref{sec::continuity_to_bm} were established. Let $\CZ$ be the Radon-Nikodym derivative between the law of the restriction of $h$ to $B(z,\tfrac{1}{4})$ and the law of a whole-plane GFF on $\C$ restricted to $B(z,\tfrac{1}{4})$ with the additive constant fixed so that its average on $\partial B(z,1)$ is equal to $0$.  Then Lemma~\ref{lem::gff_infimum} and Lemma~\ref{lem::gff_change_bc} together imply that for each $p > 0$ there exists a constant $c_p < \infty$ such that
\begin{equation}
\label{eqn::whole_plane_rn_bound}
\E[ \CZ^p ] \leq c_p.
\end{equation}

For each $\delta > 0$ we let $\phi_\delta$ be a $C_0^\infty$ function which agrees with $w \mapsto \log|w-z|$ in $B(z,1) \setminus B(z,\delta)$.  It is not hard to see that there is a constant $c_1 > 0$ so that we can find such a function $\phi_\delta$ so that $\| \phi_\delta \|_\nabla^2 \leq c_1 \log \delta^{-1}$ (e.g., by truncating the $\log$ function in a smooth manner).  Combining this with~\eqref{eqn::whole_plane_rn_bound} implies that the Radon-Nikodym derivative $\CZ_\delta$ between the law of the restriction of $h$ to $B(z,\tfrac{1}{4}) \setminus B(z,\delta)$ and the law of a $\sqrt{8/3}$-quantum cone $(\C,\wt{h},z,\infty)$ restricted to $B(z,\tfrac{1}{4}) \setminus B(z,\delta)$ with the circle average embedding satisfies the property that for each $p > 0$ there is a constant $\wt{c}_p > 0$ such that
\begin{equation}
\label{eqn::rn_moment_bound}
\E[ \CZ_\delta^p ] \leq \wt{c}_p \delta^{-c_1 p^2/2}.
\end{equation}

Proposition~\ref{prop::disk_diameter_bounds} implies that for each choice of $c_3 > 0$ there exists a constant $c_2 > 0$ such that, conditional on $E$, the probability that the $\QLE(8/3,0)$ distance defined from the quantum cone $(\C,\wt{h},z,\infty)$ between every point in $B(z,e^{-k}) \setminus B(z,e^{-k-1})$ and $\partial B(z,e^{-k-1})$ is at most $e^{-c_2 k}$ with probability at least $1-e^{-c_3 k}$.  Combining this with~\eqref{eqn::rn_moment_bound} and H\"older's inequality, we see that the same is true under $h$ (though with possibly different constants $c_2,c_3$).   Iterating this and summing over $k$ implies that there exists $\alpha > 0$ so that the probability of the first part of being $(\alpha,\beta,j)$-good is at least $1-c_0 e^{-(25/12-\epsilon)j}$.

The same change of measures argument but using~\eqref{eqn::prob_diam_bound} of Proposition~\ref{prop::ball_size_ubd} in place of Proposition~\ref{prop::disk_diameter_bounds} yields a similar lower bound of the probability of the second $(\alpha,\beta,j)$-good condition.
\end{proof}

\begin{proof}[Proof of Proposition~\ref{prop::quantum_distance_deviations}]
We will first prove the result for $k=-2$ and then explain how to extract the result for other values of $k$ from this case.  Let $a$ (resp.\ $b$) be the midpoint of the line $\re(z) = -2$ (resp.\ $\re(z) = -1$) in $\strip$.  Let $j \in \N$ and for each $0 \leq \ell \leq e^j$ we let
\[ z_\ell = -2 + \frac{\ell}{e^j}  + i\frac{\pi}{2}\]
be the midpoint of the line $\re(z) = -2 + \ell/e^j$ in $\strip$.  Let $G_j$ be the event that $a=z_0,\ldots,z_{e^j}=b$ are all $(\alpha,\beta,j)$-good.  On $G_j$, we have that
\begin{equation}
\label{eqn::dist_good_bound}
\oqdist(a,b) \leq \sum_{\ell=1}^{e^j} \oqdist(z_{\ell-1},z_\ell) \leq e^{(1-\alpha) j}.
\end{equation}
Fix $\epsilon > 0$ small so that $13/12-\epsilon > 1$.  By Lemma~\ref{lem::good_prob}, we know for a constant $c_0 > 0$ that
\begin{align}
     \p[ G_j^c ]
&\leq c_0 e^j \times e^{ - (25/12-\epsilon) j}
  = c_0 e^{- (13/12-\epsilon) j}.
  \label{eqn::good_prop_bound}
\end{align}
Fix $p > 1$ so that $(1-\alpha) p - (13/12-\epsilon) < 0$.  Let $J$ be the first $j$ so that $G_j$ occurs.  Then we have that 
\begin{align*}
 \E[ \oqdist^p(a,b)  ]
&= \sum_{j=1}^\infty \E[ \oqdist^p(a,b) \one_{\{J = j\}}  ] \\
&\leq \sum_{j=1}^\infty e^{(1-\alpha)j p} \p[J = j] \quad\text{(by~\eqref{eqn::dist_good_bound})}\\
&\leq \sum_{j=1}^\infty e^{(1-\alpha)j p} \p[J > j-1]
 \leq \sum_{j=1}^\infty e^{(1-\alpha)j p} \p[G_{j-1}^c]\\
&\leq c_0 \sum_{j=1}^\infty e^{(1-\alpha) p j } e^{-(13/12-\epsilon) (j-1)} \quad\text{(by~\eqref{eqn::good_prop_bound})}\\
&< \infty.
\end{align*}

This completes the proof of the result for $k=-2$.

We will now generalize the result to all $k \in \Z$ with $k \leq -2$.  Lemma~\ref{lem::quantum_distance_scale} implies that adding a constant $C$ to the field scales distances by the factor $e^{\gamma C /4}$ for $\gamma=\sqrt{8/3}$.  Let $X$ be the projection of $h$ onto $\CH_1(\strip)$.  Then we can write $X_u = B_{-2 u} + (Q-\gamma)u$ for $u \leq 0$ where $B$ is a standard Brownian motion with $B_0 = 0$ conditioned so that $X_u \leq 0$ for all $u \leq 0$.  Let $B_{-2(k+2)}^* = \sup_{t \in [0,2]} B_{-2(k+2-t)}$ and $X^* = \inf\{ X_{-t} : t \in [0,1]\}$.  Let $a_k$ (resp.\ $b_k$) be the midpoint of the line $\re(z) = k$ (resp.\ $\re(z) = k+1$) in $\strip$.  With $\gamma = \sqrt{8/3}$, we have that the conditional law of $\oqdist(a_k,b_k) e^{\tfrac{\gamma}{4}(X^* - (B_{-2(k+2)}^* + (Q-\gamma)(k+2))}$ given $X$ is stochastically dominated from above by the conditional law of $\oqdist(a,b)$ given~$X$.  It follows for a constant $c_1 > 0$ that
\begin{equation}
\label{eqn::dist_giv_b}
\E[ \oqdist^p(a_k,b_k) \giv X] \leq c_1 \exp\left( \frac{p \gamma}{4} \left( B_{-2(k+2)}^* + (Q-\gamma)(k+2)  - X^* \right) \right).
\end{equation}
Since $X^*$ has finite exponential moments of all orders, by H\"older's inequality it suffices to show that there exists $p_0 > 1$ and constants $c_0,c_1 > 0$ such that for all $p \in (0,p_0)$ we have that
\begin{equation}
\label{eqn::dist_giv_b_integrated_b}
\E\!\left[ \exp\left( \frac{p \gamma}{4} \left( B_{-2(k+2)}^* + (Q-\gamma) (k+2) \right) \right) \right] \leq c_0 e^{ c_1 k }.
\end{equation}
This, in turn, is not difficult to see as $X_u$ conditioned to be negative for all $u \leq 0$ is stochastically dominated from above by $X_u$ conditioned to be negative only for $u=k+2$.

Combining~\eqref{eqn::dist_giv_b} and~\eqref{eqn::dist_giv_b_integrated_b} implies~\eqref{eqn::line_to_line_bound}.  The bound~\eqref{eqn::vertical_line_bound} is proved in an analogous manner except one considers evenly spaced points along a vertical rather than horizontal segment.
\end{proof}

\subsection{Proof of moment bounds}

We now have the necessary estimates to complete the proofs of Proposition~\ref{prop::base_to_tip_expectation_bound} and Proposition~\ref{prop::left_right_shift_moment_bound}.

\begin{proof}[Proof of Proposition~\ref{prop::base_to_tip_expectation_bound}]
Fix $\alpha > 0$ small and let $\beta=\alpha/2$. Recall from the proposition statement that $F_{\alpha,\delta}$ is the event that $\eta'|_{[0,\delta]}$ is contained in $\strip_- + u_{\alpha,\delta}-1$.  We let $\gamma_1$ be the shortest $\oqdist$-length path from $-\infty$ to the midpoint of the vertical line through $\re(z) = u_{\beta,\delta}$ contained in $[\pi/4,3\pi/4] \times \R$.  We then let $\gamma_2$ be the shortest $\oqdist$-path from $a = u_{\beta,\delta} + i\pi/4 - 2$ to $b = u_{\beta,\delta} + i3\pi/4 - 2$ contained in $[u_{\beta,\delta}-5, u_{\beta,\delta}] \times [0,\pi]$. 

Let $f_\delta$ be the unique conformal map from the component of $\strip \setminus \eta'([0,\delta])$ with $+\infty$ on its boundary to $\strip$ with $|f_\delta(z)-z| \to 0$ as $z \to +\infty$.  We also let $\wt{\gamma}_3$ be the shortest path with respect to the internal $\QLE(8/3,0)$ metric associated with $h \circ f_\delta^{-1} + Q\log| (f_\delta^{-1})'|$ from $-\infty$ to the midpoint of the line with $\re(z) = u_{\beta,\delta}$ contained in $[\pi/4,3\pi/4] \times \R$ and let $\gamma_3 = f_\delta^{-1}(\wt{\gamma}_3)$.  Note that $\gamma_1$ crosses $\gamma_2$.  Moreover, standard distortion estimates for conformal maps imply that on $F_{\alpha,\delta}$ we have that $\gamma_3$ also crosses $\gamma_2$ for all $\delta > 0$ small enough.  It therefore follows that, on $F_{\alpha,\delta}$, the distance between $-\infty$ and $\eta'(\delta)$ is bounded from above by the sum of the $\oqdist$ lengths of $\gamma_1,\gamma_2,\gamma_3$.  Thus our first goal will be to show that the lengths of these three paths have a finite $p$th moment for some $p > 1$.  We will then deduce the result from this using a scaling argument.

We begin by bounding the length of $\gamma_1$.  Fix $p > 1$ and $\epsilon > 0$.  Let $D_k$ be as in the statement of Proposition~\ref{prop::quantum_distance_deviations}.  Let $n = \lceil u_{\beta,\delta} \rceil$ and let $d_1 = \sum_{k=-\infty}^n D_k$.  Then the length of $\gamma_1$ is bounded by $d_1$.  Suppose that $p > 1$.  By Jensen's inequality, with $c_\epsilon = \sum_{k=0}^
\infty e^{-\epsilon k}$, we have that
\begin{equation}
\label{eqn::d_1_p_bound}
d_1^p = \left( \sum_{k=-\infty}^n D_k \right)^p = \left( \sum_{k=-\infty}^n D_k  e^{- \epsilon (n-k)} e^{\epsilon (n-k)} \right)^p \leq c_\epsilon^{p-1} \sum_{k=-\infty}^n D_k^p e^{\epsilon (n-k) p}.
\end{equation}
Thus to bound $\E[d_1^p]$ it suffices to bound the expectation of the right side of~\eqref{eqn::d_1_p_bound}.  Proposition~\ref{prop::quantum_distance_deviations} and scaling together imply that there exists $p > 1$ and constants $c_0,c_1 > 0$ such that $\E[ D_k^p ] \leq c_0 e^{c_1 (k-n)} \delta^{c_1 \alpha}$ for each $k$.  Thus by choosing $\epsilon > 0$ sufficiently small, by inserting this into~\eqref{eqn::d_1_p_bound} we see that (possibly adjusting $c_0,c_1$)
\begin{equation}
\E[ d_1^p] \leq c_0 \delta^{c_1 \alpha}.
\end{equation}
The same argument also implies that the $p$th moments of the lengths of $\gamma_2$ and $\gamma_3$ are both at most $c_0 \delta^{c_1 \alpha}$ (possibly adjusting $c_0,c_1 > 0$).

Combining everything implies that there exists $p > 1$ such that (possibly adjusting $c_0,c_1 > 0$)
\begin{equation}
\label{eqn::initial_d_delta_moment_bound}
\E[ D_\delta^p \one_{F_{\alpha,\delta}}] \leq c_0 \delta^{c_1 \alpha}.
\end{equation}

Recall that if we add $C$ to the field then quantum natural time gets scaled by $e^{3 \gamma C/4}$, for $\gamma=\sqrt{8/3}$ (see \cite[Section~6.2]{quantum_spheres}) and quantum distance gets scaled by $e^{\gamma C/4}$ (Lemma~\ref{lem::quantum_distance_scale}).  In particular, if we add $C$ to $h$ then $D_\delta$ becomes $e^{\gamma C/4} D_{e^{3\gamma C/4} \delta}$ and $F_{\alpha,\delta}$ becomes $F_{\alpha', \delta'}$ where $\alpha' = (\alpha \log \delta  + C) / (\log \delta + 3\gamma C/4)$ and $\delta' = e^{3 \gamma C/4} \delta$.  We take $C = \tfrac{4}{3\gamma} \log \tfrac{2}{\delta}$ so that $e^{3 \gamma C/4} \delta = 2$.  Applying this in the setting of~\eqref{eqn::initial_d_delta_moment_bound}, we see that (possibly adjusting $c_0,c_1 > 0$)
\begin{align*}
\E[ D_2^p \one_{F_{c \log \delta^{-1},2}}] \leq c_0 \delta^{c_1 \alpha} \times (2\delta^{-1})^{p/3} = 2^{p/3} c_0 \delta^{c_1\alpha-p/3} \quad\text{where}\quad c = \frac{\tfrac{4}{3\gamma} - \alpha}{\log 2} + \frac{4}{3 \gamma \log \delta^{-1}}.
\end{align*}
(Here, $c$ is determined by the formula $c \log \delta^{-1} = \alpha'$.)  We assume that $\alpha > 0$ is sufficiently small so that $c > 0$ for all $\delta > 0$ sufficiently small.  Applying this with $\delta = e^{-k}$ for $k \geq 0$ we thus see that (possibly adjusting $c_0,c_1 > 0$)
\begin{align}
\label{eqn:d2_bound}
\E[ D_2^p \one_{F_{c k,2}}] \leq c_0 e^{c_1 k}.
\end{align}
We have that
\begin{align*}
\E[ D_2^p ]
&= \E[ D_2^p \one_{F_{c,2}}] + \sum_{k=2}^\infty \E[ D_2^p \one_{F_{c k,2} \setminus F_{c(k-1),2}^c}]\\
&\leq \E[ D_2^p \one_{F_{c,2}}] + \sum_{k=2}^\infty \left( \E[ D_2^{p p'} \one_{F_{c k,2}}]^{1/p'} \p[  F_{c(k-1),2}^c]^{1/q'} \right)
\end{align*}
for $p',q' > 1$ with $\tfrac{1}{p'} + \tfrac{1}{q'} = 1$.  By choosing $\alpha > 0$ sufficiently small, we can arrange so that $\p[F_{ck, 2}^c] = \p[F_{\alpha,e^{-k}}^c]$ decays to $0$ as $k \to \infty$ faster than any fixed power of $e^{-k}$.  Also, for $p' > 1$ sufficiently close to $1$, by~\eqref{eqn:d2_bound} we have that the first term in the sum above is at most a fixed power of $e^k$.  Therefore the sum above is finite.  Altogether, this implies that $\E[ D_2^p] < \infty$.  By scaling, this implies that $\E[ D_\delta^p] < \infty$ for all $\delta \in (0,1)$.  In particular, since the quantum natural time scales as the third power of quantum distance, we obtain~\eqref{eqn::base_to_tip_moment_bound}.

The final assertions of the proposition are immediate from the first and Proposition~\ref{prop::process_does_not_leave_disk}.
\end{proof}

\begin{proof}[Proof of Proposition~\ref{prop::left_right_shift_moment_bound}]
This follows from the same argument used to prove Proposition~\ref{prop::base_to_tip_expectation_bound}, except we have to explain why the analog of Proposition~\ref{prop::process_does_not_leave_disk} holds in this setting.  This, in turn, is a consequence of Proposition~\ref{prop::gff_boundary_length}.
\end{proof}

\newcommand{\ttime}{\mathfrak t}
\newcommand{\mlevrev}[1]{{\mathsf M}_{\mathrm{SPH},\mathrm{R}}^{2,#1}}
\newcommand{\rtime}{\mathfrak r}
\newcommand{\bandlaw}{{\mathsf B}}
\newcommand{\filledmetcomp}{{\mathsf C}}

\section{Reverse explorations of $\sqrt{8/3}$-LQG spheres}
\label{sec::reverse_explorations}

\newcommand{\unexplored}{\wt{U}}

In \cite{qlebm,quantum_spheres} we constructed forward explorations of doubly-marked $\sqrt{8/3}$-LQG spheres sampled from the infinite measure $\Mstwo$ by $\QLE(8/3,0)$ and $\SLE_6$, respectively.  The purpose of this section is to describe the time-reversals of the unexplored-domain processes which correspond to these explorations.  We will begin with the case of $\SLE_6$ in Section~\ref{subsec::reverse_exploration} and then do the case of $\QLE(8/3,0)$ in Section~\ref{subsec::reverse_qle}.  In Section~\ref{subsec::metric_bands}, we will collect some consequences of the properties of the time-reversal of the unexplored domain process for $\QLE(8/3,0)$.  The reason that we need to study the time-reversal of the unexplored domain process is to check that the assumptions from the characterization given in Theorem~\ref{thm::levynetbasedcharacterization} of TBM are satisfied in the setting of the $\sqrt{8/3}$-LQG sphere to complete the proof of Theorem~\ref{thm::tbm_lqg}.  In particular, the breadth first description of the $3/2$-L\'evy net is described in terms of the evolution of the boundary lengths between a collection of geodesics as one decreases the radius of a metric ball (recall Figure~\ref{fig::levy_net_sketch}).  Verifying this in the setting of the $\sqrt{8/3}$-LQG sphere in Section~\ref{sec::geodesics_and_levy_net} will use as input various independence properties of the time-reversal of thee unexplored domain process for $\QLE(8/3,0)$ which we will collect here.

\subsection{Time-reversal of $\SLE_6$ unexplored-domain process}
\label{subsec::reverse_exploration}

\begin{figure}[ht!]
\begin{center}
\includegraphics[scale=0.85]{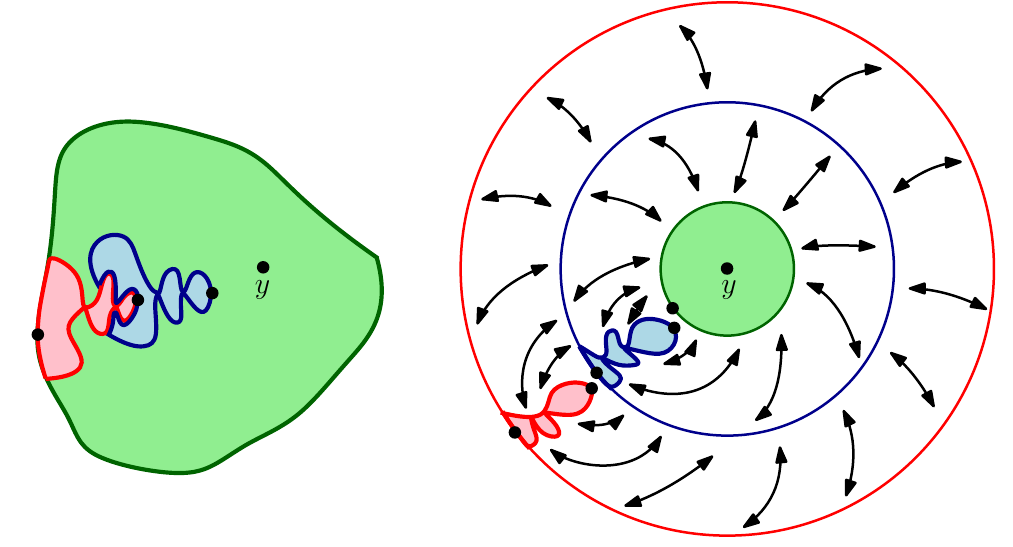}	
\end{center}
\caption{\label{fig::necklace_construction} {\bf Left:} Part of the time-reversal of the unexplored-domain process associated with a whole-plane $\SLE_6$ process $\eta'$ on a doubly-marked quantum sphere $(\CS,x,y)$ from $x$ to $y$.  If $T$ denotes the (random) amount of quantum natural time required by $\eta'$ to go from $x$ to $y$ and $\delta > 0$, then the green region corresponds to the component of $\CS \setminus \eta'([0,T-\delta])$ which contains $y$.  This surface is doubly marked by the interior point $y$ and the boundary point $\eta'(T-\delta)$.  The union of the blue and green regions corresponds to the component of $\CS \setminus \eta'([0,T-2\delta])$ containing $y$.  This surface is also doubly marked, with the interior point being equal to $y$ and the boundary point equal to $\eta'(T-2\delta)$.  The red region is defined similarly.  Each of the green, blue, and red regions may individually be viewed as a doubly-marked surface, where in this case the surface is marked by the first and last point visited by $\eta'$.
 {\bf Right:}  We separate the three surfaces on the left hand side into three ``necklaces.''  As on the left, each necklace has two marked points.  Each necklace also has two marked boundary segments, which we call the ``top'' and ``bottom'' of the necklace.  The top corresponds to the boundary segment which is not part of the circular arc (hence filled by $\eta'$) and the bottom corresponds to the part of the circular arc which is bold (see Figure~\ref{fig::first_approximation_construction} for further explanation).  If we glue together the necklaces as shown (with the tip of one necklace identified with the initial point of the next), then we can recover the left hand picture.}
\end{figure}

Suppose that $(\CS,x,y)$ has distribution given by that of a $\sqrt{8/3}$-LQG sphere decorated with a whole-plane $\SLE_6$ process $\eta'$ from $x$ to $y$ sampled from $\Mstwo$.  We assume that $\eta'$ has the quantum natural time parameterization.  For each $t$, we let $U_t$ be the component of $\CS \setminus \eta'([0,t])$ containing $y$.  We recall from \cite{quantum_spheres} that the quantum boundary length of $U_t$ evolves as the time-reversal of a $3/2$-stable L\'evy excursion $e \colon [0,T] \to \R_+$ with only upward jumps \cite[Theorem~1.2]{quantum_spheres}.  (Recall that the L\'evy excursion measure is an infinite measure.)  We let $\unexplored_t = U_{(T-t)^+}$ and note that $\unexplored_T = U_0 = \CS$.

\begin{proposition}
\label{prop::reverse_sle6_independence}
Using the notation introduced just above, we have that:
\begin{enumerate}[(i)]
\item\label{it::reverse_sle_6_bl} The quantum boundary length of $\partial \unexplored_t$ evolves in $t$ as a $3/2$-stable L\'evy excursion with only upward jumps from $0$ to $0$ of length $T$.  (We emphasize that $T$ is the length of the L\'evy excursion and is random and that $t$ can be bigger than $T$.)
\item\label{it::reverse_independence} For each $t > 0$, on $\{T > t\}$ we have that the quantum surface parameterized by $\unexplored_t$ and decorated by the path $\eta'|_{[T-t,T]}$ is conditionally independent of the quantum surface parameterized by $\CS \setminus \unexplored_t$ and decorated by the path $\eta'|_{[0,T-t]}$ given the boundary length of $\partial \unexplored_t$.
\item \label{it::reverse_removable} For each $t > 0$, on $\{T > t\}$ we have that $\partial \unexplored_t$ is a.s.\ conformally removable.
\end{enumerate}
\end{proposition}

\begin{proof}
Part~\eqref{it::reverse_sle_6_bl} is immediate from \cite[Theorem~1.2]{quantum_spheres} as mentioned above.

We will deduce part~\eqref{it::reverse_independence} from the Markov property for forward explorations of quantum spheres.  The argument will be similar to that given for metric explorations of the Brownian map just before the statement of \cite[Proposition~4.4]{map_making}.  Let $\MstwoW$ denote the distribution
\[ d \MstwoW = \one_{[0,T]}(\ttime) d\ttime d \Mstwo\]
where $T$ denotes the length of the underlying L\'evy excursion and $d\ttime$ denotes Lebesgue measure on $\R$.  Then $\MstwoW$ is a measure on triples consisting of a doubly marked quantum sphere $(\CS,x,y)$, an independent whole-plane $\SLE_6$ from $x$ to $y$, and a random time $\ttime = TU$ where $T$ is the length of the L\'evy excursion and $U$ is a uniform random variable in $[0,1]$ independent of everything else.  Note that if we integrate out $\ttime$, then the marginal distribution of everything else is given by the measure whose Radon-Nikodym derivative with respect to $\Mstwo$ is equal to $T$, the length of the L\'evy excursion.  This is the reason for the subscript $\mathrm{W}$, which stands for ``(length) weighted''.  It is proved in \cite[Proposition~4.1]{qlebm} that given the random time $\ttime$ and the quantum surface parameterized by $\CS \setminus U_\ttime$ and decorated by $\eta'|_{[0,\ttime]}$, the conditional law of the quantum surface parameterized by $U_\ttime$ and decorated by $\eta'|_{[\ttime,T]}$ is that of a quantum disk weighted by its quantum area decorated by an independent radial $\SLE_6$ targeted at a point chosen independently from the quantum measure.  Since $\ttime$ is determined by the quantum surface $\CS \setminus U_\ttime$ decorated by $\eta'|_{[0,\ttime]}$, this implies that $U_\ttime$ decorated by $\eta'|_{[\ttime,T]}$ is conditionally independent of $\CS \setminus U_\ttime$ decorated by $\eta'|_{[0,\ttime]}$ given the boundary length of $\partial U_\ttime$.

Since $U$ is uniform in $[0,1]$ independently of everything else, we note that $T-\ttime = (1-U)T$ has the same conditional distribution as $\ttime$ given everything else.  It therefore follows that $\CS \setminus U_{T-\ttime}$ decorated by $\eta'|_{[0,T-\ttime]}$ is conditionally independent of $U_{T-\ttime}$ decorated by $\eta'|_{[T-\ttime,T]}$ given the boundary length of $\partial U_{T-\ttime}$.    Since $\ttime$ is determined by $U_{T-\ttime}$ decorated by $\eta'|_{[T-\ttime,T]}$, it follows that the same conditional independence statement holds when we condition further on $\ttime$.

In the previous paragraph, we have proved the desired conditional independence property in the setting of $\MstwoW$.  We note that it is explained in \cite[Proposition~4.1]{qlebm} that the conditional distribution of $\MstwoW$ given $\ttime = t$ is that of $\Mstwo$ conditioned on the event $\{T > t\}$, i.e., the length of the L\'evy excursion is at least $t$.  Combining, we see that the conditional law of the path decorated surface consisting of the quantum surface parameterized by $U_{T-t}$ and decorated by $\eta'|_{[T-t,T]}$ given $t$ and the quantum surface parameterized by $\CS \setminus U_{T-t}$ and decorated by $\eta'|_{[0,T-t]}$ is the same under both $\Mstwo$ and $\MstwoW$.  This implies part~\eqref{it::reverse_independence}.

Part~\eqref{it::reverse_removable} is proved in the same way as part~\eqref{it::reverse_independence} since we have the a.s.\ conformal removability of $\partial U_t$ under $\Mstwo$ for each fixed $t > 0$, hence also $\partial U_\ttime$ under $\MstwoW$.
\end{proof}

Throughout, we let $\mlevrev{t}$ denote the infinite measure on doubly marked surfaces $(\unexplored_t,h)$ decorated by a path $\eta'$ as considered in Proposition~\ref{prop::reverse_sle6_independence}.  (The subscript ``${\mathrm R}$'' is to indicate that this law corresponds to a time-reversal.)  We emphasize again that $T > 0$ is random.  On the event that $t < T$, the quantum surface $(\unexplored_t,h)$ has the topology of a disk.  In this case, one marked point is on the disk boundary and the other marked point is in the interior.  The marked points respectively correspond to the starting and ending points of the restriction of $\eta'$ to $\unexplored_t$.  On the event that $t \geq T$, the quantum surface $(\unexplored_t,h) = (U_0,h)$ has the topology of a sphere.  In this case, both of the marked points are contained in the interior of the surface and they correspond to the starting and ending points of $\eta'$.

For $s, t > 0$ we note that there is a natural coupling of $\mlevrev{t}$ and $\mlevrev{t+s}$ because we can produce both laws from $\Mstwo$, as described just above.  Part~\eqref{it::reverse_independence} of Proposition~\ref{prop::reverse_sle6_independence} implies that the path decorated quantum surface which is parameterized by $\unexplored_{s+t} \setminus \unexplored_t$ is conditionally independent of the path-decorated quantum surface parameterized by $\unexplored_t$ given the quantum boundary length of $\unexplored_t$.  We note that both quantum surfaces are doubly marked: $\unexplored_t$ is marked by the initial and target points of $\eta'$  and $\unexplored_{s+t} \setminus \unexplored_t$ is also marked by the initial and target points of the $\SLE_6$.  The usual removability arguments imply that the two path decorated surfaces together with their marked points a.s.\ determine the path decorated surface parameterized by $\unexplored_{s+t}$.  In particular, this gives us a way of describing a two-step sampling procedure for producing a sample from $\mlevrev{t+s}$.  Namely, we:
\begin{itemize}
\item Produce a sample from $\mlevrev{t}$, and then,
\item Given the boundary length, we can glue on a conditionally independent surface which corresponds to another $s$ units of quantum natural time
\end{itemize}
and obtain a sample from $\mlevrev{t+s}$.  We will refer to this operation either as ``zipping in $s$ units of quantum natural time of $\SLE_6$'' or ``gluing in an $\SLE_6$ necklace with quantum natural time length~$s$.''  We note that this operation involves adding at \emph{most}~$s$ units of quantum natural time, however it may involve adding less in the case that all of the boundary length is exhausted in fewer than $s$ units (i.e., if $s+t > T$).  We note that if we iterate this procedure for long enough, then we will eventually be left with a sample from $\Mstwo$ (i.e., $\Mstwo$ is the limit of $\mlevrev{t}$ as $t \to \infty$).

\subsection{Reverse $\QLE(8/3,0)$ metric exploration}
\label{subsec::reverse_qle}

We will now collect some facts about reverse explorations by $\QLE(8/3,0)$.  Suppose that $(\CS,x,y)$ is a doubly marked quantum sphere and let $\Gamma_r$ be $\QLE(8/3,0)$ from $x$ to $y$.  Let $D = \oqdist(x,y)$.  Let $U_r = \CS \setminus \Gamma_r$ and let $\unexplored_r = \CS \setminus \Gamma_{(D-r)^+}$.  We will now collect some facts about the time-reversed unexplored domain process $\unexplored_r$.

\begin{proposition}
\label{prop::reverse_qle_markov_property}
Suppose that $(\CS,x,y)$ is a doubly-marked quantum sphere, let $(\Gamma_r)$ be a $\QLE(8/3,0)$ from $x$ to $y$, and let $D = \oqdist(x,y)$.  Then we have that
\begin{enumerate}[(i)]
\item The quantum boundary length of $\partial \unexplored_r$ evolves as a $3/2$-stable CSBP.
\item For each $r > 0$, on the event $\{D > r\}$ the quantum surfaces parameterized by $\Gamma_{D-r}$ and marked by $x$ and $\unexplored_r$ and marked by $y$ are conditionally independent given the boundary length of $\partial \unexplored_r$.
\item For each $r > 0$, on the event $\{D > r\}$ we have that $\partial \unexplored_r$ is a.s.\ conformally removable.
\end{enumerate}
\end{proposition}

\begin{proof}
The first part of the proposition is immediate from the construction of $\QLE(8/3,0)$ and the definition of the quantum distance parameterization.

The second two parts of the proposition follow from the same argument used to prove part~\eqref{it::reverse_independence} and part~\eqref{it::reverse_removable} from Proposition~\ref{prop::reverse_sle6_independence} except we use the law $\Mstwo$ weighted by the amount of quantum distance time for $\Gamma$ to go from $x$ to $y$ in place of $\MstwoW$.  That is, we let $D = \qdist(x,y)$ and then let
\[ d \MstwoD = \one_{[0,D]}(\rtime) d\rtime d\Mstwo\]
where $d\rtime$ denotes Lebesgue measure on $\R$.  Then $\MstwoD$ is a measure on triples consisting of a doubly marked quantum sphere $(\CS,x,y)$, the $\QLE(8/3,0)$ process $\Gamma$ from $x$ to $y$, and a random time $\rtime = UD$ where $D = \qdist(x,y)$ and $U$ is a uniform random variable in $[0,1]$ independent of everything else.  Then \cite[Proposition~4.1]{qlebm} implies that the conditional law of the surface parameterized by $\CS \setminus \Gamma_\rtime$ given $\Gamma_\rtime$ is that of a quantum disk weighted by its quantum area.  In particular, these two quantum surfaces are conditionally independent given their quantum boundary length.  Since $D-\rtime = (1-U) D$ has the same conditional distribution as $\rtime = UD$ given everything else, it similarly follows that the surface parameterized by $\CS \setminus \Gamma_{D-\rtime}$ is conditionally independent of $\Gamma_{D-\rtime}$ given its quantum boundary length.  Moreover, \cite[Proposition~4.1]{qlebm} implies that the law of $\MstwoD$ given $\rtime =r$ is equal to that of $\Mstwo$ conditional on $\{D > r\}$.  Therefore under $\Mstwo$, for $r > 0$ fixed and on $\{D > r\}$ we have that $\Gamma_{D-r}$ is conditionally independent of $\Gamma_{D-r}$ given its quantum boundary length.  This proves the second part of the proposition.  The third part similarly follows since for each fixed $r > 0$ we have that $\partial \Gamma_r$ is a.s.\ conformally removable, so under $\MstwoD$ we have that $\partial \Gamma_\rtime$ is also a.s.\ conformally removable.
\end{proof}

As in the case of $\QLE(8/3,0)$, we can also consider the time-reversal of the unexplored-domain process associated with the time-reversal of the $\delta$-approximation to $\QLE(8/3,0)$.  This is also a Markov process but is a bit more complicated to describe than in the case of $\QLE(8/3,0)$ (or $\SLE_6$) because the times at which the tip of the $\SLE_6$ is re-randomized are described in terms of the quantum natural time in the forward direction.  In particular, there is an asymmetry in that the final necklace in the $\delta$-approximation to $\QLE(8/3,0)$ a.s.\ does not have quantum natural time equal to $\delta$.  Note that the length of the final necklace is equal to the length of the corresponding L\'evy excursion modulo $\delta$.  In particular, once we observe the length of this necklace, the length of the remainder of the necklaces is an integer multiple of $\delta$.

Let $\rho(t,L)$ be the density at $t > 0$ with respect to Lebesgue measure for a $3/2$-stable L\'evy process with only upward jumps starting from $L$ to hit $0$ at time $t$.  For $t \in [0,\delta)$, also let $\rho_{t,\delta}(L) = \sum_{k=1}^\infty \rho(k \delta-t,L)$.  Note that for any $t \in [0,\delta)$, the evolution of such a process in the time-interval $[0,t]$ conditioned to first hit $0$ at an integer multiple of $\delta$ is the same (by a Bayes' rule calculation) as weighting the unconditioned law by $\rho_{t,\delta}(X_t) / \rho_{0,\delta}(X_0)$.  We note that this Radon-Nikodym derivative formula determines the law of a $3/2$-stable L\'evy process with only upward jumps conditioned to terminate at an integer multiple of $\delta$.  Recall that in the infinite measure on $3/2$-stable L\'evy excursions with only upward jumps one has that the length has distribution given by a constant times $t^{-5/3} dt$ where $dt$ denotes Lebesgue measure on $\R_+$.

\begin{proposition}
\label{prop::delta_reverse_time}
Suppose that $(\CS,x,y)$ is a doubly-marked quantum sphere, $\delta > 0$, and $(\Gamma_r)$ is an instance of the $\delta$-approximation to $\QLE(8/3,0)$ from $x$ to $y$ (parameterized by quantum natural time).  The time-reversal of the unexplored-domain process is Markovian.  Moreover, conditioned on having quantum natural time at least $\delta$, it can be generated using the following Markovian procedure.
\begin{itemize}
\item Step 1: sample the length $Y$ of the final $\SLE_6$ segment.  Let $Z$ be sampled from the law $\one_{\{t \geq \delta\}} t^{-5/3} dt$ (normalized to be a probability measure) and let $Y = Z \mod \delta$.
\item Step 2: sample the path-decorated surface corresponding to the final $\SLE_6$ segment.  Generate the time-reversal of the unexplored-domain process for $\SLE_6$ conditioned on lasting at least $Y$ units of time and weighted by $\rho_{0,\delta}(X)$ where $X$ is its boundary length at time $Y$.
\item Step 3: tip-rerandomization and iteration.  Randomize the location of the tip uniformly from the quantum boundary measure.  Glue in $\delta$ units of reverse $\SLE_6$ conditioned on the boundary length process for the time-reversal of the unexplored domain process hitting $0$ at an integer multiple of $\delta$ (as defined just above the statement).  Repeat until the boundary length process first hits $0$.
\end{itemize}
\end{proposition}
\begin{proof}
This is a direct description of the time-reversal.
\end{proof}

Let us now make some remarks about the time-reversal of the unexplored-domain process for the $\delta$-approximation to $\QLE(8/3,0)$.

\begin{remark}
\label{rem::reverse_qle_facts}
\begin{enumerate}[(i)]
\item First, we note that if we fix $k \in \N$ and let $U$ be the path-decorated surface which arises after performing the above Markovian procedure for $k$ steps after the terminal $\SLE_6$ segment, then on the event that the boundary length of $\partial U$ is positive, the conditional law of $U$ is that of a quantum disk weighted by its quantum area conditioned on the $\delta$-approximation to $\QLE(8/3,0)$ terminating between $k \delta$ and $(k+1) \delta$ units of quantum natural time.
\item\label{it::reverse_delta_qle_covering} Let $X$ be the $3/2$-stable L\'evy excursion which gives the boundary length associated with the time-reversal of the unexplored-domain process for the $\delta$-approximation for $\QLE(8/3,0)$.  If we run $X$ up until a given time before it has terminated, then from $X$ itself it is not possible to determine which increments of time correspond to $\SLE_6$ necklaces (since we need to know the length of the necklace whose length is not equal to $\delta$).  However, we do know that the interval of time which encodes the $k$th necklace is contained in $[(k-1)\delta,(k+2)\delta]$.  This will be useful later for proving bounds on the boundary length of the necklace.
\end{enumerate}
\end{remark}

\subsection{Filled-metric ball complements and metric bands}
\label{subsec::metric_bands}

\begin{figure}[ht!]
\begin{center}
\includegraphics[scale=0.85]{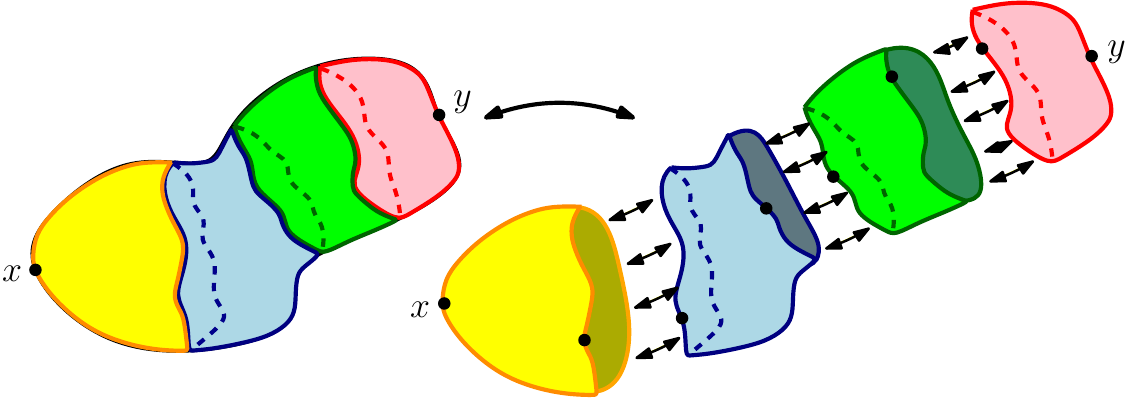}
\end{center}
\caption{\label{fig::metric_band_decomposition} {\bf Left:} An instance $(\CS,x,y)$ of a doubly-marked $\sqrt{8/3}$-LQG sphere decomposed into four metric bands.  Note that a metric band can have the topology of either a disk or an annulus.  {\bf Right:} If we mark the inside and outside of each metric band, then we can uniquely reconstruct $(\CS,x,y)$ by gluing the bands together according to boundary length, with the marked point on each band identified with the corresponding marked point on the next band.}	
\end{figure}

Suppose that we have a doubly-marked quantum sphere $(\CS,x,y)$, $(\Gamma_r)$ is the $\QLE(8/3,0)$ from $x$ to $y$, and $D = \oqdist(x,y)$.  Conditionally on the event that $\{D > r\}$, we let $\filledmetcomp_r$ be the law of the quantum surface which is parameterized by $\CS \setminus \Gamma_{D-r}$.  Proposition~\ref{prop::reverse_qle_markov_property} implies that the conditional law the surface parameterized by $\Gamma_{D-r}$ depends only on the quantum boundary length~$\ell$ of $\partial \Gamma_{D-r}$.  Moreover, by scaling, this law can be sampled from by first sampling from the law in the case that the boundary length is equal to $1$ and then scaling so that the quantum boundary length is equal to $\ell$.  Recall that this has the effect of scaling quantum distances by $\ell^{1/2}$ (Lemma~\ref{lem::quantum_distance_scale}) and quantum areas by $\ell^2$.  If we start off with such a surface of quantum boundary length $\ell$, and then we explore the metric ball in reverse for $s$ units of distance, then we refer to this surface as a (reverse) \emph{metric band} of inner boundary length $\ell$ and width $s$.  That is, conditionally on $\{D > r\}$, we call the quantum surface $\Gamma_{D-r} \setminus \Gamma_{D-r-s}$ a metric band of length $\ell$ (where $\ell$ is the quantum boundary length of $\partial \Gamma_{D-r}$) and of width $s$.  The inner boundary is $\partial \Gamma_{D-r}$ and the outer boundary is $\partial \Gamma_{D-r-s}$ (when it is non-empty).  We call $\bandlaw_{\ell,s}$ the law on such surfaces.  In other words, the law of the quantum surface parameterized by $\Gamma_{D-r} \setminus \Gamma_{D-r-s}$ given that $\{D > r\}$ and the boundary length of $\partial \Gamma_{D-r}$ is $\ell$ is $\bandlaw_{\ell,s}$.  We emphasize that $\bandlaw_{\ell,s}$ is a probability measure since it is defined from an infinite measure conditioned on a positive and finite measure event.  We make the following observations about~$\bandlaw_{\ell,s}$:

\begin{proposition}
\label{prop::metric_band_top_disk_or_annulus}
Fix $\ell, s > 0$ and suppose that $\CB$ has the law $\bandlaw_{\ell,s}$.  Then $\CB$ is topologically either an annulus or a disk (it is a disk in the case that the target point has distance less than $s$ from the boundary of the band).  Moreover, if we fix a sequence $(r_k)$ of positive numbers with $\sum_k r_k = \infty$ and we decompose $(\CS,x,y)$ into its successive bands $\CB_k = \Gamma_{D-\sum_{j=1}^{k-1} r_j} \setminus \Gamma_{D-\sum_{j=1}^k r_j}$ of width $r_k$, then the $\CB_k$ are conditionally independent given the quantum length of their inner and outer boundaries.
\end{proposition}
We note that in the statement of Proposition~\ref{prop::metric_band_top_disk_or_annulus}, there exists $k_0$ (random) such that $B_k = \emptyset$ for all $k \geq k_0$.
\begin{proof}[Proof of Proposition~\ref{prop::metric_band_top_disk_or_annulus}] 
The first assertion follows from the continuity results established earlier (Theorem~\ref{thm::continuity} and Theorem~\ref{thm::metric_completion}).

The second assertion is immediate from the construction and Proposition~\ref{prop::reverse_qle_markov_property}.
\end{proof}

Suppose that we are in the setting of Proposition~\ref{prop::metric_band_top_disk_or_annulus} and we mark the outer boundary of each $\CB_k$ with a point chosen uniformly at random from the quantum measure, so that each metric band is doubly marked (one point on the inside boundary and one point on the outside boundary).  Then the removability of each $\partial \CB_k$ implied by the construction and Proposition~\ref{prop::reverse_qle_markov_property} (combined with the usual removability arguments, e.g., \cite{SHE_WELD}) implies that there is a.s.\ a unique way to glue these doubly marked metric bands together to reconstruct the original doubly-marked quantum sphere $(\CS,x,y)$.  That is, the doubly marked bands a.s.\ determine the entire doubly-marked quantum surface.

This decomposition will be important for us in Section~\ref{subsec::interfaces_are_geodesics}, in which we show that the quantum boundary lengths between geodesics along the boundary of a filled metric ball evolve as independent $3/2$-stable CSBPs.

\section{Emergence of the $3/2$-L\'evy net}
\label{sec::geodesics_and_levy_net}

In this section we will see the $3/2$-L\'evy net structure \cite{map_making} appear in the $\sqrt{8/3}$-LQG sphere.  We will establish this by successively considering three different approximations to geodesics.

We will describe the first approximation in Section~\ref{subsubsec::approximate_interfaces}.  It is based on a reverse exploration by the $\delta$-approximation to $\QLE(8/3,0)$ (Section~\ref{subsec::reverse_qle}).  Using Proposition~\ref{prop::base_to_tip_expectation_bound} and Proposition~\ref{prop::left_right_shift_moment_bound} established in Section~\ref{sec::sle_6_moment_bounds}, we will then show in Section~\ref{subsubsec::approximate_interfaces_converge} that these first approximations to geodesics converge (at least along a subsequence) to limiting continuous paths.  These subsequential limits serve as our second approximation to geodesics.

Although it may not be obvious from the construction that the first and second approximations to geodesics are related to actual geodesics, these approximations will be useful to analyze.  This is because, as we will show in Section~\ref{subsubsec::boundary_length_csbp}, it will follow from the construction that if one considers two such second approximations to geodesics and then performs a reverse metric exploration, then the quantum lengths of the two segments of the boundary of the reverse metric exploration between the two paths evolve as independent $3/2$-stable CSBPs.  In fact, we will show that this holds more generally for any finite collection of such paths.  At this point, we will start to see some of the (breadth first) $3/2$-L\'evy net structure from \cite{map_making} to emerge.

We will see in the proofs that these second approximations to geodesics are finite length paths but we will not rule out in the construction that they can be strictly longer than an actual geodesic.  This will lead us to our third approximation to geodesics, which will be paths whose expected length is at most $(1+\epsilon)$ times the length of an actual geodesic with the additional property that the quantum boundary lengths between such paths along the boundary of a reverse metric exploration evolve approximately like independent $3/2$-stable CSBPs.  We will then use that quantum boundary lengths and quantum distances have different scaling exponents to deduce that the quantum boundary lengths between any finite collection of actual geodesics also evolve as independent $3/2$-stable CSBPs.

Once we have finished all of this, it will not require much additional work in Section~\ref{subsec::tbm_lqg} to combine the results of this article with Theorem~\ref{thm::levynetbasedcharacterization} to complete the proof of Theorem~\ref{thm::tbm_lqg}.

Throughout this section, for a doubly-marked surface $(\CS,x,y)$ and $r > 0$, we will write $\qhull{x}{r}$ for the hull of the closure of $\qball{x}{r}$ relative to $y$. That is, $\qhull{x}{r}$ is the complement of the $y$-containing component of $\CS \setminus \qhull{x}{r}$.  Equivalently, $\qhull{x}{r}$ is equal to the hull of the $\QLE(8/3,0)$ growth from $x$ to $y$ with the quantum distance parameterization stopped at time $r$.

\subsection{First approximations to geodesics}
\label{subsubsec::approximate_interfaces}

\begin{figure}[ht!]
\begin{center}
\includegraphics[scale=0.85]{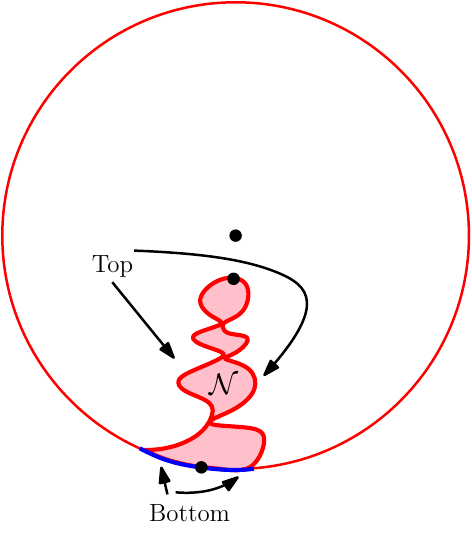}
\end{center}	
\caption{\label{fig::sle_6_necklace} Shown is an $\SLE_6$ necklace $\CN$ of length $\delta$.  When referring to the boundary of $\CN$, we mean the boundary of the region which is cut off from the target point by the corresponding $\SLE_6$.  (We will only show this part of the necklace in illustrations in subsequent figures.)  We can divide the boundary of $\CN$ into two parts: the top (heavy red) and the bottom (blue), as shown.  The top is marked by the terminal point of the $\SLE_6$ and the bottom is marked by the initial point.  If $T$ (resp.\ $B$) denotes the quantum length of the top (resp.\ bottom) of the necklace and $X$ is the $3/2$-stable L\'evy process with only upward jumps which encodes the change in the boundary length of the time-reversal of the unexplored-domain process as one glues in $\CN$ then we have that $B-T = X_\delta - X_0$.  In the case that the top is not disconnected to the bottom (as shown, corresponding to the case that the $\SLE_6$ does not wrap around its target point), then we can divide the top and the bottom into their left and right sides.  The left (resp.\ right) part of the top is the part of the top which is to the left (resp.\ right) of the marked point up until it hits the outer boundary of the necklace.  The left (resp.\ right) side of the bottom is defined similarly.}
\end{figure}

\begin{figure}[ht!]
\begin{center}
\includegraphics[scale=0.85,page=1,trim={0 0 0 1.2cm},clip]{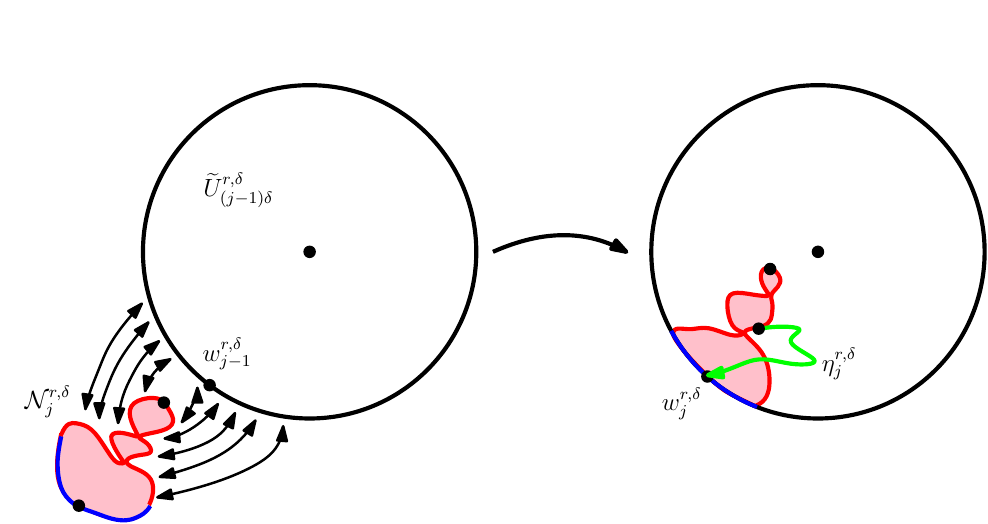}
\end{center}
\caption{\label{fig::first_approximation_construction}Illustration of one step in the construction of the first approximations to geodesics.  {\bf Left:} The disk represents the surface parameterized by $\unexplored_{(j-1)\delta}^{r,\delta}$.  Shown is the event $A_j^{r,\delta}$ that the top of the $\SLE_6$ necklace $\CN_j^{r,\delta}$ is glued to a boundary segment which contains the marked boundary point $w_{j-1}^{r,\delta}$ at step $j$.  {\bf Right:} The disk represents the surface parameterized by $\unexplored_{j\delta}^{r,\delta}$, which is formed by gluing $\CN_j^{r,\delta}$ to $\unexplored_{(j-1)\delta}^{r,\delta}$.  The path $\eta_j^{r,\delta}$, indicated in green, is a shortest path in the internal metric (recall Section~\ref{subsubsec::internal_metric}) associated with $\unexplored_{j \delta}^{r,\delta}$ which connects the marked boundary point $w_j^{r,\delta}$ at step $j$ to the marked boundary point $w_{j-1}^{r,\delta}$ from step $j-1$.}
\end{figure}

\begin{figure}[ht!]
\begin{center}
\includegraphics[scale=0.85,page=2,trim={0 1.5cm 0 0},clip]{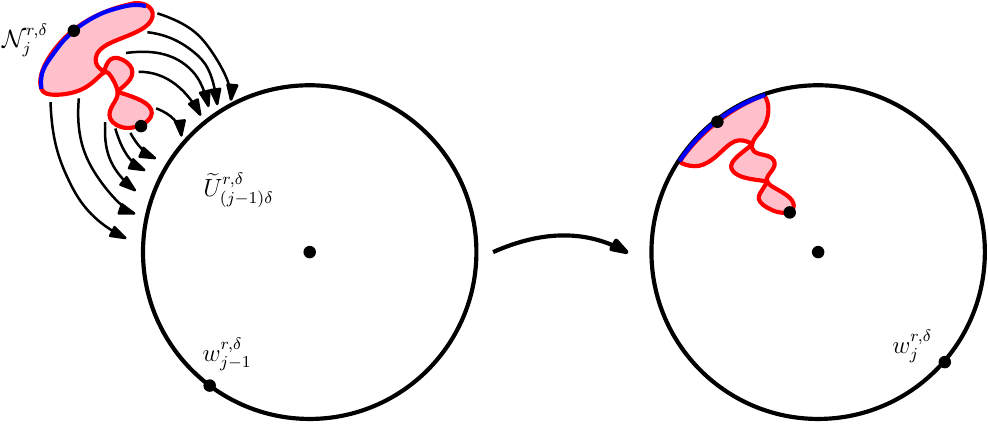}
\end{center}
\caption{\label{fig::first_approximation_construction2} (Continuation of Figure~\ref{fig::first_approximation_construction}) Illustration of one step in the construction of an approximate geodesic.  {\bf Left:} Shown is the case that $(A_j^{r,\delta})^c$ occurs, i.e., the top of the $\SLE_6$ necklace $\CN_j^{r,\delta}$ is glued to a boundary segment which does not contain the marked point $w_{j-1}^{r,\delta}$ from step $j-1$.  {\bf Right:} Shown is the surface parameterized by $\unexplored_{j\delta}^{r,\delta}$.  In this case, $\eta_j^{r,\delta}$ is the constant path which is equal to the point $w_{j-1}^{r,\delta} = w_j^{r,\delta}$.}	
\end{figure}

Fix $r,\delta > 0$.  Suppose that we have a doubly-marked quantum sphere $(\CS,x,y)$ which has distribution $\Mstwo$ and which is decorated by an instance $\Gamma^\delta$ of the $\delta$-approximation to $\QLE(8/3,0)$ from $x$ to $y$.  Assume that $\Gamma^\delta$ is parameterized by quantum natural time and for each $t \geq 0$ we let $U_t^\delta = \CS \setminus \Gamma_t^\delta$.  We also let $\unexplored_t^\delta = U_{T-t}^\delta$ where $T$ is the amount of quantum natural time required by $\Gamma^\delta$ to go from $x$ to $y$ and let $X_t^\delta$ be the quantum boundary length of $\partial \unexplored_t^\delta$.  We also let $\tau_r^\delta = \inf\{t \geq 0 : \int_0^t (X_s^\delta)^{-1} ds \geq r\}$ be the first time that $r$ units of quantum distance time have accumulated in the time-reversal of $\Gamma^\delta$.  We then let $\sigma_r^\delta$ be the first time after time $\tau_r^\delta$ that one of the $\SLE_6$ necklaces in $\Gamma^\delta$ is finished being glued in and set $\unexplored_t^{r,\delta} = \unexplored_{\sigma_r^\delta+t}^\delta$.  For each $u \geq 0$, we let $X_u^{r,\delta}$ be the quantum boundary length of $\partial \unexplored_u^{r,\delta}$.  By Proposition~\ref{prop::delta_reverse_time}, $X^{r,\delta}$ evolves as a $3/2$-stable L\'evy process with only upward jumps conditioned to first hit $0$ at a time which is an integer multiple of $\delta$.

We augment the construction of $\unexplored^{r,\delta}$ by simultaneously building what we will call a \emph{first approximation} to a geodesic as follows.  For each $j$, we let $\CN_j^{r,\delta}$ be the $j$th $\SLE_6$ necklace which is glued to $\unexplored^{r,\delta}$ in the reverse $\delta$-approximation to $\QLE(8/3,0)$ exploration where we take the indexing so that $\CN_1^{r,\delta}$ is the necklace being glued into $\unexplored^{r,\delta}$ starting at time~$0$.  We note that $\CN_j^{r,\delta}$ is encoded by $X^{r,\delta}|_{[(j-1)\delta,j\delta]}$ and the corresponding collection of quantum disks.  We can divide the outer boundary of $\CN_j^{r,\delta}$ into two parts: the bottom and the top (see Figure~\ref{fig::sle_6_necklace}).  The second part is what gets glued to $\unexplored_{(j-1)\delta}^{r,\delta}$ and it is marked by the tip of the $\SLE_6$ segment and the bottom is marked by the initial point of the path.  Let~$T_j^{r,\delta}$ (resp.\ $B_j^{r,\delta}$) denote the quantum length of the top (resp.\ bottom) of~$\CN_j^{r,\delta}$.  In the case that the top is not disjoint from the bottom (as illustrated in Figure~\ref{fig::first_approximation_construction}, corresponding to when the $\SLE_6$ has not wrapped around its target point), then we can also divide the top and bottom into their left and right sides.  The left (resp.\ right) side of the top is the part which is to the left (resp.\ right) of the marked point up until it hits the outer boundary of the necklace.  The left (resp.\ right) side of the bottom is the part which is to the left (resp.\ right) side of the bottom up until it hits top.

We are now going to derive formulas for $T_j^{r,\delta}$ and $B_j^{r,\delta}$.  In order to motivate these formulas, we will first recall an analogous formula in the context of a quantum wedge.  Namely, suppose that $(\h,h,0,\infty)$ is a $\sqrt{8/3}$-quantum wedge and $\eta'$ is an independent $\SLE_6$ on $\h$ from $0$ to $\infty$ which has been parameterized by quantum natural time.  Let $L_t$ (resp.\ $R_t$) denote the change in the boundary length of the outer boundary of $\eta'([0,t])$ relative to time $0$ (so that $L_0 = R_0 = 0$).  Then the boundary length of the part of the outer boundary of $\eta'([0,t])$ which is to the left (resp.\ right) of $\eta'(t)$ is given by $L_t - \inf_{s \in [0,t]} L_s$ (resp.\ $R_t - \inf_{s \in [0,t]} R_s$).  These lengths correspond to the left and right sides of the top of $\eta'([0,t])$.  Therefore the length of the top of the outer boundary of $\eta'([0,t])$ is given by $L_t + R_t - \inf_{s \in [0,t]} L_s - \inf_{s \in [0,t]} R_s$.  Similarly, the length of the interval that $\eta'([0,t])$ has separated from $\infty$ (corresponding to the bottom of $\eta'([0,t])$) is equal to $-\inf_{s \in [0,t]} L_s - \inf_{s \in [0,t]} R_s$.

We will now extend the formulas from the setting of a forward exploration of a quantum wedge to the setting of a reverse exploration of a quantum sphere.  The main difference between these two cases is that when one explores a sphere or disk with an $\SLE_6$, there is the possibility that it can wrap around its target point.  In terms of necklaces, this corresponds to when the top is disconnected from the bottom and happens when $T_j^{r,\delta} = X_{(j-1)\delta}^{r,\delta}$.  We note that $B_j^{r,\delta} - T_j^{r,\delta} = X_{j \delta}^{r,\delta} - X_{(j-1) \delta}^{r,\delta}$.  On the event that $T_j^{r,\delta} < X_{(j-1)\delta}^{r,\delta}$, for $u \in [0,\delta]$ we let $X_u^{j,r,\delta,L}$ (resp.\ $X_u^{j,r,\delta,R}$) denote the change in the left (resp.\ right) boundary length of the $\SLE_6$ which forms $\CN_j^{r,\delta}$ as it is being zipped in so that
\[ X_{(j-1)\delta + u}^{r,\delta} - X_{(j-1)\delta}^{r,\delta} = X_u^{j,r,\delta,L} + X_u^{j,r,\delta,R}.\]
On the event that $T_j^{r,\delta} < X_{(j-1)\delta}^{r,\delta}$, we have that
\begin{equation}
\label{eqn::top_length}
T_j^{r,\delta} = - \left( \inf_{u \in [0,\delta]} X_u^{j,r,\delta,L}  + \inf_{u \in [0,\delta]} X_u^{j,r,\delta,R} \right)
\end{equation}
and 
\begin{equation}
\label{eqn::bottom_length}
B_j^{r,\delta} = X_\delta^{j,r,\delta,L} + X_\delta^{j,r,\delta,R} - \left( \inf_{u \in [0,\delta]} X_u^{j,r,\delta,L}  + \inf_{u \in [0,\delta]} X_u^{j,r,\delta,R} \right).
\end{equation}
When one performs a reverse $\SLE_6$ exploration of a quantum sphere as described in Section~\ref{subsec::reverse_exploration}, the left and right boundary length processes of the $\SLE_6$ necklaces are independent $3/2$-stable L\'evy processes with only upward jumps up until the first time that the length of the top is equal to the length of the outer boundary of the previous necklace.  The reason for this is that the overall boundary length process is a $3/2$-stable L\'evy process with only upward jumps.  The left (resp.\ right) boundary length process can then be generated from the overall boundary length process by considering those jumps which are to the left (resp.\ right) of the tip and each such jump is to the left (resp.\ right) independently with probability $1/2$.  Indeed, whether a jump is to the left or right corresponds to the orientation of the boundary of the corresponding quantum disk.  In the setting of a reverse $\delta$-approximation to $\QLE(8/3,0)$ exploration, the situation is a little bit more complicated because the boundary length process for the $\SLE_6$ necklaces (except for the first one) is conditioned to first hit $0$ at an integer multiple of $\delta$.  When the overall boundary length is bounded from below, it is not difficult to see that the Radon-Nikodym derivative between the conditioned and unconditioned processes run for $\delta$ units of time is bounded from above and below by constants.  Therefore Lemma~\ref{lem::stable_maximum} implies that $T_j^{r,\delta}$ has an exponential moment on the event that the overall boundary length is bounded from below in the corresponding time interval.  This will be important for us in our later arguments.

We suppose that $w_0^{r,\delta}$ is picked uniformly from $\partial \unexplored_0^{r,\delta}$ using the quantum boundary measure.  Assume that we have defined $w_0^{r,\delta},\ldots,w_{j-1}^{r,\delta}$.  Then we inductively define~$w_j^{r,\delta}$ as follows.  If $w_{j-1}^{r,\delta}$ is contained in the interval of $\partial \unexplored_{(j-1)\delta}^{r,\delta}$ to which the top of $\CN_j^{r,\delta}$ is glued, then we take~$w_j^{r,\delta}$ to be equal to the marked point on the bottom of~$\CN_j^{r,\delta}$ (see Figure~\ref{fig::first_approximation_construction}).  Otherwise, we take~$w_j^{r,\delta}$ to be equal to~$w_{j-1}^{r,\delta}$ (see Figure~\ref{fig::first_approximation_construction2}).

We then form a path $\eta^{r,\delta}$, our first approximation to a geodesic, by connecting the points $w_0^{r,\delta},\ldots,w_n^{r,\delta}$ with paths $\eta_j^{r,\delta}$ where we take $\eta_j^{r,\delta}$ to be the shortest path in the internal metric of the surface which has been explored by time $j$ (i.e., the surface parameterized by $\unexplored_{j\delta}^{r,\delta}$) between $w_j^{r,\delta}$ and $w_{j-1}^{r,\delta}$.  Note that $T_j^{r,\delta}$ is typically of order $\delta^{2/3}$ while the quantum length of $\partial \unexplored_{j\delta}^{r,\delta}$ is typically of order $1$.  Thus, the probability that $w_j^{r,\delta} \neq w_{j-1}^{r,\delta}$ is of order $\delta^{2/3}$.  In particular, the number of $j$ such that $w_j^{r,\delta} \neq w_{j-1}^{r,\delta}$ is of order $\delta^{-1/3}$.

We have so far defined a single path $\eta^{r,\delta}$.  By repeating this construction with independently chosen initial points on $\partial \unexplored_0^{r,\delta}$, we can construct many such paths.

\subsection{Second approximations to geodesics}
\label{subsubsec::approximate_interfaces_converge}

Fix $r > 0$.  We will now show that the joint law of $(\CS,x,y)$ and the paths $\eta^{r,\delta}$ constructed in Section~\ref{subsubsec::approximate_interfaces} (first approximations to geodesics) just above converges weakly, at least along a subsequence $(\delta_k)$, to a limiting doubly-marked quantum sphere $(\CS,x,y)$ with law $\Mstwo$ decorated by a path~$\eta^r$ which connects a uniformly random point on the boundary of the reverse metric exploration at time $r$ to $x$.  Since $\Mstwo$ is an infinite measure, we need to clarify what we mean by weak convergence.  In this context, we mean that for each $a > 0$ the sequence of probability measures given by $\Mstwo$ conditioned on the total mass of $\CS$ being at least $a$ converge weakly.

The exact topology that we use here is not important, but to be concrete we will make the following choice.  By applying a conformal transformation, we can parameterize $(\CS,x,y)$ using $\s^2$ with $x$ (resp.\ $y$) taken to the south (resp.\ north) pole and the starting point of $\eta^{r,\delta}$ taken to a fixed point on the equator.  We recall that the $\eta^{r,\delta}$ are parameterized according to arc length using $\oqdist$.  We use the uniform topology on paths on $\s^2$ and the weak topology on measures on $\s^2$ for the area measure which encodes the quantum surface.

We will refer to the path $\eta^r$ as our \emph{second approximation} to a geodesic because it has finite $\oqdist$-length from a point on $\partial \qhull{x}{\oqdist(x,y)-r}$ to $x$.  In the process of proving the existence of $\eta^r$, we will also show that it has certain properties that will be useful for us in the next section.  We will later show that the quantum boundary length of the two segments along the boundary of a metric ball between two such paths started at uniformly random points evolve as independent $3/2$-stable CSBPs and, more generally, that the same is true for any finite number of paths.

\begin{proposition}
\label{prop::subsequentially_limiting_paths_exist}
Fix $r > 0$.  There exists a sequence $(\delta_k)$ of positive numbers with $\delta_k \to 0$ as $k \to \infty$ such that the following is true.  The joint law of the doubly marked quantum surfaces $(\CS,x,y)$ and paths $\eta^{r,\delta_k}$ converges weakly (using the topology described just above) to that of a limiting doubly marked quantum surface/path pair $(\CS,x,y)$, $\eta^r$ where the marginal of $(\CS,x,y)$ is given by $\Mstwo$ and the following hold.
	\begin{enumerate}[(i)]
	\item\label{it::path_on_ball_boundary} 
	Almost surely, $\eta^r(t) \in \partial \qhull{x}{\oqdist(x,y)-(r+t)}$ for all $t \in [0,\oqdist(x,y)-r]$.
	\item\label{it::inside_outside} For each $t \geq 0$, given the quantum boundary length of $\partial \qhull{x}{\oqdist(x,y)-(r+t)}$, the quantum surface parameterized by $\qhull{x}{\oqdist(x,y)-(r+t)}$ and marked by the pair $(\eta^r(t),x)$ is independent of the quantum surface parameterized by $\CS \setminus \qhull{x}{\oqdist(x,y)-(r+t)}$ and marked by the pair $(\eta^r(t),y)$.
	\end{enumerate}
Fix $T > 0$, $C > 1$, and let $E_{C,T}^r$ be the event that the quantum boundary length of $\partial \qhull{x}{\oqdist(x,y)-(r+t)}$ stays in $[C^{-1},C]$ for $t \in [0,T]$ and let $\ell_T^r$ be the arc length of $\eta^r|_{[0,T]}$.  Then there exists a constant $K > 0$ depending only on $C,T$ such that $\E[ \ell_T^r \one_{E_{C,T}^r}] \leq K$.  In particular, for every $\epsilon > 0$ there exists $\delta > 0$ such that if we have an event $Q$ which occurs with probability at most $\delta$ then $\E[ \ell_T^r \one_{E_{C,T}^r \cap Q}] \leq \epsilon$.

Finally, by passing to a further subsequence if necessary, we can construct a coupling of a countable collection of paths which each satisfy~\eqref{it::path_on_ball_boundary} and~\eqref{it::inside_outside}, which start at a countable dense set of points chosen i.i.d.\ from the quantum boundary measure on $\partial \qhull{x}{ \oqdist(x,y) -r}$, and do not cross. 
\end{proposition}

We will break the proof of Proposition~\ref{prop::subsequentially_limiting_paths_exist} into several steps which are carried out in Sections~\ref{subsubsec::path_does_not_trace_metric_ball_boundary}--\ref{subsubsec::moment_bounds}.  The part of the proof contained in Section~\ref{subsubsec::path_does_not_trace_metric_ball_boundary} is instructive to read on a first reading because it provides some intuition as to why the second approximations should be related to geodesics.  The estimates from Sections~\ref{subsubsec::law_of_necklace_given_glued}--\ref{subsubsec::moment_bounds} may be skipped on a first reading, since the material here is mainly technical and is focused on transferring the moment bounds from Section~\ref{sec::sle_6_moment_bounds} to the present setting.

We will establish the statement regarding the evolution of the quantum boundary lengths between a finite number of paths as in Proposition~\ref{prop::subsequentially_limiting_paths_exist} in Section~\ref{subsubsec::boundary_length_csbp}.

\subsubsection{Step count distance passes to limit}
\label{subsubsec::path_does_not_trace_metric_ball_boundary}

We begin by establishing a lemma which we will later argue implies part~\eqref{it::path_on_ball_boundary} of Proposition~\ref{prop::subsequentially_limiting_paths_exist}.  This will be important because it will imply that along any subsequence which~$\eta^{r,\delta}$ converges we have that the limiting path $\eta^r$ does not trace along $\partial \qhull{x}{\oqdist(x,y) - (r+t)}$ for any value of $t$.  Equivalently, this will imply that $\eta^r$ is a continuous path if we parameterize it according to its distance from $x$ and the proof will show that this is in fact the natural parameterization to use for~$\eta^r$.

\begin{lemma}
\label{lem::time_parameterization_converges}
There exists a constant $c > 0$ such that the following is true.  For each $j$, we let $A_j^{r,\delta}$ be the event that $w_j^{r,\delta} \neq w_{j-1}^{r,\delta}$ and let $I_j^{r,\delta} = \one_{A_j^{r,\delta}}$.  Fix any value of $t > 0$ and let
\[ N = \min\left\{ m \geq 1 : c^{-1} \delta^{1/3} \sum_{j=1}^m I_j^{r,\delta} \geq t \right\}.\]
On the event that $\oqdist(x,y) > r+t$, we have that $\oqdist(w_N^{r,\delta},x)$ converges in probability as $\delta \to 0$ to $\oqdist(x,y)-(r+t)$.
\end{lemma}
\begin{proof}
Note that $\sum_{j=1}^m I_j^{r,\delta}$ counts the number of times that the marked point moves in $m \delta$ units of quantum natural time.  That is, $\sum_{j=1}^m I_j^{r,\delta}$ is the ``step count distance'' of $w_j^{r,\delta}$ to $\partial \qhull{x}{\oqdist(x,y)-r}$ because it counts the number of steps that the marked point has taken after $m$ $\SLE_6$ necklaces have been added in the reverse exploration.

For each $s$, we let $\ul{s} = \lfloor \delta^{-1} s \rfloor \delta$.  Assume that $u > 0$ is fixed, let $\epsilon > 0$, and $\tau_\epsilon^{r,\delta} = u \wedge \inf\{s \geq 0 : X_s^{r,\delta} = \epsilon\}$.  Note that $X_s^{r,\delta}$ is a non-negative {\cadlag} process with only upward jumps.  Moreover, it is not difficult to see that the law of $X^{r,\delta}|_{[0,\tau_\epsilon^{r,\delta}]}$ converges in total variation as $\delta \to 0$ to that of a $3/2$-stable L\'evy process with only upward jumps run up to the corresponding time.  It is therefore easy to see that (in probability)
\[ \left| \int_0^{\tau_\epsilon^{r,\delta}} \frac{1}{X_{\ul{s}}^{r,\delta}} ds - \int_0^{\tau_\epsilon^{r,\delta}} \frac{1}{X_s^{r,\delta}} ds \right| \to 0 \quad\text{as}\quad \delta \to 0.\]
Let $\CF_s^{r,\delta}$ be the filtration generated by $X_s^{r,\delta}$ and recall from~\eqref{eqn::top_length} that $T_j^{r,\delta}$ is the quantum boundary length of the top of $\CN_j^{r,\delta}$.  We assume that $\delta > 0$ is sufficiently small so that $\delta^{2/3} \leq \epsilon$.  Let $Q_j^{r,\delta} = \{ T_j^{r,\delta} < X_{(j-1)\delta}^{r,\delta}\}$.  On the event that $j \delta \leq \tau_\epsilon^{r,\delta}$ so that $X_{j\delta}^{r,\delta} \geq \epsilon$, using that $\one_{Q_{j+1}^{r,\delta}} = 1 - \one_{(Q_{j+1}^{r,\delta})^c}$, we have for a constant $c > 0$ that
\begin{equation}
\label{eqn::cond_prob_hit}
\p[ A_{j+1}^{r,\delta}, Q_{j+1}^{r,\delta} \giv \CF_{j\delta}^{r,\delta}] = \E\!\left[ \frac{T_{j+1}^{r,\delta}}{X_{j\delta}^{r,\delta}} \one_{Q_{j+1}^{r,\delta}} \giv \CF_{j\delta}^{r,\delta} \right] = \frac{c \delta^{2/3}}{X_{j\delta}^{r,\delta}} - \p[ (Q_{j+1}^{r,\delta})^c \giv \CF_{j \delta}^{r,\delta}].
\end{equation}
(The constant $c$ appearing in~\eqref{eqn::cond_prob_hit} is the value of $c$ that we take in the statement of the lemma.)  Let $n = \lfloor \delta^{-1} \tau_\epsilon^{r,\delta} / c \rfloor$ and let $G_\epsilon^{r,\delta}$ be the event that $T_j^{r,\delta} < X_{j\delta}^{r,\delta}$ for all $j$ such that $\tau_\epsilon^{r,\delta} \geq j\delta$.  By Lemma~\ref{lem::stable_maximum}, we have that $\p[ G_\epsilon^{r,\delta}] \to 1$ as $\delta \to 0$ with $r,\epsilon$ fixed.  Consequently, it follows that
\[ \delta^{1/3} \sum_{j=1}^n I_j^{r,\delta} \one_{(Q_j^{r,\delta})^c} \to 0\]
in probability as $\delta \to 0$ with $r,\epsilon$ fixed.  Using that $\p[ (Q_{j+1}^{r,\delta})^c \giv \CF_{j \delta}^{r,\delta}] \to 0$ as $\delta \to 0$ faster than any power of $\delta$ (Lemma~\ref{lem::stable_maximum}) on the event that $j \delta \leq \tau_\epsilon^{r,\delta}$, and using the notation $o(1)$ to indicate terms which tend to $0$ as $\delta \to 0$ with $r,\epsilon$ fixed, we have that
\begin{align}
  &  \E\!\left[ \left( c^{-1} \delta^{1/3} \sum_{j=1}^n I_j^{r,\delta} \one_{Q_j^{r,\delta}} - \int_0^{\tau_\epsilon^{r,\delta}} \frac{1}{X_{\ul{s}}^{r,\delta}} ds \right)^2 \right]  \notag\\
=&  \delta^{2/3} \E\!\left[ \sum_{j,k=1}^n ( c^{-1} I_j^{r,\delta} \one_{Q_j^{r,\delta}} - \delta^{2/3} (X_{(j-1) \delta}^{r,\delta})^{-1}) (c^{-1} I_k^{r,\delta} \one_{Q_k^{r,\delta}} - \delta^{2/3} (X_{(k-1) \delta}^{r,\delta})^{-1} ) \right] +o(1)  \notag\\
=&  \delta^{2/3} \E\!\left[ \sum_{j=1}^n (c^{-1} I_j^{r,\delta} \one_{Q_j^{r,\delta}} - \delta^{2/3} (X_{(j-1) \delta}^{r,\delta})^{-1})^2 \right] + o(1)  \quad\text{(by~\eqref{eqn::cond_prob_hit})} \notag\\
=& \delta^{2/3}  \E\!\left[ \sum_{j=1}^n \left( c^{-2} I_j^{r,\delta} \one_{Q_j^{r,\delta}} + \delta^{4/3} (X_{(j-1) \delta}^{r,\delta})^{-2} - 2 c^{-1} \delta^{2/3} I_j^{r,\delta} (X_{(j-1) \delta}^{r,\delta})^{-1} \one_{Q_j^{r,\delta}} \right) \right] + o(1) \notag\\
=& \delta^{2/3} \E\!\left[ \sum_{j=1}^n \left( c^{-2} I_j^{r,\delta} \one_{Q_j^{r,\delta}}  - \delta^{4/3} (X_{(j-1) \delta}^{r,\delta})^{-2} \right) \right] + o(1) \quad\text{(by~\eqref{eqn::cond_prob_hit})} \notag\\
=& \delta^{2/3} \E\!\left[ \sum_{j=1}^n c^{-2} I_j^{r,\delta} \one_{Q_j^{r,\delta}} \right] - \delta^2 \E\!\left[ \sum_{j=1}^n (X_{(j-1)\delta}^{r,\delta})^{-2} \right] + o(1). \label{eqn::qnt_two_sums}
\end{align}
For the first summand in~\eqref{eqn::qnt_two_sums} we have that
\begin{align}
 c^{-2} \delta^{2/3} \E\!\left[ \sum_{j=1}^n I_j^{r,\delta} \one_{Q_j^{r,\delta}} \right]
&= c^{-1} \delta^{4/3} \E\!\left[ \sum_{j=1}^n (X_{(j-1) \delta}^{r,\delta})^{-1} \right] + o(1) \quad\text{(by~\eqref{eqn::cond_prob_hit})} \notag\\
&\leq c^{-1} u \epsilon^{-1} \delta^{1/3} + o(1) \to 0 \quad\text{as}\quad \delta \to 0. \label{eqn::qnt_two_sums1}
\end{align}
To bound the second summand in~\eqref{eqn::qnt_two_sums}, we can use the deterministic bound
\begin{equation}
 \delta^2 \sum_{j=1}^n (X_{j\delta}^{r,\delta})^{-2} \leq u \epsilon^{-2} \delta \to 0 \quad\text{as}\quad \delta \to 0. \label{eqn::qnt_two_sums2}
\end{equation}
Combining~\eqref{eqn::qnt_two_sums1} and~\eqref{eqn::qnt_two_sums2} implies that~\eqref{eqn::qnt_two_sums} tends to $0$ as $\delta \to 0$.  This completes the proof because the boundary of the time-reversal of the $\delta$-approximation to the reverse metric exploration at quantum distance time $r+\int_0^u (X_s^{r,\delta})^{-1} ds$ converges as $\delta \to 0$ to the boundary of the radius $\oqdist(x,y)-(r + \int_0^u (X_s^{r,\delta})^{-1} ds)$ ball.
\end{proof}

\subsubsection{Conditional law of necklace given top glued to marked point}
\label{subsubsec::law_of_necklace_given_glued}

Conditioning on the event $A_j^{r,\delta}$ that $w_j^{r,\delta} \neq w_{j-1}^{r,\delta}$ introduces a bias into the law of $\CN_j^{r,\delta}$ because necklaces with longer top boundary lengths are more likely to be glued to a given marked boundary point.  As we will see in the following lemma, this bias corresponds to weighting the law of $\CN_j^{r,\delta}$ by the quantum boundary length $T_j^{r,\delta}$ of its top.

\begin{lemma}
\label{lem::good_necklace_law}
We have that:
\begin{enumerate}[(i)]
\item The conditional law of $\CN_j^{r,\delta}$ given $A_j^{r,\delta}$ and $X_{(j-1)\delta}^{r,\delta}$ on $\{T_j^{r,\delta} < X_{(j-1)\delta}^{r,\delta}\}$ is that of an $\SLE_6$ necklace weighted by $T_j^{r,\delta}$.
\item Given $A_j^{r,\delta}$, $\{T_j^{r,\delta} < X_{(j-1)\delta}^{r,\delta}\}$, and $X_{(j-1) \delta}^{r,\delta}$ we have that $w_{j-1}^{r,\delta}$ is distributed uniformly from the quantum boundary measure on the boundary of the top of $\CN_j^{r,\delta}$.
\end{enumerate}
\end{lemma}
\begin{proof}
The first assertion of the lemma is a standard sort of Bayes' rule style calculation.  Fix an event $\CA$ such that $\p[ \CN_j^{r,\delta} \in \CA \giv X_{(j-1)\delta}^{r,\delta}] > 0$ and $\CA \subseteq \{ T_j^{r,\delta} < X_{(j-1)\delta}^{r,\delta} \}$.  We have that
\begin{align}
     \p[ \CN_j^{r,\delta} \in \CA \giv A_j^{r,\delta}, X_{(j-1)\delta}^{r,\delta}]
&= \frac{\p[ A_j^{r,\delta} \giv \CN_j^{r,\delta} \in \CA, X_{(j-1)\delta}^{r,\delta}]}{\p[A_j^{r,\delta} \giv X_{(j-1)\delta}^{r,\delta}]} \p[ \CN_j^{r,\delta} \in \CA \giv X_{(j-1)\delta}^{r,\delta}]. \label{eqn::necklace_radon_form}
\end{align} 
We can read off from~\eqref{eqn::necklace_radon_form} the Radon-Nikodym derivative of the law of $\CN_j^{r,\delta}$ given $A_j^{r,\delta}$, $X_{(j-1)\delta}^{r,\delta}$ on the event that $\{T_j^{r,\delta} < X_{(j-1)\delta}^{r,\delta}\}$ with respect to the unconditioned law of $\CN_j^{r,\delta}$.  Fix $\epsilon,a,b > 0$.  Assume that on $\CA$ we have that $T_j^{r,\delta} \in [a,a+\epsilon]$ where $a+\epsilon < X_{(j-1)\delta}^{r,\delta}$.  Then we have that
\begin{equation}
\label{eqn::conditioned_on_bl}
\frac{a}{X_{(j-1)\delta}^{r,\delta}} \leq \p[ A_j^{r,\delta} \giv \CN_j^{r,\delta} \in \CA, X_{(j-1)\delta}^{r,\delta}] \leq \frac{a+\epsilon}{X_{(j-1)\delta}^{r,\delta}}.
\end{equation}
The first assertion follows by combining~\eqref{eqn::necklace_radon_form} and~\eqref{eqn::conditioned_on_bl} and sending $\epsilon \to 0$.

The second assertion of the lemma is obvious from the construction.
\end{proof}

\subsubsection{Comparison of explored surface to a quantum disk}
\label{subsubsec::comparison_to_disk}

In order to make use of Proposition~\ref{prop::base_to_tip_expectation_bound} and Proposition~\ref{prop::left_right_shift_moment_bound} in the proof of Proposition~\ref{prop::subsequentially_limiting_paths_exist} given just below we will need to make a comparison between:
\begin{itemize}
\item the quantum surface which arises when running the time-reversal to the $\delta$-approximation to $\QLE(8/3,0)$ in the setting of a $\sqrt{8/3}$-quantum sphere and
\item a $\sqrt{8/3}$-quantum wedge.
\end{itemize}
We will accomplish this with a cutting/gluing argument.

Conditional on the boundary length of $\partial \unexplored_{j \delta}^{r,\delta}$ being equal to $L$, we note that the law of the surface parameterized by $\unexplored_{j \delta}^{r,\delta}$ can be sampled from as follows.  Let $e$ be a $3/2$-stable L\'evy excursion with only upward jumps conditioned to have maximum at least $L$.  Let $t_L$ be the last time that $e$ hits $L$ and let $\wt{e}_t = e_{t_L-t}$.  That is, $\wt{e}$ is the time-reversal of $e$ starting from when $e$ last hits $L$.  We consider $\wt{e}$ conditioned on the following event.  Let $S$ be the amount of L\'evy process time elapsed by $\wt{e}$ when it has been run for $j \delta$ units of L\'evy process time and then $r$ units of CSBP time (i.e., after performing a time-change as in~\eqref{eqn::levy_to_csbp}).  Then we condition $\wt{e}$ on the event that it first hits $0$ in the interval between $\delta \lfloor S/\delta \rfloor$ and $ \delta \lceil S/\delta \rceil$.  The surface parameterized by $\unexplored_{j\delta}$ is then constructed by associating with each jump of $\wt{e}$ a conditionally independent quantum disk whose boundary length is equal to the size of the jump.  In the first $\delta \lfloor S/\delta \rfloor$ units of quantum natural time, each chunk of surface which corresponds to $\delta$-units of quantum natural time corresponds to an $\SLE_6$ necklace and the necklaces are glued together by gluing the tip of one necklace onto the previous necklaces at a uniformly random point chosen from the quantum boundary measure.  The final segment of L\'evy process time corresponds to an $\SLE_6$ necklace whose quantum natural time length is smaller than $\delta$.

We can make a comparison between the law of the surface parameterized by $\unexplored_{j \delta}^{r,\delta}$ and that of a quantum disk weighted by its quantum area with quantum boundary length equal to $L$ decorated by the $\delta$-approximation to $\QLE(8/3,0)$ as follows.  First, we recall that this latter law can be encoded using the time-reversal of a $3/2$-stable L\'evy excursion with maximum at least $L$ starting from where it last hits $L$ and then run until it first hits $0$.

Let $\rho_\delta(s,t,L)$ be the probability of the following event for the time-reversal $\wt{e}$ of a $3/2$-stable L\'evy excursion $e$ starting from when it last hits $L$.  Let $S$ be the amount of L\'evy process time elapsed after $\wt{e}$ has evolved for $s$ units of CSBP time and let $R$ be the time at which $\wt{e}$ first hits $0$.  Then $\rho_\delta(s,t,L)$ is the probability of the event that $R+t$ is in the interval between $\delta \lfloor (S+t)/\delta \rfloor$ and $\delta \lceil (S+t)/\delta \rceil$.  We also let $\rho_{j,\delta}(r,L)$ denote the probability of the following event.  Let $S$ be the amount of L\'evy process time elapsed by $\wt{e}$ when it has been run for $j \delta$ units of L\'evy process time and then $r$ units of CSBP time.  Then $\rho_{j,\delta}(r,L)$ is the probability that $\wt{e}$ first hits $0$ in the time interval between $\delta \lfloor S/\delta \rfloor$ and $\delta \lceil S/\delta \rceil$.  Then the Radon-Nikodym derivative between the law of the process which encodes the $\delta$-approximation to (forward) $\QLE(8/3,0)$ on $\unexplored_{j \delta}^{r,\delta}$ up until $r-\zeta$ units of quantum distance time after $j\delta$ units of quantum natural time with respect to the law of a quantum disk weighted by its quantum area decorated by the $\delta$-approximation to (forward) $\QLE(8/3,0)$ up until the same time, by a Bayes' rule calculation, is equal to
\begin{equation}
\label{eqn::quantum_disk_rn_form}
\frac{\rho_\delta(\zeta,T,X)}{\rho_{j,\delta}(r,L)}.
\end{equation}
Here, $X$ is the boundary length of the surface at this time and $T$ is the amount of L\'evy process time which has elapsed.  For $r$, $\zeta$, and $L$ fixed, it is easy to see that this Radon-Nikodym derivative is a bounded, continuous function of $X$ and $T$ (and the bound only depends on $r$, $\zeta$, $L$).  Moreover, if $C > 1$ and $r$, $\zeta$ are fixed, the bound is also uniform in $L \in [C^{-1},C]$.

We consider three laws on disk-homeomorphic growth-process-decorated quantum surfaces with fixed quantum boundary length $L$:

\begin{enumerate}
\item[Law 1:] A quantum disk weighted by its quantum area with quantum boundary length equal to $L$ (i.e., $\Mdonel$) decorated by the $\delta$-approximation to $\QLE(8/3,0)$ run for $j\delta$ units of quantum natural time then for $r-\zeta$ units of quantum distance time conditioned not to hit the uniformly random marked point.
\item[Law 2:] A quantum disk weighted by its quantum area with quantum boundary length equal to $L$ (i.e., $\Mdonel$) decorated by the $\delta$-approximation to $\QLE(8/3,0)$ for $j\delta$ units of quantum natural time and then for $r-\zeta$ units of quantum distance time weighted by the Radon-Nikodym derivative in~\eqref{eqn::quantum_disk_rn_form} (with the value of $\zeta$ fixed).
\item[Law 3:] The growth-process-decorated quantum surface corresponding to exploring $\unexplored_{j\delta}^{r,\delta}$ with the $\delta$-approximation to $\QLE(8/3,0)$ conditioned on the boundary length of $\partial \unexplored_{j \delta}^{r,\delta}$ being equal to $L$.
\end{enumerate}

Then we know that:
\begin{itemize}
\item We can transform from Law 3 to Law 2 by cutting out the last $\zeta$ units of quantum distance time of the $\QLE(8/3,0)$ and then gluing in a quantum disk weighted by quantum area decorated by a uniformly random marked point.  The continuation of the growth process is given by the $\delta$-approximation to $\QLE(8/3,0)$.
\item We can transform from Law 2 to Law 1 by unweighting by the Radon-Nikodym derivative in~\eqref{eqn::quantum_disk_rn_form}.
\end{itemize}

As we will see momentarily, Law 1 is the one which is easiest to make the comparison with a $\sqrt{8/3}$-quantum wedge (hence apply Proposition~\ref{prop::base_to_tip_expectation_bound} and Proposition~\ref{prop::left_right_shift_moment_bound}).  This is because when we parameterize a quantum disk by $\strip$, the local behavior of the field near the marked points at $\pm \infty$ is the same as that of a $\sqrt{8/3}$-quantum wedge near its origin (i.e., the finite marked point).  On the other hand, to prove Proposition~\ref{prop::subsequentially_limiting_paths_exist} we will need to work with Law~3.

\subsubsection{Comparison of explored surface near $w_j^{r,\delta}$ to a $\sqrt{8/3}$-quantum wedge}
\label{subsubsec::comparison_to_wedge}

We are now going to introduce events on which we will truncate when making the comparison to a $\sqrt{8/3}$-quantum wedge.  In what follows, we will indicate a quantity associated with Law 1 (resp.\ Law 2) using the notation $\dot{a}$ (resp.\ $\ddot{a}$).  In other words, one (resp.\ two) dots indicates Law 1 (resp.\ Law 2).  We will indicate quantities associated with Law 3 in a manner which is consistent with the notation from the preceding text.

Suppose that $(\D,\ddot{h})$ is a quantum surface with Law 2 described just above in the case that $j=0$.  We assume that we have taken the embedding of the surface so that the marked point is equal to~$0$.  Fix a function $\phi \in C_0^\infty(\strip)$ with $\phi \geq 0$ and $\int \phi(z) dz = 1$.  For each $r >0$, $M,C > 1$, and $\zeta \in (0,r)$ we let $\Psi_M^\phi$ be the set of conformal transformations $\psi \colon \D \to D$ where $D \subseteq \strip$ contains $\supp(\phi)$ with $|\psi'(z)| \in [M^{-1},M]$ for all $z \in \psi^{-1}(\supp(\phi))$ and let $G_{\zeta,M,C}^{r,\delta}$ be the event that:
\[ \inf\{  (\ddot{h} - Q\log|\psi'|,|\psi'|^2 \phi \circ \psi) : \psi \in \Psi_M^\phi \} \geq -C.\]

\begin{lemma}
\label{lem::g_happens}
For $r,M$ fixed, the probability under the law considered just above for which $G_{\zeta,M,C}^{r,\delta}$ occurs tends to $1$ as $C \to \infty$ uniformly in $\delta$.
\end{lemma}
\begin{proof}
This follows from the argument given in \cite[Proposition~10.18 and Proposition~10.19]{dms2014mating} as in \cite[Section~10]{dms2014mating}.
\end{proof}

Let $(\strip,h_j^{r,\delta}, \Gamma^{r,\delta,j})$ be the growth-process decorated surface with Law 3 and we let $(\strip,\dot{h}_j^{r,\delta},\dot{\Gamma}^{r,\delta,j})$ and $(\strip,\ddot{h}_j^{r,\delta},\ddot{\Gamma}^{r,\delta,j})$ be growth-process decorated surfaces with Law 1 and Law 2, respectively.  We assume that $(\strip,h_j^{r,\delta},\Gamma^{r,\delta,j})$ and $(\strip,\ddot{h}_j^{r,\delta},\ddot{\Gamma}^{r,\delta,j})$ have been coupled together so that the surfaces parameterized by $\Gamma^{r,\delta,j}$ and $\ddot{\Gamma}^{r,\delta,j}$ agree except for the last $\zeta$ units of quantum distance time.  In other words, it is possible to transform from the former to the latter using the cutting/gluing operation described at the end of Section~\ref{subsubsec::comparison_to_disk} just above.  We take the embedding for $(\strip,\ddot{h}_j^{r,\delta})$ into $\strip$ by taking the tip of the $\SLE_6$ necklace just glued in (i.e., at time $j$) to be $-\infty$ and we then pick another point uniformly from the quantum boundary measure in the complement of the interval with quantum length $(2C)^{-1}$ centered at the tip and send this point to $+\infty$.  We take the horizontal translation so that the target point $\ddot{z}_j^{r,\delta}$ of $\ddot{\Gamma}^{r,\delta,j}$ has real part equal to $0$.

For each $j$, we let $\ddot{w}_j^{r,\delta}$ (resp.\ $\ddot{\CN}_j^{r,\delta}$) be the point on $\partial \ddot{\Gamma}_{j \delta}^{r,\delta}$ (resp.\ $\SLE_6$ necklace) which corresponds to $w_j^{r,\delta}$ (resp.\ $\CN_j^{r,\delta}$).  Under the coupling that we have constructed, we have that $\ddot{\CN}_j^{r,\delta}$ is equal to $\CN_j^{r,\delta}$ (as path decorated quantum surfaces).

We also let $F_{j,M,C}^{r,\delta}$ be the event that:
\begin{enumerate}
\item The quantum boundary length of $(\strip, \ddot{h}_j^{r,\delta})$ is in $[C^{-1},C]$.
\item The quantum area of $(\strip,\ddot{h}_j^{r,\delta})$ is in $[C^{-1},C]$.
\item The Euclidean distance between $\partial \strip \cup \ddot{\Gamma}_{j\delta}^{r,\delta,j}$ and the support of $\phi$ is at least $M^{-1/2}$.  The same is also true with $\ddot{z}_j^{r,\delta}$ in place of $\supp(\phi)$.
\end{enumerate}

Note that the third condition of the definition of $F_{j,M,C}^{r,\delta}$ implies that the following is true.  Let $\psi$ be the unique conformal transformation $\D \to \strip \setminus \ddot{\Gamma}_{j\delta}^{r,\delta,j}$ with $\psi(0) = \ddot{z}_j^{r,\delta}$ and $\psi'(0) > 0$.  Then, by the distortion theorem, there exists a constant $c_0 > 0$ such that $|\psi'(z)| \geq c_0 M^{-1/2}$ for all $z \in \psi^{-1}(\supp(\phi))$.  In particular, $|\psi'(z)| \in [M^{-1},M]$ for all $z \in \psi^{-1}(\supp(\phi))$ provided $M$ is at least some universal constant.  Thus if we assume that we are working on the event $G_{\zeta,M,C}^{r,\delta}$ so that $\psi \in \Psi_M^\phi$, by the change of coordinates formula for quantum surfaces we have that $(\ddot{h}_j^{r,\delta},\phi) = (\ddot{h} - Q\log|\psi'|, |\psi'|^2 \phi \circ \psi) \geq -C$.

\begin{lemma}
\label{eqn::rn_bound}
For each $C > 1$ and $\zeta > 0$ there exists a constant $K > 0$ such that on the event that the quantum boundary length of $(\strip,\ddot{h}_j^{r,\delta})$ is in $[C^{-1},C]$ we have that the Radon-Nikodym derivative between the law of $(\strip,\ddot{h}_j^{r,\delta},\ddot{\Gamma}^{r,\delta,j})$ and $(\strip,\dot{h}_j^{r,\delta},\dot{\Gamma}^{r,\delta,j})$ is at most $K$.
\end{lemma}
\begin{proof}
This follows by combining the observations made just after~\eqref{eqn::quantum_disk_rn_form}.
\end{proof}

\begin{lemma}
\label{lem::bessel_phi_maximum}
Suppose that $(\strip,\wh{h})$ has the Bessel quantum disk law conditioned on the event that $\sup_{r \in \R} (\wh{h},\phi(\cdot+r)) \geq 0$ and let $r^*$ be the value of $r \in \R$ at which the supremum is achieved.  Let $Y^*$ be equal to the value of the projection of $\wh{h}$ onto $\CH_1(\strip)$ at $r^*$.  There exist constants $c_0,c_1 > 0$ such that
\[ \p[ |Y^* - (\wh{h},\phi(\cdot+r^*))| \geq u] \leq c_0 e^{-c_1 u^2} \quad\text{for all}\quad u \geq 0.\]
\end{lemma}
\begin{proof}
This is immediate from the construction.
\end{proof}

\begin{lemma}
\label{lem::strip_projection}
We assume that we are working on $G_{\zeta,M,C}^{r,\delta} \cap F_{j,M,C}^{r,\delta}$.  There exist constants $c_0,c_1 > 0$ depending only on $C,M,\zeta$ such that the following is true.  The probability that the supremum of the projection of $\ddot{h}_j^{r,\delta}$ onto $\CH_1(\strip)$ is smaller than $u$ is at most $c_0 e^{-c_1 u^2}$ for all $u \in \R_-$.
\end{lemma}
\begin{proof}
Let $\psi \colon \D \to \strip \setminus \ddot{\Gamma}_{j\delta}^{r,\delta,j}$ be the unique conformal map with $\psi(0) = \ddot{z}_j^{r,\delta}$ and $\psi'(0) > 0$.  As explained above, it follows from the definition of the event $G_{\zeta,M,C}^{r,\delta} \cap F_{j,M,C}^{r,\delta}$ that $(\ddot{h} - Q\log|\psi'|,|\psi'|^2 \phi \circ \psi) \geq -C$.  Applying the change of coordinates rule for quantum surfaces, this implies that $(\ddot{h}_j^{r,\delta},\phi) \geq -C$.  Note that on $G_{\zeta,M,C}^{r,\delta} \cap F_{j,M,C}^{r,\delta}$, the law of $\ddot{h}_j^{r,\delta}$ (modulo horizontal translation) is absolutely continuous with bounded Radon-Nikodym derivative with respect to the law on distributions which comes from the Bessel law conditioned on quantum disks so that $\sup_{r \in \R} (\wh{h},\phi( \cdot+r)) \geq -C$.  Consequently, the result follows by applying Lemma~\ref{lem::bessel_phi_maximum}.
\end{proof}

Assuming that $\zeta,M,C$ are fixed, we can choose $c$ sufficiently large so that with $c_0,c_1$ as in the statement of Lemma~\ref{lem::strip_projection} we have with
\begin{equation}
\label{eqn::eta_0_def}
u_0 = -c\sqrt{\log \delta^{-1}}
\end{equation}
that $c_0 e^{-c_1 u_0^2} \leq \delta^2$.  For each $j \in \N$ and $\alpha > 0$, we let $\ddot{u}_{j,\alpha,\delta}^r \in \R$ be where the projection of $\ddot{h}_j^{r,\delta}$ onto $\CH_1(\strip)$ first hits $\alpha \log \delta$; we take $\ddot{u}_{j,\alpha,\delta}^r = +\infty$ if the supremum of this projection is smaller than $\alpha \log \delta$.  We also let $H_{j,\alpha}^{r,\delta}$ be the event that
\begin{enumerate}
\item The supremum of the projection of $\ddot{h}_j$ onto $\CH_1(\strip)$ is at least $u_0$.
\item $T_j^{r,\delta} \leq \delta^{2/3-\alpha}$ where $T_j^{r,\delta}$ is the quantum length of the top of~$\CN_j^{r,\delta}$ (equivalently, of~$\ddot{\CN}_j^{r,\delta}$).
\item $\ddot{\CN}_j^{r,\delta}$ is contained in $\strip_- + \ddot{u}_{j,\alpha,\delta}^r$.
\item Let $(\strip,\acute{h}_{j-1})$ be the quantum surface which is given by re-embedding the quantum surface $(\strip,\ddot{h}_{j-1}^{r,\delta})$ so that the point on $\partial \strip$ where the tip of $\ddot{\CN}_j^{r,\delta}$ is glued to form the quantum surface $(\strip,\ddot{h}_j^{r,\delta})$ is sent to $-\infty$ (with $+\infty$ fixed and the horizontal translation left unspecified).  Let $\acute{u}_{j-1,\alpha,\delta}^r$ be where the projection of $\acute{h}_{j-1}$ onto $\CH_1(\strip)$ first hits $\alpha \log \delta$.  Then the interval of the boundary of $(\strip, \acute{h}_{j-1})$ to where $\ddot{\CN}_j^{r,\delta}$ gets glued to is contained in $\partial \strip_- + \acute{u}_{j-1,\alpha,\delta}^r$.
\end{enumerate}
We then let $E_{j,\zeta,M,C,\alpha}^{r,\delta} = G_{\zeta,M,C}^{r,\delta} \cap F_{j,M,C}^{r,\delta} \cap H_{j,\alpha}^{r,\delta}$ and $E_{\zeta,M,C,\alpha}^{n,r,\delta} = \cap_{j=1}^n E_{j,\zeta,M,C,\alpha}^{r,\delta}$.  

We will now combine the estimates established earlier to get that it is possible to adjust the parameters in the definition of $E_{\zeta,M,C,\alpha}^{n,r,\delta}$ so that it occurs with probability as close to~$1$ as we like.

\begin{lemma}
\label{lem::good_event_happens}
For every $\epsilon,a_0 > 0$ there exists $M, C> 1$, $\alpha,\zeta,\delta_0 > 0$, and $\phi \in C_0^\infty(\strip)$ such that $\delta \in (0,\delta_0)$ implies that $\p[ (E_{\zeta,M,C,\alpha}^{n,r,\delta})^c] \leq \epsilon$ where $n = \lfloor a_0 \delta^{-1} \rfloor$.
\end{lemma}
\begin{proof}
We explained just after the definition of $G_{\zeta,M,C}^{r,\delta}$ why there exists $M,C > 1$ and $\zeta > 0$ such that $\p[(G_{\zeta,M,C}^{r,\delta})^c] \leq \epsilon$.  Therefore it is left to explain why we have the corresponding property for $\cap_{j=1}^n F_{j,M,C}^{r,\delta}$ and $\cap_{j=1}^n H_{j,\alpha}^{r,\delta}$.

From the definition of $F_{j,M,C}^{r,\delta}$, elementary distortion estimates for conformal maps, and elementary estimates for L\'evy processes, it is easy to see that by choosing $M,C > 1$ sufficiently large and by making the support of $\phi$ sufficiently small, we have that $\p[ (\cap_{j=1}^n F_{j,M,C}^{r,\delta})^c] \leq \epsilon$.

It is left to explain why $\p[ (\cap_{j=1}^n H_{j,\alpha}^{r,\delta})^c ] \leq \epsilon$.  The first two parts of the definition follow from Lemma~\ref{lem::stable_maximum} and Lemma~\ref{lem::strip_projection}.  The second two parts of the definition respectively follow from~\eqref{eqn::path_displacement} of Proposition~\ref{prop::base_to_tip_expectation_bound} and~\eqref{eqn::shift_probability} of Proposition~\ref{prop::left_right_shift_moment_bound}.
\end{proof}

\subsubsection{Moment bounds}
\label{subsubsec::moment_bounds}

For each $j$, we let $D_j^{r,\delta}$ (resp.\ $\ddot{D}_j^{r,\delta}$) denote the quantum distance between $w_j^{r,\delta}$ (resp.\ $\ddot{w}_j^{r,\delta}$) and $w_{j+1}^{r,\delta}$ (resp.\ $\ddot{w}_{j+1}^{r,\delta}$) with respect to the internal metric of $(\strip,h_j^{r,\delta})$ (resp.\ $(\strip,\ddot{h}_j^{r,\delta})$).  We let $S_j^{r,\delta}$ be the event that the shortest path from $w_j^{r,\delta}$ to $w_{j+1}^{r,\delta}$ does not hit the part of the surface that we cut out in order to transform from $(\strip,h_j^{r,\delta})$ to $(\strip,\ddot{h}_j^{r,\delta})$.  On $S_j^{r,\delta}$, we have that $D_j^{r,\delta} = \ddot{D}_j^{r,\delta}$.  We note that this is the case for all $1 \leq j \leq n$ on $E_{\zeta,M,C,\alpha}^{n,r,\delta}$.

Fix $a_0 > 0$ and let $n = \lfloor a_0 \delta^{-1} \rfloor$ as in the statement of Lemma~\ref{lem::good_event_happens}.    Suppose that $Q$ is any event.  Using that $D_j^{r,\delta} = 0$ on $(A_j^{r,\delta})^c$ in the last step, we have that
\begin{align}
\E\!\left[ \left( \sum_{j=1}^n  D_j^{r,\delta} \right) \one_{E_{\zeta,M,C,\alpha}^{n,r,\delta}} \one_Q \right]
&= \E\!\left[ \left( \sum_{j=1}^n  \ddot{D}_j^{r,\delta} \right) \one_{E_{\zeta,M,C,\alpha}^{n,r,\delta}} \one_Q \right] \notag\\
&\leq \sum_{j=1}^n \E\!\left[ \ddot{D}_j^{r,\delta} \one_{F_{j,M,C}^{r,\delta} \cap H_{j,\alpha}^{r,\delta}} \one_Q \right] \notag\\
&= \sum_{j=1}^n \E [ \ddot{D}_j^{r,\delta} \one_{F_{j,M,C}^{r,\delta} \cap H_{j,\alpha}^{r,\delta}} \one_Q \giv A_j^{r,\delta} ] \p[ A_j^{r,\delta}]. \label{eqn::distance_sum_bound}
\end{align}

We next aim to bound the right hand side of~\eqref{eqn::distance_sum_bound}.

\begin{lemma}
\label{lem::quantum_disk_moment_bounds}
There exist constants $c_0,\sigma > 0$ such that
\begin{equation}
\label{eqn::disk_moment_bound}
\E[ \ddot{D}_j^{r,\delta}  \one_{F_{j,M,C}^{r,\delta} \cap H_{j,\alpha}^{r,\delta}} \one_Q \giv A_j^{r,\delta} ] \leq c_0 \delta^{1/3} \p[ Q]^\sigma.
\end{equation}
\end{lemma}
\begin{proof}
We let $\acute{D}_j^{r,\delta}$ be the quantum distance (with respect to the internal metric of $(\strip,\ddot{h}_j^{r,\delta})$) between the tip of $\ddot{\CN}_j^{r,\delta}$ and a point which is chosen uniformly at random from the quantum measure on the top of $\ddot{\CN}_j^{r,\delta}$.  Conditionally on $A_j^{r,\delta}$, we have by Lemma~\ref{lem::good_necklace_law} that $\acute{D}_j^{r,\delta} \stackrel{d}{=} \ddot{D}_j^{r,\delta}$.  Let $p > 1$ be such that both Proposition~\ref{prop::base_to_tip_expectation_bound} and Proposition~\ref{prop::left_right_shift_moment_bound} apply and let $q \in (1,\infty)$ be conjugate to $p$, i.e., $p^{-1} + q^{-1} = 1$.  We begin by noting that:
\begin{align*}
\E[ \ddot{D}_j^{r,\delta}  \one_{F_{j,M,C}^{r,\delta} \cap H_{j,\alpha}^{r,\delta}} \giv A_j^{r,\delta} ]
&= \E[ \acute{D}_j^{r,\delta}  \one_{F_{j,M,C}^{r,\delta} \cap H_{j,\alpha}^{r,\delta}} \giv A_j^{r,\delta} ]\\
&\leq c_0 \E[ \acute{D}_j^{r,\delta} T_j^{r,\delta} \one_{F_{j,M,C}^{r,\delta} \cap H_{j,\alpha}^{r,\delta}}] \quad\text{(by Lemma~\ref{lem::good_necklace_law})}\\
&\leq c_0 \E[ (\acute{D}_j^{r,\delta})^p \one_{F_{j,M,C}^{r,\delta} \cap H_{j,\alpha}^{r,\delta}} ]^{1/p} \E[ (T_j^{r,\delta})^q ]^{1/q} \quad\text{(H\"older's inequality)}\\
&\leq c_1 \E[ (\acute{D}_j^{r,\delta})^p \one_{F_{j,M,C}^{r,\delta} \cap H_{j,\alpha}^{r,\delta}} ]^{1/p} \quad\text{(by Lemma~\ref{lem::stable_maximum})}.
\end{align*}
We note that the constant $c_1$ depends on $q$.  By Lemma~\ref{eqn::rn_bound}, we know that there exists a constant $K > 0$ such that
\[ \E[ (\acute{D}_j^{r,\delta})^p \one_{F_{j,M,C}^{r,\delta} \cap H_{j,\alpha}^{r,\delta}} ]^{1/p} \leq K \dot{\E}[ (\acute{D}_j^{r,\delta})^p \one_{H_{j,\alpha}^{r,\delta}}]^{1/p}\]
where $\dot{\E}$ denotes the expectation under the law $(\strip,\dot{h}_j^{r,\delta},\dot{\Gamma}^{r,\delta,j})$.  We let $\acute{D}_j^{1,r,\delta}$ denote the quantum distance between the base and the tip of $\ddot{\CN}_j^{r,\delta}$ and we let $\acute{D}_j^{2,r,\delta}$ denote the quantum distance between the tip of $\ddot{\CN}_j^{r,\delta}$ and the uniformly random point $\ddot{w}_j^{r,\delta}$ on the top of $\ddot{\CN}_j^{r,\delta}$ in the surface which arises after cutting out $\ddot{\CN}_j^{r,\delta}$.  We will establish~\eqref{eqn::disk_moment_bound} by bounding the $p$th moments of $\acute{D}_j^{1,r,\delta}$ and $\acute{D}_j^{2,r,\delta}$.

By the definition of the event $H_{j,\alpha}^{r,\delta}$, we have that $\re(\ddot{w}_j^{r,\delta}) \leq \ddot{u}_{j,\alpha,\delta}^r$.  Note that the law of the field $\dot{h}_j^{r,\delta}$ in $\strip_- + \dot{u}_{j,\alpha,\delta}^r$ is absolutely continuous with bounded Radon-Nikodym derivative to the law of a $\sqrt{8/3}$-quantum wedge with the usual embedding into $\strip$ restricted to the part of $\strip$ up to where the projection of the field onto $\CH_1(\strip)$ first hits $\alpha \log \delta$.  Consequently, it follows from Proposition~\ref{prop::left_right_shift_moment_bound} that for a constant $c_2 > 0$ we have that
\begin{equation}
\label{eqn::left_right_shift_moment_bound_new}
\dot{\E}[ (\acute{D}_j^{2,r,\delta})^p  \one_{H_{j,\alpha}^{r,\delta}} ]^{1/p} \leq c_2 \delta^{1/3}.
\end{equation}
It similarly follows from Proposition~\ref{prop::base_to_tip_expectation_bound} that, by possibly increasing the value of $c_2 > 0$, we have that
\begin{equation}
\label{eqn::base_to_tip_moment_bound_new}
\dot{\E}[ (\acute{D}_j^{1,r,\delta})^p  \one_{H_{j,\alpha}^{r,\delta}} ]^{1/p} \leq c_2 \delta^{1/3}.
\end{equation}
Combining~\eqref{eqn::left_right_shift_moment_bound_new} and~\eqref{eqn::base_to_tip_moment_bound_new} implies the result.
\end{proof}

\begin{proof}[Proof of Proposition~\ref{prop::subsequentially_limiting_paths_exist}]

We take the path that we have constructed and parameterize it according to arc length with respect to the quantum distance.  Using~\eqref{eqn::distance_sum_bound} and the fact that the conditional probability of $A_j^{r,\delta}$ given that the boundary length is not too short is of order $\delta^{2/3}$ (since $T_j^{r,\delta}$ is of order $\delta^{2/3}$ and has exponential moments), it follows from Lemma~\ref{lem::good_event_happens} and Lemma~\ref{lem::quantum_disk_moment_bounds} that the length of $\eta^{r,\delta}|_{[0,T]}$ is tight as $\delta \to 0$.  Since the length of $\eta^{r,\delta}|_{[0,T]}$ is equal to its Lipschitz constant times $T$ (as we assume $\eta^{r,\delta}$ to be parameterized according to arc length), it follows that the law of $\eta^{r,\delta}|_{[0,T]}$ is in fact tight as $\delta \to 0$.  This completes the proof of the tightness of the law of $\eta^{r,\delta}|_{[0,T]}$ for each $T > 0$.

Let $\eta^r$ be any subsequential limit.  Lemma~\ref{lem::time_parameterization_converges} implies that for any fixed $t > 0$ we have that $\eta^r(t) \in \partial \qhull{x}{\oqdist(x,y)-(t+r)}$ a.s.  Therefore this holds a.s.\ for all $t \in \Q_+$ simultaneously and, combining with the continuity of $\eta^r$, we have that $\eta^r(t) \in \partial \qhull{x}{\oqdist(x,y)-(t+r)}$ for all $t > 0$ a.s.  (Note that this property does not imply that $\eta^r$ is a geodesic since it could ``spiral'' around as it approaches $x$.)

The conditional independence statement in the limit is immediate since it holds for the approximations.

The final assertion of the proposition is immediate from the argument given above.
\end{proof}

\subsubsection{Boundary lengths between second approximations of geodesics}
\label{subsubsec::boundary_length_csbp}

\begin{figure}[ht!]
\begin{center}
\includegraphics[scale=0.85,page=3,trim={0 0 7cm 0},clip]{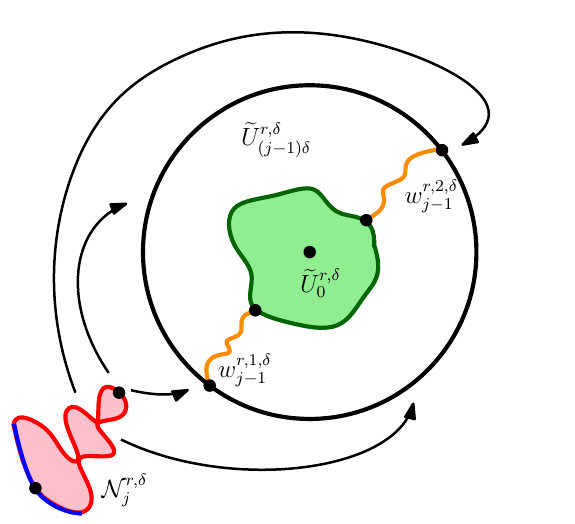}
\end{center}
\vspace{-0.03\textheight}
\caption{\label{fig::bl_independent_csbp} Illustration of the argument to prove Proposition~\ref{prop::csbps_between_paths}, which states that the boundary lengths between second approximations to geodesics evolve as independent $3/2$-stable CSBPs.  The green region parameterizes $\unexplored_0^{r,\delta}$, which we recall is equal to the reverse metric exploration at time $r$ and the disk parameterizes $\unexplored_{(j-1)\delta}^{r,\delta}$.  The orange paths are first approximations to geodesics starting from points $w_0^{r,1,\delta},w_0^{r,2,\delta}$, which are independently chosen from the quantum boundary measure on $\partial \unexplored_0^{r,\delta}$.  Shown is the $\SLE_6$ necklace $\CN_j^{r,\delta}$ which is about to be glued to the surface parameterized by $\unexplored_{(j-1)\delta}^{r,\delta}$ to form $\unexplored_{j\delta}^{r,\delta}$.  In the case that the top of $\CN_j^{r,\delta}$ is contained in the counterclockwise (resp.\ clockwise) segment from $w_{j-1}^{r,1,\delta}$ to $w_{j-1}^{r,2,\delta}$, the boundary length of the corresponding segment gets an increment of $\delta$ units of L\'evy process time.  In the case that the top of $\CN_j^{r,\delta}$ is glued to an interval which contains either $w_{j-1}^{r,1,\delta}$ or $w_{j-1}^{r,2,\delta}$, then the boundary lengths of both segments are changed.  Since the probability that this happens is of order $\delta^{2/3}$ (i.e., proportional to the quantum length of the top of $\CN_j^{r,\delta}$) and there are of order $\delta^{-1}$ necklaces overall, the number of such necklaces will be of order $\delta^{-1/3}$.  Since the change to the boundary lengths which result from such a necklace is of order $\delta^{2/3}$, the overall change to the boundary lengths which results from such necklaces will be of order $\delta^{1/3}$, hence tend to $0$ in the $\delta \to 0$ limit.}
\end{figure}

\begin{proposition}
\label{prop::csbps_between_paths}
Fix $r > 0$ and suppose that $(\CS,x,y)$ has the law $\Mstwo$ conditioned so that $\oqdist(x,y) > r$.  Suppose that $x_1,\ldots,x_k$ are picked independently from the quantum boundary measure on $\partial \qhull{x}{\oqdist(x,y)-r}$ and then reordered to be counterclockwise.  We let $\eta_1^r,\ldots,\eta_k^r$ be second approximations to geodesics starting from $x_1,\ldots,x_k$ as constructed in Proposition~\ref{prop::subsequentially_limiting_paths_exist}.  For each $1 \leq j \leq k$ and $t  \in [0,\oqdist(x,y)-r]$ we let $X_t^{r,j}$ be the quantum boundary length of the counterclockwise segment of $\partial \qhull{x}{\oqdist(x,y)-(r+t)}$  between $\eta_j^r(t)$ and $\eta_{j+1}^r(t)$ (with the convention that $\eta_{k+1}^r = \eta_1^r$).  Given $X_0^{r,1},\ldots,X_0^{r,k}$, the processes $X_t^{r,j}$ evolve as independent $3/2$-stable CSBPs with initial values $X_0^{r,j}$.
\end{proposition}
In the setting of Proposition~\ref{prop::csbps_between_paths}, the processes $X^{r,j}$ are independent of each other.  The time at which $X^{r,j}$ first hits $0$ corresponds to when the paths $\eta^{r,j}$ and $\eta^{r,j+1}$ intersect and merge into each other.  The time $\sup\{t > 0 : \max_{1 \leq j \leq k} X_t^{r,j} > 0\}$ is equal to $\oqdist(x,y) - r$.  We also note that $\Mstwo$ conditioned on $\oqdist(x,y) > r$ is a probability measure since this event has finite and positive $\Mstwo$ measure by \cite[Proposition~4.2]{qlebm}.
\begin{proof}[Proof of Proposition~\ref{prop::csbps_between_paths}]
We will prove the result in the case that $k=2$ for simplicity; the proof for general values of $k \in \N$ with $k \geq 3$ follows from the same argument.  See Figure~\ref{fig::bl_independent_csbp} for an illustration.  We will prove the result by showing that $X_t^{r,1},X_t^{r,2}$ have the property that if we reparameterize the time for each using the time change $\int_0^t X_s^{r,j} ds$, then the resulting processes evolve as independent $3/2$-stable L\'evy processes with only upward jumps and stopped at the first time that they hit $0$.  This suffices because if we invert the time change, then the Lamperti transform (recall~\eqref{eqn::csbp_to_levy}) implies that the resulting processes (i.e., we recover $X^{r,1},X^{r,2}$) are independent $3/2$-stable CSBPs.

We fix $\delta > 0$ and consider the boundary lengths between two points as in the construction of the first approximations to geodesics.  As earlier, for each $j$ we let $w_j^{r,1,\delta}, w_j^{r,2,\delta}$ be the locations of the two marked points when we have glued on the $j$th $\SLE_6$ necklace.  We let $X_t^{r,1,\delta}$, $X_t^{r,2,\delta}$ be the quantum boundary lengths between the points and assume that we have the quantum distance parameterization for the overall boundary length process $X_t^{r,\delta}$ in the $\delta$-approximation to the reverse metric exploration.  We let $\sigma_j^{r,\delta}$ be the $j$th time that the top of a necklace gets glued to one of the marked points and we let $\tau_j^{r,\delta}$ be the end time of that necklace ($\tau_j^{r,\delta}$ occurs $\delta$ units of quantum natural time after $\sigma_j^{r,\delta}$).  For each $t > 0$ we let $I_t^{r,k,\delta} = 1$ (resp.\ $I_t^{r,k,\delta} = 0$) if the starting point of the necklace being glued in at time $t$ is in the counterclockwise (resp.\ clockwise) segment between $w_j^{r,k,\delta}$ and $w_j^{r,3-k,\delta}$.  Let
\[ A_t^{r,k,\delta} = \int_0^t I_s^{r,k,\delta} X_s^{r,\delta} ds.\]
We will first argue that $A_t^{r,k,\delta} -  \int_0^t X_s^{r,k,\delta} ds \to 0$ in $L^1$ as $\delta \to 0$.

For each $s \geq 0$ we let $\ol{s}$ be the time that the necklace being glued in at time $s$ first appears in the reverse exploration.  We have that
\begin{equation}
 \label{eqn::int_diff_bound}
\begin{split}
     &\E\!\left| \int_0^t I_s^{r,k,\delta} X_s^{r,\delta} ds - \int_0^t X_s^{r,k,\delta} ds \right| 
\leq \E\!\left| \int_0^t I_s^{r,k,\delta} X_{\ol{s}}^{r,\delta} ds - \int_0^t X_{\ol{s}}^{r,k,\delta} ds \right| + \\
&\quad\quad\quad\quad\quad \int_0^t \E |X_s^{r,k,\delta} - X_{\ol{s}}^{r,k,\delta}| ds +
 \int_0^t \E| X_s^{r,\delta} - X_{\ol{s}}^{r,\delta}| ds.
\end{split}
\end{equation}

We note that the second term on the right hand side of~\eqref{eqn::int_diff_bound} tends to $0$ as $\delta \to 0$ because we have for each fixed $s \in [0,t]$ that $|X_s^{r,k,\delta} - X_{\ol{s}}^{r,k,\delta}| \to 0$ in probability as $\delta \to 0$ and there exists a constant $c > 0$ and $p > 1$ so that $\E | X_s^{r,k,\delta} - X_{\ol{s}}^{r,k,\delta}|^p \leq c$ for all $s \in [0,t]$.  The third term on the right hand side of~\eqref{eqn::int_diff_bound} tends to $0$ as $\delta \to 0$ for the same reason.

We will now argue that the first term on the right hand side of~\eqref{eqn::int_diff_bound} tends to $0$ if we take a limit as $\delta \to 0$.  To see this, we assume that $t$ is at a time at which a necklace is first being glued in and let $L$ be the number of necklaces which have been glued in by time $t$.  Let $[t_j^{r,\delta},t_{j+1}^{r,\delta})$ be the time-interval in which the $j$th necklace is being glued in.  Since the values of $I_s^{r,k,\delta}, X_{\ol{s}}^{r,\delta}, X_{\ol{s}}^{r,k,\delta}$ do not change in $[t_j^{r,\delta},t_{j+1}^{r,\delta})$, we can write the first term on the right hand side of~\eqref{eqn::int_diff_bound} as the expectation of the absolute value of
\begin{equation}
\begin{split}
\label{eqn::first_term_rep}
 \int_0^t \left( I_s^{r,k,\delta} - \frac{X_{\ol{s}}^{r,k,\delta}}{X_{\ol{s}}^{r,\delta}} \right) X_{\ol{s}}^{r,\delta} ds &= \sum_{j=0}^L \Delta_{j}^{r,\delta} X_{t_j^{r,\delta}}^{r,\delta}(t_{j+1}^{r,\delta}-t_j^{r,\delta}) \quad\text{where}\\
 \Delta_{j}^{r,\delta} &= \left( I_{t_j^{r,\delta}}^{r,k,\delta} - \frac{X_{t_j^{r,\delta}}^{r,k,\delta}}{X_{t_j^{r,\delta}}^{r,\delta}} \right). 
 \end{split}
\end{equation}
Note that $|\Delta_{t_j^{r,\delta}}^{r,\delta}| \leq 1$ for each $k$ and
\begin{align*}
\E[ I_{t_j^{r,\delta}}^{r,k,\delta} \giv X_{t_j^{r,\delta}}^{r,k,\delta} , X_{t_j^{r,\delta}}^{r,\delta} ] = \frac{X_{t_j^{r,\delta}}^{r,k,\delta}}{X_{t_j^{r,\delta}}^{r,\delta}}.
\end{align*}
Consequently, $M_n^{r,\delta} = \sum_{j=0}^n \Delta_{j}^{r,\delta} X_{t_j^{r,\delta}}^{r,\delta} (t_{j+1}^{r,\delta} - t_j^{r,\delta})$ is a martingale.  Note that we can write
\begin{align}
	X_{t_j^{r,\delta}}^{r,\delta} (t_{j+1}^{r,\delta} - t_j^{r,\delta})
&= \int_{t_j^{r,\delta}}^{t_{j+1}^{r,\delta}} X_{t_j^{r,\delta}}^{r,\delta} ds
 \leq \frac{X_{t_j^{r,\delta}}^{r,\delta}}{\inf_{s \in [t_j^{r,\delta},t_{j+1}^{r,\delta}]} X_s^{r,\delta}} \int_{t_j^{r,\delta}}^{t_{j+1}^{r,\delta}} X_s^{r,\delta} ds \notag\\
&= \delta \frac{X_{t_j^{r,\delta}}^{r,\delta}}{\inf_{s \in [t_j^{r,\delta},t_{j+1}^{r,\delta}]} X_s^{r,\delta}}. \label{eqn::time_ratio_stable}
\end{align}
Due to the definition of the times $t_j^{r,\delta}$, we note that $\inf_{s \in [t_j^{r,\delta},t_{j+1}^{r,\delta}]} X_s^{r,\delta}$ gives the infimum of a $3/2$-stable L\'evy process with only upward jumps starting from $X_{t_j^{r,\delta}}^{r,\delta}$ in a time interval of (L\'evy process time) $\delta$.  Therefore if we fix $\epsilon > 0$ and let $n_\epsilon = \min\{j \in \N : X_{t_j^{r,\delta}}^{r,\delta} < \epsilon\}$ it follows from Lemma~\ref{lem::stable_maximum} that on $j \leq n_\epsilon$ we have that the ratio on the right hand side of~\eqref{eqn::time_ratio_stable} has finite moments of all orders.  In particular, $M_{n \wedge n_\epsilon}$ is an $L^2$ martingale each of whose $O(\delta^{-1})$ increments have second moment which is $O(\delta^2)$.
It therefore follows that the first term on the right hand side of~\eqref{eqn::int_diff_bound} tends to $0$ as $\delta \to 0$.

We let $B_t^{r,j,\delta}$ be the right-continuous inverse of $A_t^{r,j,\delta}$.  For a given value of $t > 0$ and each $k$, we also let $\wt{\tau}_k^{r,j,\delta} = t \wedge A_{\tau_k^{r,\delta}}^{r,j,\delta}$ and $\wt{\sigma}_k^{r,j,\delta} = t \wedge A_{\sigma_k^{r,\delta}}^{r,j,\delta}$.  Then we note that we can write
\begin{equation}
\label{eqn::lbtk_sum}
X_{B_t^{r,j,\delta}}^{r,j,\delta} = \sum_\ell \left( X_{B_{\wt{\sigma}_{\ell+1}^{r,j,\delta}}^{r,j,\delta}}^{r,j,\delta} - X_{B_{\wt{\tau}_\ell^{r,j,\delta}}^{r,j,\delta}}^{r,j,\delta} \right) + \sum_\ell \left( X_{B_{\wt{\tau}_\ell^{r,j,\delta}}^{r,j,\delta}}^{r,j,\delta} - X_{B_{\wt{\sigma}_\ell^{r,j,\delta}}^{r,j,\delta}}^{r,j,\delta} \right).
\end{equation}
To finish the proof, we need to show that in the limit as $\delta \to 0$ we have that $X_{B_t^{r,j,\delta}}^{r,j,\delta}$ evolves as a $3/2$-stable L\'evy process.  We will establish this by showing that the first term in~\eqref{eqn::lbtk_sum} in the $\delta \to 0$ limit evolves as $3/2$-stable L\'evy process and the second term in~\eqref{eqn::lbtk_sum} tends to $0$ as $\delta \to 0$.

We begin with the second term in the right hand side of~\eqref{eqn::lbtk_sum}.  We note that the probability that a necklace hits one of the marked points is proportional to the quantum length of the top of the necklace.  By Lemma~\ref{lem::stable_maximum}, we know that the probability for this length to be larger than $c \delta^{2/3}$ for $c > 0$ decays exponentially in $c$.  Since the total number of necklaces is of order $\delta^{-1}$, we see that there will be with high probability $\delta^{-1/3}$ necklaces whose top is glued to one of the marked points.  The change in the boundary length for the left (resp.\ right) side of~\eqref{eqn::lbtk_sum} evolves like a $3/2$-stable L\'evy process conditioned to hit $0$ for the first time at an integer multiple of $\delta$ and these L\'evy processes are independent across necklaces.  So the overall magnitude of the error which comes from necklaces of this type is dominated by the sum of the supremum of the absolute value of order $\delta^{-1/3}$ $3/2$-stable L\'evy processes, each run for time $\delta$.  The expectation of the supremum of the absolute value of such a process is of order $\delta^{2/3}$, so the overall error is of order $\delta^{-1/3} \times \delta^{2/3} = \delta^{1/3}$.  We conclude that the amount of change which comes from these time intervals tends to $0$ as $\delta \to 0$.

We now turn to the first term in the right hand side of~\eqref{eqn::lbtk_sum}.  In each of the other intervals we know that the boundary length evolves as a $3/2$-stable L\'evy process conditioned to hit $0$ for the first time at an integer multiple of $\delta$. The total amount of L\'evy process time for each of the two sides is equal to $t$ minus the time which corresponds to those necklaces whose top was glued to a marked point.  As we have just mentioned above, this corresponds to time of order $\delta^{2/3}$ and therefore makes a negligible contribution as $\delta \to 0$.
\end{proof}

\subsection{Third approximations to geodesics and the $3/2$-L\'evy net}
\label{subsec::interfaces_are_geodesics}

We will now show that the statement of Proposition~\ref{prop::csbps_between_paths} holds in the setting of geodesics starting from the boundary of a filled metric ball.

\begin{proposition}
\label{prop::paths_are_geodesics}
Fix $r > 0$ and suppose that $(\CS,x,y)$ has the law $\Mstwo$ conditioned so that $\oqdist(x,y) > r$.  Suppose that $x_1,\ldots,x_k$ are picked independently from the quantum boundary measure on $\partial \qhull{x}{\oqdist(x,y)-r}$ and then reordered to be counterclockwise.  We let $\eta_1^r,\ldots,\eta_k^r$ be the a.s.\ unique (recall Proposition~\ref{prop::boundary_to_inside_unique}) geodesics from $x_1,\ldots,x_k$ to $x$.  For each $1 \leq j \leq k$ and $t  \in [0,\oqdist(x,y)-r]$ we let $X_t^{j,r}$ be the quantum boundary length of the counterclockwise segment of $\partial \qhull{x}{\oqdist(x,y)-(r+t)}$  between $\eta_j^r(t)$ and $\eta_{j+1}^r(t)$ (with the convention that $\eta_{k+1}^r = \eta_1^r$).  Given $X_0^{r,1},\ldots,X_0^{r,k}$, the processes $X_t^{r,j}$ evolve as independent $3/2$-stable CSBPs with initial values $X_0^{r,j}$.
\end{proposition}

In order to prove Proposition~\ref{prop::paths_are_geodesics}, we will need to construct our third approximations to geodesics.  We will carry this out in Section~\ref{subsubsec::third_approximations}.  We will then compare these third approximations with the second approximations in Section~\ref{subsubsec::paths_are_geodesics_proof}.  This comparison together with a scaling argument will lead to Proposition~\ref{prop::paths_are_geodesics}.

\subsubsection{Construction of third approximations to geodesics}
\label{subsubsec::third_approximations}

\begin{figure}[ht!]
\begin{center}
\includegraphics[scale=0.85]{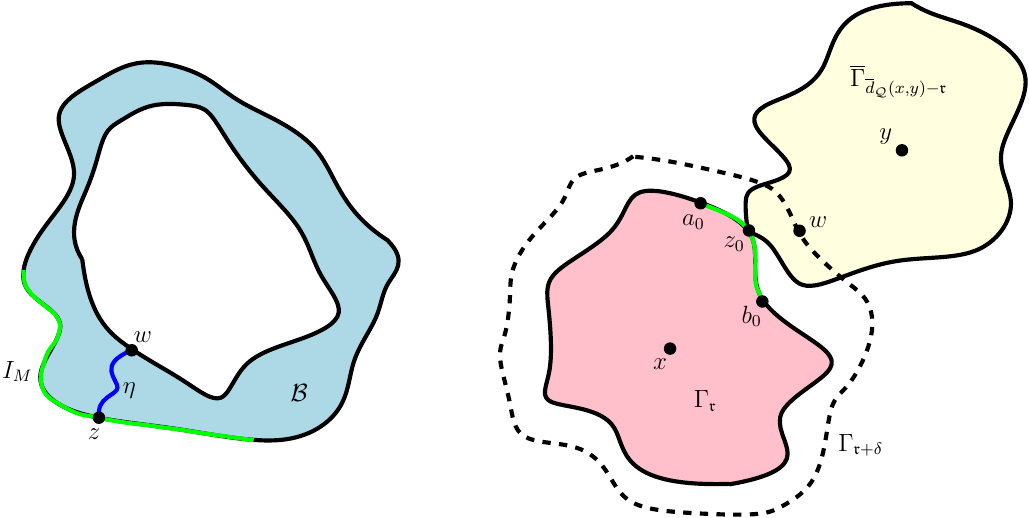}	
\end{center}
\caption{\label{fig::approximate_geodesic_band_illustration}  {\bf Left:} Illustration of the statement of Lemma~\ref{lem::approximate_geodesic_in_band}.  Shown is a metric band $\CB$ with inner boundary length $L$ and width $1$ together with a second approximation of a geodesic $\eta$ starting from a uniformly random point $w$ chosen on the inner boundary of $\CB$ and the length $M$ interval $I_M$ starting from the point $z$ where $\eta$ terminates on the outer boundary of $\CB$.  Lemma~\ref{lem::approximate_geodesic_in_band} implies that, if $M$ is large enough, then the expected distance from $w$ to $I_M$ is at most $1+\epsilon$. {\bf Right:} Illustration of the setup of the proof Lemma~\ref{lem::band_distances_grow}.}
\end{figure}

\begin{lemma}
\label{lem::approximate_geodesic_in_band}
For each $\epsilon > 0$ and $C > 1$ there exists $L_0,M_0 > 0$ such that for all $L \geq L_0$ and $M \geq M_0$ the following is true.  Suppose that $\CB$ has the law $\bandlaw_{L,1}$ (i.e., is a metric band with inner boundary length~$L$ and width~$1$) and that $w$ is chosen uniformly at random from the quantum measure restricted to the inner boundary of the band.  Let $\eta$ be the path from $w$ to the outer boundary of the band as constructed in Proposition~\ref{prop::subsequentially_limiting_paths_exist} (i.e., a second approximation to a geodesic), let $z$ be the point on the outside of the band where this path terminates, and let $I_M$ be the interval of quantum length $M$ on the outside of the band which is centered at $z$.  Let $E_C$ be the event that the quantum boundary length of the outer boundary of the reverse metric exploration starting from the inner boundary of $\CB$ and terminating at the outer boundary of $\CB$ stays in $[C^{-1} L, CL]$.  Conditionally on $E_C$, we have that the expected distance inside of $\CB$ starting from $w$ to $I_M$ is at most $1+\epsilon$.
\end{lemma}

See Figure~\ref{fig::approximate_geodesic_band_illustration} for an illustration of the statement of Lemma~\ref{lem::approximate_geodesic_in_band}.  We call a shortest length path in a metric band as in the statement of Lemma~\ref{lem::approximate_geodesic_in_band} which connects $w$ to the closest point to $w$ along $I_M$ a \emph{third approximation} to a geodesic.  We first record the following before giving the proof of Lemma~\ref{lem::approximate_geodesic_in_band}.

\begin{lemma}
\label{lem::band_distances_grow}
For each $\epsilon, D > 0$ there exists $L_0,M_0 > 0$ such that for all $L \geq L_0$ and $M \geq M_0$ the following is true.  Suppose that $\CB$ has the law $\bandlaw_{L,1}$ (i.e., is a metric band with inner boundary length $L$ and width $1$) and that $w$ is chosen uniformly at random from the quantum measure restricted to the inner boundary of the band.  Let $\eta$ be the path from $w$ to the outer boundary of the band as constructed in Proposition~\ref{prop::subsequentially_limiting_paths_exist} (i.e., a second approximation to a geodesic) let $z$ be the point on the outside of the band where this path terminates, and let $I_M$ be the interval of quantum length $M$ on the outside of the band which is centered at $z$.  The probability that the quantum distance between the complement of $I_M$ in the outer boundary to $w$ is at least $D$ is at least $1-\epsilon$.
\end{lemma}
\begin{proof}
Let $\MstwoD$ be as in the proof of Proposition~\ref{prop::reverse_qle_markov_property}.  Suppose that $(\CS,x,y)$ and $\rtime$ are sampled from $\MstwoD$.  As explained in the proof of Proposition~\ref{prop::reverse_qle_markov_property}, we know that the conditional law of $(\CS,x,y)$ given $\rtime = r$ is that of a sample from $\Mstwo$ conditioned on $\{\oqdist(x,y) > r\}$.

Let $\Gamma_r$ be the $\QLE(8/3,0)$ metric ball starting from $x$ and targeted at $y$.  Fix $\delta > 0$.  Note that the conditional law of the quantum surface parameterized by $\Gamma_{\oqdist(x,y) - \rtime} \setminus \Gamma_{\oqdist(x,y)-\rtime-\delta}$ given that the quantum boundary length of $\partial \Gamma_{\oqdist(x,y) - \rtime}$ is $L$ is that of $\bandlaw_{L,\delta}$.  Since the conditional laws of $\rtime$ and $\oqdist(x,y)-\rtime$ given everything else are the same, it follows that the conditional law of the quantum surface parameterized by $\Gamma_\rtime \setminus \Gamma_{\rtime - \delta}$ given that the quantum length of $\partial \Gamma_\rtime$ is equal to $L$ is that of $\bandlaw_{L,\delta}$.  Before we study $\Gamma_\rtime \setminus \Gamma_{\rtime - \delta}$, we will first study the quantum surface parameterized by $\Gamma_{\rtime + \delta} \setminus \Gamma_\rtime$.

Let $\ol{\Gamma}_r$ be the $\QLE(8/3,0)$ metric ball starting from $y$ and targeted at $x$.  Let $w$ be the unique element of $\ol{\Gamma}_{\oqdist(x,y)-\rtime-\delta} \cap \Gamma_{\rtime+\delta}$.  In other words, $w$ is the point on $\partial \Gamma_{\rtime + \delta}$ which is first hit by $\ol{\Gamma}_r$.  Let $z_0$ be the point on $\partial \Gamma_\rtime$ which is closest to $w$ using the internal metric in $\Gamma_{\rtime+\delta} \setminus \Gamma_\rtime$.  Note that in fact $z_0$ is the first point on $\Gamma_\rtime$ which is hit by $\ol{\Gamma}_r$.  Let $J_{N,\delta}$ be the interval on $\partial \Gamma_\rtime$ centered at $z_0$ of length $N \delta^2$ and let $a_0$, $b_0$ be its left and right endpoints.  Let $\eta_0$ be the shortest path from $a_0$ to $b_0$ in $\CS \setminus \Gamma_\rtime$ (if there is more than one, breaking ties by taking the one which is rightmost).  Let $d_0$ be the distance in $\CS \setminus \Gamma_{\rtime}$ from $\eta_0$ to $z_0$.  Then note that the distance from $w$ to $\partial \Gamma_\rtime \setminus J_{N,\delta}$ using the internal metric in $\Gamma_{\rtime+\delta} \setminus \Gamma_\rtime$ is at least $d_0 - \delta$.  Indeed, this follows because:
\begin{itemize}
\item Computing distances to $z_0$ versus $w$ changes them by at most $\delta$ by the triangle inequality.
\item Any path from a point on $\partial \Gamma_\rtime \setminus J_{N,\delta}$ to $z_0$ in $\Gamma_{\rtime+\delta} \setminus \Gamma_\rtime$ must pass $\eta_0$.
\item The distance from $\eta_0$ to $z_0$ gets smaller if we use the internal metric of $\CS \setminus \Gamma_{\rtime}$ instead of the internal metric in $\Gamma_{\rtime+\delta} \setminus \Gamma_\rtime$.
\end{itemize}
  We are now going to show that $\delta^{-1} d_0$ converges in law as $\delta \to 0$ to an $N$-dependent limit which, in turn, tends to $\infty$ as $N \to \infty$.

We know that the quantum surface parameterized by $\CS \setminus \Gamma_{\rtime}$ is a quantum disk weighted by its quantum area.  If we add $\tfrac{4}{\gamma} \log \delta^{-1}$ to the field so that distances are multiplied by $\delta^{-1}$ and boundary lengths are multiplied by $\delta^{-2}$ and send $\delta \to 0$, then the law of the surface near $z_0$ will converge to a $\sqrt{8/3}$-LQG wedge, say $(\h,h,0,\infty)$.  Let $p < 0 < q$ be such that $\nu_h([p,0]) = N/2$ and $\nu_h([0,q]) = N/2$.  Then the law of $\delta^{-1} d_0$ will converge to that of the distance from $0$ of the shortest path from $p$ to $q$ computed using the metric associated with $h$.  By the scaling properties of the metric and boundary lengths (i.e., if we add $\tfrac{4}{\gamma} \log C$ to $h$ then distances get multiplied by $C$ and boundary lengths by $C^2$), it follows that the distance to $0$ of the aforementioned shortest path converges in probability to $\infty$ as $N \to \infty$.

Note that the total variation distance between the conditional law of $(\rtime,\rtime+\delta)$ given everything else and that of $(\rtime-\delta,\rtime)$ tends to $0$ as $\delta \to 0$.  It therefore follows from the above that if we define $z$, $J_{N,\delta}$, and $a_0$ in the same way but in terms of $\Gamma_\rtime \setminus \Gamma_{\rtime - \delta}$ then we also have that $\delta^{-1} d_0$ converges to the same limit as $\delta \to 0$.  The same analysis applies with $e_0$ in place of $d_0$.

Since $\bandlaw_{\delta^{-2} L,1}$ can be obtained from $\bandlaw_{L,\delta}$ by adding $\tfrac{4}{\gamma} \log \delta^{-1}$ to the field, it follows from the above that we have established the following statement.  For every $\epsilon, D > 0$ there exists $L_0,N_0$ such that for every $L \geq L_0$ and $N \geq N_0$ the following is true.  Suppose that $\CB$ has law $\bandlaw_{L,1}$ and let $z_0$ be the point on the outer boundary of the band which is closest to $w$.  Note that the distance between $z_0$ and $w$ is equal to $1$.  Let $J_N$ be the interval of quantum length $N$ on the outside of $\CB$ centered at $z_0$.  Then the probability that the distance between $w$ and the complement of $J_N$ is at least $D$ is at least $1-\epsilon/2$.  On the other hand, Proposition~\ref{prop::subsequentially_limiting_paths_exist} implies that by possibly increasing the value of $N$, we have that the probability that the length of $\eta$ is at most the distance between $w$ and complement of $J_N$ is at least $1-\epsilon/2$.  In particular, on the intersection of these two events (which occurs with probability at least $1-\epsilon$), we have that $J_N \subseteq I_{2N}$ and therefore with $M=2N$ the distance of $w$ to the complement of $I_M$ is at least $D$.

\end{proof}

\begin{proof}[Proof of Lemma~\ref{lem::approximate_geodesic_in_band}]
For each $j$, we let $I_j$ be the interval of quantum length $j$ centered at $z$ as in the statement of the proposition and let $X_j$ be the distance from $w$ to $I_j$ inside of $\CB$.  Then we have that $X_{j+1} \leq X_j$ for every $j$.  We also have that $X_0$ is at most the length of $\eta$.  We also know that $X_j = 1$ on the event $F_j$ that the geodesic terminates in $I_j$ since $\CB$ has width $1$.  Then we have that
\begin{align*}
   \E[ X_j \giv E_C ]
&= \E[ X_j (\one_{F_j} + \one_{F_j^c}) \giv E_C ] 
 \leq 1 + \E[ X_0 \one_{F_j^c} \giv E_C].
\end{align*}
Lemma~\ref{lem::band_distances_grow} implies that by adjusting the parameters, we can make $\p[F_j^c \giv E_C]$ as small as we like.  Therefore the result follows from the uniform integrability of the length of $\eta$ on $E_C$ established in Proposition~\ref{prop::subsequentially_limiting_paths_exist}.
\end{proof}

\subsubsection{Subsequential limits of rescalings of concatenations of third approximations of geodesics are geodesics}
\label{subsubsec::paths_are_geodesics_proof}

\begin{figure}[ht!]
\begin{center}
\includegraphics[scale=0.85]{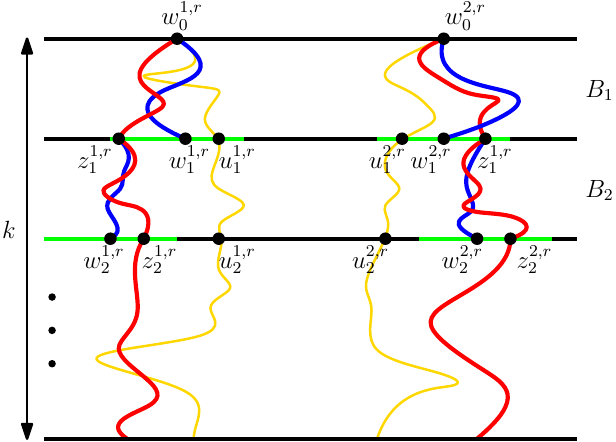}
\end{center}
\vspace{-0.025\textheight}
\caption{\label{fig::geodesic_layers} Illustration of the comparison argument used to prove Proposition~\ref{prop::paths_are_geodesics}, which implies that the boundary lengths between geodesics from the boundary of the reverse metric exploration up to time $r$ back to the root evolve as independent $3/2$-stable CSBPs.  Each of the $k$ layers shown represents a metric band of a fixed width (where each band is as illustrated in Figure~\ref{fig::metric_band_decomposition}).  The two orange paths are second approximations to a geodesic across all $k$ metric bands.  Each blue path represents a second approximation to a geodesic across a single band starting from the terminal point of the corresponding blue path across the previous band, which is a third approximation to a geodesic.  Note that on a high probability event the terminal point of each of the red paths is contained in a green interval centered at the corresponding second approximation to a geodesic.  These green intervals each have quantum length equal to a fixed constant $M$.  The evolution of the boundary lengths between any pair of blue paths, orange paths, or pair consisting of a blue path and an orange path is given by independent $3/2$-stable CSBPs.  By construction, the evolution of the boundary lengths between the red paths is then close to being that of independent $3/2$-stable CSBPs.  Due to the way that boundary lengths and quantum distances scale, this error can be made to be arbitrarily small by first taking $k$ to be large and then rescaling.  We will then argue that we can make the red paths as close to geodesics as we like by taking $\epsilon > 0$ very small, which in turn implies that boundary lengths between geodesics evolve as independent $3/2$-stable CSBPs.}	
\end{figure}

We can now complete the proof of Proposition~\ref{prop::paths_are_geodesics}.  See Figure~\ref{fig::geodesic_layers} for an illustration of the argument.

\begin{proof}[Proof of Proposition~\ref{prop::paths_are_geodesics}]
Fix $\alpha \in (0,1)$.  Fix $\epsilon > 0$, $C > 1$, and let $L_0,M_0$ be as in Lemma~\ref{lem::approximate_geodesic_in_band} for these values of $\epsilon, C$.  Fix $r,s >0$ and assume that $(\CS,x,y)$ is sampled from $\Mstwo$ conditioned on both:
\begin{itemize}
\item $\oqdist(x,y) > r\sqrt{L}$ and
\item the quantum boundary length of the boundary of the reverse metric exploration at quantum distance time $r \sqrt{L}$ being in $[L,CL]$.
\end{itemize}
That is, the quantum boundary length of $\partial \qhull{x}{\oqdist(x,y)-r\sqrt{L}}$ is contained in $[L,CL]$.  Let $(B_j)$ be the sequence of width-$1$ metric bands in the reverse exploration from $\CS$ starting from $\partial \qhull{x}{\oqdist(x,y)-r\sqrt{L}}$ and targeted towards $x$.  We let $E$ be the event that the boundary length of the reverse metric exploration starting from quantum distance time $r\sqrt{L}$ stays in $[C^{-1} L ,C L]$ for quantum distance time $s \sqrt{L}$.  By scaling quantum lengths by $L^{-1}$ so that quantum distances scale by $L^{-1/2}$ (recall also the scaling rules for $3/2$-stable CSBPs from Section~\ref{subsec::csbp_estimates}), we observe that the conditional probability of $E$ assigned by $\Mstwo$ conditioned as described just above is bounded from below by a positive constant which depends only on $C$ and $s$.

Assume that we have chosen $L \geq L_0$, $M \geq M_0$, and that we have picked $w_0^{1,r}$ from the quantum measure on the inner boundary of $B_1$.  Let $w_1^{1,r}$ be the point where the second approximation to a geodesic starting from $w_0^{1,r}$ hits the outside of $B_1$.  Lemma~\ref{lem::approximate_geodesic_in_band} implies that, conditionally on $E$, the expected length of the shortest path starting from $w_0^{1,r}$ and terminating in the quantum length $M$ interval centered at $w_1^{1,r}$ is at most $1+\epsilon$.  Let $z_1^{1,r}$ be the terminal point of this path.  Assuming that $w_1^{1,r},\ldots,w_k^{1,r}$ and $z_1^{1,r},\ldots,z_k^{1,r}$ have been defined, we let $w_{k+1}^{1,r}$ be the terminal point of the second approximation of a geodesic starting from $z_k^{1,r}$ across the band $B_{k+1}$ and we let~$z_{k+1}^{1,r}$ be the terminal point of the shortest path starting from $z_k^{1,r}$ and terminating in the boundary length $M$ interval centered at $w_{k+1}^{1,r}$.  Lemma~\ref{lem::approximate_geodesic_in_band} implies that, conditionally on $E$, the expected length of the this path is at most $1+\epsilon$.

Let $u_0^{1,r},u_1^{1,r},\ldots$ be the points on the inner boundaries of the successive metric bands visited by a second approximation to a geodesic starting from $u_0^{1,r} = w_0^{1,r}$.

For each $k$, we let $S_k^{1,r}$ be the quantum length of the shorter boundary segment between $z_k^{1,r}$ and $u_k^{1,r}$.  By the triangle inequality, we have that $S_k^{1,r}$ is at most the sum of $M$ and the boundary length distance between $u_k^{1,r}$ and $w_k^{1,r}$.  The law of the latter, in turn, can be sampled from by evolving a $3/2$-stable CSBP starting from $S_{k-1}^{1,r}$ for one unit of time.  That is,
\[ S_k^{1,r} \leq M + \Delta_k^{1,r} + S_{k-1}^{1,r}\]
where $\Delta_k^{1,r} = Y_1^k - Y_0^k$ and $Y^k$ has the law of a $3/2$-stable CSBP starting from $Y_0^k = S_{k-1}^{1,r}$.  This implies that we can write
\[ S_k^{1,r} \leq M k + \sum_{j=1}^k \Delta_j^{1,r}.\]
Since $\E[ \Delta_j^{1,r}] = 0$ for all $j$ (as a $3/2$-stable CSBP is a martingale), we have that
\begin{equation}
\label{eqn::boundary_displacement_ubd}
\E[S_k^{1,r}] \leq M k.
\end{equation}

Recall that if we rescale so that distances are multiplied by $L^{-1/2}$, then quantum lengths are rescaled by the factor $L^{-1}$.  Therefore if we rescale so that distances are rescaled by $L^{-1/2}$, we have a pair of paths $\gamma^{r,1}$ and $\wt{\gamma}^{r,1}$ which connect $\partial \qhull{x}{\oqdist(x,y)-r}$ to $\partial \qhull{x}{\oqdist(x,y)-(r+s)}$ where $\gamma^{r,1}$ (resp.\ $\wt{\gamma}^{r,1}$) is a rescaled second approximation to a geodesic (resp.\ rescaled concatenation of third approximations to geodesics).  Moreover, the expectation of the quantum length of the segment of $\partial \qhull{x}{\oqdist(x,y)-(r+s)}$ which connects the tip of $\gamma^{r,1}$ to the tip of $\wt{\gamma}^{r,1}$ is at most
\begin{equation}
\label{eqn::tip_distance}
	M L^{1/2} \times  L^{-1} = M L^{-1/2}.
\end{equation}
Also, the conditional expectation of the length of $\wt{\gamma}^{r,1}$ is at most $(1+\epsilon) s$ given $E$.  Suppose that $\gamma^{r,2},\wt{\gamma}^{r,2}$ is another pair of such paths starting from a uniformly random point on $\partial \qhull{x}{\oqdist(x,y)-r}$ which is conditionally independent (given the surface) of the starting point of $\gamma^{1,r},\wt{\gamma}^{1,r}$.  Then Proposition~\ref{prop::csbps_between_paths} implies that the boundary lengths of the two boundary segments between $\gamma^{1,r}$ and $\gamma^{2,r}$ evolve as independent $3/2$-stable CSBPs when performing a reverse metric exploration.  Indeed, Proposition~\ref{prop::csbps_between_paths} implies that this is the case for second approximations to geodesics and, as this property is scale invariant, it also holds for rescalings of second approximations to geodesics.

We will now take limits first as $L \to \infty$, then as $\epsilon \to 0$, and then finally as $C \to \infty$ to complete the proof.

\begin{itemize}
\item Step 1: limit as $L \to \infty$.  Since we can take $M$ to be of constant order as $L \to \infty$, it follows from~\eqref{eqn::tip_distance} that by taking $L$ to be very large we can arrange so that the boundary length distance between the tips of $\gamma^{1,r},\wt{\gamma}^{1,r}$ and $\gamma^{2,r},\wt{\gamma}^{2,r}$ is arbitrarily small.  In particular, as $L \to \infty$, we find that the boundary lengths of the two segments of $\partial \qhull{x}{\oqdist(x,y)-(r+s)}$ has law given by starting with the boundary lengths between the two segments of $\partial \qhull{x}{\oqdist(x,y)-r}$ and then evolving according to a pair of independent $3/2$-stable CSBPs for $s$ units of time.  Since this holds for each fixed $s$, we have that the finite dimensional distributions of the boundary lengths of the two segments of $\partial \qhull{x}{\oqdist(x,y)-(r+s)}$ are the same as the finite dimensional distributions of a pair of independent $3/2$-stable CSBPs.
\item Step 2: limit as $\epsilon \to 0$.  By an abuse of notation, we write $\wt{\gamma}^{1,r}$, $\wt{\gamma}^{2,r}$ for the paths which arise after taking the limit as $L \to \infty$ from the previous step.  As $\epsilon \to 0$, the length of $\wt{\gamma}^{1,r}$ and $\wt{\gamma}^{2,r}$ converges to $s$.  If we parameterize these paths according to quantum distance, then this in particular implies that their Lipschitz constant converges to $1$ and thus admit subsequential limits as $\epsilon \to 0$ which are in fact geodesics.  Proposition~\ref{prop::boundary_to_inside_unique} implies that there is only possible limit and therefore Proposition~\ref{prop::boundary_to_inside_unique} implies that $\wt{\gamma}^{1,r}$, $\wt{\gamma}^{2,r}$ converge as $\epsilon \to 0$ to the a.s.\ unique geodesic which connects their starting points back to $x$.
\item Step 3: limit as $C \to \infty$.  Since $\p[E] \to 1$ as $C \to \infty$, the proof is complete.
\end{itemize}

\end{proof}

\subsubsection{$\sqrt{8/3}$-LQG unembedded metric net is the $3/2$-stable L\'evy net}
\label{subsubsec::metric_net_is_levy_net}

We will now combine Proposition~\ref{prop::reverse_qle_markov_property} with Proposition~\ref{prop::paths_are_geodesics} to show that the unembedded metric net from $x$ to $y$ in a sample $(\CS,x,y)$ from $\Mstwo$ is the $3/2$-stable L\'evy net of \cite{map_making}.

We recall that there are several equivalent definitions of the $3/2$-stable L\'evy net which are given in \cite{map_making}.  We will make the comparison here between the construction of the $3/2$-stable L\'evy net given in \cite[Section~3.6]{map_making}, which is based on a breadth-first approach.

We are now going to give a brief review of the definition and basic properties of the $3/2$-stable L\'evy net \cite[Definition~3.1]{map_making}.  The $3/2$-stable L\'evy net is an infinite measure on $\treeequivspace$ which can be sampled from as follows.  Suppose that $X^\LN \colon [0,T] \to \R_+$ is the cadlag modification of the time-reversal of a $3/2$-stable L\'evy excursion with only upward jumps sampled from the infinite measure on such excursions \cite[Chapter VIII.4]{bertoin96levy}.  Note that $X^\LN$ has downward jumps as it is the time-reversal of a process with upward jumps.  We note that the lifetime $T$ of $X^\LN$ is not deterministic.  Let $Y$ be the height process associated with $X^\LN$.  This means that for each $t \in [0,T]$, we have that $Y_t$ is equal to the amount of local time that $X^\LN|_{[t,T]}$ spends at its running infimum.  We then take $Y_t^\LN = -(Y_{t/T} - \sup_{s \in [0,T]} Y_s)$ and define $K^\LN$ as follows.  We have that $(s,t) \in \T_2$ with $s < t$ is in $K^\LN$ if and only if:
\begin{enumerate}[(i)]
\item $X_{sT}^\LN = X_{t T}^\LN$ and the horizontal chord from $(sT,X_{sT}^\LN)$ to $(tT,X_{tT}^\LN)$ lies below the graph of $X^\LN$, or
\item $Y_s^\LN = Y_t^\LN$ and the horizontal chord from $(s,Y_s^\LN)$ to $(t,Y_t^\LN)$ lies below the graph of $Y^\LN$, or
\item $Y_s^\LN = Y_t^\LN$ and the horizontal chords from $(s,Y_s^\LN)$ to $(0,Y_s^\LN)$ and from $(t,Y_t^\LN)$ to $(1,Y_t^\LN)$ lie below the graph of $Y^\LN$.
\end{enumerate}
We also have that if $(s,t) \in K^\LN$ then $(t,s) \in K^\LN$.  Then $K^\LN$ defines a topologically closed equivalence relation on $\T_1$ and $(Y^\LN,K^\LN)$ takes values in $\treeequivspace$.

Let us recall a few additional facts about the L\'evy net.  The process $Y^\LN$ attains its minimum of $0$ at a unique time and this corresponds to the root of the tree encoded by $Y^\LN$.  The process $Y^\LN$ also attains its maximum value $D$ at a unique time and this corresponds to the dual root.  The reason for this terminology is that if we quotient $\T_1$ by the equivalence relation described by $K^\LN$ then the resulting topological space can be viewed as a gluing of two trees: the tree encoded by $Y^\LN$ and the looptree encoded by $X^\LN$ \cite[Figure~3.3]{map_making}.  The root of the L\'evy net is the root of the former and the dual root is the root of the latter.

Although the topological space obtained by quotienting $\T_1$ by the equivalence relation $K^\LN$ does not a priori come with the structure of a metric, we refer to $Y^\LN$ as the contour function of the \emph{geodesic tree} $\CT^\LN$ and we note that distances are defined in $\CT^\LN$ in the usual way for real trees.  We let $\eta$ be the geodesic in $\CT^\LN$ which connects the root to the dual root.  For each $r \geq 0$, we let $B_r^\LN$ be the metric ball of radius $r$ in $\CT^\LN$ centered at the root.  It is explained in the beginning of \cite[Section~3.6]{map_making} that it is possible to define an overall boundary length $Z_r^\LN$ of $\partial B_r^\LN$.  Then we moreover have that $r \mapsto Z_{D-r}^\LN$ for $r \in [0,D]$ evolves as a $3/2$-stable CSBP \cite[Proposition~3.14]{map_making}.  This is the starting point for the \emph{breadth-first} construction of the L\'evy net (instead of the depth-first construction, which is the one defined using $X^\LN$).  Whenever $Z_{D-r}^\LN$ makes an upward jump, it corresponds to $(s,t) \in K^\LN$ with $s < t$ and $t T$ is a jump time of $X^\LN$ \cite[Proposition~3.15]{map_making}.  The size of the upward jump made by $Z_{D-r}^\LN$ turns out to be the same as the size of the downward jump made by $X^\LN$ at time $t T$ \cite[Proposition~3.15]{map_making}.  More generally, there is a boundary length measure defined on $\partial B_r^\LN$ with total boundary length described by $Z^\LN$ and which is right-continuous \cite[Proposition~3.26]{map_making}.  Two points in $a, b \in \CT^\LN$ are equivalent under $K^\LN$ (which induces an equivalence relation on $\CT^\LN$) if and only if $a,b \in \partial B_r^\LN$ for some $r \geq 0$ and the boundary length of either the clockwise or counterclockwise segment of $\partial B_r^\LN$ from $a$ to $b$ is zero \cite[Proposition~3.25]{map_making}.
\begin{enumerate}[(i)]
\item\label{it:tree_prop1} Suppose that we fix a value of $r$ and then sample $z_1,\ldots,z_n$ uniformly from the boundary length measure on $\partial B_{D-r}^\LN$ then reorder $z_1,\ldots,z_n$ to be counterclockwise.  Then we have that the boundary lengths between the geodesics from the $z_i$ to the root evolve as independent $3/2$-stable CSBPs.  The amount of time that it takes the boundary length between $z_i$ and $z_j$ to reach $0$ is equal to $1/2$ the distance between $z_i$ and $z_j$ in $\CT^\LN$ \cite[Proposition~3.19]{map_making}.
\item\label{it:tree_prop2} Finally, if we condition on the boundary length $Z_{D-r}^\LN$ of $\partial B_{D-r}^\LN$, then the inside and the outside of $B_{D-r}^\LN$ are independent together with the set of points in each which are equivalent under the equivalence relation defined by $K^\LN$ (\cite[Proposition~3.25]{map_making} together with the beginning of \cite[Section~3.6]{map_making}). 
\end{enumerate}
As explained in the proof of Lemma~\ref{lem:geodesic_tree_is_the_same}, the form of the boundary length evolution between geodesics characterizes the whole geodesic tree in the following sense.  Suppose that $\wt{\CT}$ is a random rooted plane tree and for each $r \geq 0$ we let $\wt{B}_r$ be the set of points in $\wt{\CT}$ at distance at most $r$ from the root of $\wt{\CT}$.  Suppose that (for each $r \geq 0$) $\partial \wt{B}_r$ comes equipped with a length measure so that properties~\eqref{it:tree_prop1}, \eqref{it:tree_prop2} hold.  Then $\wt{\CT}$ has the same law as $\CT^\LN$.

\begin{proposition}
\label{prop::metric_net_law}
Suppose that $(\CS,x,y)$ is an instance of the $\sqrt{8/3}$-LQG sphere.  Then the unembedded metric net of $\CS$ has the law of a $3/2$-stable L\'evy net.
\end{proposition}

Throughout, we will let $(\CS,x,y)$ be an instance of the $\sqrt{8/3}$-LQG sphere.  Let $(Y^\LQG,K^\LQG)$ be the unembedded metric net of $\CS$.  We will proceed by first showing (Lemma~\ref{lem:geodesic_tree_is_the_same}) that the leftmost geodesic tree in the metric net of $\CS$ (i.e., the tree encoded by $Y^\LQG$) has the same law as the geodesic tree of the L\'evy net.  Upon showing this, as explained in \cite[Proposition~3.7]{map_making}, there exists a coupling of $Y^\LQG$ and an instance of the $3/2$-stable L\'evy net $(Y^\LN,K^\LN)$ where $Y^\LN = Y^\LQG$.  We will then show that the associated equivalence relation $K^\LQG$ is equal to $K^\LN$.  The first step is to show that $K^\LQG \supseteq K^\LN$ (Lemma~\ref{lem:levy_net_jump_equiv_dense}) and the second step is to show that $K^\LQG \subseteq K^\LN$ (Lemma~\ref{lem:net_equiv_implies_levy_equiv}).

\begin{lemma}
\label{lem:geodesic_tree_is_the_same}
Up to a monotone reparameterization, we have that $Y^\LQG$ has the law of the contour function of the leftmost geodesic tree in the $3/2$-stable L\'evy net.
\end{lemma}
\begin{proof}
Since we have shown in Proposition~\ref{prop::reverse_qle_markov_property} that the overall boundary length of $\partial \qhull{x}{\oqdist(x,y)-s}$ for $s \in [0,\oqdist(x,y)]$ evolves as a $3/2$-stable CSBP excursion, it follows that we can couple $(\CS,x,y)$ with a $3/2$-stable L\'evy net so that the overall boundary length processes agree.

We recall from the construction given in \cite[Section~3.6]{map_making}, on the event that the length $d = \oqdist(x,y)$ of the $3/2$-stable CSBP excursion used to generate the L\'evy net is at least $r$ for $r > 0$ fixed, the boundary lengths between geodesics starting from equally spaced points on the boundary of a ball of radius $d-r$ evolve as independent $3/2$-stable CSBPs as the radius of the ball decreases from $d-r$ to $0$.  The same is also true if the geodesics start from randomly chosen points on the boundary of the ball and then ordered to be counterclockwise.

For any fixed value of $r$, Proposition~\ref{prop::paths_are_geodesics} implies that the same is true for the boundary lengths between the geodesics in a reverse metric exploration of $(\CS,x,y)$ which start from a finite number of points chosen i.i.d.\ from the quantum boundary measure on $\partial \qhull{x}{\oqdist(x,y)-r}$.  Therefore, for a fixed value of $r$, we can couple these boundary lengths to be the same as in the L\'evy net.  By sending the number of geodesics considered to $\infty$, we can couple so that the evolution of the boundary lengths between all of the (leftmost) geodesics starting from $\partial \qhull{x}{\oqdist(x,y)-r}$ to $x$ agree with the corresponding boundary length evolutions in the L\'evy net instance.

Suppose that $0 < r_1 < \cdots < r_k$ are fixed and we are working on the event that $\oqdist(x,y) > r_k$.  As explained above, we can couple all of the (leftmost) geodesics from $\partial \qhull{x}{\oqdist(x,y)-r_1}$ to $\partial \qhull{x}{\oqdist(x,y)-r_2}$ with those in the L\'evy net so that the evolution of the boundary length between any pair agrees.  Recall from Proposition~\ref{prop::reverse_qle_markov_property} that $\qhull{x}{\oqdist(x,y)-r_2}$ is conditionally independent of $\CS \setminus \qhull{x}{\oqdist(x,y)-r_2}$ given its boundary length.  The same is also true in the case of the L\'evy net.  Therefore we can couple all of the (leftmost) geodesics from $\partial \qhull{x}{\oqdist(x,y)-r_2}$ to $\partial \qhull{x}{\oqdist(x,y)-r_3}$ with those in the L\'evy net so that the evolution of the boundary length between any pair agrees.  By iterating this, we can more generally couple all of the (leftmost) geodesics from $\partial \qhull{x}{\oqdist(x,y)-r_j}$ to $\partial \qhull{x}{\oqdist(x,y)-r_{j+1}}$ with those in the L\'evy net so that the evolution of the boundary length between any pair agrees for $1 \leq j \leq k-1$.  Since each geodesic from $\partial \qhull{x}{\oqdist(x,y)-r_j}$ to $x$ consists of a concatenation of geodesics from $\partial \qhull{x}{\oqdist(x,y)-r_i}$ to  $\partial \qhull{x}{\oqdist(x,y)-r_{i+1}}$ for each $i \leq j \leq k$ (taking $r_{k+1} = \oqdist(x,y)$), it follows that under this coupling we have that the boundary lengths between all of the pairs of (leftmost) geodesics from every point of $\partial \qhull{x}{\oqdist(x,y)-r_j}$ to $x$ agrees with that in the L\'evy net for each $1 \leq j \leq k$.

Sending the number of $r$ values to a countable dense set, we obtain a coupling in which the whole leftmost geodesic trees agree, from which the result follows.
\end{proof}

Let $Y^\LN = Y^\LQG$.  As we mentioned above, $Y^\LN$ can be coupled with a $3/2$-stable L\'evy net instance $(Y^\LN,K^\LN)$.  Let $X^\LN \colon [0,T] \to \R_+$ be the corresponding time-reversed L\'evy excursion.  We now proceed to show that $K^\LQG = K^\LN$.

\begin{lemma}
\label{lem:levy_net_jump_equiv_dense}
We have that $K^\LQG \supseteq K^\LN$.
\end{lemma}
\begin{proof}
We are first going to argue that if $(s,t) \in K^\LN$ with $s < t$ and $t T$ is a jump time of $X^\LN$ then $(s,t) \in K^\LQG$.  We will then show that $K^\LN$ is equal to the closure of the set of such pairs of times together with the set of pairs of times which are equivalent in the real tree encoded by $Y^\LN = Y^\LQG$.  Together, this will imply that $K^\LQG \supseteq K^\LN$.

The first step is to argue that the locations at which the bubbles appear when performing a reverse metric exploration in $\CS$ are the same as in the L\'evy net instance encoded by $Y^\LN$.  To show that this is the case, fix $\delta > 0$ and for each $j \in \N$ we let $\tau_{j,\delta}$ be the $j$th time that the overall boundary length process for the reverse metric exploration makes an upward jump of size at least $\delta$.  This corresponds to a bubble $U_{j,\delta}$ of boundary length at least $\delta$ which is cut off from $y$ by $\partial \qhull{x}{\oqdist(x,y)-\tau_{j,\delta}}$.  Then we have that $\partial U_{j,\delta} \cap \partial \qhull{x}{d(x,y)-\tau_{j,\delta}}$ consists of a single point $x_{j,\delta}$.   It therefore follows that there are two distinct geodesics in $\CS$ from $x$ to $x_{j,\delta}$.  One of them is the leftmost geodesic from $x_{j,\delta}$ to $x$ and the other is the limit of the leftmost geodesic from $z$ to $x$ as $z \in \partial \qhull{x}{\oqdist(x,y)-\tau_{j,\delta}}$ tends to $x_{j,\delta}$ from the right.

Since the overall boundary length process for the reverse metric exploration of $\CS$ is the same as for $(Y^\LN,K^\LN)$, it follows that there is a corresponding jump in the boundary length measure for the L\'evy net.  Let $\eta$ be the unique geodesic from $y$ to $x$ and so that $\eta(\tau_{j,\delta})$ is the unique intersection point of $\eta$ with $\partial \qhull{x}{d(x,y)-\tau_{j,\delta}}$.  We will show that the location of $x_{j,\delta}$ is the same as in the corresponding jump in the L\'evy net boundary length measure in that the counterclockwise boundary length distance from $x_{j,\delta}$ to $\eta(\tau_{j,\delta})$ is the same as for the location of the jump in the L\'evy net measure.  Since we have chosen $\delta > 0$ to be arbitrary, we will get that the attachment points for all of the bubbles cut off by a metric exploration in $\CS$ are the same as the locations of the jumps in the boundary length measure for the L\'evy net.

To prove the claim, we fix $\epsilon > 0$.  For each $k \in \N$, we fix points on $\partial \fb{x}{d(x,y)-k \epsilon}$ in counterclockwise order with boundary length spacing $\epsilon$ starting from $\eta(k \epsilon)$.  Then we know that the boundary lengths between the geodesics starting from these points back to $x$ evolve as independent $3/2$-stable CSBPs.  By construction, the boundary lengths between the corresponding geodesics in the L\'evy net are the same.  Therefore if the overall boundary length process makes an upward jump of size at least $\delta$ in $[k\epsilon, (k+1)\epsilon]$, then the boundary length between a pair of one of the aforementioned geodesics also makes an upward jump of size at least $\delta$ and the pair is the same for both the L\'evy net and for $\CS$.  By sending $\epsilon \to 0$, we see that the counterclockwise boundary length distance between each $x_{j,\delta}$ to $\eta(\tau_{j,\delta})$ is the same as the corresponding boundary length distance in the L\'evy net.

Suppose that $(s,t) \in K^\LN$ with $s < t$.  If $t T$ is a jump time of $X^\LN$, then by what we have just proved above we have that $(s,t) \in K^\LQG$.  Suppose that $t T$ is not a jump time of $X^\LN$.  By the definition of $K^\LN$, we have that the horizontal chord connecting $(sT,X_{sT}^\LN)$ and $(t T,X_{t T}^\LN)$ lies below the graph of $X^\LN|_{[sT, t T]}$.  Then there exists a sequence of times $t_k$ so that $X^\LN$ has a downward jump at time $t_k T$ such that if $s_k < t_k$ is such that $X_{s_k T}^\LN = X_{t_k T}^\LN$ and the horizontal chord from $(s_k T,X_{s_k T}^\LN)$ to $(t_k T,X_{t_k T}^\LN)$ lies below the graph of $X^\LN|_{[s_k T,t_k T]}$ and $s_k \uparrow s$ and $t_k \downarrow t$ as $k \to \infty$.  Then $(s_k,t_k) \in K^\LQG$ for each $k \in \N$ by what we explained at the beginning of the proof.  Since $K^\LQG$ is closed, it follows that $(s,t) \in K^\LQG$.
\end{proof}

\begin{lemma}
\label{lem:net_equiv_implies_levy_equiv}
We have that $K^\LQG \subseteq K^\LN$.  In particular, $K^\LQG = K^\LN$.
\end{lemma}
\begin{proof}
Suppose that $(s,t) \in K^\LQG$ and $s < t$.  If $t T$ is a jump time of $X^\LN$, then as we explained in the proof of Lemma~\ref{lem:levy_net_jump_equiv_dense} we have that $(s,t) \in K^\LN$.  We may therefore assume that $t T$ is not a jump time of $X^\LN$.  Then there exists $r > 0$ so that $Y_s^\LQG = Y_t^\LQG = \oqdist(x,y)-r$.  We recall from the breadth-first construction of the L\'evy net that the boundary length measure is defined for all radii simultaneously and is right-continuous.  Let $\rho^\LN \colon \T_1 \to \CT^\LN$ be the projection map from $\T_1$ to the tree encoded by $Y^\LN$.  We consider two possibilities: either one of the boundary lengths along the clockwise or counterclockwise segments of $\partial B_{\oqdist(x,y)-r}^\LN$ from $\rho^\LN(s)$ to $\rho^\LN(t)$ is equal to zero or both boundary lengths are positive.  If one of the boundary lengths is equal to zero, then it follows that $(s,t) \in K^\LN$ by the breadth-first construction of the L\'evy net.  Suppose that both boundary lengths are positive.  Let $\rho^\LQG \colon \T_1 \to \CS$ be the map which embeds the (completion of) the leftmost geodesic tree into $\CS$.  We will obtain a contradiction by showing that $\rho^\LQG(s) \neq \rho^\LQG(t)$.

Fix $\epsilon > 0$ small and rational and let $k = \lceil r /\epsilon\rceil - 1$.  Since $t T$ is not a jump time for $X^\LN$, it follows that for all such $\epsilon > 0$ sufficiently small we can find $u,v \in \partial B_{\oqdist(x,y)-k \epsilon}^\LN$ to the root of $\CT^\LN$ and with counterclockwise boundary length distance from the unique geodesic from the dual root to the root in $\CT^\LN$ given by a multiple of $\epsilon$ so that the following is true.  The geodesics from $u,v$ to the root pass through the counterclockwise segment of $\partial B_{\oqdist(x,y)-r}^\LN$ from $\rho^\LN(s)$ to $\rho^\LN(t)$ before merging.  By the proof of Lemma~\ref{lem:geodesic_tree_is_the_same}, this implies that the corresponding leftmost geodesics in $\CS$ also do not merge before passing through the counterclockwise segment of $\partial \fb{x}{\oqdist(x,y)-r}$ from $\rho^\LQG(s)$ to $\rho^\LQG(t)$.  In particular, this interval has non-empty interior.  We can likewise find a pair of points so that the leftmost geodesics to $x$ do not merge before passing through the clockwise segment of $\partial \fb{x}{\oqdist(x,y)-r}$ from $\rho^\LQG(s)$ to $\rho^\LQG(t)$.  In particular, this interval also has non-empty interior.  This implies that $\rho^\LQG(s) \neq \rho^\LQG(t)$ so that $(s,t) \notin K^\LQG$ as desired.
\end{proof}

\begin{proof}[Proof of Proposition~\ref{prop::metric_net_law}]
As explained above, this follows by combining Lemmas~\ref{lem:levy_net_jump_equiv_dense} and~\ref{lem:net_equiv_implies_levy_equiv}.
\end{proof}

\subsection{Proof of Theorem~\ref{thm::tbm_lqg} and Corollary~\ref{cor::disk_plane_equivalences}}
\label{subsec::tbm_lqg}

\begin{proof}[Proof of Theorem~\ref{thm::tbm_lqg}]
Proposition~\ref{prop::metric_net_law} implies that, in a $\sqrt{8/3}$-LQG sphere sampled from $\Mstwo$, we have that the metric net from $x$ to $y$ has the same law as in the $3/2$-stable L\'evy net.  \cite[Proposition~5.11]{dms2014mating} implies that the construction is invariant under the operation of resampling $x$ and $y$ independently from the quantum area measure.  We also have the conditional independence of the unexplored region going in the forward direction from the construction of $\QLE(8/3,0)$ given in \cite{qlebm}.  Therefore all of the hypotheses of Theorem~\ref{thm::levynetbasedcharacterization} are satisfied, hence our metric measure space is a.s.\ isometric to TBM.  If we condition on the total mass of the surface being equal to $1$ then the resulting metric measure space is isometric to the standard unit-area Brownian map measure.
\end{proof}

\begin{proof}[Proof of Corollary~\ref{cor::disk_plane_equivalences}]
As explained just after the statement of Corollary~\ref{cor::disk_plane_equivalences}, this immediately follows from Theorem~\ref{thm::tbm_lqg}. 
\end{proof}

\section{Open problems}
\label{sec::open_problems}

We now state a number of open problems which are related to the present work.

\begin{problem}
\label{prob::ball_dimension_boundary}
Compute the Hausdorff dimension of the outer boundary of a $\QLE(8/3,0)$ process, stopped at a deterministic time $r$. In other words, consider the outer boundary of a $\oqdist$ metric ball of radius $r$, interpret this as a random closed subset of the Euclidean sphere or plane, and compute its (Euclidean) Hausdorff dimension.
\end{problem}

To begin to think about this problem, suppose that $z$ is chosen from the boundary measure on a filled metric ball boundary.  What does that surface look like locally near $z$?  We understand that the ``outside'' of the filled metric ball near $z$ should look locally like a weight $2$ quantum wedge, and that the inside should be an independent random surface---somewhat analogous to a quantum wedge---that corresponds to the local behavior of a filled metric ball at a typical boundary point. If we had some basic results about the interplay between metric, measure, and conformal structure near $z$, such as what sort of (presumably logarithmic) singularity the GFF might have near $z$, this could help us understand the number and size of the Euclidean balls required to cover the boundary.

\noindent{\it Update:} The Hausdorff dimension of the boundary of a metric ball in $\gamma$-LQG was computed in \cite{gwynne2019ballboundary,gps2020ball} (see also the update to Problem~\ref{prob::othergamma} below).  The question of computing the dimension of the outer boundary is still open.

\begin{problem}
\label{prob::geodesic_dimension}
Compute the Hausdorff dimension of a $\sqrt{8/3}$-LQG geodesic (interpreted as a random closed subset of the Euclidean sphere or plane).
\end{problem}

As in the case of a metric ball, we can also consider the local structure near a point $z$ chosen at random from the length measure of a geodesic between some distinct points $a$ and $b$.  This $z$ lies on the boundary of a metric ball (of appropriate radius) centered at $a$, and also on the boundary of a metric ball centered at $z$. These two ball boundaries divide the local picture near $z$ into four pieces, two of which look like independent weight $2$ wedges, and the other two of which look like the surfaces one gets by zooming in near metric ball boundaries. As before, if we knew what type of GFF thick point $z$ corresponded to, this could enable to extract the dimension.

We emphasize that the KPZ formula cannot be applied in the case of either Problem~\ref{prob::ball_dimension_boundary} or Problem~\ref{prob::geodesic_dimension} because in both cases the corresponding random fractal is a.s.\ determined by the underlying quantum surface.  To the best of our knowledge, there are no existing physics predictions for the answers to Problem~\ref{prob::ball_dimension_boundary} and Problem~\ref{prob::geodesic_dimension}.

\begin{problem}
\label{prob::geodesics_removable}
Show that a geodesic between two quantum typical points on a $\sqrt{8/3}$-LQG sphere is a.s.\ conformally removable.
\end{problem}

We note that the coordinate change trick used to prove the removability of the outer boundary of $\QLE$ given in \cite{ms2013qle} does not apply in this particular setting because we do not have an explicit description of the field which describes the quantum surface in a geodesic slice.  Related removability questions include establishing the removability of $\SLE_\kappa$ for $\kappa \in [4,8)$ as well as the entire $\QLE(8/3,0)$ trace (as opposed to just its outer boundary).

{\it Update:} Problem~\ref{prob::geodesics_removable} was solved in \cite{mq2018geodesics}.

In \cite{qle_determined}, we will show that the embedding of TBM into $\sqrt{8/3}$-LQG constructed in this article is a.s.\ determined by the instance of TBM, up to M\"obius transformation.  This implies that TBM comes equipped with a unique \emph{conformal structure}, which in turn implies that we can define Brownian motion on TBM, up to time-change, by taking the inverse image of a Brownian motion on the corresponding $\sqrt{8/3}$-LQG instance under the embedding map.  The existence of the process with the correct time change was constructed in \cite{b2013liouville,grv2013lbm} and some rough estimates of its associated heat kernel have been obtained in \cite{mrvz2014liouville,ak2014liouville}. Following the standard intuition from heat kernel theory, one might guess that the probability that a Brownian motion gets from $x$ to $y$ in some very small $\epsilon$ amount of time should scale with $\epsilon$ in a way that depends on the metric distance between $x$ and $y$ (since any path that gets from $x$ to $y$ in a very short time would probably take roughly the shortest possible path). This leads to the following question (left deliberately vague for now), which could in principle be addressed using the techniques of this paper independently of \cite{qle_determined}.

\begin{problem}
\label{prob::brownian_heat_kernel_estimates}
Relate the heat kernel for Liouville Brownian motion in the case that $\gamma = \sqrt{8/3}$ to the $\QLE(8/3,0)$ metric.
\end{problem}

It has been conjectured that the heat kernel $p_t(x,y)$ should satisfy (for some constants $c_0,c_1 > 0$) the bound
\begin{equation}
\label{eqn::liouville_heat_kernel}
\frac{c_0}{t} \exp\left( - \frac{\oqdist(x,y)^{4/3}}{c_0 t^{1/3}} \right) \leq p_t(x,y) \leq \frac{c_1}{t} \exp\left( - \frac{\oqdist(x,y)^{4/3}}{c_1 t^{1/3}} \right).
\end{equation}
See, for example, the discussion in \cite{fb2009kpz}.

A number of versions of the KPZ relation \cite{kpz-scaling} have been made sense of rigorously in the context of LQG \cite{benjaminischrammkpz, grv-kpz,bjrv-gmt-duality,dms2014mating,shef-renormalization,DS08,rhodes-vargas-log-kpz,ghm2015kpz, arukpz}.  One of the differences between these formulations is how the ``quantum dimension'' of the fractal set is computed.

\begin{problem}
\label{prob::metric_kpz}
Does the KPZ formula hold when one computes Hausdorff dimensions using $\QLE(8/3,0)$ metric balls?
\end{problem}

{\it Update:} The KPZ formula was shown to hold for $\gamma$-LQG metric balls in \cite{gp2019kpz} (see also the update to Problem~\ref{prob::othergamma} below).

In this article, we have constructed the metric space structure for $\sqrt{8/3}$-LQG and have shown that in the case of a quantum sphere, quantum disk, and quantum cone the corresponding metric measure space has the same law as in the case of TBM, the Brownian disk, and the Brownian plane, respectively.  The construction of the metric is a local property of the surface, so we also obtain the metric for any other $\sqrt{8/3}$-LQG surface.  One particular example is the torus.  The natural law on $\sqrt{8/3}$-torii is described in \cite{drvtorus} and the Brownian torus, the scaling limit of certain types of random planar maps, will be constructed in \cite{bettinelli_miermont_torus}.

\begin{problem}
\label{prob::other_topologies}
Show that the $\sqrt{8/3}$-LQG torus of \cite{drvtorus}, endowed with the metric defined by $\QLE(8/3,0)$ using the methods of this paper, agrees in law (as a random metric measure space) with the Brownian torus of \cite{bettinelli_miermont_torus}.
\end{problem}

Finally, a major open problem is to rigorously describe an analog of TBM that corresponds to $\gamma$-LQG with $\gamma \not = \sqrt{8/3}$, to extend the results of this paper to that setting.  A partial step in this direction appears in \cite{distanceexponent}, which shows the existence of a certain distance scaling exponent (but does not compute it explicitly).

\begin{problem}
\label{prob::othergamma}
Construct a metric on $\gamma$-LQG when $\gamma \not = \sqrt{8/3}$. Work out the appropriate dimension and scaling relations (as discussed in Section~\ref{subsubsec::scalingremark}).
\end{problem}

\noindent{\it Update.  The metric for $\gamma \in (0,2)$ was constructed in \cite{dddf2019tightness,dfgps2019weak,gm2019local,gm2019confluence,gm2019uniqueness,gm2019conformalcov}.}

\bibliographystyle{hmralphaabbrv}
\addcontentsline{toc}{section}{References}
\bibliography{sle_kappa_rho}

\bigskip

\filbreak
\begingroup
\small
\parindent=0pt

\bigskip
\vtop{
\hsize=5.3in
Statistical Laboratory, DPMMS\\
University of Cambridge\\
Cambridge, UK}

\bigskip
\vtop{
\hsize=5.3in
Department of Mathematics\\
Massachusetts Institute of Technology\\
Cambridge, MA, USA } \endgroup \filbreak

\end{document}

%% file: acknowledgements.tex
We have benefited from conversations about this work with many people, a partial list of whom includes Omer Angel, Itai Benjamini, Nicolas Curien, Hugo Duminil-Copin, Amir Dembo, Bertrand Duplantier, Ewain Gwynne, Nina Holden, Jean-Fran{\c{c}}ois Le Gall, Gregory Miermont, R\'emi Rhodes, Steffen Rohde, Oded Schramm, Stanislav Smirnov, Xin Sun, Vincent Vargas, Menglu Wang, Samuel Watson, Wendelin Werner, David Wilson, and Hao Wu.

%% file: support_acknowledgements.tex
We would also like to thank the Isaac Newton Institute (INI) for Mathematical Sciences, Cambridge, for its support and hospitality during the program on Random Geometry where part of this work was completed.  J.M.'s work was also partially supported by DMS-1204894 and J.M.\ thanks Institut Henri Poincar\'e for support as a holder of the Poincar\'e chair, during which part of this work was completed.  S.S.'s work was also partially supported by DMS-1209044, DMS-1712862, a fellowship from the Simons Foundation, and EPSRC grants {EP/L018896/1} and {EP/I03372X/1}.